\documentclass[10pt]{amsart}

\usepackage{amscd}
\usepackage{amsmath}
\usepackage{amsfonts}
\usepackage{amssymb}
\usepackage{amsthm}
\usepackage{enumerate}
\usepackage[T1]{fontenc}
\usepackage{mathrsfs}
\usepackage{stackrel}
\usepackage[all, cmtip]{xy}
\usepackage{yhmath}
\usepackage{graphicx}
\usepackage{hyperref}

\DeclareMathOperator{\Sp}{Sp}

\DeclareMathOperator{\coker}{coker}
\DeclareMathOperator{\coim}{Coim}
\DeclareMathOperator{\im}{Im}

\begin{document}
	\newtheorem{thm}{Theorem}[section]
	\newtheorem{lem}[thm]{Lemma}
	\newtheorem{defn}[thm]{Definition}
	\newtheorem{prop}[thm]{Proposition}
	\newtheorem{cor}[thm]{Corollary}
	\newtheorem*{rmk}{Remark}
	\newtheorem*{rmks}{Remarks}
	
	\newcommand{\Zp}{{\mathbb{Z}_p}}
	\newcommand{\Qp}{{\mathbb{Q}_p}}
	\newcommand{\Fp}{{\mathbb{F}_p}}
	\newcommand{\A}{\mathcal{A}}
	\newcommand{\B}{\mathcal{B}}
	\newcommand{\C}{\mathcal{C}}
	\newcommand{\D}{\mathcal{D}}
	\newcommand{\E}{\mathcal{E}}
	\newcommand{\F}{\mathcal{F}}
	\newcommand{\G}{\mathcal{G}}
	\newcommand{\I}{\mathcal{I}}
	\newcommand{\J}{\mathcal{J}}
	\renewcommand{\L}{\mathcal{L}}
	\newcommand{\sL}{\mathscr{L}}
	\newcommand{\M}{\mathcal{M}}
	\newcommand{\sM}{\mathscr{M}}
	\newcommand{\N}{\mathcal{N}}
	\newcommand{\sN}{\mathscr{N}}
	\renewcommand{\O}{\mathcal{O}}
	\newcommand{\cP}{\mathcal{P}}
	\newcommand{\Q}{\mathcal{Q}}
	\newcommand{\R}{\mathcal{R}}
	\newcommand{\cS}{\mathcal{S}}
	\newcommand{\T}{\mathcal{T}}
	\newcommand{\U}{\mathcal{U}}
	\newcommand{\sU}{\mathscr{U}}
	\newcommand{\sV}{\mathscr{V}}
	\newcommand{\V}{\mathcal{V}}
	\newcommand{\W}{\mathcal{W}}
	\newcommand{\X}{\mathcal{X}}
	\newcommand{\Y}{\mathcal{Y}}
	\newcommand{\Z}{\mathcal{Z}}
	\newcommand{\PFS}{\mathbf{PreFS}}
	\newcommand{\h}[1]{\widehat{#1}}
	\newcommand{\hA}{\h{A}}
	\newcommand{\hK}[1]{\h{#1_K}}
	\newcommand{\hsULK}{\hK{\sU(\L)}}
	\newcommand{\hsUFK}{\hK{\sU(\F)}}
	\newcommand{\hsULnK}{\hK{\sU(\pi^n\L)}}
	\newcommand{\hn}[1]{\h{#1_n}}
	\newcommand{\hnK}[1]{\h{#1_{n,K}}}
	\newcommand{\w}[1]{\wideparen{#1}}
	\newcommand{\wK}[1]{\wideparen{#1_K}}
	\newcommand{\invlim}{\varprojlim}
	\newcommand{\dirlim}{\varinjlim}
	\newcommand{\fr}[1]{\mathfrak{{#1}}}
	\newcommand{\LRU}[1]{\sU(\mathscr{#1})}
	\newcommand{\et}{\acute et}
	
	\newcommand{\ts}[1]{\texorpdfstring{$#1$}{}}
	\newcommand{\st}{\mid}
	\newcommand{\be}{\begin{enumerate}[{(}a{)}]}
		\newcommand{\ee}{\end{enumerate}}
	\newcommand{\qmb}[1]{\quad\mbox{#1}\quad}
	\let\le=\leqslant  \let\leq=\leqslant
	\let\ge=\geqslant  \let\geq=\geqslant

	\title[Six operations for $\w{\D}$-modules]{Six operations for $\w{\D}$-modules on rigid analytic spaces}
	\author{Andreas Bode}
	\address{Bergische Universit\"{a}t Wuppertal, Fakult\"{a}t f\"{u}r Mathematik und Naturwissenschaften, Gaussstr. 20, 42119 Wuppertal, Germany}
	\email{abode@uni-wuppertal.de}
	\begin{abstract}
		We introduce all six operations for $\w{\D}$-modules on smooth rigid analytic spaces by considering the derived category of complete bornological $\w{\D}$-modules. We then focus on a full subcategory which should be thought of as consisting of complexes with coadmissible cohomology, and establish analogues of some classical results: Kashiwara's equivalence, stability of coadmissibility under extraordinary inverse image for smooth morphisms and direct image for projective morphisms, as well as the computation of relative de Rham cohomology.
	\end{abstract}
	\maketitle
	
	\tableofcontents
	\section{Introduction}
	The theory of $\D$-modules, modules over the sheaf of differential operators on some smooth complex algebraic variety, forms an important part of geometric representation theory and algebraic geometry more generally. There are two main motivations for this, one representation theoretic and one of a cohomological nature.
	
	Firstly, twisted $\D$-modules on a flag variety are equivalent to the category of associated Lie algebra representations with a fixed central character (Beilinson--Bernstein localization, \cite{BB}).
	
	Secondly, the (derived) category of regular holonomic $\D$-modules allows for a six-functor formalism in the sense of Grothendieck, which gets (almost) identified under the Riemann--Hilbert correspondence with the six-functor formalism for constructible sheaves (see \cite[Theorem 7.1.1, Theorem 7.2.2]{Hotta}).
	
	In particular, it can be useful to consider $\D$-modules as `coefficient objects' for suitable cohomology theories (see e.g. the work of Berthelot \cite{Berthelot}, where $\mathscr{D}^\dag$-modules play the role of `coefficient objects' for rigid cohomology).
	
	This paper is concerned with the corresponding theory on rigid analytic spaces. We fix througout a complete nonarchimedean field $K$ of mixed characteristic $(0, p)$.
	
	In \cite{DcapOne}, Ardakov--Wadsley began the study of coadmissible $\w{\D}$-modules on smooth rigid analytic $K$-varieties, with an analogue of Beilinson--Bernstein localization being proved in \cite{Ardakov}. The theory of six functors for $\w{\D}$-modules however lacked so far a suitable derived categorical framework. While the notion of coadmissibility corresponds naturally to coherence in the algebraic theory, there is currently no well-behaved definition of quasi-coherence (as is not unusual in rigid analytic geometry), making it unclear in which category relevant resolutions should be taken. As a consequence, previous work involving $\w{\D}$-module operations had to restrict itself to cases where the functors are exact (see Kashiwara's equivalence in \cite{DcapTwo}), or where higher derived functors could be computed purely algebraically (e.g. the Cech complex computations in \cite{Bodeproper} or the Ext groups in \cite{DcapThree}).
	
	The main challenge is thus to find a large category of $\w{\D}$-modules which 
	\begin{enumerate}[(i)]
		\item is retaining the analytic flavour of the theory, so that e.g. the functor $\h{\otimes}_\O$ makes sense; and
		\item is well-behaved from a homological standpoint, i.e. allows for the required resolutions to derive functors like $f_*$ or $\h{\otimes}_\O$.
	\end{enumerate}  
	
	In this paper, we present one such category: let $\h{\B}c_K$ denote the quasi-abelian category of complete bornological $K$-vector spaces and let $LH(\h{\B}c_K)$ denote its left heart, i.e. its `abelian envelope'. For any smooth rigid analytic $K$-variety, we then introduce the category $\mathrm{Mod}_{\mathrm{Shv}(X, LH(\h{\B}c_K))}(\w{\D}_X)$ and define analogues of the usual six operations on the (unbounded) derived category 
	\begin{equation*}
		\mathrm{D}(\w{\D}_X):=\mathrm{D}(\mathrm{Mod}_{\mathrm{Shv}(X, LH(\h{\B}c_K))}(\w{\D}_X)).
	\end{equation*}
	In a slight abuse of terminology, we sometimes refer to this derived category as the derived category of complete bornological $\w{\D}_X$-modules. 
	
	While we reserve the study of holonomicity for a follow-up paper, we investigate a subcategory $\mathrm{D}_\C(\w{\D}_X)$ of `$\C$-complexes', which plays the role of $\mathrm{D}^b_\mathrm{coh}(\D)$ in the classical theory. We obtain the following results:
	
	\begin{thm}
		\label{IntroA}
		Let $X$ be a smooth (equidimensional) rigid analytic $K$-variety. The category $\mathrm{Mod}_{\mathrm{Shv}(X, LH(\h{\B}c_K))}(\w{\D}_X)$ is a Grothendieck abelian category admitting flat resolutions.\\
		In particular, there are functors
		\begin{enumerate}[(i)]
			\item $-\widetilde{\otimes}^{\mathbb{L}}_{\O_X}-: \mathrm{D}(\w{\D}_X)\times \mathrm{D}(\w{\D}_X)\to \mathrm{D}(\w{\D}_X)$
			\item $\mathbb{D}= \mathrm{R}\mathcal{H}om_{\w{\D}_X}(-, \w{\D}_X)\widetilde{\otimes}_{\O_X}\Omega_X^{-1}[\mathrm{dim}X]: \mathrm{D}(\w{\D}_X)\to \mathrm{D}(\w{\D}_X)^\mathrm{op}$.
		\end{enumerate}
		For any morphism $f: X\to Y$ of smooth (equidimensional) rigid analytic $K$-varieties, there are functors
		\begin{enumerate}[(i)]
			\setcounter{enumi}{2}
			\item $f_+$, $f_!=\mathbb{D}f_+\mathbb{D}: \mathrm{D}(\w{\D}_X)\to \mathrm{D}(\w{\D}_Y)$
			\item $f^!$, $f^+=\mathbb{D}f^!\mathbb{D}: \mathrm{D}(\w{\D}_Y)\to \mathrm{D}(\w{\D}_X)$. 
		\end{enumerate}
	\end{thm}
	\begin{proof}
		There is an equivalence of categories $LH(\h{\B}c_K)\cong LH(\mathrm{Ind}(\mathrm{Ban}_K))$ by Proposition \ref{dissandL}.(iv). By Corollary \ref{IndBansheafy}, $\mathrm{Ind}(\mathrm{Ban}_K)$ satisfies all the conditions for the general sheaf theory developed in section 3. We can thus invoke Theorem \ref{Shvnice} to deduce the categorical properties of 
		\begin{equation*}
			\mathrm{Mod}_{\mathrm{Shv}(X, LH(\h{\B}c_K))}(\w{\D}_X)\cong LH(\mathrm{Mod}_{\mathrm{Shv}(X, \mathrm{Ind}(\mathrm{Ban}_K))}(\w{\D}_X)).
		\end{equation*}
		The functors are then explicitly defined in section 7.
	\end{proof}
	
	\begin{thm}
		\label{IntroB}
		Let $X$ be a smooth (equidimensional) rigid analytic $K$-variety. Then there is a canonical exact and fully faithful functor
		\begin{equation*}
			(-)^b: \{\text{coadmissible $\w{\D}_X$-modules}\}\to \mathrm{Mod}_{\mathrm{Shv}(X, LH(\h{\B}c_K))}(\w{\D}_X)
		\end{equation*}
		under which the coadmissible tensor product $\w{\otimes}$ from \cite[Lemma 7.3]{DcapOne} agrees with $\widetilde{\otimes}$.
	\end{thm}
	\begin{proof}
		This is Corollary \ref{embedding} and Theorem \ref{Dcapembedding}, together with Proposition \ref{coadandtor} for the result on tensor products.
	\end{proof}
	
	\begin{thm}
		\label{IntroC}
		Let $X$ be a smooth (equidimensional) rigid analytic $K$-variety.
		\begin{enumerate}[(i)]
			\item The category $\mathrm{D}_\C(\w{\D}_X)$ of $\C$-complexes is a triangulated subcategory. A bounded complex is a $\C$-complex if and only if its cohomology groups are coadmissible.
			\item (Kashiwara). Let $i: X\to Y$ be a closed embedding of smooth rigid analytic $K$-varieties. Then $i_+$ and $i^!$ yield an equivalence of categories between $\mathrm{D}_\C(\w{\D}_X)$ and $\mathrm{D}_{\C^X}(\w{\D}_Y)$, where the latter is the full subcategory of $\C$-complexes on $Y$ whose cohomology groups are supported on $X$.
			\item $\C$-complexes are preserved by $f^!$ whenever $f$ is smooth.
		\end{enumerate}
		If $K$ is discretely valued, we have additionally:
		\begin{enumerate}[(i)]
			\setcounter{enumi}{3}
			\item $\C$-complexes are preserved by $f_+$ whenever $f$ is projective.
			\item $\C$-complexes are preserved by $\mathbb{D}$, and $\mathbb{D}^2\cong \mathrm{id}$ on $\C$-complexes.
		\end{enumerate}
	\end{thm}
	\begin{proof}
		(i) is done in Proposition \ref{Ccomplextriangulated} and Corollary \ref{boundedCcomplex}.\\
		(ii) is Theorem \ref{KashiwaraCcomplex}.\\
		(iii) is Theorem \ref{smoothinvimage}.\\
		(iv) is Corollary \ref{KiehlCcomplex}, while (v) is the content of Theorem \ref{dualCcomplex}.
	\end{proof}
	We also verify that our definitions are consistent with the earlier work mentioned above.
	
	We point out that the category $\mathrm{D}(\w{\D}_X)$ is too large to promise good behaviour, and that we only really expect good results when we remain within the realm of $\C$-complexes. For example, it is not clear whether $\mathbb{D}^2\cong \mathrm{id}$ in general.
	
	We make a couple of remarks:
	\begin{enumerate}[(i)]
		\item We prefer the category of bornological vector spaces over that of locally convex topological vector spaces, as the completed tensor product has better properties here: $\h{\B} c_K$ is a closed symmetric monoidal category, whereas this is not true in the topological setting. Bambozzi, Ben-Bassat and Kremnitzer have recently studied bornologies in the context of nonarchimedean geometry, see \cite{Bamdagger}, \cite{BamStein}, and it should be evident that the categorical aspects of this paper are largely influenced by their insights.
		\item There is however one subtlety in this choice: Contrary to the claim in \cite[Lemma 3.53]{Bamdagger}, $\h{\B}c_K$ is in fact not an elementary quasi-abelian category in the sense of Schneiders \cite[Definition 2.1.10]{Schneiders}. As a consequence, the sheaf theory of $\h{\B}c_K$ is very difficult to control: it is not clear that $\mathrm{Shv}(X, \h{\B}c_K)$ is a quasi-abelian category, and the analogue of the usual sheafification functor does not agree with sheafification over $LH(\h{\B}c_K)$ and might not necessarily produce sheaves. To fix this, we embed $\h{\B}c_K$ fully faithfully into $\mathrm{Ind}(\mathrm{Ban}_K)$, which is elementary quasi-abelian and in turns embeds into $LH(\mathrm{Ind}(\mathrm{Ban}_K))\cong LH(\h{\B}c_K)$. It is therefore within this larger framework that we develop our sheaf theory. For instance, $\mathrm{Preshv}(X, \h{\B}c_K)$ can be regarded as a full subcategory of $\mathrm{Preshv}(X, \mathrm{Ind}(\mathrm{Ban}_K))$ and of $\mathrm{Preshv}(X, LH(\h{\B}c_K))$, and we apply sheafification in either of these larger categories. In particular, the sheafification of some presheaf $\F\in \mathrm{Preshv}(X, \h{\B}c_K)$ need not be in $\mathrm{Shv}(X, \h{\B}c_K)$, but only in $\mathrm{Shv}(X, \mathrm{Ind}(\mathrm{Ban}_K))$. This also explains what might have been perceived by the reader as notational quirks in the Theorems above: while it is true that $\mathrm{D}(\C)\cong \mathrm{D}(LH(\C))$ for any quasi-abelian category $\C$, we could not have written `$\mathrm{D}(\mathrm{Mod}_{\mathrm{Shv}(X, \h{\B}c_K)}(\w{\D}_X))$' instead of `$\mathrm{D}(\mathrm{Mod}_{\mathrm{Shv}(X, LH(\h{\B}c_K))}(\w{\D}_X))$'. Similarly, the derived completed tensor product is denoted by $\widetilde{\otimes}^{\mathbb{L}}$ to reflect that it is obtained from the closed symmetric monoidal structure on $LH(\h{\B}c_K)$ (or equivalently, the one on $\mathrm{Ind}(\mathrm{Ban}_K)$ or on $LH(\mathrm{Ind}(\mathrm{Ban}_K))$. In other words, it is given by applying $\mathrm{H}^0(\h{\otimes}^{\mathbb{L}})$ sectionwise and then sheafifying in the sense above (!). 
		\item Analogously to Theorem \ref{IntroB}, one can embed coherent $\O_X$-modules into the category $\mathrm{Mod}_{\mathrm{Shv}(X, LH(\h{\B}c_K))}(\O_X)$, but we are currently not aware of any applications of this fact.
		\item The question of how to reconcile homological algebra and analysis lies also at the heart of Clausen--Scholze's far-reaching work on Condensed Mathematics \cite{Condensed} -- in particular, they provide a categorical framework in which natural analogues of quasi-coherence can be discussed on the derived level. It is natural to ask whether our approach can be linked to or even reformulated in condensed language. While we are not aware of a way to see $\mathrm{D}(\w{\D}_X)$ within the category of condensed $\w{\D}_X$-modules, we would expect that $\mathrm{D}_\C(\w{\D}_X)$ can be recovered, as it should be the result of glueing (in a suitable $\infty$-categorical sense) the natural categories $\mathrm{D}^b_\mathrm{coh}(\D_n)$ along localization functors (see section 8 for precise definitions). 
	\end{enumerate}
	With this categorical framework at hand, the theory of $\w{\D}$-modules on rigid analytic spaces becomes powerful and flexible enough to allow for various constructions which can inform the study of $p$-adic representations via Beilinson--Bernstein localization. For instance, one can now define in complete generality analogues of the intertwining functors introduced in \cite{BBCass}.
	
	We briefly mention possible further generalizations of the approach presented here. The restriction on the ground field $K$ in Theorem \ref{IntroC} is a consequence of us making use of earlier results from \cite{DcapThree} and \cite{Bodeproper}, which were only formulated in the discretely valued case. While the theory of $\w{\D}$-modules was initially only available for $K$ discretely valued, Ardakov has succeeded in generalizing large parts of the framework to complete nonarchimedean fields in \cite{Ardakov} and \cite{Ardakovinduction}. We are confident that in the same spirit, the constraint in Theorem \ref{IntroC} can be lifted. For instance, our discussion of direct images along projective morphisms relies on \cite{Bodeproper}, but it should be possible to generalize this further along the lines of the argument in \cite{Kiehl}, even for proper morphisms.
	
	While we do not discuss it in this paper, we expect most of the results to also hold in the context of equivariant $\w{\D}$-modules with respect to some group action. This would be particularly valuable, as the Beilinson--Bernstein localization of \cite{Ardakov} would then allow us to use the operations discussed here for the study of admissible locally analytic representations of $p$-adic groups (without equivariance, we `only' obtain information about the Lie algebra representations).    
	
	\subsection*{Structure of the paper}
	In section 2, we recall some background on completed enveloping algebras of Lie--Rinehart algebras and generalize earlier results to the non-discretely valued case: we show in Theorem \ref{UcapFS} that the completed enveloping algebra of a smooth Lie--Rinehart algebra $L$ over an affinoid $K$-algebra $A$ is Fr\'echet--Stein in the sense of \cite{ST} (generalizing \cite[Theorem 6.4]{DcapOne} in the discretely-valued case and dropping some additional conditions from \cite[Theorem 4.1.11]{Ardakov}). With Corollary \ref{fflatcompl}, we then generalize \cite[Corollary 3.2]{DcapThree} to the non-discrete case, proving that the map $U_A(L)\to \w{U_A(L)}$ is faithfully flat.
	 
	If $X=\Sp A$ is a smooth affinoid, then $\w{\D}_X(X)$ is given as the completed enveloping algebra $\w{U_A(\T_X(X))}$, making algebras of this type one of the key players in this paper.
	
	In section 3, we provide background on quasi-abelian categories and develop the general theory of sheaves on G-topological spaces valued in a quasi-abelian category. Most of the results are already known or straightforward generalizations from \cite{Schneiders}, but we have decided to give a careful treatment for the convenience of the reader, particularly as in many cases it is not easy to locate a reference of the required generality.
	
	We show that under some mild conditions, our category of sheaves (as well as categories of modules over a sheaf of rings) behaves perfectly fine from a homological perspective: Its derived category is equivalent to the derived category of a Grothendieck abelian category, and we can form the usual derived functors.
	
	In section 4, we specialize to the case where the quasi-abelian category in question is $\h{\B}c_K$, the category of complete bornological vector spaces. As $\h{\B}c_K$ itself does not satisfy the conditions from section 3, we embed it into the category $\mathrm{Ind}(\mathrm{Ban}_K)$, which does. In practice, we usually retain $\h{\B}c_K$ for explicit, analytic arguments and then ensure that our constructions are preserved under the embedding.
	
	Section 5 can be understood as an attempt to integrate the algebras from section 2 into the categorical framework from section 4: We discuss Fr\'echet--Stein algebras and coadmissible modules in $\h{\B}c_K$ and spend quite some time to verify that (at least for all algebras which we are considering) many basic properties of coadmissible modules known in a Fr\'echet context also hold verbatim in the bornological world. For this purpose, we need to introduce the notion of nuclearity relative to a Noetherian Banach algebra, as well as various related notions. This is the most technical part of the paper, the results of which will be used heavily throughout the later sections, especially in sections 8 and 9.
	
	In section 6, we discuss the sheaf $\w{\D}_X$ from a bornological viewpoint. We prove Theorem \ref{IntroB} (see Theorem \ref{Dcapembedding}) and give a locally free resolution of $\O_X$ (the Spencer resolution), which makes the link between $\w{\D}_X$-module theory and the de Rham complex explicit.\\
	In section 7, we define the six operations in strict analogy to the classical picture, proving Theorem \ref{IntroA}.
	
	In section 8, we introduce $\C$-complexes and prove a couple of structure properties (including Theorem \ref{IntroC}.(i)), before we establish the remainder of Theorem \ref{IntroC} in section 9 (see Theorem \ref{KashiwaraCcomplex}, Theorem \ref{smoothinvimage}, Corollary \ref{KiehlCcomplex}, Theorem \ref{dualCcomplex}).  
	
	\subsection*{Acknowledgements} 
	The idea to use bornological methods for $\w{\D}_X$-modules was already in the air when I was working in Oxford in 2018/19, thanks to conversations with Konstantin Ardakov and Kobi Kremnitzer. I am grateful for many helpful discussions with Simon Wadsley, in particular concerning the definition of $\C$-complexes. More broadly, I would like to thank Konstantin Ardakov, Federico Bambozzi, Thomas Bitoun, Guido Bosco, Chirantan Chowdhury, Nicolas Dupr\'e, Chris Lazda, Marc Levine, Vytautas Paskunas, Vincent Pilloni, Tobias Schmidt and Simon Wadsley for their encouragement and interest in this work.
	
	The results of this paper were first presented in the form of a mini-course at the Number Theory Seminar of the ENS, Lyon. I am grateful to the organizers as well as all participants for their very helpful questions and feedback.
	
	A first version of this paper was built upon the erroneous \cite[Lemma 3.53]{Bamdagger} to discuss complete bornological sheaves. I am very grateful to Jack Kelly and Finn Wiersig for pointing out the mistake and their help in ironing out these subtleties.
	
	The author acknowledges support from the ERC Grant HiCoShiVa (Higher coherent cohomology of Shimura Varieties, Grant ID 818856, PI: Pilloni).
	
	\subsection*{Conventions and notation}
	Throughout, $K$ denotes a complete nonarchimedean field of mixed characteristic $(0, p)$ with valuation ring $R$ and residue field $k$. We fix an element $\pi\in R$ with $0<|\pi|<1$.
	
	All our rigid analytic spaces are assumed to be quasi-separated and quasi-paracompact (cf \cite[Definition 8.2/12]{Boschlectures}). For simplicity, we also assume all our smooth rigid analytic spaces to be equidimensional.
	
	Given a normed $K$-vector space $V$, we denote by $V^\circ$ its unit ball. Given an $R$-module $M$, we sometimes write $M_K$ instead of $M\otimes_R K$. 
	
	We ignore all set-theoretic issues by working tacitly within a suitable Gro\-then\-dieck universe.

	\section{Algebraic background: Completed enveloping algebras}
	
	\subsection{Fr\'echet--Stein algebras and coadmissible modules}
	
	We begin by recalling the theory of coadmissible modules over Fr\'echet--Stein algebras as developed in \cite{ST}. We then discuss the example of completed enveloping algebras for smooth Lie--Rinehart algebras, generalizing some results from \cite{Bode1}, \cite{DcapOne} and \cite{Ardakov}.
	
	\begin{defn}
		\label{FSandcoaddef}
		A (two-sided) Fr\'echet $K$-algebra $A$ is called \textbf{Fr\'echet--Stein} if $A\cong \varprojlim A_n$, where each $A_n$ is a (two-sided) Noetherian Banach $K$-algebra, with each connecting morphism $A_{n+1}\to A_n$ flat on both sides with dense image.
		
		A left $A$-module $M$ over a Fr\'echet--Stein algebra $A\cong \varprojlim A_n$ is called \textbf{coadmissible} if $M\cong \varprojlim M_n$, where $M_n$ is a finitely generated left $A_n$-module, and the natural morphism $A_n\otimes_{A_{n+1}}M_{n+1}\to M_n$ is an isomorphism for each $n$.
	\end{defn}
	
	We denote the category of coadmissible $A$-modules by $\C_A$. It has the following properties:
	
	\begin{prop}
		\label{coadprops}
		Let $A$ be a Fr\'echet--Stein $K$-algebra.
		\begin{enumerate}[(i)]
			\item The category $\C_A$ is an abelian category, containing all finitely presented $A$-modules.
			\item Any coadmissible $A$-module can be equipped with a canonical Fr\'echet topology such that any $A$-module morphism between coadmissible $A$-modules is continuous with closed image.
			\item Closed submodules of coadmissible modules are coadmissible, quotients of coadmissible modules by closed submodules are coadmissible.
			\item If $M\cong \varprojlim M_n$ is a coadmissible $A$-module, then the natural morphism 
			\begin{equation*}
				A_n\otimes_AM\to M_n
			\end{equation*}
			is an isomorphism.
			\item If $M\cong \varprojlim M_n$ is a coadmissible $A$-module, then $\varprojlim^{(1)}M_n=0$, i.e. the augmented Roos complex
			\begin{equation*}
				0\to M\to \prod M_n\to \prod M_n\to 0
			\end{equation*}
			is exact.
		\end{enumerate}
	\end{prop}
	\begin{proof}
		Compare \cite[section 3]{ST} -- more concretely, this is \cite[Corollary 3.5, Corollary 3.4.(v)]{ST}, \cite[remark after Lemma 3.6]{ST}, \cite[Lemma 3.6]{ST}, \cite[Corollary 3.1]{ST} and \cite[Theorem B]{ST}, respectively.
	\end{proof}
	
	While the category $\C_A$ is abelian, it is too small to afford a good derived theory due to the absence of projective or injective resolutions (see e.g. \cite[Theorem 3.10]{Zabradi}). Our goal in the following chapters will therefore be to embed $\C_A$ into a suitable category of `complete' $A$-modules and study its derived category.
	
	In this subsection, we discuss the class of Fr\'echet--Stein algebras which will be the main focus of this paper: completed enveloping algebras of smooth Lie--Rinehart algebras.
	
	Let $A$ be an affinoid $K$-algebra. Recall that an \textbf{admissible affine formal model} for $A$ is an $R$-subalgebra $\A\subseteq A$ which is topologically of finite presentation and satisfies $\A\otimes_R K\cong A$ via the natural morphism. Note that any admissible affine formal model can be obtained as the image $\theta(R\langle x_1, \hdots, x_n\rangle)$ for some surjection $\theta: K\langle x_1, \hdots, x_n\rangle \to A$. In particular, admissible affine formal models are open in $A$ and contained in the unit ball for the corresponding residue norm on $A$.
	
	We now fix an affinoid $K$-algebra $A$ with admissible affine formal model $\A$.
	
	\begin{defn}
		A \textbf{smooth $(K,A)$-Lie--Rinehart algebra} is a finitely generated projective $A$-module equipped with a $K$-linear Lie bracket and a map of Lie algebras $\rho: L\to \mathrm{Der}_K(A)$ satisfying
		\begin{equation*}
			[\partial, a\cdot \partial']=\rho(\partial)(a)\cdot \partial'+a\cdot [\partial, \partial']
		\end{equation*}
		for all $\partial, \partial'\in L$, $a\in A$.
		
		An \textbf{$(R, \A)$-Lie lattice} inside $L$ is a finitely generated $\A$-submodule $\L$ of $L$ which spans $L$, is closed under the Lie bracket, and satisfies $\rho(\L)(\A)\subseteq \A$.
	\end{defn} 
	
	For example, if $A=K$, a smooth $(K, K)$-Lie--Rinehart algebra is simply a finite-dimensional Lie algebra over $K$.
	
	Given a smooth $(K, A)$-Lie--Rinehart algebra $L$ with $(R, \A)$-Lie lattice $\L$, we can form the universal enveloping algebra $U_\A(\L)$ as in \cite{Rinehart}, analogously to the universal enveloping algebra of a Lie algebra. It is a ring generated by $\A$ and $\L$ such that
	\begin{enumerate}[(i)]
		\item multiplying an element of $\L$ on the left by an element of $\A$ is given by the $\A$-module structure on $\L$.
		\item the commutator of two elements in $\L$ is given by the Lie bracket.
		\item if $a\in \A$, $\partial\in \L$, then
		\begin{equation*}
			\partial\cdot a-a\cdot \partial=\rho(\partial)(a).
		\end{equation*}
	\end{enumerate}
	Note that if $\L$ is an $(R, \A)$-Lie lattice, then so is $\pi^n\L$ for any $n\geq 0$. We can thus define the \textbf{completed enveloping algebra}
	\begin{equation*}
		\w{U_A(L)}:=\varprojlim \left(\h{U_\A(\pi^n\L)}\otimes_RK\right).
	\end{equation*}
	This does not depend on the choice of Lie lattice (see \cite[Lemma 6.2]{DcapOne}), nor on the choice of admissible affine formal model (see \cite[Proposition 6.2]{DcapOne}).
	
	We claim that $\w{U_A(L)}$ is a Fr\'echet--Stein algebra. In the case when $K$ is a discretely valued field, this was done in \cite[Theorem 3.5]{Bode1} building on \cite[Theorem 6.4]{DcapOne}, while Ardakov \cite{Ardakov} gave a proof for arbitrary $K$ as long as $L$ is free as an $A$-module. Here, we adapt the proof from \cite{Bode1} to our more general setting.
	
	Fix an $(R, \A)$-Lie lattice $\L$ and let $\U_n$ denote the image of $U_\A(\pi^n\L)$ in $U_A(L)\cong U_\A(\pi^n\L)\otimes_R K$.
	
	Note that $U_{\A}(\pi^n\L)$ carries a natural filtration, and we equip $\U_n$ with the corresponding quotient filtration. The PBW theorem (\cite[Theorem 3.1]{Rinehart}) yields a surjective ring map $\mathrm{Sym}_{\A}\pi^n\L\to \mathrm{gr}(U_{\A}(\pi^n\L))$, which is even an isomorphism if $\L$ is a projective $\A$-module. It follows that each filtered piece $\mathrm{Fil}_i\U_n$ is a finitely generated $\A$-module.
	
	\begin{defn}
		Suppose that $K$ is not discretely valued.
		
		An $R$-algebra $\B$ is called \textbf{almost left Noetherian} if for any left ideal $I\subseteq \B$ and any $\epsilon\in \mathfrak{m}$, there exists a finitely generated left ideal $J\subseteq I$ such that $\epsilon I\subseteq J$.
	\end{defn}
	
	We define almost right Noetherian rings analogously and use the term `almost Noetherian' to mean `almost left Noetherian and almost right Noetherian'.
	
	We remark that if $K$ is discretely valued, then $\A$ is Noetherian (as it is a quotient of $R\langle x_1, \hdots, x_r\rangle$), and if $K$ is not discretely valued, then $\A$ is still almost Noetherian by \cite[Satz 5.1]{Kiehl}.
	
	\begin{lem}
		\label{Unanoeth}
		\leavevmode
		\begin{enumerate}[(i)]
			\item If $K$ is discretely valued, then $\h{U_{\A}(\pi^n\L)}$ is Noetherian for all $n\geq 0$. If $K$ is not discretely valued, then $\h{U_{\A}(\pi^n\L)}$ is almost Noetherian for $n\geq 1$. 
			\item The algebra $\h{U_\A(\pi^n\L)}\otimes_RK$ is a Noetherian Banach $K$-algebra for any $n\geq 1$.
			\item The natural surjection $U_\A(\pi^n\L)\to \U_n$ induces an isomorphism 
			\begin{equation*}
				\h{U_\A(\pi^n\L)}\otimes_RK\to \h{\U_n}\otimes_R K
			\end{equation*}
			for any $n\geq 1$.
		\end{enumerate}
	\end{lem}
	
	\begin{proof}
		In the discretely valued case, this was done in \cite{DcapOne}: $U_{\A}(\pi^n\L)$ is Noetherian by the PBW theorem, so its $\pi$-adic completion is also Noetherian. Hence $\h{U_{\A}(\pi^n\L)}\otimes_RK$ is Noetherian Banach, and (iii) was proved in \cite[Lemma 2.5]{DcapOne}.
		
		If $n\geq 1$, then $[\pi^n\L, \pi^n\L]\subseteq \pi \cdot \pi^n\L$ and $\rho(\pi^n\L)(\A)\subseteq \pi \A$, so that
		\begin{equation*}
			\frac{\h{U_\A(\pi^n\L)}}{\pi \h{U_\A(\pi^n\L)}}\cong \frac{U_\A(\pi^n\L)}{\pi U_\A(\pi^n\L)}
		\end{equation*}
		is a \emph{commutative} $R/\pi$-algebra of finite type. Since $R/\pi[x_1, \hdots, x_m]$ is almost Noetherian for all $m$ by \cite[Satz 5.1]{Kiehl}, (i) and (ii) follow from \cite[Proposition 2.10]{Bodegl}.
		
		For (iii), we follow the same argument as in \cite[Lemma 2.5]{DcapOne}: Let $T$ denote the torsion submodule of $U_\A(\pi^n\L)$ so that
		\begin{equation*}
			0\to T\to U_\A(\pi^n\L)\to \U_n\to 0
		\end{equation*}
		is a short exact sequence, where $\U_n$ is $\pi$-torsionfree. In particular, the sequence remains exact after tensoring with $R/\pi^j$ for any $j$, and taking the limit yields a short exact sequence
		\begin{equation*}
			0\to \h{T}\to \h{U_\A(\pi^n\L)}\to\h{\U_n}\to 0. 
		\end{equation*}
		Note that $T/\pi^2T\cong \h{T}/\pi^2\h{T}$, so by almost Noetherianity, there exist $t_1, \hdots, t_r\in T$ such that $\pi \h{T}\subseteq \sum \h{U_{\A}(\pi^n\L)}\cdot t_i+\pi^2\h{T}$.
		
		By \cite[Lemma 2.8.(i)]{Bodegl}, $\pi\cdot \h{T}$ is then contained in the ideal generated by the images of the $t_i$. Since each $t_i$ is annihilated by some $\pi^{m_i}$, it follows that $\pi^m\h{T}=0$ for any $m\geq 1+\mathrm{max}\{m_i\}$. Therefore, the morphism $\h{U_\A(\pi^n\L)}\to \h{\U_n}$ becomes an isomorphism after tensoring with $K$.
	\end{proof}
	
	Since the connecting maps $\h{\U_{n+1}}\otimes K\to \h{\U_n}\otimes K$ clearly have dense image (each of them contains the enveloping algebra $U_A(L)$), it remains to prove flatness. Following ideas from \cite[section 5.3]{Emerton}, this requires us to understand some ring-theoretic properties of $\h{\U_n}$ and $\h{\U_{n+1}}$ as well as an intermediate ring.
	
	\begin{lem}
		\label{Unptor}
		\leavevmode
		\begin{enumerate}[(i)]
			\item The ring $\U_n$ is $\pi$-adically separated for all sufficiently large $n$.
			\item The graded ring $\mathrm{gr}\U_n$ has bounded $\pi$-torsion for each $n$.
			\item For sufficiently large $n$, $\mathrm{gr}\U_n$ is $\pi$-torsionfree.
		\end{enumerate}
	\end{lem}
	\begin{proof}
		For (i), the proof in \cite[Lemma 3.6, Lemma 3.7]{Bode1} goes through unchanged.
		
		For (ii), note that each graded piece $\mathrm{gr}_i\U_n$ is a finitely generated $\A$-module and therefore $\pi$-adically separated by \cite[Lemma 2.13.(ii)]{Bodegl}. Thus $\mathrm{gr}\U_n$ is $\pi$-adically separated. Since $(\mathrm{gr}\U_n)/\pi \mathrm{gr}\U_n$ is a quotient of $\mathrm{Sym}_{\A}(\pi^n\L)\otimes_RR/\pi$, it is a commutative $R/\pi$-algebra of finite type and hence almost Noetherian. In particular, $\mathrm{gr}\U_n$ has bounded $\pi$-torsion by \cite[Corollary 2.9]{Bodegl}.
		
		For (iii), we can use the same argument as in \cite[Lemma 3.1, Corollary 3.1]{DcapThree}: by (ii), $\mathrm{gr}\U_1$ has bounded $\pi$-torsion, annihilated by $\pi^m$, say. Now consider $n\geq m$ and the natural $R$-module surjection $\mu: \mathrm{gr}\U_1\to \mathrm{gr}\U_{n+1}$ given by multiplication by $\pi^{ni}$ on the $i$ graded piece. If $x\in \mathrm{Fil}_i\U_1$ such that $\overline{x}\in \mathrm{gr}_i\U_1$ is in the kernel of $\mu$, then $\pi^{ni}x\in \mathrm{Fil}_{i-1}\U_{n+1}\subseteq \mathrm{Fil}_{i-1}\U_1$, so $\overline{x}$ is $\pi$-torsion in $\mathrm{gr}\U_1$. Now if $y\in \mathrm{Fil}_i\U_{n+1}$ such that $\overline{y}\in \mathrm{gr}_i\U_{n+1}$ is annihilated by $\pi$, there exists $\overline{x}\in \mathrm{gr}_i\U_1$ such that $\mu(\overline{x})=\overline{y}$. Thus $\pi x\in \mathrm{ker}\mu$, and $x$ is therefore annihilated by $\pi^n=\pi^m\cdot \pi^{n-m}$ by the above. But then 
		\begin{equation*}
			\pi^{ni}x=\pi^{n(i-1)}\pi^nx\in \pi^{n(i-1)}\mathrm{Fil}_{i-1}\U_1\subseteq \mathrm{Fil}_{i-1}\U_{n+1},
		\end{equation*} 
		and hence $\overline{y}=0$.
	\end{proof}
	
	We now choose $n$ large enough so that $\U_n$ is $\pi$-adically separated and $\mathrm{gr}\U_n$ is $\pi$-torsionfree. In particular, $\U_{n+1}$ is also $\pi$-adically separated and we have an inclusion $\U_n\subseteq U_A(L)=\U_{n+1}\otimes_RK\subseteq \h{\U_{n+1}}\otimes_RK$.
	
	We then denote by $\V_n\subseteq \h{\U_{n+1}}\otimes_RK$ the $R$-submodule $\V_n:=\h{\U_{n+1}}+\U_n$.
	
	\begin{lem}
		\label{Vnsep}
		Assume that $\U_n$ is $\pi$-adically separated and $\mathrm{gr}\U_n$ $\pi$-torsionfree.
		\begin{enumerate}[(i)]
			\item For any natural number $k$, we have
			\begin{equation*}
				\h{\U_{n+1}}=\mathrm{Fil}_k\U_{n+1}+\h{\U_{n+1}}\cdot(\pi^{n+1}\L)^{k+1}.
			\end{equation*}
			\item For any $x\in \V_n$ and any $r\in \mathbb{N}$, there exists a natural number $k$ such that 
			\begin{equation*}
				x\in \mathrm{Fil}_k\U_n+ \pi^r\h{\U_{n+1}}\cdot (\pi^{n+1}\L)^{k+1}.
			\end{equation*}
			\item For any $r$, we have $\V_n/\pi^r\V_n\cong \U_n/\pi^r\U_n$. In particular, $\h{\V_n}\cong \h{\U_n}$.
			\item $\V_n\subseteq \h{\U_{n+1}}\otimes_RK$ is a $\pi$-adically separated subring.
		\end{enumerate}
	\end{lem}
	\begin{proof}
		For (i), note that the natural $R$-linear map $\mathrm{Fil}_k \U_{n+1}\oplus \U_{n+1}\cdot (\pi^{n+1}\L)^{k+1}\to \U_{n+1}$ is surjective by construction of $U_{\A}(\pi^{n+1}\L)$. Since $\U_{n+1}$ is $\pi$-torsionfree, the corresponding morphism between the $\pi$-adic completions
		\begin{equation*}
			\left(\mathrm{Fil}_k\U_{n+1}\oplus \U_{n+1}\cdot (\pi^{n+1}\L)^{k+1}\right)\ \h{}\to \h{\U_{n+1}}
		\end{equation*} 
		is also surjective. But now $\mathrm{Fil}_k\U_{n+1}$ is a finitely generated $\A$-module, hence already complete by \cite[Lemma 2.13.(iii)]{Bodegl}, and the $\pi$-adic completion of 
		\begin{equation*}
			\U_{n+1}\cdot (\pi^{n+1}\L)^{k+1}\subseteq \U_{n+1} 
		\end{equation*}
		clearly surjects onto $\h{\U_{n+1}}\cdot (\pi^{n+1}\L)^{k+1}$. Thus we have $\h{\U_{n+1}}=\mathrm{Fil}_k\U_{n+1}+\h{\U_{n+1}}\cdot (\pi^{n+1}\L)^{k+1}$ for any $k$, as required.
		
		For (ii), let $x\in \V_n=\h{\U_{n+1}}+\U_n$, so that we can write $x=a+b$ with $a\in \h{\U_{n+1}}$, $b\in \U_n$. Now there exists some $a_1\in \U_{n+1}$ and $a_2\in \h{\U_{n+1}}$ such that $a=a_1+\pi^ra_2$, since $\U_{n+1}/\pi^r\U_{n+1}\cong \h{\U_{n+1}}/\pi^r\h{\U_{n+1}}$. Choose $k$ such that $a_1\in \mathrm{Fil}_k\U_{n+1}\subseteq \mathrm{Fil}_k\U_n$ and $b\in \mathrm{Fil}_k\U_n$. By (i), we can write $a_2=a_3+a_4$ with $a_3\in \mathrm{Fil}_k\U_{n+1}$ and $a_4\in \h{\U_{n+1}}\cdot (\pi^{n+1}\L)^{k+1}$. Then $x=(a_1+\pi^ra_3+b)+\pi^ra_4\in \mathrm{Fil}_k\U_n+\pi^r\h{\U_{n+1}}\cdot(\pi^{n+1}\L)^{k+1}$, as required. 
		
		For (iii), the natural morphism $\U_n/\pi^r\U_n\to \V_n/\pi^r\V_n$ is surjective by (ii). For injectivity, note that $\pi^r\h{\U_{n+1}}\cap \U_{n+1}=\pi^r\U_{n+1}$ for any $r$, since $\U_{n+1}/\pi^r\U_{n+1}\cong \h{\U_{n+1}}/\pi^r\h{\U_{n+1}}$. Thus $\h{\U_{n+1}}\cap U_A(L)=\U_{n+1}$ and hence $\pi^r(\h{\U_{n+1}}+\U_n)\cap \U_n=\pi^r\U_n$.
		
		For (iv), we need to show that $\h{\U_{n+1}}\cdot \U_n$ and $\U_n\cdot \h{\U_{n+1}}$ are contained in $\h{\U_{n+1}}+\U_n$. Let $a\in \h{\U_{n+1}}$ and $b\in \U_n$. Note that there exists some $i$ such that $\pi^ib\in \U_{n+1}$ -- for instance, if $b\in \mathrm{Fil}_i \U_n$, then $\pi^ib\in \U_{n+1}$. Moreover, we can find $a'\in \pi^i \h{\U_{n+1}}$ such that $a-a'\in \U_{n+1}$. Writing $a\cdot b=(a-a')b+a'b$ and $b\cdot a=b(a-a')+ba'$ shows that $\V_n$ is indeed a subring.
		
		It remains to show that $\V_n$ is $\pi$-adically separated. Let $x\in \cap \pi^m\V_n$. Let $r$ be a natural number. By (ii), there exists $k$ such that we can write $x=a+b$ with $a\in \mathrm{Fil}_k\U_n$ and $b\in \pi^r\h{\U_{n+1}}\cdot(\pi^{n+1}\L)^{k+1}$.
		
		Note that $b\in\pi^r\pi^{k+1}\h{\U_{n+1}}\cdot(\pi^n\L)^{k+1}\subseteq \pi^{r+k+1}\V_n$. But $x$ is also in $\pi^{r+k+1}\V_n$ by assumption, so $a\in \mathrm{Fil}_k\U_n\cap \pi^{r+k+1}\V_n$. As $\U_n/\pi^{r+k+1}\U_n\cong \V_n/\pi^{r+k+1}\V_n$, it follows that $a\in \mathrm{Fil}_k\U_n\cap \pi^{r+k+1}\U_n$. But since $\mathrm{gr}\U_n$ is $\pi$-torsionfree, this forces $a\in \pi^{r+k+1}\mathrm{Fil}_k\U_n\subseteq \pi^r \h{\U_{n+1}}$.
		
		Thus $x=a+b\in \pi^r\h{\U_{n+1}}$. But $r$ was arbitrary and $\h{\U_{n+1}}$ is $\pi$-adically separated, so $x=0$.   
	\end{proof}
	
	\begin{thm}
		\label{UcapFS}
		For all sufficiently large $n$, the morphism $\h{\U_{n+1}}\otimes_RK\to \h{\U_n}\otimes_R K$ is flat on both sides. In particular, $\w{U_A(L)}$ is a Fr\'echet--Stein algebra.
	\end{thm}
	\begin{proof}
		We can assume without loss of generality that $\U_n$ is $\pi$-adically separated and $\mathrm{gr}\U_n$ is $\pi$-torsionfree. In this case, we check that $\V_n=\h{\U_{n+1}}+\U_n\subseteq \h{\U_{n+1}}\otimes_RK$ satisfies the conditions in \cite[Proposition 2.16]{Bodegl}:
		
		By Lemma \ref{Vnsep}.(iv) and (iii)¸, $\V_n$ is a $\pi$-adically separated subring of $\U_{n+1}\otimes_RK$ such that $\V_n/\pi\V_n\cong \U_n/\pi\U_n$ is almost Noetherian. It contains $\h{\U_{n+1}}$, which in turn is an almost Noetherian, $\pi$-torsionfree $R$-algebra by Lemma \ref{Unanoeth}(i). 
		
		We can thus apply \cite[Proposition 2.16]{Bodegl} to deduce that 
		\begin{equation*}
			\V_n\otimes_RK=\h{\U_{n+1}}\otimes_RK\to \h{\V_n}\otimes_RK
		\end{equation*}
		is flat on both sides. As $\h{\V_n}\cong \h{\U_n}$ by Lemma \ref{Vnsep}.(iii), the result follows.
	\end{proof}
	
	\subsection{Faithful flatness}
	
	Let $A$ be an affinoid $K$-algebra and let $L$ be a smooth $(K, A)$-Lie--Rinehart algebra. Write $\A\subseteq A$ for an admissible affine formal model of $A$. Let $\L\subseteq L$ be an $(R, \A)$-Lie lattice. Throughout, we write $U=U_A(L)$, $\w{U}=\w{U_A(L)}$, $\U_n$ is the image of $U_\A(\pi^n\L)$ in $U$, and $U_n=\h{\U_n}\otimes_R K$. 
	
	By Lemma \ref{Unptor}, we can assume that $\mathrm{gr}\U_n$ is $\pi$-torsionfree for all $n$, after possibly replacing $\L$ by some rescaled lattice $\pi^{n_0}\L$. In particular $\U_n$ is a \textbf{deformable} $R$-algebra in the sense of \cite[Definition 3.5]{AW}.

	We recall the following result:
	\begin{lem}[{\cite[Lemma 4.14]{Bode1}}]
		\label{discreteflatness}
		Assume that $K$ is discretely valued. Then the natural map $U\to \w{U}$ is flat on both sides.
	\end{lem}
	\begin{proof}
		By assumption and \cite[Theorem 3.1]{Rinehart}, $\U_n$ is Noetherian for any $n\geq 0$. As $\pi$ is central, \cite[3.2.3.(iv)]{Berthelot} implies that the $\pi$-adic completion $\h{\U_n}$ is flat over $\U_n$, and hence $U_n=\h{\U_n}\otimes_R K$ is flat over $U$.
		
		Now if $0\to M_1\to M_2\to M_3\to 0$ is a short exact sequence of finitely generated $U$-modules, then
		\begin{equation*}
			0\to U_n\otimes_U M_1\to U_n\otimes_U M_2\to U_n\otimes_U M_3\to 0
		\end{equation*}
		is exact by the above. Since $M_i$ is finitely presented over $U$ for each $i$, $\w{U}\otimes_UM_i$ is finitely presented over $\w{U}$ and thus coadmissible. In particular, $\w{U}\otimes_U M_i\cong \varprojlim U_n\otimes_UM_i$, so that the exactness of
		\begin{equation*}
			0\to \w{U}\otimes_U M_1\to \w{U}\otimes_U M_2\to \w{U}\otimes_U M_3\to 0
		\end{equation*}
		follows from the above and Mittag-Leffler for coadmissible $\w{U}$-modules (\cite[Theorem B]{ST}).
	\end{proof}
	It was shown in \cite[Corollary 3.2]{DcapThree} that $\w{U}$ is even faithfully flat if $K$ is discretely valued.
	
	Our goal in this subsection is to prove the analogous statement for non-discretely valued $K$, when we can no longer rely on $\U_n$ to be Noetherian.
	
	
	
	\begin{lem}
		\label{gradedNakayama}
		Let $\M=\oplus_i\mathrm{gr}_i\M$ be an $\mathbb{N}$-graded $\mathrm{gr}\U_n$-module such that each graded piece $\mathrm{gr}_i\M$ is a finitely generated $\A$-module. If $\M\otimes_R k$ is finitely generated over $(\mathrm{gr}\U_n)\otimes_R k$, then $\M$ is finitely generated over $\mathrm{gr}\U_n$.
	\end{lem}
	\begin{proof}
		A similar argument can be found in the proof of \cite[Lemma 053C]{stacksproj}.
		Let $m_1, \hdots, m_l\in \M$ be homogeneous elements whose images generate $\M\otimes_R k$. Suppose the degree of $m_j$ is $e_j$, say. We then have a morphism of graded $\mathrm{gr}\U_n$-modules
		\begin{equation*}
			\theta: \oplus_{j=1}^l\mathrm{gr}\U_n(-e_j)\to \M,
		\end{equation*}
		where $\mathrm{gr}\U_n(-e_j)$ is $\mathrm{gr}\U_n$ with the gradation shifted by $e_j$. Thus $\theta=\oplus \theta_i$, where the $i$th graded piece $\theta_i$ is given by
		\begin{equation*}
			\theta_i: \oplus_{j=1}^l\mathrm{gr}_{i-e_j}\U_n\to \mathrm{gr}_i\M.
		\end{equation*}
		For each $i$, $\theta_i\otimes_R k$ is surjective, and $\mathrm{gr}_i\M$ is a finitely generated $\A$-module. Thus Nakayama's lemma implies that $\theta_i$ is surjective for each $i$. Therefore, $\theta$ is surjective and $\M$ is finitely generated.
	\end{proof}
	
	Recall from \cite[Chapter I, Definition 5.1, Theorem 5.7]{Zariskian} that a filtered $\U_n$-module $\M$ is said to carry a \textbf{good filtration} if $\mathrm{gr}\M$ is finitely generated over $\mathrm{gr}\U_n$.
	
	Let $\M$ be a finitely generated $\pi$-torsionfree $\U_0$-module. Pick a surjection
	\begin{equation*}
		\U_0^r\to \M,
	\end{equation*}
	inducing a good filtration on $\M$ as the quotient filtration. As in \cite[subsection 3.1]{DcapThree}, we define the \textbf{$n$th deformation} $\M_n\subseteq \M$ to be the image of $\U_n^r$ inside $\M$, equipped with the corresponding quotient filtration:
	\begin{equation*}
		\begin{xy}
			\xymatrix{
				\U_n^r\ar[r]\ar[d] & \M_n\ar[d]\\
				\U_0^r\ar[r] & \M.
			}
		\end{xy}
	\end{equation*}
	By construction, $\M_n$ is a finitely generated $\pi$-torsionfree $\U_n$-module, endowed with a filtration satisfying
	\begin{equation*}
		\mathrm{Fil}_i\M_n=\sum_{j\leq i} \pi^{nj}\mathrm{Fil}_j\M.
	\end{equation*}
	Note that this is again a good filtration, and each filtered piece is a finitely generated $\A$-module.
	
	\begin{lem}
		\label{gradedtorsion}
		Let $\M$ be a finitely generated $\pi$-torsionfree $\U_0$-module. 
		\begin{enumerate}[(i)]
			\item $\mathrm{gr}\M$ has bounded $\pi$-torsion.
			\item There exists $n_0$ such that $\mathrm{gr}\M_n$ is $\pi$-torsionfree for all $n\geq n_0$.
		\end{enumerate}
	\end{lem}
	\begin{proof}
		Let $T\subseteq \mathrm{gr}\M$ be the torsion submodule of $\mathrm{gr}\M$, so that we have a short exact sequence of graded $\mathrm{gr}\U_0$-modules
		\begin{equation*}
			0\to T\to \mathrm{gr}\M\to S\to 0,
		\end{equation*}
		where $S$ is $\pi$-torsionfree.
		
		Since $\mathrm{gr}_i\M$ is a finitely generated $\A$-module, so is the graded piece $\mathrm{gr}_iS$ for each $i$. As $\mathrm{gr}_iS$ is $\pi$-torsionfree, \cite[Proposition 1.9.14]{Abbes} implies that $\mathrm{gr}_i T$ is a finitely generated $\A$-module.
		
		Now $S$ is $\pi$-torsionfree and therefore flat over $R$, so $T\otimes_R k$ injects into $(\mathrm{gr}\M)\otimes_R k$. As $(\mathrm{gr}\U_0)\otimes_R k$ is Noetherian, this shows that $T\otimes_R k$ is finitely generated, and we can invoke Lemma \ref{gradedNakayama} to deduce that $T$ is finitely generated. In particular $\mathrm{gr}\M$ has bounded $\pi$-torsion. We have thus proven (i).
		
		For (ii), the discussion of \cite[Corollary 3.1]{DcapThree} goes through unchanged, compare also Lemma \ref{Unptor}.(iii).
	\end{proof}

	\begin{lem}
		\label{fpfortf}
		Let $\M_n$ be a finitely generated $\pi$-torsionfree $\U_n$-module equipped with a good filtration such that $\mathrm{gr}\M_n$ is $\pi$-torsionfree. Then $\M_n$ is finitely presented. 
	\end{lem} 
	\begin{proof}
		Let $\U_n^r\to \M_n$ be a surjection inducing a good filtration on $\M_n$ with $\mathrm{gr}\M_n$ $\pi$-torsionfree. Let $J_n\subseteq \U_n^r$ denote the kernel and equip it with the subspace filtration. By construction, we have a short exact sequence
		\begin{equation*}
			0\to \mathrm{Fil}_iJ_n\to \mathrm{Fil}_i\U_n^r\to \mathrm{Fil}_i\M_n\to 0
		\end{equation*}
		for each $i$. Now $\mathrm{Fil}_i\M_n$ is a finitely generated $\A$-module and is $\pi$-torsionfree by assumption, so by \cite[Proposition 1.9.14]{Abbes}, $\mathrm{Fil}_iJ_n$ is a finitely generated $\A$-module for each $i$.
		
		The short exact sequences above yield a short exact sequence
		\begin{equation*}
			0\to \mathrm{gr}J_n\to \mathrm{gr}\U_n^r\to \mathrm{gr}\M_n\to 0
		\end{equation*}
		by \cite[Chapter I, Theorem 4.2.4]{Zariskian}. 
		
		By assumption, $\mathrm{gr}\M_n$ is $\pi$-torsionfree, so flat over $R$, so the natural map
		\begin{equation*}
			(\mathrm{gr}J_n)\otimes_R k\to (\mathrm{gr}\U_n^r)\otimes_R k
		\end{equation*}
		is injective. But $(\mathrm{gr}\U_n)\otimes_Rk$ is a commutative $k$-algebra of finite type, hence Noetherian, so $(\mathrm{gr}J_n)\otimes_R k$ is a finitely generated $(\mathrm{gr}\U_n)\otimes_R k$-module. Since $\mathrm{Fil}_i J_n$ is a finitely generated $\A$-module, so is $\mathrm{gr}_i J_n$ for each $i$, so we can once more conclude from Lemma \ref{gradedNakayama} that $\mathrm{gr}J_n$ is finitely generated. Now we invoke \cite[Chapter I, Theorem 5.7]{Zariskian} to deduce that $J_n$ is finitely generated.
	\end{proof}

	\begin{thm}
		\label{flatcompletion}
		The map $U\to \w{U}$ is flat on both sides.
	\end{thm}
	\begin{proof}
		We only show flatness on the right. Flatness on the left can be proven in the same way.
		
		Let $M$ be a finitely generated left $U$-module and consider a short exact sequence
		\begin{equation*}
			0\to I\to U^r\to M\to 0
		\end{equation*}
		Let $n\geq 0$. Let $\M_n$ denote the image of $\U_n^r$ in $M$, and set $\I_n=\U_n^r\cap I$, so that we have a short exact sequence
		\begin{equation*}
			0\to \I_n\to \U_n^r\to \M_n\to 0.
		\end{equation*}
		By Lemma \ref{gradedtorsion}, if we choose $n$ large enough, $\mathrm{gr}\M_n$ is $\pi$-torsionfree, and then $\I_n$ is finitely generated by Lemma \ref{fpfortf}. Moreover, the subspace filtration on $\I_n$ is a good filtration, as we have seen in the proof of Lemma \ref{fpfortf} that $\mathrm{gr}\I_n$ is finitely generated.
		
		Since $\M_n$ itself is $\pi$-torsionfree, the sequence above remains exact after $\pi$-adic completion and applying $\otimes_RK$, i.e.
		\begin{equation*}
			0\to \h{\I_n}\otimes_RK\to U_n^r\to \h{\M_n}\otimes_RK\to 0
		\end{equation*}
		is an exact sequence of $U_n$-modules. 
		
		We now have the following diagram with exact rows
		\begin{equation*}
			\begin{xy}
				\xymatrix{
					&U_n\otimes_U I\ar[r]\ar[d]& U_n^r\ar[r]\ar[d]& U_n\otimes_U M\ar[r] \ar[d] & 0\\
					0\ar[r] & \h{\I_n}\otimes_RK\ar[r]& U_n^r\ar[r]& \h{\M_n}\otimes_RK\ar[r] & 0
				}
			\end{xy}
		\end{equation*}	
		where the vertical arrow in the middle is the identity map. In particular, we can see that whenever $\N$ is a finitely generated $\pi$-torsionfree $\U_n$-module, then the natural morphism
		\begin{equation*}
			U_n\otimes_U (\N\otimes_RK)\to \h{\N}\otimes_RK
		\end{equation*}
		is surjective.
		
		But now $\I_n$ is also a finitely generated $\pi$-torsionfree $\U_n$-module, so not only is the third vertical arrow in the diagram above surjective, the first vertical arrow is also surjective. This forces an isomorphism
		\begin{equation*}
			U_n\otimes_U M\cong \h{\M_n}\otimes_RK
		\end{equation*}
		for all $n$ with the property that $\mathrm{gr}\M_n$ is $\pi$-torsionfree.
		
		Note that the short exact sequence
		\begin{equation*}
			0\to \mathrm{gr}\I_n\to \mathrm{gr}\U_n^r\to \mathrm{gr}\M_n\to 0
		\end{equation*}
		shows that $\mathrm{gr}\I_n$ is also $\pi$-torsionfree, as we assume $\mathrm{gr}\U_n$ to be $\pi$-torsionfree, so
		\begin{equation*}
			U_n\otimes_U I\cong \h{\I_n}\otimes_RK
		\end{equation*}
		by the same argument as above.
		
		We have thus identified the short exact sequence
		\begin{equation*}
			0\to \h{\I}\otimes_RK\to U_n^r\to \h{\M_n}\otimes_RK\to 0
		\end{equation*}
		with
		\begin{equation*}
			0\to U_n\otimes_U I\to U_n^r\to U_n\otimes_UM\to 0
		\end{equation*}
		for all sufficiently large $n$.

		Now note that the finitely presented module $\w{U}\otimes_U M$ is coadmissible, so that 
		\begin{equation*}
			\w{U}\otimes_UM\cong \varprojlim U_n\otimes_U M,
		\end{equation*}
		similarly for $I$ instead of $M$. As in Lemma \ref{discreteflatness}, taking the limit of the short exact sequences above and invoking \cite[Theorem B]{ST} thus yields a short exact sequence
		\begin{equation*}
			0\to \w{U}\otimes_U I\to \w{U}^r\to \w{U}\otimes_U M\to 0,
		\end{equation*}
		showing that $\w{U}$ is a flat right $U$-module.
	\end{proof}
	
	Following \cite[section 3]{DcapThree}, we can even strengthen this result.
	\begin{cor}
		\label{fflatcompl}
		The map $U\to \w{U}$ is faithfully flat on both sides.
	\end{cor}
	\begin{proof}
		Let $M\neq 0$ be a finitely generated $U$-module. It suffices to prove that $\w{U}\otimes_UM\neq 0$, and hence it suffices to prove that there exists some $n$ with 
		\begin{equation*}
			U_n\otimes_U M\cong U_n\otimes_{\w{U}}\w{U}\otimes_UM\neq 0.
		\end{equation*} 
		
		As before, let $\M\subseteq M$ be a finitely generated $\U_0$-module spanning $M$. Choose a good filtration and let $\M_n\subseteq \M$ denote the $n$th deformation. By Lemma \ref{gradedtorsion}, there exists $n_0$ such that $\mathrm{gr}\M_n$ is $\pi$-torsionfree for all $n\geq n_0$.
		
		By the proof of Theorem \ref{flatcompletion}, if $n\geq n_0$, then
		\begin{equation*}
			U_n\otimes_U M\cong \h{\M_n}\otimes_R K.
		\end{equation*}
		
		We first claim that if $n\geq n_0$, then $\h{\M_n}\neq 0$. Suppose for a contradiction that 
		\begin{equation*}
			\h{\M_n}=\varprojlim \M_n/\pi^r\M_n=0.
		\end{equation*}
		In particular, $\M_n/\pi\M_n=0$, i.e. $\M_n=\pi\M_n$. Since $\mathrm{gr}\M_n$ is $\pi$-torsionfree, this forces $\mathrm{Fil}_i\M_n=\pi \mathrm{Fil}_i\M_n$ for all $i$. But $\mathrm{Fil}_i\M_n$ is a finitely generated $\A$-module, so Nakayama's lemma implies now that $\mathrm{Fil}_i\M_n=0$ for all $i$, hence $\M_n=0$. As $\M_n\otimes_RK=M\neq 0$, this produces the desired contradiction.
		
		Finally, we need to show that $\h{\M_n}$ is not $\pi$-torsion. Note that the same argument as in the proof of Theorem \ref{flatcompletion} shows that 
		\begin{equation*}
			\h{\M_n}\cong \h{\U_n}\otimes_{\U_n}\M_n.
		\end{equation*}
		In particular, $\h{\M_n}$ is a finitely generated $\h{\U_n}$-module, so if it was a torsion module, it would be annihilated by some $\pi^s$.
		
		Since the morphism
		\begin{equation*}
			\varprojlim_j \M_n/\pi^j\M_n\to \M_n/\pi^{s+1}\M_n
		\end{equation*}
		is always surjective, this implies that $\M_n/\pi^{s+1}\M_n$ is annihilated by $\pi^s$, i.e.
		\begin{equation*}
			\pi^s\M_n=\pi^{s+1}\M_n.
		\end{equation*}
		Since $\M_n$ is $\pi$-torsionfree, this yields once again $\M_n=\pi \M_n$, leading to the same contradiction as before. Hence $\h{\M_n}$ is not $\pi$-torsion, and $\h{\M_n}\otimes_RK\neq 0$ as soon as $n\geq n_0$. The result follows. 
	\end{proof}

	\section{Categorical background: Sheaves valued in quasi-abelian categories}
	\subsection{Quasi-abelian categories}
	We begin by recalling the key definitions from Schneiders' framework of quasi-abelian categories \cite{Schneiders}.
	
	Suppose that $\C$ is an additive category admitting kernels and cokernels. Recall that a morphism $f$ in $\C$ is called \textbf{strict} if the natural morphism $\coim f\to \im f$ is an isomorphism. In particular, if $\C$ is abelian, then any morphism in $\C$ is strict.
	
	We use the definition of strictness to formulate the following generalization of the notion of an abelian category.
	\begin{defn}
		An additive category $\C$ admitting kernels and cokernels is called \textbf{quasi-abelian} if 
		\begin{enumerate}[(i)]
			\item for any Cartesian square
			\begin{equation*}
				\begin{xy}
					\xymatrix{
						E'\ar[d] \ar[r]^{f'}&F'\ar[d]\\
						E\ar[r]^f &F
					}
				\end{xy}
			\end{equation*}
			with $f$ a strict epimorphism, $f'$ is also a strict epimorphism.
			\item for any Cocartesian square
			\begin{equation*}
				\begin{xy}
					\xymatrix{
						E\ar[r]^f\ar[d] & F\ar[d]\\
						E'\ar[r]^{f'} & F'
					}
				\end{xy}
			\end{equation*}
			with $f$ a strict monomorphism, $f'$ is also a strict monomorphism.
		\end{enumerate}
	\end{defn}
	A sequence
	\begin{equation*}
		\begin{xy}
			\xymatrix{
				E\ar[r]^u& F\ar[r]^v & G
			}
		\end{xy}
	\end{equation*}
	in $\C$ is called \textbf{strictly exact} (at $F$) if $u$ is strict and $\mathrm{Im}u\cong \mathrm{ker} v$. The class of short strictly exact sequences gives $\C$ the structure of an exact category in the sense of Quillen \cite{Quillen}, and we can form the derived category $\mathrm{D}(\C)$ by taking the quotient of the homotopy category by the full subcategory of strictly exact complexes.
	
	The derived category $\mathrm{D}(\C)$ admits a canonical t-structure called the left t-structure. Its heart $LH(\C)$, called the \textbf{left heart} of $\C$, is an abelian category which can be thought of as an `abelian envelope' of $\C$. We first give an explicit description of $LH(\C)$ before making this statement more precise.
	Denote by $K(\C)$ the homotopy category of $\C$. The left heart $LH(\C)$ is the localization of the full subcategory of $K(\C)$ consisting of complexes 
	\begin{equation*}
		\begin{xy}
			\xymatrix{
				E=(0\ar[r]& E^{-1}\ar[r]^{f_E} &E^0\ar[r] &0)
			}
		\end{xy}
	\end{equation*}
	with $f_E$ a monomorphism, by the multiplicative system consisting of morphisms $u: E\to F$ such that
	\begin{equation*}
		\begin{xy}
			\xymatrix{
				E^{-1}\ar[r]^{f_E} \ar[d]_{u_{-1}}& E^0\ar[d]^{u_0}\\
				F^{-1}\ar[r]^{f_F}& F^0
			}
		\end{xy}
	\end{equation*}
	is a Cartesian and Cocartesian square in $\C$.
	
	The functor $I: \C\to LH(\C)$ sending $M$ to $0\to0\to M\to 0$ identifies $\C$ with the full subcategory consisting of those $E\in LH(\C)$ such that $f_E$ is strict (see \cite[Corollary 1.2.28, Proposition 1.2.29]{Schneiders}).
	
	Moreover, a sequence 
	\begin{equation*}
		0\to E\to F\to G\to 0
	\end{equation*}
	in $\C$ is short strictly exact if and only if it is exact in $LH(\C)$ (\cite[Corollary 1.2.28]{Schneiders}), so that we obtain a natural morphism $\mathrm{D}(\C)\to \mathrm{D}(LH(\C))$, which turns out to be an equivalence of triangulated categories (\cite[Proposition 1.2.32]{Schneiders}).
	
	In summary: while $\C$ might not be abelian, it embeds into an abelian category $LH(\C)$ which has the same derived category.
	
	The embedding $I: \C\to LH(\C)$ admits a left adjoint $C: LH(\C)\to \C$, given by taking the cokernel of any monomorphism representing the object (\cite[Proposition 1.2.27]{Schneiders}).
	
	If $(M^\bullet, d^\bullet)$ is a complex in $\C$, then $\mathrm{H}^j(M^\bullet)\in LH(\C)$ is its $j$th cohomology, viewing $M^\bullet$ as an object in $\mathrm{D}(LH(\C))$. Explicitly, 
	\begin{equation*}
		\mathrm{H}^j(M^\bullet)=(0\to \coim d^{j-1}\to \ker d^j\to 0).
	\end{equation*}
	In particular, $\mathrm{H}^j(M^\bullet)$ is an object of $\C$ if and only if the differential $M^{j-1}\to M^j$ is strict.
	
	We note that if $E\in LH(\mathcal{C})$ is represented by a monomorphism $f_E: E^{-1}\to E^0$, then we have a short exact sequence
	\begin{equation*}
		0\to I(E^{-1})\to I(E^0)\to E\to 0
	\end{equation*}
	in $LH(\mathcal{C})$: as $I$ is a right adjoint, it preserves all monomorphisms, and $E$ is evidently the cokernel of $I(E^{-1})\to I(E^0)$. 
	
	If $f: E^{-1}\to E^0$ is any (not necessarily monomorphic) morphism in $\C$, we sometimes write
	\begin{equation*}
		\begin{xy}
			\xymatrix{[E^{-1}\ar[r]^{f}& E^0]}
		\end{xy}
	\end{equation*}   
	for the object 
	\begin{equation*}
		(0\to \coim f\to E^0\to 0)\in LH(\C).
	\end{equation*}
	So for example, $\mathrm{H}^j(M^\bullet)=[M^{j-1}\to \ker d^j]$ in the example above.
	
	The following example might help to illustrate the notions introduced so far and motivate how they might be relevant in $p$-adic geometry. Consider the Tate algebra $K\langle x\rangle$ in one variable, which is the ring of analytic functions on the closed unit disk $D$. The de Rham cohomology of $D$ is then computed (thanks to the acyclicity of coherent modules on affinoid spaces) as the cohomology of the complex
	\begin{equation*}
		\begin{xy}
			\xymatrix{
				0\ar[r]& K\langle x\rangle \ar[r]^{\mathrm{d}}& K\langle x\rangle \mathrm{d}x\ar[r]& 0,
			}
		\end{xy}
	\end{equation*}
	where $\mathrm{d}(f)=\frac{\mathrm{d}f}{\mathrm{d}x}\mathrm{d}x$.
	
	Thus $\mathrm{H}^0_\mathrm{dR}(D)=K$ as expected, but if we take the cokernel of $\mathrm{d}$ in the category of $K$-vector spaces, we obtain an infinite-dimensional vector space. Even worse, while the terms in the complex can be equipped with natural Banach structures, the set-theoretic image of $\mathrm{d}$ is dense, so that the quotient could only carry the indiscrete topology. 
	
	The complex begins to make somewhat more sense when we are working in the quasi-abelian category of Banach spaces (or, as we will do later, in the category of complete bornological spaces): the fact that $\mathrm{d}$ has dense image without being surjective is precisely saying that $\mathrm{d}$ is an epimorphism in the category of Banach spaces, but \emph{not} a strict epimorphism. Taking cohomology in this setup, we have
	\begin{equation*}
		\mathrm{H}^1_\mathrm{dR}(D)=[K\langle x\rangle \to K\langle x\rangle \mathrm{d}x]\in LH(\mathrm{Ban}_K),
	\end{equation*}
	and while $C\mathrm{H}^1_\mathrm{dR}(D)=0$ as one would like, all the curious, non-strict behaviour is still recorded on the level of the left heart.
	
	We now turn to functors between quasi-abelian categories.
	
	If $F:\C\to \C'$ is a functor between quasi-abelian categories, we say that it is 
	\begin{enumerate}[(i)]
		\item \textbf{left exact} if it sends a strictly exact sequence
		\begin{equation*}
			0\to E\to E'\to E''\to 0
		\end{equation*}
		to a strictly exact sequence
		\begin{equation*}
			0\to F(E)\to F(E')\to F(E'').
		\end{equation*}
		\item \textbf{strongly left exact} if it sends a strictly exact sequence
		\begin{equation*}
			0\to E\to E'\to E''
		\end{equation*}
		to a strictly exact sequence
		\begin{equation*}
			0\to F(E)\to F(E')\to F(E'').
		\end{equation*}
	\end{enumerate}
	Equivalently, $F$ is left exact (resp. strongly left exact) if and only if it preserves the kernel of any strict morphism (resp. any morphism).
	
	The corresponding notions of right exactness are defined dually (in particular, in terms of strictly \emph{co}exact sequences, see \cite[Definition 1.1.17]{Schneiders}).
	
	A functor is then said to be \textbf{(strongly) exact} if it is both (strongly) left exact and (strongly) right exact.
	
	If $F: \C\to \C'$ is a left exact functor between quasi-abelian categories such that $\C$ contains an $F$-injective subcategory (see \cite[Definition 1.3.2]{Schneiders}),  we can define a right derived functor $\mathrm{R}F: \mathrm{D}^+(\C)\to \mathrm{D}^+(\C')$ as usual. Embedding $LH(\C)$ into $\mathrm{D}^+(LH(\C))\cong \mathrm{D}^+(\C)$ then allows us to form the functor $\widetilde{F}=\mathrm{H}^0(\mathrm{R}F): LH(\C)\to LH(\C')$.
	
	If $F$ is in fact strongly left exact, then $\widetilde{F}$ is left exact and right derivable, and the equivalence $\mathrm{D}^+(\C)\cong \mathrm{D}^+(LH(\C'))$ identifies $\mathrm{R}F$ with $\mathrm{R}\widetilde{F}$ (see \cite[Proposition 1.3.14]{Schneiders}).
	
	In the case of right exact functors, we obtain an analogous statement even without requiring strong right exactness (see \cite[Proposition 1.3.15]{Schneiders}). This asymmetry stems from our preference for the left t-structure.
	
	We should stress that for $E\in \C$ and $F$ a right exact functor, $\mathrm{H}^0(\mathbb{L}F(E))\in LH(\C')$ need not be isomorphic to $I(F(E))$. We only know that $C\mathrm{H}^0(\mathbb{L}F(E))\cong F(E)$.
	
	A particularly convenient situation is when $F$ is actually an exact functor. In this case, $F$ obviously lifts to a functor between the derived categories, and the corresponding $\widetilde{F}$ is a functor between the left hearts which is explicitly given by
	\begin{equation*}
		\widetilde{F}(0\to E^{-1}\to E^0\to 0)=[F(E^{-1})\to F(E^0)].
	\end{equation*}
	If $F$ is even strongly exact, then $\widetilde{F}$ is exact, and we have the natural isomorphism $\mathrm{D}(I)\mathrm{D}(F)\cong \mathrm{D}(\widetilde{F})\circ \mathrm{D}(I)$.
	
	We also need the notion of a quasi-elementary quasi-abelian category (see \cite[section 2.1]{Schneiders}).
	
	Let $\C$ be a cocomplete quasi-abelian category.
	
	An object $E$ of $\C$ is called \textbf{small} if the functor $\mathrm{Hom}_\C(E, -)$ commutes with direct sums, and we call $E$ \textbf{tiny} if $\mathrm{Hom}_\C(E, -)$ commutes with all filtered colimits.
	
	A set $\mathcal{G}$ of objects in $\C$ is called a \textbf{strictly generating set} if for any $E\in \C$, there exist $G_i\in \mathcal{G}$, $i\in \mathcal{I}$ for some indexing set $\mathcal{I}$, with a strict epimorphism $\oplus_{\mathcal{I}} G_i\to E$.
	\begin{defn}
		We say that a quasi-abelian category $\C$ is \textbf{quasi-elementary} (resp. \textbf{elementary}) if it is cocomplete and admits a strictly generating set of small (resp. tiny) projectives.
	\end{defn}
	Let $\C$ be a quasi-elementary category. Then \cite[Proposition 2.1.15.c)]{Schneiders} states that $\C$ has enough projectives\footnote{A quasi-abelian category $\C$ is said to have enough projectives if any object in $\C$ admits a strict epimorphism from some projective object.}, and $LH(\C)$ has enough projectives and enough injectives.
	
	By \cite[Proposition 2.1.15]{Schneiders}, $\C$ has exact products and strongly exact direct sums. If $\C$ is even elementary, all filtered colimits are strongly exact (\cite[Proposition 2.1.16]{Schneiders}).
	
	We note in particular that if $\C$ is a quasi-elementary quasi-abelian category, then $LH(\C)$ is a Grothendieck abelian category: by \cite[Proposition 2.1.12]{Schneiders}, $LH(\C)$ is elementary, so it has exact filtered colimits and a generator by definition. 
	
	\subsection{Homotopy limits}
	
	In this paper, we will usually work with unbounded derived categories. We refer to \cite{Spalt} and \cite[chapter 14]{KS} for background on the theory of derived functors in this context. 
	
	Let $\C$ be an abelian category. In order to extend statements about $\mathrm{D}^+(\C)$ to $\mathrm{D}(\C)$, it can be convenient to write chain complexes as a suitable limit of their truncations. In this subsection, we make this precise by recalling the notion of a homotopy limit in $\mathrm{D}(\C)$, which will also be useful in our discussion of coadmissibility in section 8. 
	
	Let $\C$ be an abelian category with the property that the derived category $\mathrm{D}(\C)$ admits countable products. By \cite[Lemma 07D9]{stacksproj}, this is e.g. the case when $\C$ is Grothendieck abelian, or more generally if $\C$ has countable products and enough injectives (\cite[Lemma 0BK6]{stacksproj}).
	
	Let $(M_n^\bullet, f_n)$ be an inverse system in $\mathrm{D}(\C)$, indexed by $\mathbb{N}$. Following \cite[section 08TB]{stacksproj}, we call $M^\bullet\in \mathrm{D}(\C)$ a \textbf{homotopy limit} of $M_n^\bullet$ if there is a distinguished triangle
	\begin{equation*}
		M^\bullet\to \prod M_n^\bullet\to \prod M_n^\bullet,
	\end{equation*}  
	where the morphism $\prod M_n^\bullet\to \prod M_n^\bullet$ is given by $\mathrm{id}_{M_n^\bullet}-f_{n+1}$ on the $n$th factor. We write
	\begin{equation*}
		M^\bullet\cong \mathrm{holim}M_n^\bullet
	\end{equation*}
	in this case. Note that the mapping cone construction ensures that homotopy limits always exist in $\mathrm{D}(\C)$ and are unique up to (non-unique!) isomorphism.
	
	\begin{lem}
		\label{homotopyproductexact}
		Let $\C$ be an abelian category with exact countable products. Let $(M_n^\bullet, f_n)$ be an inverse system in $\mathrm{D}(\C)$. Then 
		\begin{enumerate}[(i)]
			\item $\mathrm{D}(\C)$ has countable products.
			\item $\prod M_n^\bullet$ exists and can be obtained by taking termwise products of any chain complexes representing $M_n^\bullet$. 
			\item We have
			\begin{equation*}
				\mathrm{H}^j(\prod M_n^\bullet)\cong \prod \mathrm{H}^j(M_n^\bullet)
			\end{equation*}
			for all $j$.
		\end{enumerate}
		In particular, the long exact sequence of cohomology of the defining distinguished triangle yields short exact sequences
		\begin{equation*}
			0\to \mathrm{R}^1\varprojlim_n\mathrm{H}^{j-1}(M_n^\bullet)\to \mathrm{H}^j(\mathrm{holim}M_n^\bullet)\to \varprojlim \mathrm{H}^j(M_n^\bullet)\to 0
		\end{equation*}
		in $\C$ for all $j$.
	\end{lem}
	\begin{proof}
		This is \cite[Lemma 07KC]{stacksproj}. The long exact sequence of cohomology can then be written as 
		\begin{equation*}
			\hdots \to \prod\mathrm{H}^{j-1}(M_n^\bullet)\to \mathrm{H}^j(M^\bullet)\to \prod \mathrm{H}^j(M_n^\bullet)\to \prod \mathrm{H}^j(M_n^\bullet)\to \hdots
		\end{equation*}
		in this case, and we can invoke e.g. \cite[Lemma 3.76]{BamStein} to see that 
		\begin{equation*}
			\mathrm{ker}\left(\prod \mathrm{H}^j(M_n^\bullet)\to \prod \mathrm{H}^j(M_n^\bullet)\right)\cong \varprojlim \mathrm{H}^j(M_n^\bullet)
		\end{equation*} 
		and 
		\begin{equation*}
			\mathrm{coker}\left(\prod \mathrm{H}^j(M_n^\bullet)\to \prod \mathrm{H}^j(M_n^\bullet)\right)\cong \mathrm{R}^1\varprojlim \mathrm{H}^j(M_n^\bullet). 
		\end{equation*}
		We thus obtain the desired short exact sequence.
	\end{proof}
	
	We now discuss how the notion of homotopy limit can be used to pass from bounded to unbounded complexes. 
	
	Let $\C$ be an abelian category with countable products and enough injectives, and let $M^\bullet\in \mathrm{Ch}(\C)$ be a (possibly unbounded) chain complex. 
	
	If $\C$ has exact products, then Lemma \ref{homotopyproductexact} ensures that 
	\begin{equation*}
		M^\bullet\cong \mathrm{holim} \tau^{\geq -n}M^\bullet,
	\end{equation*}
	see also \cite[Remark 2.3]{BN}.
	
	If products in $\C$ are not necessarily exact, we can at least say the following:
	
	By \cite[Lemma 070F]{stacksproj}, there exist chain complexes $I_n^\bullet$ together with morphisms
	\begin{equation*}
		\begin{xy}
			\xymatrix{\hdots \ar[r]& \tau^{\geq -2}M^\bullet\ar[r]\ar[d]&\tau^{\geq -1}M^\bullet\ar[d]\\
				\hdots \ar[r]&I_2^\bullet\ar[r]&I_1^\bullet}
		\end{xy}
	\end{equation*}
	such that
	\begin{enumerate}[(i)]
		\item each vertical morphism is a quasi-isomorphism.
		\item each $I_n^\bullet$ is a bounded below complex consisting of injective objects.
		\item the morphisms $I_{n+1}^\bullet\to I_n^\bullet$ are termwise split surjections (in particular, $\mathrm{ker}(I_{n+1}^i\to I_n^i)$ is also injective).
	\end{enumerate}
	
	\begin{lem}
		\label{holimtruncation}
		Let $\C$ be an abelian category with countable products and enough injectives. Let $M^\bullet\in \mathrm{Ch}(\C)$ and let $\tau^{\geq -n}M^\bullet\to I_n^\bullet$ be as above. Then $I^\bullet=\varprojlim I_n^\bullet$ exists in $\mathrm{Ch}(\C)$ and is $K$-injective\footnote{Recall that a chain complex $X^\bullet$ is called $K$-injective (or homotopically injective) if $\mathrm{Hom}_\C^\bullet(Y^\bullet, X^\bullet)$ is acyclic for any acyclic complex $Y^\bullet\in K(\C)$. By the work of Spaltenstein \cite{Spalt}, $K$-injective complexes play the role of injective resolutions in the setting of unbounded complexes. See \cite{Spalt}, \cite[chapter 14]{KS}, \cite[tag 070G]{stacksproj} for details.}. The natural morphism
		\begin{equation*}
			M^\bullet\to I^\bullet
		\end{equation*}
		is a quasi-isomorphism in $\mathrm{D}(\C)$ if and only if $M^\bullet\cong \mathrm{holim}\tau^{\geq -n}M^\bullet$ via the natural morphisms.
	\end{lem}
	\begin{proof}
		This is \cite[Lemma 070M]{stacksproj}.
	\end{proof}
	
	The following result will be used later.
	
	\begin{lem}
		\label{exactcohom}
		Let $F: \E\to \F$ be an additive functor between abelian categories with exact products, admitting a right derived functor $\mathrm{R}F: \mathrm{D}(\E)\to \mathrm{D}(\F)$ between the \emph{unbounded} derived categories. Let $M^\bullet\in \mathrm{D}(\E)$. 
		\begin{enumerate}[(i)]
			\item If $\mathrm{R}F$ commutes with products and $\mathrm{H}^j(M^\bullet)$ is $F$-acyclic for each $j$, then
			\begin{equation*}
				\mathrm{H}^j(\mathrm{R}F(M^\bullet))\cong F(\mathrm{H}^j(M^\bullet))
			\end{equation*}
			for each $j$.
			\item If $M^\bullet$ is bounded below and there exists $i$ such that $\mathrm{R}F(\mathrm{H}^j(M^\bullet))\cong \mathrm{R}^iF(\mathrm{H}^j(M^\bullet))$ for each $j$, then
			\begin{equation*}
				\mathrm{H}^j(\mathrm{R}F(M^\bullet))\cong \mathrm{R}^iF(\mathrm{H}^{j-i}(M^\bullet))
			\end{equation*}
			for each $j$. The boundedness assumption can be dropped if $\mathrm{R}F$ commutes with products.
		\end{enumerate}
	\end{lem}
	\begin{proof}
		\begin{enumerate}[(i)]
			\item If $M^\bullet$ is bounded below, we can represent it by an $F$-acyclic complex, and an easy induction shows that the kernels and images in each degree are also $F$-acyclic. This proves the statement whenever $M^\bullet$ is bounded below. In general, note that $M^\bullet\cong \mathrm{holim}\tau^{\geq -n}M^\bullet$ by Lemma \ref{homotopyproductexact}, \cite[Remark 2.3]{BN}. As $\mathrm{R}F$ commutes with products, we have a distinguished triangle 
			\begin{equation*}
				\mathrm{R}F(M^\bullet)\to \prod \mathrm{R}F(\tau^{\geq -n}M^\bullet)\to \prod \mathrm{R}F(\tau^{\geq -n}M^\bullet).
			\end{equation*}
			Since $\mathrm{H}^j(\mathrm{R}F(\tau^{\geq -n}M^\bullet))\cong F(\mathrm{H}^j(M^\bullet))$ for each $j\geq -n$ by the earlier argument, this inverse system is eventually constant, so that Lemma \ref{homotopyproductexact} implies
			\begin{equation*}
				\mathrm{H}^j(\mathrm{R}F(M^\bullet))\cong \varprojlim \mathrm{H}^j(\mathrm{R}F(\tau^{\geq -n}M^\bullet))\cong F(\mathrm{H}^j(M^\bullet)),
			\end{equation*} 
			finishing the proof.
			\item In the case when $M^\bullet$ is bounded, this is a straightforward induction on the cohomological lengh $l(M^\bullet)$. If $l(M^\bullet)=0$, the statement is true by assumption. If $l(M^\bullet)>0$, with $n$ being the largest integer such that $\mathrm{H}^n(M^\bullet)\neq 0$, apply the induction hypothesis to the long exact sequence of cohomology associated to the distinguished triangle
			\begin{equation*}
				\mathrm{R}F(\tau^{<n}M^\bullet)\to \mathrm{R}F(M^\bullet)\to\mathrm{R}F(\tau^{\geq n}M^\bullet).
			\end{equation*}
			If $M^\bullet$ is only bounded below, then $\tau^{<n}M^\bullet$ is bounded for any $n$, and the distinguished triangle
			\begin{equation*}
				\mathrm{R}F(\tau^{<n}M^\bullet)\to \mathrm{R}F(M^\bullet)\to\mathrm{R}F(\tau^{\geq n}M^\bullet)
			\end{equation*}
			shows that 
			\begin{align*}
				\mathrm{H}^j(\mathrm{R}F(M^\bullet))&\cong \mathrm{H}^j(\mathrm{R}F(\tau^{<n}M^\bullet))\\
				&\cong \mathrm{R}^iF(\mathrm{H}^{j-i}(\tau^{<n}M^\bullet))\\
				&\cong \mathrm{R}^iF(\mathrm{H}^{j-i}(M^\bullet))
			\end{align*}
			for any $j<n$, since $\mathrm{R}F(\tau^{\geq n}M^\bullet)$ is concentrated in degrees $\geq n$. As $n$ was arbitrary, this proves the statement in the case when $M^\bullet$ is bounded below. If $\mathrm{R}F$ commutes with products, we can use the same argument as in (i) to get the same result for unbounded complexes.\qedhere
		\end{enumerate}
	\end{proof}
	The dual statement for left derivable functors holds mutatis mutandis, using homotopy colimits.
	
	\subsection{Modules over monoids}
	We now summarize a couple of standard properties regarding monoidal categories and modules over monoids. Most of the results in this subsection are well-known, even though it is not easy to locate proofs of the required generality in the literature. For a brief overview, see \cite[section 1.3.8]{Meyer}.
	
	Let $\C$ be a quasi-abelian category which is also closed symmetric monoidal. We denote by $T: \C\times \C\to \C$ the tensor product in $\C$, and $H: \C^{\mathrm{op}}\times \C\to \C$ the internal Hom functor. The unit object is denoted $U\in \C$.

	\begin{lem}[{\cite[Corollary 1.5.4]{Schneiders}}]
		\label{extendclosed}
		If $\C$ has enough projectives such that for any projective $P\in \C$, $T(P, -)$ is an exact functor sending projectives to projectives, then $T$ and $H$ extend to functors on $LH(\C)$, giving it the structure of a closed symmetric monoidal category.
	\end{lem}
	
	We remark that the conditions in the lemma are satisfied as soon as we can find a subcategory $\mathcal{P}\subseteq \C$ of projective objects such that
	\begin{enumerate}[(i)]
		\item any $X\in \C$ admits a strict epimorphism $P\to X$ for some $P\in \mathcal{P}$
		\item for any $P\in \mathcal{P}$, $T(P, -)$ is exact and preserves $\mathcal{P}$.
	\end{enumerate}
	In fact, in this case an object of $\C$ is projective if and only if it is a direct summand of an object in $\mathcal{P}$.
	
	We summarize this condition by saying that `$\C$ has \textbf{enough flat projectives stable under $T$}'.
	
	Using the associativity of $T$, we can write unambigiously (up to natural isomorphism) triple products as $T(A, B, C)\cong T(A, T(B, C))$ for $A, B, C\in \C$.\\
	By associativity of $T$, an easy Yoneda argument shows that $H(A, H(B, C))\cong H(T(B, A), C)$ naturally for $A, B, C\in \C$.
	
	Let $\R$ be a monoid in $\C$, i.e. an object $\R$ in $\C$ with a multiplication morphism $T(\R, \R)\to \R$ and a unit morphism $U\to \R$ satisfying the usual axioms. We can then form the category $\mathrm{Mod}_{\C}(\R)$ of left $\mathcal{R}$-modules in $\C$ by considering objects $M\in \C$ together with action morphisms $T(\R, M)\to M$, again satisfying the usual unit and associativity axioms.
	
	When the base category is understood, we will drop the subscript and write $\mathrm{Mod}(\R)$ instead of $\mathrm{Mod}_{\C}(\R)$.
	
	\begin{lem}
		\label{ModR}
		Let $\C$ be a closed symmetric monoidal quasi-abelian category and let $\R$ be a monoid in $\C$.
		\begin{enumerate}[(i)]
			\item $\mathrm{Mod}(\R)$ is a quasi-abelian category.
			\item The forgetful functor $\mathrm{forget}: \mathrm{Mod}(\R)\to \C$ commutes with limits and colimits. In particular, $\mathrm{forget}$ is strongly exact, and a morphism in $\mathrm{Mod}(\R)$ is strict if and only if it is strict as a morphism in $\C$.
			\item If $\C$ is (co)complete, so is $\mathrm{Mod}(\R)$.
			\item If $\C$ is quasi-elementary (resp. elementary), so is $\mathrm{Mod}(\R)$.
		\end{enumerate}
	\end{lem}
	\begin{proof}
		(i) and (ii) are \cite[Proposition 1.5.1]{Schneiders}. For (iii), if $(M_i)$ is a collection of $\R$-modules, it is straightforward to endow $\oplus M_i\in \C$ and $\prod M_i\in \C$ with $\R$-module structures and to verify that they form sums and products in $\mathrm{Mod}(\R)$. Completeness and cocompleteness follow, since $\mathrm{Mod}(\R)$ also admits kernels and cokernels. Finally, (iv) is \cite[Proposition 2.1.18.c)]{Schneiders}.
	\end{proof}
	Naturally, if $\C$ itself is abelian, then so is $\mathrm{Mod}(\R)$. We will often reduce to this situation using the following proposition.
	
	\begin{prop}
		\label{LHIR}
		Let $\C$ be a closed symmetric monoidal quasi-abelian category. Suppose that $LH(\C)$ is a closed symmetric monoidal category with unit object $I(U)$, tensor product $\widetilde{T}$ and inner hom $\widetilde{H}$, such that there is a natural isomorphism 
		\begin{equation*}
			\widetilde{T}(I(X), I(Y))\cong I(T(X, Y))
		\end{equation*}
		for $X, Y\in \C$. Let $\R$ be a monoid in $\C$.
		
		Then $I(\R)$ is a monoid in $LH(\C)$ and there is a functor $J: \mathrm{Mod}(\R)\to \mathrm{Mod}(I(\R))$ inducing an equivalence of categories
		\begin{equation*}
			LH(\mathrm{Mod}(\R))\cong \mathrm{Mod}(I(\R)).
		\end{equation*}  
	\end{prop}
	\begin{proof}
		The isomorphism $\widetilde{T}(I(\R), I(\R))\cong I(T(\R, \R))$ makes it clear that $I(\R)$ is a monoid, and any action map $T(\R, M)\to M$ in $\C$ yields an action map $\widetilde{T}(I(\R), I(M))\to I(M)$ simply by applying $I$. Thus $I$ induces a functor $J: \mathrm{Mod}(\R)\to \mathrm{Mod}(I(\R))$. \\
		To prove the equivalence of categories, we apply \cite[Proposition 1.2.36]{Schneiders}, so that it remains to check the following:
		\begin{enumerate}[(a)]
			\item $J$ is fully faithful.
			\item If $A\to J(M)$ is a monomorphism for some $A\in \mathrm{Mod}(I(\R))$, $M\in \mathrm{Mod}(\R)$, then $A$ lies in the essential image of $J$.
			\item Any object in $\mathrm{Mod}(I(\R))$ admits an epimorphism from an object in the essential image of $J$.
		\end{enumerate}
		For (a), $\mathrm{Hom}_{\mathrm{Mod}(\R)}(M, N)$ consists of those morphisms $M\to N$ in $\C$ inducing a commutative square
		\begin{equation*}
			\begin{xy}
				\xymatrix{
					T(\R, M)\ar[r]\ar[d]& T(\R, N)\ar[d]\\
					M\ar[r] & N.
				}
			\end{xy}
		\end{equation*}
		Since $I$ is fully faithful, the functor $I$ maps this bijectively to the set of morphisms $I(M)\to I(N)$ in $LH(\C)$ such that the corresponding diagram commutes in $LH(\C)$. This proves fully faithfulness by the natural isomorphism $\widetilde{T}(I(\R), I(M))\cong I(T(\R, M))$.
		
		For (b), let $A\to J(M)$ be a monomorphism in $\mathrm{Mod}(I(\R))$. By Lemma \ref{ModR}.(ii), this is a monomorphism in $LH(\C)$, so $A$ is in the essential image of $I$ by \cite[Proposition 1.2.29]{Schneiders}. Hence $A\cong I(A')$ for some $A'\in \C$. But then the action map $I(T(\R, A'))\cong\widetilde{T}(I(\R), I(A'))\to I(A')$ together with the fully faithfulness of $I$ make $A'$ an $\R$-module such that $A=J(A')$.
		
		For (c), let $A=(0\to A^{-1}\to A^0\to 0)$ be an object of $\mathrm{Mod}(I(\R))$. Then $I(A^0)\to A$ is an epimorphism in $LH(\C)$, and the action map $\widetilde{T}(I(\R), A)\to A$ is always an epimorphism (admitting a section via the unit morphism). Here $\widetilde{T}(I(\R), A)$ is an $I(\R)$-module via multiplication on the first factor. In the same way, $T(\R, A^0)$ is naturally an $\R$-module, and hence $J(T(\R, A^0))$ is an $I(\R)$-module such that the natural morphism
		\begin{equation*}
			J(T(\R, A^0))\cong \widetilde{T}(I(\R), I(A^0))\to \widetilde{T}(I(\R), A)\to A
		\end{equation*} 
		is an ($I(\R)$-linear) epimorphism as required.
		
		Thus all the conditions in \cite[Proposition 1.2.36]{Schneiders} are satisfied, so that $J$ induces an isomorphism $LH(\mathrm{Mod}(\R))\cong \mathrm{Mod}(I(\R))$.
	\end{proof}
	
	In fact, we can view the above as an instance of a more general phenomenon:
	
	\begin{defn}
		Let $(\C, U, T), (\C', U', T')$ be symmetric monoidal categories. An additive functor $F: \C\to \C'$ is called \textbf{lax symmetric monoidal} if it comes equipped with natural morphisms $U'\to F(U)$ and $T'(F(X), F(Y))\to F(T(X,Y))$ for $X, Y\in \C$, respecting the natural isomorphisms for the unit, symmetry and associativity axioms.
		
		A lax symmetric monoidal functor $F$ is called \textbf{strong symmetric monoidal} if the morphisms $U'\to F(U)$ and $T'(F(X), F(Y))\to F(T(X, Y))$ are actually isomorphisms.
	\end{defn}
	
	\begin{prop}
		\label{laxmodulechange}
		Let $F: \C\to \C'$ be a lax symmetric monoidal functor between symmetric monoidal categories. If $\R$ is a monoid in $\C$, then $F(\R)$ is a monoid in $\C'$, and $F$ induces a functor $\mathrm{Mod}_{\C}(\R)\to \mathrm{Mod}_{\C'}(F(\R))$.
	\end{prop}
	\begin{proof}
		This is well-known, but it appears difficult to locate a detailed proof in the literature. 
		
		If $\R$ is a monoid in $\C$, then the multiplication map $m: T(\R, \R)\to \R$ gives rise to a morphism $T'(F(\R), F(\R))\to F(T(\R, \R))\to F(\R)$, and the unit morphism $U\to \R$ induces a morphism $U'\to F(U)\to F(\R)$. It is easy to verify that this satisfies the unit and associativity axioms. For instance, the unit axiom requires that
		\begin{equation*}
			\begin{xy}
				\xymatrix{T'(U', F(\R))\ar[r]\ar[rd]&T'(F(\R), F(\R))\ar[d]\\
				&F(\R)}
			\end{xy}
		\end{equation*}
		commutes, where the horizontal map is induced by $U'\to F(R)$, the vertical map is the multiplication map and the diagonal arrow is the natural isomorphism $T'(U', F(\R))\to F(\R)$.
		
		But this fits into a diagram
		\begin{equation*}
			\begin{xy}
				\xymatrix{T'(U', F(\R))\ar@/^1pc/[rrr]\ar[r]\ar[rd] & T'(F(\R), F(\R))\ar[r]\ar[d]& F(T(\R, \R))\ar[d]& F(T(U, \R))\ar[l]\ar[ld]\\
					& F(\R)\ar[r]^= &F(\R)}
			\end{xy}
		\end{equation*}
		where the outer square commutes since $F$ is lax monoidal, the right triangle commutes since $\R$ is a monoid, and the middle square commutes by definition of the multiplication map. The remaining properties are checked similarly.
		
		In the same way, if $M$ is an $\R$-module with action map $T(\R, M)\to M$, then we obtain a morphism $T'(F(\R), F(M))\to F(T(\R, M))\to F(M)$. Easy diagram chases confirm that this indeed defines an $F(\R)$-module structure on $F(M)$.
		
		An $\R$-module morphism $M\to N$ is a morphism in $\C$ fitting into a commutative square
		\begin{equation*}
			\begin{xy}
				\xymatrix{T(\R, M)\ar[r]\ar[d]&T(\R, N)\ar[d]\\
					M\ar[r]&N,}
			\end{xy}
		\end{equation*}
		 and thus $F$ sends $\R$-module morphisms to $F(\R)$-module morphisms by definition of the $F(\R)$-module structure. In particular, as $F: \C\to \C'$ is a functor, $F$ induces a functor $\mathrm{Mod}(\R)\to \mathrm{Mod}(F(\R))$.
	\end{proof}
	
	Let $\C$ be a symmetric monoidal quasi-abelian category such that the tensor product $T$ is left derivable, such that the induced functor $\widetilde{T}=\mathrm{H}^0(\mathbb{L}T)$ defines a symmetric monoidal structure on $LH(\C)$ (with unit object $I(U)$). In this case the natural embedding $I: \C\to LH(\C)$ is lax symmetric monoidal:
	
	In fact, if $X\in \C$ can be written as the cokernel of a strict morphism between flat objects $P^{-1}\to P^0$, then 
	\begin{equation*}
		\widetilde{T}(I(X), I(Y))\cong [T(P^{-1}, Y)\to T(P^0, Y)]\in LH(\C),
	\end{equation*}
	while $T(X, Y)$ is the cokernel of $T(P^{-1}, Y)\to T(P^0, Y)$ in $\C$. In particular, we have $C\widetilde{T}(I(X), I(Y))\cong T(X, Y)$, and adjunction provides the natural morphism $\widetilde{T}(I(X), I(Y))\to I(T(X, Y))$. It is easy to verify that this indeed turns $I$ into a lax symmetric monoidal functor.
	
	Proposition \ref{LHIR} can then be interpreted as follows: if $I$ is actually strong symmetric monoidal, then the induced functor between the module categories becomes an equivalence on the level of left hearts.
	
	We remark that in general, $I$ need not be strong symmetric monoidal: the above description of $\widetilde{T}$ makes it clear that this will happen if and only if $T(-, Y)$ sends strict morphisms between flat objects to strict morphisms, for each $Y\in \C$.	
	
	We now verify that tensor products and internal homs of modules inherit natural module structures, and introduce tensor products over $\R$ and internal $\R$-module homs.
	
	Let $\R_1$, $\R_2$, $\R$ be monoids in a closed symmetric monoidal quasi-abelian category $\C$. If $M$ is a left $\R_1$-module and $N$ is a right $\R_2$-module, it is immediate that $T(M, N)$ carries the structure of an $(\R_1, \R_2)$-bimodule.
	
	If $M$ is a right $\R_1$-module and $N$ is a right $\R_2$-module, then $H(M, N)$ becomes an $(\R_1, \R_2)$-bimodule in the following way:
	
	The left $\R_1$-module structure is obtained from the adjunctions
	\begin{align*}
		\mathrm{Hom}(T(\R_1, H(M, N)), H(M, N))&\cong \mathrm{Hom}(H(M, N), H(\R_1, H(M, N)))\\
		&\cong \mathrm{Hom}(H(M, N), H(T(M, \R_1), N)),
	\end{align*}
	while the right $\R_2$-module structure 
	\begin{equation*}
		T(H(M, N), \R_2)\to H(M, N)
	\end{equation*}
	corresponds via adjunction to the morphism
	\begin{equation*}
		T(M, H(M, N), \R_2)\to T(N, \R_2)\to N,
	\end{equation*}
	where the first morphism is induced from $T(M, H(M, N))\to N$ corresponding to the identity on $H(M, N)$. In all cases, it is straightforward to verify the action axioms.
	
	We now define the functor $T_\R(-,-): \mathrm{Mod}(\R^\mathrm{op})\times \mathrm{Mod}(\R)\to \C$ by letting $T_\R(M, N)$ be the coequalizer of
	\begin{equation*}
		\begin{xy}
			\xymatrix{
				T(M, \R, N)\ar@<1ex>[r]^f \ar[r]_g& T(M, N),
			}
		\end{xy}
	\end{equation*}
	where $f$ is the morphism induced from the action map $T(M, \R)\to M$, and $g$ is the morphism induced from the action map $T(\R, N)\to N$.
	
	As before, if $M$ is an $(\R_1, \R)$-bimodule and $N$ is an $(\R, \R_2)$-bimodule, the bimodule structure of $T(M, N)$ descends, giving $T_\R(M, N)$ a natural structure of an $(\R_1, \R_2)$-bimodule.
	
	We now turn to the internal Hom functor. Note that if $N\in \mathrm{Mod}(\R)$, then the action map $T(\R, N)\to N$ induces via adjunction a natural morphism $N\to H(\R, N)$.
	For any $M, N\in \mathrm{Mod}(\R)$, $H_\R(M, N)$ is now defined to be the equalizer of 
	\begin{equation*}
		\begin{xy}
			\xymatrix{
				H(M, N)\ar@<1ex>[r]^f \ar[r]_g&  H(T(\R,M),N),}
		\end{xy}
	\end{equation*}
	where $f$ is the morphism obtained by applying $H(-, N)$ to the action map 
	\begin{equation*}
		T(\R, M)\to M,
	\end{equation*}
	and $g$ is the morphism obtained by applying $H(M, -)$ to the morphism
	\begin{equation*}
		N\to H(R, N)
	\end{equation*}
	given above, and applying tensor-hom adjunction once more.
	
	If $M$ is an $(\R, \R_1)$-bimodule and $N$ is an $(\R, \R_2)$-bimodule, then $H_\R(M, N)$ inherits the structure of an $(\R_1, \R_2)$-bimodule.
	
	\begin{lem} Let $\C$ be a closed symmetric monoidal quasi-abelian category, and let $\R$ be a monoid in $\C$.
		\begin{enumerate}[(i)]
			\item  There is a natural isomorphism
			\begin{equation*}
				\mathrm{Hom}_{\C}(T_\R(M, N), X)\cong \mathrm{Hom}_{\mathrm{Mod}(\R)}(N, H(M, X))
			\end{equation*}
			for $M \in \mathrm{Mod}(\R^\mathrm{op})$, $N\in \mathrm{Mod}(\R)$, $X\in \C$.
			\item There is a natural isomorphism
			\begin{equation*}
				\mathrm{Hom}_{\mathrm{Mod}(\R)}(T(M, X), N)\cong \mathrm{Hom}_\C(X, H_\R(M, N))
			\end{equation*}
			for $M, N\in \mathrm{Mod}(\R)$, $X\in \C$.
		\end{enumerate}
	\end{lem}
	\begin{proof}
		\begin{enumerate}[(i)]
			\item As $\mathrm{Hom}_\C(-, X)$ sends cokernels in $\C$ to kernels, we know that 
			\begin{equation*}
				\mathrm{Hom}_\C(T_\R(M, N), X)
			\end{equation*}
			is the equalizer of
			\begin{equation*}
				\begin{xy}
					\xymatrix{
						\mathrm{Hom}_\C(T(M, N), X)\ar@<1ex>[r] \ar[r]& \mathrm{Hom}_\C(T(M, \R, N), X).}
				\end{xy}
			\end{equation*}
			Also note that $\mathrm{Hom}_{\mathrm{Mod}(\R)}(N, H(M, X))$ is the equalizer of
			\begin{equation*}
				\begin{xy}
					\xymatrix{
						\mathrm{Hom}_\C(N, H(M, X))\ar@<1ex>[r] \ar[r]& \mathrm{Hom}_\C(T(\R, N), H(M, X)).}
				\end{xy}
			\end{equation*}
			The isomorphism thus follows from the adjunction between $T$ and $H$, which identifies the two diagrams above.
			\item Entirely analogous to (i), describing both sides as the equalizer of
			\begin{equation*}
				\begin{xy}
					\xymatrix{
						\mathrm{Hom}_\C(T(M, X), N)\ar@<1ex>[r] \ar[r]& \mathrm{Hom}_\C(T(\R, M, X), N).}
				\end{xy}\qedhere
			\end{equation*}
		\end{enumerate}
	\end{proof}
	In the same way, these adjunctions hold when working with suitable bimodule categories as above.
	\begin{lem}
		\label{tensorwithfree}
		There is a natural isomorphism
		\begin{equation*}
			T_\R(M, T(\R, N))\cong T(M, N) 
		\end{equation*}
		for $M\in \mathrm{Mod}(\R^\mathrm{op})$, $N\in \C$.
	\end{lem}
	\begin{proof}
		It suffices to show that $T(M, N)$ is the cokernel of the natural morphism
		\begin{equation*}
			T(M, \R, \R, N)\to T(M, \R, N),
		\end{equation*}
		which is just obtained by applying $T(-, N)$ to the strictly exact sequence
		\begin{equation*}
			T(M, \R, \R)\to T(M, \R)\to M\to 0
		\end{equation*}
		(see e.g. \cite[proof of Lemma 2.9]{BamStein}). This is obvious as $T(-, N)$ preserves cokernels.
	\end{proof}
	
	If we are in the situation of Proposition \ref{LHIR}, we have in theory two ways of obtaining a tensor product on the level of left hearts: we can derive the tensor product $T_{\R}$, or we can use the identification of $LH(\mathrm{Mod}(\R))$ with $\mathrm{Mod}(I(\R))$ and use $\widetilde{T}_{I(\R)}$. The next Lemma ensures that these two constructions coincide.
	
	\begin{lem}
		\label{reltensoronLH}
		Let $\C$ be a closed symmetric monoidal quasi-abelian category which has enough flat objects. Let $\R$ be a monoid in $\C$, and suppose that $I$ induces an equivalences of categories
		\begin{equation*}
			\widetilde{J}: LH(\mathrm{Mod}(\R))\to \mathrm{Mod}(I(\R))
		\end{equation*}
		Then there is a natural isomorphism
		\begin{equation*}
			\widetilde{(T_{\R})}(M, N):=\mathrm{H}^0(\mathbb{L}T_{\R}(M, N))\cong \widetilde{T}_{I(\R)}(\widetilde{J}(M), \widetilde{J}(N))\in LH(\C)
		\end{equation*}
		for $M\in LH(\mathrm{Mod}(\R^{\mathrm{op}}))$, $N\in LH(\mathrm{Mod}(\R))$.
	\end{lem}
	
	\begin{proof}
		The natural transformation $T(\mathrm{forget}(-), \mathrm{forget}(-))\to T_{\R}(-, -)$ of functors from $\mathrm{Mod}(\R^{\mathrm{op}})\times \mathrm{Mod}(\R)\to \C$ induces a natural morphism
		\begin{equation*}
			\widetilde{T}(\widetilde{J}(M), \widetilde{J}(N))\to \widetilde{(T_{\R})}(M, N)
		\end{equation*}
		for $M\in LH(\mathrm{Mod}(\R^{\mathrm{op}}))$, $N\in LH(\mathrm{Mod}(\R))$, which via the universal property of coequalizers induces a natural morphism
		\begin{equation*}
			\widetilde{T}_{I(\R)}(\widetilde{J}(M), \widetilde{J}(N))\to \widetilde{(T_{\R})}(M, N).
		\end{equation*}
		To check that this is an isomorphism, it suffices to consider the case where $M=I(T(X, \R))$, $N=I(T(\R, Y))$ for some flat objects $X, Y\in \C$.
		
		By flatness, we have $I(T(\R, Y))\cong \widetilde{T}(I(\R), I(Y))$ in $LH(\C)$, and hence $\widetilde{J}(N)\cong \widetilde{T}(I(\R), I(Y))$, with the left $I(\R)$-module structure coming from the first factor. Likewise $\widetilde{J}(M)=\widetilde{T}(I(X), I(\R))$.
		
		In this case, 
		\begin{equation*}
			\widetilde{T}_{I(\R)}(\widetilde{J}(M), \widetilde{J}(N))\cong \widetilde{T}(I(X), I(\R), I(Y))
		\end{equation*}
		by Lemma \ref{tensorwithfree}, whereas
		\begin{align*}
			\widetilde{(T_{\R})}(M, N)&\cong IT_{\R}(T(X, \R), T(\R, Y))\\
			&\cong IT(X, \R, Y)
		\end{align*}
		again by Lemma \ref{tensorwithfree} and the flatness assumption. This proves the Lemma, since $X$ and $Y$ are flat and $\widetilde{T}(-, -)=\mathrm{H}^0(\mathbb{L}T(-, -))$.
	\end{proof}

	\subsection{Categories of sheaves}
	We can also extend the above discussion to categories of sheaves with values in a quasi-abelian category. While the theory of sheaves on topological spaces was already discussed in \cite{Schneiders}, we will also need some basic results for sheaves on G-topological spaces. We will therefore reprove some statements from \cite{Schneiders} on this level of generality, avoiding the use of arguments involving stalks. Most of the results given here hold more generally for sites, but we will not need this level of generality.
	\begin{defn}[{\cite[Definition 9.1.1/1]{BGR}}]
		A \textbf{G-topological space} is a set $X$ together with the data $(\mathrm{Op}(X), \mathrm{Cov})$, where $\mathrm{Op}(X)$ is a collection of subsets of $X$ (the admissible open subsets of $X$) and $\mathrm{Cov}$ associates to each admissible open subset $U$ of $X$ a collection $\mathrm{Cov}(U)$ of coverings of $U$ by admissible open subsets such that
		\begin{enumerate}[(i)]
			\item $\mathrm{Op}(X)$ is closed under finite intersections.
			\item For any $U\in \mathrm{Op}(X)$, $\{U\}\in \mathrm{Cov}(U)$.
			\item If $\{U_i\}_{i\in I}\in \mathrm{Cov}(U)$ and $\{V_{ij}\}_{j\in J_i}\in \mathrm{Cov}(U_i)$ for each $i\in I$, then $\{V_{ij}\}_{i, j}\in \mathrm{Cov}(U)$.
			\item If $U, V\in \mathrm{Op}(X)$ and $V\subset U$, $\{U_i\}\in \mathrm{Cov}(U)$, then $\{V\cap U_i\}\in \mathrm{Cov}(V)$.
		\end{enumerate}
	\end{defn}
	Viewing the power set $\mathcal{P}(X)$ as a category with the inclusion maps as morphisms, we see that a G-topological space is the same as a set $X$ together with a Grothendieck topology on a full subcategory of $\mathcal{P}(X)$.
	
	A \textbf{presheaf} on $X$ with values in a complete quasi-abelian category $\C$ is a functor $\mathrm{Op}(X)^\mathrm{op}\to \C$. We denote the category of presheaves by $\mathrm{Preshv}(X, \C)$.
	
	An object $\mathcal{F}$ of $\mathrm{Preshv}(X, \C)$ is a \textbf{sheaf} (or a $\C$-sheaf) if for any $U\in \mathrm{Op}(X)$ and any covering $(U_i)_{i\in I}$ of $U$, $\mathcal{F}(U)$ is the equalizer of the usual restriction diagram
	\begin{equation*}
		\begin{xy}\xymatrix{
				\prod_{i\in I} \mathcal{F}(U_i)\ar@<1ex>[r] \ar[r]& \prod_{i, j\in I} \mathcal{F}(U_i\cap U_j).
			}
		\end{xy}
	\end{equation*}
	The full subcategory of $\mathrm{Preshv}(X, \C)$ consisting of sheaves is denoted by $\mathrm{Shv}(X, \C)$.
	
	In this subsection, we suppose that $\C$ satisfies the following:
	\begin{enumerate}[(i)]
		\item $\C$ is an elementary quasi-abelian category.
		\item $\C$ is a closed symmetric monoidal category with enough flat projectives stable under $T$.
	\end{enumerate}
	In this case, we have the sheafification functor $\mathcal{F}\mapsto \mathcal{F}^a$ as a left adjoint to the inclusion $\mathrm{Shv}(X, \C)\to \mathrm{Preshv}(X, \C)$ (see \cite[subsection 2.2.1]{Schneiders} for details).
	
	More explicitly, $\mathcal{F}^a=h(h(\mathcal{F}))$, where
	\begin{equation*}
		h(\mathcal{F})(U)=\underset{\mathfrak{U}\in \mathrm{Cov}(U)}{\mathrm{colim}}\check{\mathrm{H}}^0(\mathfrak{U}, \mathcal{F})=\mathrm{colim} \left(\mathrm{ker}\left(\prod_{U_i\in \mathfrak{U}}\mathcal{F}(U_i)\to \prod \mathcal{F}(U_i\cap U_j)\right)\right)
	\end{equation*}
	for any admissible open subset $U$. The proof that $(-)^a$ is a left adjoint given in \cite[Proposition 2.2.6]{Schneiders} still works in this generality. We stress that this construction rests crucially on $\C$ being elementary: the proofs in \cite{Schneiders} rely on testing against tiny projective objects $X$ and noting that in this case
	\begin{equation*}
		\mathrm{Hom}_{\C}(X, h(\mathcal{F})(U))\cong h(\mathrm{Hom}_{\C}(X, \mathcal{F}(-)))(U)
	\end{equation*} 
	in the category of abelian groups.
	
	\begin{lem}
		\label{Presheaflimits}
		Let $X$ be a G-topological space.
		\begin{enumerate}[(i)]
			\item $\mathrm{Preshv}(X, \C)$ is a quasi-abelian category.
			\item $\mathrm{Preshv}(X, \C)$ and $\mathrm{Shv}(X, \C)$ are complete and cocomplete. Moreover, limits in $\mathrm{Shv}(X, \C)$ are the same as in $\mathrm{Preshv}(X, \C)$, a colimit in $\mathrm{Shv}(X, \C)$ is the sheafification of the corresponding colimit in $\mathrm{Preshv}(X, \C)$.
			\item Sheafification is strongly exact. 
		\end{enumerate}
	\end{lem}
	\begin{proof}
		\begin{enumerate}[(i)]
			\item follows from \cite[Proposition 1.4.9]{Schneiders}.
			\item $\mathrm{Preshv}(X, \C)=\mathrm{Fun}(\mathrm{Op}(X)^\mathrm{op}, \C)$ is complete and cocomplete, since $\C$ is complete and cocomplete. The remaining statements follow from the fact that $\mathrm{Shv}(X, \C)$ is a reflective subcategory of $\mathrm{Preshv}(X, \C)$.
			\item By adjunction, sheafification is strongly right exact, so it suffices to show that $h$ preserves kernels. But this is immediate from the strong left exactness of products and strong exactness of filtered colimits in $\C$.\qedhere
		\end{enumerate}
	\end{proof}
	\begin{cor}
		$\mathrm{Shv}(X, \C)$ is a quasi-abelian category. 
	\end{cor}
	\begin{proof}
		We have already seen that $\mathrm{Shv}(X, \C)$ is an additive category with kernels and cokernels. It remains to show that if
		\begin{equation*}
			\begin{xy}
				\xymatrix{
					E'\ar[r]^{f'} \ar[d]&F'\ar[d]\\
					E\ar[r]_f&F
				}
			\end{xy}
		\end{equation*}
		is a Cartesian square with $f$ a strict epimorphism, then $f'$ is a strict epimorphism, as well as the dual statement.
		
		Given a diagram as above, let $F^\circ$ be the coimage of $f$ in the category of presheaves. We know from the above that $(F^\circ)^a$ is the coimage of $f$ in $\mathrm{Shv}(X, \C)$, and hence isomorphic to $F$. Moreover, $f^\circ: E\to F^\circ$ is a strict epimorphism in $\mathrm{Preshv}(X, \C)$.
		
		Consider the diagram of presheaves
		\begin{equation*}
			\begin{xy}
				\xymatrix{
					E'^\circ\ar[r]^{f'^\circ}\ar[d]&F'^\circ\ar[r]\ar[d] &F'\ar[d]\\
					E\ar[r] & F^\circ\ar[r] & F,
				}
			\end{xy}
		\end{equation*}
		where $F'^\circ=F'\times_F F^\circ$ and $E'^\circ=F'^\circ\times_{F^\circ}E$ in $\mathrm{Preshv}(X, \C)$.
		
		Then $f'^\circ: E'^\circ\to F'^\circ$ is a strict epimorphism, as $\mathrm{Preshv}(X, \C)$ is quasi-abelian. But as $(-)^a$ sends Cartesian squares to Cartesian squares and $(F^\circ)^a\cong F$, we get $(F'^\circ)^a\cong F'$, $(E'^\circ)^a\cong E'$ and $(f'^\circ)^a\cong f'$. Since sheafification preserves strict epimorphisms, this shows that $f'$ is a strict epimorphism in $\mathrm{Shv}(X, \C)$.
		
		Similarly, if 
		\begin{equation*}
			\begin{xy}
				\xymatrix{
					E'\ar[r]^{f'} &F'\\
					E\ar[u] \ar[r]_f & F\ar[u]
				}
			\end{xy}
		\end{equation*} is Cocartesian with $f$ a strict monomorphism, let $F'^\circ$ be the pushout of
		\begin{equation*}
			\begin{xy}
				\xymatrix{
					E'&\\
					E\ar[u] \ar[r] & F
				}
			\end{xy}
		\end{equation*} 
		in $\mathrm{Preshv}(X, \C)$, so that $(F'^\circ)^a\cong F'$ by Lemma \ref{Presheaflimits}. As kernels in $\mathrm{Shv}(X, \C)$ are the same as in $\mathrm{Preshv}(X, \C)$, we know that $f$ is also a strict monomorphism in $\mathrm{Preshv}(X, \C)$. Hence $E'\to F'^\circ$ is a strict monomorphism, and applying sheafification yields that $f'$ is a strict monomorphism.
		
		This shows that $\mathrm{Shv}(X, \C)$ is a quasi-abelian category.
	\end{proof}
	We remark that the same proof implies more generally that a reflective subcategory of a quasi-abelian category with strongly exact reflector is quasi-abelian.
	\begin{lem}
		\label{LHsheaves}
		We have canonical equivalences of categories
		\begin{equation*}
			LH(\mathrm{Preshv}(X, \C))\cong \mathrm{Preshv}(X, LH(\C))
		\end{equation*}
		and
		\begin{equation*}
			LH(\mathrm{Shv}(X, \C))\cong \mathrm{Shv}(X, LH(\C)).
		\end{equation*}
	\end{lem}
	\begin{proof}
		For presheaves, this follows from \cite[Proposition 1.4.15]{Schneiders}. For sheaves, the proof in \cite[Proposition 2.2.12]{Schneiders} generalizes to our setting.
	\end{proof}
	
	By \cite[Proposition 2.2.12]{Schneiders}, the inclusion functor 
	\begin{equation*}
		I: \mathrm{Shv}(X, \C)\to LH(\mathrm{Shv}(X, \C))
	\end{equation*}
	gets identified under this equivalence with the `section-wise' inclusion functor 
	\begin{equation*}
		\mathrm{Shv}(X, \C)\to \mathrm{Shv}(X, LH(\C)),
	\end{equation*}
	similarly for presheaves.
	\begin{lem}
		\label{Iacommute}
		$I$ commutes with sheafification, i.e. there is a natural isomorphism $I((\mathcal{F})^a)\cong (I(\mathcal{F}))^a$ for $\mathcal{F}\in \mathrm{Preshv}(X, \C)$.
	\end{lem} 
	\begin{proof}
		By \cite[Proposition 2.1.16]{Schneiders}, $I: \C\to LH(\C)$ commutes with filtered colimits. As $I$ also preserves limits (so in particular, kernels and products), the result follows from the definition of the sheafification functor.
	\end{proof}
	
	We can now endow $\mathrm{Shv}(X, \C)$ with the structure of a closed symmetric monoidal category as follows (see \cite[subsection 2.2.3]{Schneiders}):
	
	Given $\M, \N\in \mathrm{Preshv}(X, \C)$, we define the presheaf tensor product
	\begin{equation*}
		\T_\mathrm{pre}(\M, \N): U\mapsto T(\M(U), \N(U)).
	\end{equation*}
	If $\M, \N$ are sheaves, we denote by $\T(\M, \N)$ the sheafification $\T_{\mathrm{pre}}(\M, \N)^a$.
	
	The internal Hom $\mathcal{H}(\M, \N)$ for $\M, \N\in \mathrm{Preshv}(X, \C)$ is given by setting $\mathcal{H}(\M, \N)(U)$ to be the equalizer of
	\begin{equation*}
		\begin{xy}
			\xymatrix{
				\prod_{V} H(\M(V), \N(V))\ar@<1ex>[r]^f \ar[r]_g& \prod_{W\subseteq V} H(\M(V), \N(W)),
			}
		\end{xy}
	\end{equation*}
	where the second product is over all pairs of admissible open subsets $(W, V)$ with $W\subseteq V$, where $f$ is induced by the natural morphisms $H(M(V), N(V))\to H(M(V), N(W))$, and $g$ is induced by $H(M(W), N(W))\to H(M(V), N(W))$.
	
	\begin{lem}
		\label{shvclosed}
		\leavevmode
		\begin{enumerate}[(i)]
			\item There is a natural isomorphism 
			\begin{equation*}
				\mathrm{Hom}_{\mathrm{Preshv}(X, \C)}(\T_\mathrm{pre}(\M, \N), \M')\cong \mathrm{Hom}_{\mathrm{Preshv}(X, \C)}(\N, \mathcal{H}(\M, \M')) 
			\end{equation*}
			for $\M, \M', \N\in \mathrm{Preshv}(X, \C)$.
			\item If $\M, \M'\in \mathrm{Shv}(X, \C)$, then $\mathcal{H}(\M, \M')\in \mathrm{Shv}(X, \C)$, and there is a natural isomorphism
			\begin{equation*}
				\mathrm{Hom}_{\mathrm{Shv}(X, \C)}(\T(\M, \N), \M')\cong \mathrm{Hom}_{\mathrm{Shv}(X, \C)}(\N, \mathcal{H}(\M, \M'))
			\end{equation*}
			for $\M, \M', \N\in \mathrm{Shv}(X, \C)$.
			\item $\T_\mathrm{pre}$ and $\mathcal{H}$ turn $\mathrm{Preshv}(X, \C)$ into a closed symmetric monoidal category (with unit the constant presheaf $U$). $\T$ and $\mathcal{H}$ turn $\mathrm{Shv}(X, \C)$ into a closed symmetric monoidal category.
		\end{enumerate}
	\end{lem}
	\begin{proof}
		For (i), the same argument as in \cite[Proposition 18.2.3]{KS} applies. For (ii), the computation that $\mathcal{H}$ preserves sheaves is the same as in \cite[Proposition 2.2.14]{Schneiders}, and adjunction follows from the adjoint properties of sheafification. Part (iii) then follows immediately.
	\end{proof}
	
	\begin{lem}
		\label{Tacommute}
		There is a natural isomorphism
		\begin{equation*}
			\T_\mathrm{pre}(\M, \N)^a\cong \T(\M^a, \N^a)
		\end{equation*}
		for $\M, \N\in \mathrm{Preshv}(X, \C)$.
	\end{lem}
	\begin{proof}
		This is clear from the adjunctions above.
	\end{proof}
	
	\begin{lem}
		Let $\R$ be a monoid in $\mathrm{Shv}(X, \C)$. Then $\T_{\R}(\M, \N)$ is the sheafification of
		\begin{equation*}
			U\mapsto T_{\R(U)}(\M(U), \N(U))
		\end{equation*}
		for any $\M \in \mathrm{Mod}(\R^{\mathrm{op}})$, $\N\in \mathrm{Mod}(\R)$.
	\end{lem}
	
	\begin{proof}
		Let $\T_{\R, \mathrm{pre}}(\M, \N)$ denote the presheaf
		\begin{equation*}
			U\mapsto T_{\R(U)}(\M(U), \N(U)).
		\end{equation*}
		We thus have a strictly exact sequence of presheaves
		\begin{equation*}
			\T_{\mathrm{pre}}(\M, \R, \N)\to \T_{\mathrm{pre}}(\M, \N)\to \T_{\R, \mathrm{pre}}(\M, \N)\to 0.
		\end{equation*}
		Now apply sheafification, which is exact, to obtain the desired isomorphism.
	\end{proof}
	
	\begin{lem}
		\label{genforLH}
		Let $\E$ be a cocomplete quasi-abelian category with exact direct sums. If $\E$ admits a strictly generating set, then $LH(\E)$ admits a generator.
	\end{lem}
	\begin{proof}
		This follows straightforwardly from the fact that for any
		\begin{equation*}
			X=(0\to X^{-1}\to X^0\to 0)\in LH(\E),
		\end{equation*} 
		the morphism $I(X^0)\to X$ is an epimorphism, and $I$ commutes with direct sums by the dual of \cite[Corollary 1.4.7]{Schneiders}.
	\end{proof}
	
	\begin{lem}
		\label{genforshv}
		If $\C$ admits a generator, so does $\mathrm{Shv}(X, \C)$.
	\end{lem}
	\begin{proof}
		This is a modification of \cite[Proposition 2.2.11]{Schneiders} where we avoid the use of stalks. Given $A\in \C$, $U \subseteq X$ open, let $A_U^\mathrm{pre}$ be the presheaf
		$$
		V\mapsto \begin{cases} A \ \text{if $V\subseteq U$}\\
			0 \ \mathrm{otherwise}, \end{cases}
		$$
		and let $A_U$ be its sheafification. For any $\C$-sheaf $\F$ on $X$, there is a morphism
		\begin{equation*}
			\oplus_U \F(U)_U\to \F.
		\end{equation*}
		It is clear that the morphism $\oplus_U (\F(U)_U^\mathrm{pre}\to \F$ is a strict epimorphism of pre\-sheaves, as we can check sectionwise, so $\oplus \F(U)_U\to \F$ is a strict epimorphism, since sheafification is strongly exact (and commutes with direct sums). 
		
		If $G$ is a generator for $\C$, there exists for each $U$ a set $\I_U$ and a strict epimorphism
		\begin{equation*}
			\oplus_{\I_U}G\to \F(U).
		\end{equation*} 
		Thus we obtain a strict epimorphism $\oplus_U\oplus_{\I_U}G_U^\mathrm{pre}\to \oplus_U F(U)_U^\mathrm{pre}$ of presheaves, and sheafification yields the strict epimorphism
		\begin{equation*}
			\oplus_U \oplus_{\I_U} G_U\to \oplus_U \F(U)_U\to \F.
		\end{equation*}
		Hence $\oplus_U G_U$ is a generator, as required.
	\end{proof}
	
	\begin{thm}
		\label{Shvnice}
		\leavevmode
		\begin{enumerate}[(i)]
			\item The category $\mathrm{Shv}(X, \C)$ is a closed symmetric monoidal category.
			\item The category $\mathrm{Shv}(X, LH(\C))$ is Grothendieck.
			\item $\mathrm{Shv}(X, \C)$ has enough flat objects (i.e., for each $\mathcal{F}\in \mathrm{Shv}(X, \C)$, there exists a strict epimorphism $\mathcal{P}\to \mathcal{F}$ for some $\mathcal{P}$ such that $\T(-, \mathcal{P})$ is exact).
			\item If $\R$ is a monoid in $\mathrm{Shv}(X, \C)$, then $\mathrm{Mod}(\R)$ is a quasi-abelian category, complete and cocomplete, and $LH(\mathrm{Mod}(\R))$ is Grothendieck. 
			\item If $\R$ is a monoid, then $\mathrm{Mod}(\R)$ has enough flat objects (i.e. each left $\R$-module admits a strict epimorphism from some $\mathcal{P}$ such that 
			\begin{equation*}
				\T_\R(-, \mathcal{P}): \mathrm{Mod}(\R^\mathrm{op})\to \mathrm{Shv}(X, \C)
			\end{equation*} 
			is exact).
		\end{enumerate}
	\end{thm}
	
	\begin{proof}
		(i) is clear from Lemma \ref{shvclosed}.
		
		For (ii), it follows from Lemma \ref{Presheaflimits} that $\mathrm{Shv}(X, LH(\C))$ is a complete and cocomplete abelian category. Since $\mathrm{Preshv}(X, LH(\C))$ has exact filtered colimits, so does $\mathrm{Shv}(X, LH(\C))$, as sheafification is exact. Hence $\mathrm{Shv}(X, LH(\C))$ is Grothendieck by Lemma \ref{genforshv}, as $LH(\C)$ is elementary abelian by \cite[Proposition 2.1.12]{Schneiders}.
		
		For (iii), we follow the same line of reasoning as in \cite[Proposition 2.3.10]{Schneiders}. By arguments as above, every $\F\in \mathrm{Shv}(X, \C)$ admits a strict epimorphism from some sheaf of the form 
		\begin{equation*}
			\mathcal{P}=\oplus_\I (P_i)_{U_i},
		\end{equation*} 
		where $(P_i)_{i\in \I}$ is a family of projective objects in $\C$ and $U_i\subseteq X$ open. It follows from Lemma \ref{Tacommute} that for any $\A\in \mathrm{Shv}(X, \C)$, $\T(\A, \mathcal{P})\cong (\T_\mathrm{pre}(\A, \oplus (P_i)_{U_i}^\mathrm{pre}))^a$. Thus $\T(-, \mathcal{P})$ is exact.
		
		It remains to prove (iv) and (v). If $\R$ is a monoid in $\mathrm{Shv}(X, \C)$, then $\mathrm{Mod}(\R)$ is complete, cocomplete and quasi-abelian by Lemma \ref{ModR}. By \cite[Proposition 2.1.9]{Schneiders} and the strong exactness of sheafification, $\mathrm{Mod}(\R)$ has strongly exact filtered colimits, so by the dual of \cite[Corollary 1.4.7]{Schneiders}, $LH(\mathrm{Mod}(\R))$ is a cocomplete abelian category, with exact filtered colimits by \cite[Proposition 1.4.17]{Schneiders}. By \cite[Proposition 2.1.18]{Schneiders}, $\mathrm{Mod}(\R)$ admits a strictly generating set and has exact direct sums by \cite[Proposition 2.2.8]{Schneiders}. Thus $LH(\mathrm{Mod}(\R))$ admits a generator by Lemma \ref{genforLH}, and is hence Grothendieck.
		
		For (v), note that if $\mathcal{P}$ is flat as a $\C$-sheaf, then $\T(\R, \mathcal{P})$ is a flat $\R$-module by Lemma \ref{tensorwithfree}.
	\end{proof}
	It follows that in this situation, both $LH(\mathrm{Shv}(X, \C))$ and $LH(\mathrm{Mod}(\R))$ have enough injectives, and $\mathrm{Shv}(X, \C)$ and $\mathrm{Mod}(\R)$ admit flat resolutions for their respective tensor products. 
	
	\subsection{Sheaves on bases}
	
	In this subsection, we provide a brief discussion of bases.
	\begin{defn}
		Let $X$ be a G-topological space with category of admissible opens $\mathrm{Op}(X)$ and admissible coverings $\mathrm{Cov}(U)$, $U\in \mathrm{Op}(X)$. A full subcategory $\B\subseteq \mathrm{Op}(X)$ is called a \textbf{basis} for the G-topology if every $U\in \mathrm{Op}(X)$ admits an admissible covering consisting of objects in $\B$.
	\end{defn}
	
	Note that $\B$ itself need not be the category of admissible opens for some G-topology -- in fact, it is not uncommon for $\B$ to be not closed under intersection. For example, if $X$ is a rigid analytic space, let $X_w$ denote the collection of all affinoid subdomains of $X$. Then $X_w$ is a basis for the strong Grothendieck topology on $X$, but it does not constitute a G-topology if $X$ is not separated. 
	
	If $U$ is an admissible open in $X$, we call an admissible covering of $U$ a $\B$-covering if each piece of the covering is an object of $\B$.
	
	\begin{defn}
		Let $\B$ be a basis for $X$ and let $\C$ be a complete  quasi-abelian category. A \textbf{$\C$-presheaf on $\B$} is a functor $\B^{\mathrm{op}}\to \C$. We say that a $\C$-presheaf $\F$ on $\B$ is a \textbf{sheaf} if for any $U\in \B$, any $\B$-covering $(U_i)_{i\in I}$, and any $\B$-covering $(V_{i, j, k})_{k}$ of $U_i\cap U_j$, the sections  $\F(U)$ form the equalizer of the usual diagram
		\begin{equation*}
			\begin{xy}
				\xymatrix{\prod \F(U_i)\ar@<1ex>[r] \ar[r]& \prod_{i, j, k} \F(V_{i, j, k}).}
			\end{xy}
		\end{equation*}
	\end{defn}
	We denote the category of $\C$-sheaves on $\B$ by $\mathrm{Shv}(\B, \C)$. Note that in the case when $\B$ is in fact a site (so that $\B$ is closed under finite intersections), this is consistent with our earlier definition: if $(U_i)_{i\in I}$ is a $\B$-covering of $U\in \B$ and $\F$ is a $\C$-sheaf on the site $\B$, then  the sheaf condition implies that the morphism 
	\begin{equation*}
		\F(U_i\cap U_j)\to \prod_k \F(V_{i, j, k})
	\end{equation*}
	is a strict monomorphism for any $\B$-covering $(V_{i, j, k})$ of $U_i\cap U_j$. As products preserve strict monomorphisms, this shows that $\F$ is a $\C$-sheaf on the basis $\B$ in the sense of the definition above. The other direction is trivial.
	
	We now fix once again an elementary quasi-abelian category $\C$ which is closed symmetric monoidal with enough flat projectives stable under $T$.
	
	We will repeatedly make use of the following result.
	
	\begin{lem}
		\label{extfrombasis}
		Let $\B$ be a basis for $X$. The restriction functor yields an equivalence of categories
		\begin{equation*}
			\mathrm{Shv}(X, \C)\cong \mathrm{Shv}(\B, \C).
		\end{equation*}
	\end{lem}
	\begin{proof}
		The proofs in \cite[Appendix A]{DcapOne}, \cite[Proposition 9.2.3/1]{BGR} apply in our categorical setting with minor modifications. As we will need the construction later, we sketch the definition of the quasi-inverse.
		
		Let $\F\in \mathrm{Shv}(\B, \C)$. Let $U$ be an admissible open in $X$, $\mathfrak{U}=(U_i)_{i\in I}$ a $\B$-covering of $U$ and $(V_{i, j, k})_k$ a $\B$-covering of $U_i\cap U_j$ for each $i, j$. Note that if $V_{i, j, k}$ admits a $\B$-covering $(W_{i, j, k, l})_l$, then the natural morphism
		\begin{equation*}
			\F(V_{i, j, k})\to \prod_l \F(W_{i, j, k, l})
		\end{equation*}
		is a strict monomorphism due to the sheaf condition. In particular, $\prod_{i, j, k} \F(V_{i, j, k})\to \prod_{i, j, k, l}\F(W_{i, j, k, l})$ is a strict monomorphism, so the equalizer of
		\begin{equation*}
			\prod_i \F(U_i)\to \prod_{i, j, k}\F(V_{i, j, k})
		\end{equation*}
		is isomorphic to the equalizer of
		\begin{equation*}
			\prod_i \F(U_i)\to \prod_{i, j, k, l} \F(W_{i, j, k, l}).
		\end{equation*}
		If now $(V'_{i, j, k'})$ is another $\B$-covering of $U_i\cap U_j$, we can pick $\B$-coverings for each $V_{i, j, k}\cap V'_{i, j, k'}$ and apply the above to deduce that the equalizer of 
		\begin{equation*}
			\prod_i \F(U_i)\to \prod_{i, j, k} \F(V_{i, j, k})
		\end{equation*}
		does not depend on the choice of coverings $(V_{i, j, k})$, but only on $\mathfrak{U}$. We denote this equalizer by $\mathrm{H}^0(\mathfrak{U}, \F)$ and set
		\begin{equation*}
			\F^{\mathrm{ext}}(U):=\mathrm{colim}_{\mathfrak{U}} \mathrm{H}^0(\mathfrak{U}, \F),
		\end{equation*}
		with the colimit ranging over all $\B$-coverings $\mathfrak{U}$ of $U$. The universal property of equalizers makes $\F^{\mathrm{ext}}$ a presheaf on $X$ with the property that $\F^{\mathrm{ext}}|_\B=\F$ thanks to the sheaf condition. 
		
		To verify that $\F^{\mathrm{ext}}$ is a sheaf on $X$, we apply the same argument as in \cite[Proposition 2.2.4]{Schneiders} to reduce to the case $\C=\mathrm{Ab}$, where the argument is well-known, see e.g. \cite[Lemma 9.2.2/3]{BGR}.
	\end{proof}
	
	In practice, the basis $\B$ will usually consist of all affinoid subdomains of a rigid analytic $K$-variety, or a suitable subclass of them. With this in mind, we say that a basis $\B$ is \textbf{locally a site} if for each $U\in \B$, the set $\B_U:=\{V\in \B: V\subseteq U\}$ is closed under finite intersections. In particular, $\B_U$ makes $U$ a G-topological space. 
	
	For example, while $X_w$ is not a site if $X$ is not separated, each affinoid $U\subseteq X$ is separated, so that $U_w$ is both a basis for the strong G-topology on $U$ and a site in its own right.
	
	If $\B$ is locally a site, then the sheafification functors $(-)^a: \mathrm{Preshv}(\B_U, \C)\to \mathrm{Shv}(\B_U, \C)$ for each $U\in \B$ glue to give a sheafification functor $(-)^a: \mathrm{Preshv}(\B, \C)\to \mathrm{Shv}(\B, \C)$ as the left adjoint to the inclusion functor. 
	
	\begin{lem}
		\label{basisandsheafification}
		Let $\B$ be a basis for $X$ which is locally a site. There is a natural isomorphism
		\begin{equation*}
			(\F|_{\B})^a\cong \F^a|_{\B}
		\end{equation*}
		for any $\F\in \mathrm{Preshv}(X, \C)$.
		
		In particular, if $\F\in \mathrm{Preshv}(X, \C)$ such that $\F|_{\B}\in \mathrm{Shv}(\B, \C)$, then $(\F|_{\B})^{\mathrm{ext}}\cong \F^a$ and $\F|_{\B}\cong \F^a|_{\B}$.
	\end{lem}
	\begin{proof}
		This follows immediately from the fact that for any $U\in \B$, the set of $\B$-coverings of $U$ is cofinal amongst all admissible open coverings of $U$.
	\end{proof}
	
	\begin{lem}
		\label{extstrongmonoidal}
		Let $\B$ be a basis for $X$ which is locally a site. The functor $(-)^{\mathrm{ext}}: \mathrm{Shv}(\B, \C)\to \mathrm{Shv}(X, \C)$ is strong symmetric monoidal.
	\end{lem}
	\begin{proof}
		It suffices to show that there are natural isomorphisms
		\begin{equation*}
			\T(\F^{\mathrm{ext}}, \G^{\mathrm{ext}})|_{\B}\cong \T_{\B}(\F, \G)
		\end{equation*}
		for $\F, \G\in \mathrm{Shv}(\B, \C)$, where $\T$ denotes the sheaf tensor product on $X$ and $\T_\B$ the sheaf tensor product on $\B$. As both are obtained as the respective sheafifications of the presheaf tensor product, the statement follows from the previous lemma.
	\end{proof}
	
	Since sheafification is (by definition of the sheaf tensor product) a strong monoidal functor, it follows that if $\R$ is a monoid in $\mathrm{Preshv}(\B, \C)$, then $(\R^a)^{\mathrm{ext}}$ is a monoid in $\mathrm{Shv}(X, \C)$, and to endow an object $\F\in \mathrm{Shv}(X, \C)$ with an $(\R^a)^{\mathrm{ext}}$-module structure, it suffices to specify an $\R$-module structure on $\F|_\B$ with respect to the presheaf tensor product.
	
	We conclude this subsection by discussing how bases may help with the explicit computation of functors.
	
	\begin{lem}
		\label{tensoronqcompact}
		Let $\B$ be a basis for $X$ consisting of quasi-compact admissible opens, and assume that $U_i\cap U_j$ is quasi-compact for any $U_i, U_j\in \B$. If $F:\C\to \C$ is an exact functor commuting with direct sums and $\F\in \mathrm{Shv}(X, \C)$, then the sheaf $(F\F)^a\in \mathrm{Shv}(X, \C)$ associated with
		\begin{equation*}
			U\mapsto F(\F(U)), \ U\in \mathrm{Op}(X)
		\end{equation*}
		satisfies
		\begin{equation*}
			(F\F)^a(U)=F(\F(U))
		\end{equation*}
		for any quasi-compact admissible open subset $U$ of $X$.
		
	\end{lem}
	
	\begin{proof}
		Let $U\in \B$, $\mathfrak{U}=(U_i)$ be a finite covering by objects in $\B$ and for each $i, j$, let $(V_{i, j, k})_k$ be a finite covering of $U_i\cap U_j$ by objects in $\B$. Since finite products agree with finite direct sums, the assumptions on $F$ imply that 
		\begin{equation*}
			0\to F(\F(U))\to \prod F(\F(U_i))\to \prod F(\F(V_{i, j, k}))
		\end{equation*}
		is still strictly exact. By the quasi-compactness assumption, any $\B$-covering admits a finite subcovering, so this shows that $F\circ \F$ defines a sheaf on $\B$, and we can apply Lemma \ref{basisandsheafification} to deduce
		\begin{equation*}
			(F\F)^a(U)=F(\F(U))
		\end{equation*}
		for any $U\in \B$. 
		
		If $U\in \mathrm{Op}(X)$ is quasi-compact, but not necessarily in $\B$, pick again a finite $\B$-covering $(U_i)$ for each $i,j$ a finite $\B$-covering $(V_{i, j, k})_k$ of $U_i\cap U_j$ to deduce in the same way that
		\begin{equation*}
			(F\F)^a(U)=F(\F(U)).\qedhere
		\end{equation*}
	\end{proof}
	
	\begin{lem}
		Let $\B$ be basis for $X$ and suppose that any finite intersection of elements of $\B$ is quasi-compact. Let $F: \C\to \C$ an exact functor commuting with direct sums.
		
		If $U\in \mathrm{Op}(X)$, $\mathfrak{V}=(V_i)$ a finite $\B$-covering of $U$ and $\F\in \mathrm{Shv}(X, \C)$ such that the Cech complex $\check{C}^{\bullet}(\mathfrak{V}, \F)$ is a strict complex, then
		\begin{equation*}
			\check{\mathrm{H}}^j(\mathfrak{V}, (F\circ\F)^a)\cong F(\check{\mathrm{H}}^j(\mathfrak{V}, \F))
		\end{equation*}
		for each $j$.
	\end{lem}
	\begin{proof}
		Since $F$ is exact and commutes with direct sums, the above implies that the Cech complex $\check{C}^{\bullet}(\mathfrak{V}, (F\F)^a)$ is obtained by applying $F$ to $\check{C}^{\bullet}(\mathfrak{V}, \F)$. The result thus follows by exactness of $F$.
	\end{proof}

	\subsection{Derived tensors and Homs}
	Let $X$ be a G-topological space. We suppose that $\C$ is a quasi-abelian category with the following properties:
	\begin{enumerate}[(i)]
		\item $\C$ is an elementary quasi-abelian category.
		\item $\C$ is a closed symmetric monoidal category with enough flat projectives stable under $T$.
		\item Every object of $\C$ is flat, i.e. $T(A, -)$ is exact for all $A\in \C$. 
	\end{enumerate}
	
	In particular, $LH(\C)$ is a closed symmetric monoidal category by Lemma \ref{extendclosed} (with functors $\widetilde{T}$ and $\widetilde{H}$).
	
	We begin by deriving the tensors and homs for $\C$ before passing to the corresponding sheaf categories. Exactness of $T$ immediately yields the derived functor
	\begin{equation*}
		\mathbb{L}T(-, -): \mathrm{D}(\C)\times \mathrm{D}(\C)\to \mathrm{D}(\C).
	\end{equation*}
	We also need to discuss the corresponding tensor product for the left heart.\\
	Note that $\widetilde{T}: LH(\C)\times LH(\C)\to LH(\C)$ is constructed as the zeroth cohomology of $T$, so that our exactness assumption translates into a natural isomorphism $\widetilde{T}(I(A), I(B))\cong IT(A, B)$ for $A, B\in \C$. In particular, if we denote by $S$ the essential image of $I: \C\to LH(\C)$, then $(S, S)$ is $\widetilde{T}$-projective in the sense of \cite[Definition 10.3.9]{KS}. It then follows from \cite[Theorem 14.4.8]{KS} that we have an adjoint pair of derived functors
	\begin{align*}
		\mathbb{L}\widetilde{T}(-, -): \mathrm{D}(LH(\C))\times \mathrm{D}(LH(\C))\to \mathrm{D}(LH(\C))\\
		\mathrm{R}\widetilde{H}(-,-): \mathrm{D}(LH(\C))^\mathrm{op}\times \mathrm{D}(LH(\C))\to \mathrm{D}(LH(\C))
	\end{align*}
	between the unbounded derived categories.
	
	Since $\widetilde{T}(I(A), I(B))\cong IT(A, B)$, the equivalence $\mathrm{D}(\C)\cong \mathrm{D}(LH(\C))$ induced by $I$ identifies $\mathbb{L}T$ with $\mathbb{L}\widetilde{T}$. We also could have used the fact that $LH(\C)$ has enough projectives to derive $\widetilde{T}$, but we have opted for this argument to motivate what follows.
	
	Note that we do not make any claims about $H$ being derivable (reflecting the fact that $\C$ need not have enough injectives, while $LH(\C)$ does). However, if $H$ is right derivable, then it follows from \cite[Proposition 1.3.16]{Schneiders} and uniqueness of adjunction that $I$ identifies $\mathrm{R}H$ with $\mathrm{R}\widetilde{H}$.
	
	We will now employ the same strategy for the categories $\mathrm{Shv}(X, \C)$ and $\mathrm{Shv}(X, LH(\C))$, and finally for categories of modules. We recall from Lemma \ref{shvclosed} that the closed symmetric monoidal structure on $\C$ lifts to a closed symmetric monoidal structure $\T$, $\mathcal{H}$ on $\mathrm{Shv}(X, \C)$, and the closed symmetric monoidal structure on $LH(\C)$ lifts to a closed symmetric monoidal structure $\widetilde{\T}$, $\widetilde{\mathcal{H}}$ on $\mathrm{Shv}(X, LH(\C))$.
	\begin{lem}
		\label{sheaftexact}
		The functor $\T(\mathcal{F}, -)$ is exact for each $\mathcal{F}\in \mathrm{Shv}(X, \C)$.
	\end{lem}
	\begin{proof}
		Let
		\begin{equation*}
			0\to \mathcal{G}_1\to \mathcal{G}_2\to \mathcal{G}_3\to 0
		\end{equation*}
		be a strictly exact sequence in $\mathrm{Shv}(X, \C)$, and let $\mathcal{G}_0\in \mathrm{Preshv}(X, \C)$ be the presheaf cokernel of $\mathcal{G}_1\to \mathcal{G}_2$. As sheafification preserves cokernels, we have $\mathcal{G}_3\cong \mathcal{G}_0^a$.
		
		Then by exactness of $T$, we have that
		\begin{equation*}
			0\to \T_\mathrm{pre}(\mathcal{F}, \mathcal{G}_1)\to\T_\mathrm{pre}(\mathcal{F}, \mathcal{G}_2)\to\T_\mathrm{pre}(\mathcal{F}, \mathcal{G}_0)\to 0
		\end{equation*} 
		is a strictly exact sequence of presheaves, and the result follows from exactness of sheafification and Lemma \ref{Tacommute}.
	\end{proof}
	We thus obtain the corresponding derived functor
	\begin{equation*}
		\mathbb{L}\T(-, -): \mathrm{D}(\mathrm{Shv}(X, \C))\times \mathrm{D}(\mathrm{Shv}(X, \C))\to \mathrm{D}(\mathrm{Shv}(X, \C))
	\end{equation*}
	between the unbounded derived categories.
	
	\begin{lem}
		\label{strictsheaftensor}
		There is a natural isomorphism $\widetilde{\T}(I(\mathcal{F}), I(\mathcal{G}))\cong I\T(\mathcal{F}, \mathcal{G})$ for $\mathcal{F}$, $\mathcal{G}\in \mathrm{Shv}(X, \C)$.
	\end{lem}
	\begin{proof}
		We have the following sequence of natural isomorphisms
		\begin{align*}
			\widetilde{\T}(I(\mathcal{F}), I(\mathcal{G}))&=\left(\widetilde{\T}_\mathrm{pre}(I(\mathcal{F}), I(\mathcal{G})) \right)^a\\
			&\cong \left(I\T_\mathrm{pre}(\mathcal{F}, \mathcal{G}) \right)^a\\
			&\cong I\left( (\T_\mathrm{pre}(\mathcal{F}, \mathcal{G}))^a\right)\\
			&=I\T(\mathcal{F}, \mathcal{G}),
		\end{align*}
		using Lemma \ref{Iacommute} and the earlier observation that $\widetilde{T}(I(A), I(B))\cong IT(A, B)$ for $A, B\in \C$.
	\end{proof}
	
	\begin{prop}
		The functors $\widetilde{\T}$ and $\widetilde{\mathcal{H}}$ can be derived to an adjoint pair of functors
		\begin{align*}
			\mathbb{L}\widetilde{\T}(-, -): \mathrm{D}(\mathrm{Shv}(X, LH(\C)))\times \mathrm{D}(\mathrm{Shv}(X, LH(\C)))\to \mathrm{D}(\mathrm{Shv}(X, LH(\C)))\\
			\mathrm{R}\widetilde{\mathcal{H}}(-,-): \mathrm{D}(\mathrm{Shv}(X, LH(\C)))^\mathrm{op}\times \mathrm{D}(\mathrm{Shv}(X, LH(\C)))\to \mathrm{D}(\mathrm{Shv}(X, LH(\C)))
		\end{align*}
		between the unbounded derived categories.
	\end{prop}
	\begin{proof}
		Let $S$ denote the essential image of the functor 
		\begin{equation*}
			I: \mathrm{Shv}(X, \C)\to \mathrm{Shv}(X, LH(\C)).
		\end{equation*} 
		By Lemma \ref{sheaftexact} and Lemma \ref{strictsheaftensor}, $(S,S)$ is $\widetilde{\T}$-projective, so that we can apply again \cite[Theorem 14.4.8]{KS}.
	\end{proof}
	
	As before, the equivalence $\mathrm{D}(\mathrm{Shv}(X, \C))\cong \mathrm{D}(\mathrm{Shv}(X, LH(\C)))$ identifies $\mathbb{L}\T$ with $\mathbb{L}\widetilde{\T}$ thanks to Lemma \ref{strictsheaftensor}. If $\mathcal{H}$ is derivable, then $\mathrm{R}\mathcal{H}$ can be identified with $\mathrm{R}\widetilde{\mathcal{H}}$.
	
	Now let $\R$ be a monoid in $\mathrm{Shv}(X, \C)$. It follows from Proposition \ref{LHIR} that $I(\R)$ is a monoid in $\mathrm{Shv}(X, LH(\C))$.
	\begin{prop}
		\label{deriverelthom}
		The functors $\widetilde{\T}_{I(\R)}$ and $\widetilde{\mathcal{H}}_{I(\R)}$ can be derived to functors
		\begin{align*}
			\mathbb{L}\widetilde{\T}_{I(\R)}: \mathrm{D}(\mathrm{Mod}(I(\R^\mathrm{op})))\times \mathrm{D}(\mathrm{Mod}(I(\R)))\to \mathrm{D}(\mathrm{Shv}(X, LH(\C)))\\
			\mathrm{R}\widetilde{\mathcal{H}}_{I(\R)}: \mathrm{D}(\mathrm{Mod}(I(\R)))\times \mathrm{D}(\mathrm{Mod}(I(\R)))\to \mathrm{D}(\mathrm{Shv}(X, LH(\C)))
		\end{align*}
		between the unbounded derived categories.
		
		$\mathbb{L}\widetilde{\T}_{I(\R)}$ is a left adjoint of $\mathrm{R}\widetilde{\mathcal{H}}$, while $\mathrm{R}\widetilde{\mathcal{H}}_{I(\R)}$ is a right adjoint to $\mathbb{L}\widetilde{\T}$.
	\end{prop}
	\begin{proof}
		By Theorem \ref{Shvnice}, there are enough flat left $I(\R)$-modules and enough flat right $I(\R)$-modules. If $F^l$ resp. $F^r$ denotes the full subcategory of flat left resp. flat right $I(\R)$-modules, then $(F^r, F^l)$ is $\widetilde{\T}_{I(\R)}$-projective, by the same argument as in \cite[Proposition 18.5.4]{KS}. We then apply \cite[Theorem 14.4.8]{KS} to the adjoint pairs $(\widetilde{\T}_{I(\R)}, \widetilde{\mathcal{H}})$ and $(\widetilde{\T}, \widetilde{\mathcal{H}}_{I(\R)})$ from Lemma \ref{shvclosed}.
	\end{proof}
	It is straightforward that the above results also generalize to suitable bimodule categories. We omit the details here.
	
	We remark that the argument in \cite[Proposition 18.5.4]{KS} invokes the snake lemma, which relies on the underlying category being abelian. In particular, there does not seem to be a straightforward way to apply the same reasoning directly on the quasi-abelian level. We will now derive $\T_\R$ by passing through the left heart.
	
	Note that by definition, $\T_\R(\M, \N)$ is the coequalizer (in $\mathrm{Shv}(X, \C)$) of
	\begin{equation*}
		\begin{xy}
			\xymatrix{
				\T(\M, \R, \N)\ar@<1ex>[r] \ar[r]& \T(\M, \N)
			}
		\end{xy}
	\end{equation*}
	for $\M\in \mathrm{Mod}(\R^\mathrm{op})$, $\N\in \mathrm{Mod}(\R)$, while $\widetilde{\T}_{I(\R)}(I(\M), I(\N))$ is the coequalizer (in $\mathrm{Shv}(X, LH(\C))$) of
	\begin{equation*}
		\begin{xy}
			\xymatrix{
				I\T(\M, \R, \N)\ar@<1ex>[r] \ar[r]& I\T(\M, \N).
			}
		\end{xy}
	\end{equation*}
	In particular, $C\widetilde{\T}_{I(\R)}(I(\M), I(\N))\cong \T_\R(\M, \N)$. Thus this description makes it clear that
	\begin{equation*}
		\widetilde{\T}_{I(\R)}(I(\M), I(\N))\cong I\T_\R(\M, \N)
	\end{equation*}
	if and only if $\widetilde{T}_{I(\R)}(I(\M), I(\N))$ is in the essential image of $I$, i.e. if and only if the morphism
	\begin{equation*}
		\T(\M, \R, \N)\to \T(\M, \N)
	\end{equation*}
	given by $\T(\mathrm{act}_\M, \N)-\T(\M, \mathrm{act}_\N)$ is strict.
	\begin{prop}
		\label{deriverelqabtensor}
		The functor $\T_\R$ is explicitly left derivable to a functor
		\begin{equation*}
			\mathbb{L}\T_\R(-, -): \mathrm{D}^-(\mathrm{Mod}(\R^\mathrm{op}))\times \mathrm{D}^-(\mathrm{Mod}(\R))\to \mathrm{D}^-(\mathrm{Shv}(X, \C))
		\end{equation*} 
		such that there is a natural isomorphism $I\mathbb{L}\T_\R(\M^\bullet, \N^\bullet)\cong \mathbb{L}\widetilde{\T}_{I(\R)}(I(\M^\bullet), I(\N^\bullet))$ for $\M^\bullet\in\mathrm{D}^-(\mathrm{Mod}(\R^\mathrm{op}))$, $\N^\bullet\in \mathrm{D}^-(\mathrm{Mod}(\R))$. 
	\end{prop}
	\begin{proof}
		We say that a left $\R$-module $\N$ is $I$-flat if the functor
		\begin{equation*}
			\widetilde{\T}_{I(\R)}(I(-), I(\N)): \mathrm{Mod}(\R^\mathrm{op})\to \mathrm{Shv}(X, \C) 
		\end{equation*}
		is exact, analogously for right modules. Let $\mathcal{F}^l$ be the full subcategory of $\mathrm{Mod}(\R)$ consisting of all $I$-flat left $\R$-modules $\N$ with the property that for all right $\R$-modules $\M$, the natural morphism
		\begin{equation*}
			\widetilde{\T}_{I(\R)}(I(\M), I(\N))\cong I\T_\R(\M, \N)
		\end{equation*}
		is an isomorphism (equivalently, $\widetilde{\T}_{I(\R)}(I(\M), I(\N))$ is in the essential image of $I$). Define $\mathcal{F}^r$ mutatis mutandis as a full subcategory of $\mathrm{Mod}(\R^\mathrm{op})$. We claim that $(\mathcal{F}^r, \mathcal{F}^l)$ is $\T_\R$-projective.
		
		Firstly, for any $\N\in \mathrm{Mod}(\R)$, $\T(\R, \N)\in \mathcal{F}^l$ by Lemma \ref{tensorwithfree}, as $I\T(\R, \N)\cong \widetilde{\T}(I(\R), I(\N))$ by Lemma \ref{strictsheaftensor};  analogously for $\mathcal{F}^r$. Hence there are enough objects in $\F^r$ and $\F^l$.
		
		Secondly, if
		\begin{equation*}
			0\to \N_1\to \N_2\to \N_3\to 0
		\end{equation*}
		is strictly exact in $\mathrm{Mod}(\R)$ and $\N_2, \N_3\in \mathcal{F}^l$, then 
		\begin{equation*}
			0\to I(\N_1)\to I(\N_2)\to I(\N_3)\to 0
		\end{equation*}
		is exact in $\mathrm{Mod}(I(\R))$. It follows from the same argument as in part (c) of the proof of \cite[Proposition 18.5.4]{KS} that $\N_1$ is $I$-flat. In particular, if $\M$ is a right $\R$-module, then
		\begin{equation*}
			0\to \widetilde{\T}_{I(\R)}(I(\M), I(\N_1))\to\widetilde{\T}_{I(\R)}(I(\M), I(\N_2))\to \widetilde{\T}_{I(\R)}(I(\M), I(\N_3))\to 0
		\end{equation*}
		is exact by the same argument as in \cite[Proposition 18.5.4.(ii)]{KS}. Since the second term is in the essential image of $I$, so is the first term by \cite[Proposition 1.2.29.(b)]{Schneiders}. Therefore, $\N_1\in \mathcal{F}^l$. The same argument applies for right modules.
		
		Lastly, let 
		\begin{equation*}
			0\to \N_1\to \N_2\to \N_3\to 0
		\end{equation*}
		be strictly exact in $\mathcal{F}^l$ and let $\M\in \mathcal{F}^r$. We need to verify that
		\begin{equation*}
			0\to \T_\R(\M, \N_1)\to \T_\R(\M, \N_2)\to \T_\R(\M, \N_3)\to 0
		\end{equation*} 
		is strictly exact. By \cite[Corollary 1.2.28]{Schneiders}, it suffices to check exactness after applying $I$. But applying $I$ yields the sequence
		\begin{equation*}
			0\to \widetilde{\T}_{I(R)}(I(\M), I(\N_1))\to \widetilde{\T}_{I(\R)}(I(\M), I(\N_2))\to \widetilde{\T}_{I(\R)}(I(\M), I(\N_3))\to 0
		\end{equation*}
		by definition of $\mathcal{F}^l$, and this is exact by $I$-flatness.
		
		Thus $(\mathcal{F}^r, \mathcal{F}^l)$ is $\T_\R$-projective, and we obtain the desired derived functor. A similar argument shows that $(I(\F^r), I(\F^l))$ is $\widetilde{\T}_{I(\R)}$-projective. Thus 
		\begin{equation*}
			I\mathbb{L}\T_\R(\M, \N)\cong \mathbb{L}\widetilde{\T}_{I(\R)}(I(\M), I(\N)),
		\end{equation*} 
		as required.
	\end{proof}
	While we expect that Spaltenstein's formalism for unbounded derived functors as in \cite{Spalt}, \cite[chapter 14]{KS} can also be developed for quasi-abelian categories, we will content ourselves with using $\mathbb{L}\widetilde{\T}_{I(\R)}$ and the equivalence $\mathrm{D}(\C)\cong \mathrm{D}(LH(\C))$. The above then justifies that this does not affect our calculations. 
	\begin{lem}
		If $\mathcal{H}_\R$ is right derivable, then the equivalence 
		\begin{equation*}
			\mathrm{D}(\mathrm{Mod}(\R))\cong \mathrm{D}(\mathrm{Mod}(I(\R)))
		\end{equation*} 
		from Proposition \ref{LHIR} identifies $\mathrm{R}\mathcal{H}_\R$ with $\mathrm{R}\widetilde{\mathcal{H}}_{I(\R)}$.
	\end{lem}
	\begin{proof}
		By \cite[Proposition 1.3.16]{Schneiders}, if $\mathcal{H}_\R$ is right derivable, then $\mathrm{R}\mathcal{H}_\R$ is a right adjoint for $\mathbb{L}\T$, which under $I$ gets identified with $\mathbb{L}\widetilde{\T}$ by the previous Proposition, so that the claim follows from uniqueness of adjunction.
	\end{proof}

	\subsection{Derived direct and inverse image}
	We keep our assumptions on $\C$ from the beginning of subsection 3.6. Let $f:X\to Y$ be a morphism of G-topological spaces. In this subsection, we define derived direct and inverse image functors both for the category of sheaves and for the category of modules over a sheaf monoid.
	
	We begin by introducing the \textbf{presheaf inverse image}
	\begin{align*}
		f^{-1}_\mathrm{pre}: \mathrm{Preshv}(Y, \C)\to \mathrm{Preshv}(X, \C)\\
		f^{-1}(\mathcal{F})(U)=\underset{V\supseteq f(U)}{\mathrm{colim}} \mathcal{F}(V)
	\end{align*}
	and the \textbf{direct image functor}
	\begin{align*}
		f_*: \mathrm{Preshv}(X, \C)\to \mathrm{Preshv}(Y, \C)\\
		f_*(\mathcal{F})(U)=\mathcal{F}(f^{-1}U).
	\end{align*}
	It is immediate from the definition that $f_*$ sends sheaves on $X$ to sheaves on $Y$.
	\begin{lem}
		There is a natural isomorphism
		\begin{equation*}
			\mathrm{Hom}_{\mathrm{Preshv}(X, \C)}(f^{-1}_\mathrm{pre}(\mathcal{F}), \mathcal{G})\cong \mathrm{Hom}_{\mathrm{Preshv}(Y, \C)}(\mathcal{F}, f_*\mathcal{G})
		\end{equation*}
		for $\mathcal{F}\in \mathrm{Preshv}(Y, \C)$, $\mathcal{G}\in \mathrm{Preshv}(X, \C)$.
	\end{lem}
	\begin{proof}
		This is \cite[Theorem 2.3.3.(i)]{KS}.
	\end{proof}
	In particular, $f^{-1}_\mathrm{pre}$ is strongly right exact. Since $\Gamma(V, -)$ is strongly left exact for any admissible open $V\subseteq Y$ by the adjunction
	\begin{equation*}
		\mathrm{Hom}_{\mathrm{Preshv}(V, \C)}(E_V^\mathrm{pre}, \mathcal{F})\cong \mathrm{Hom}_\C(E, \Gamma(V, \mathcal{F})),  
	\end{equation*}
	it follows again from the strong exactness of filtered colimits that $f^{-1}_\mathrm{pre}$ is strongly exact.
	
	Since $(-)^a$ is also strongly exact, $f^{-1}(\mathcal{F})=(f^{-1}_\mathrm{pre})(\mathcal{F})^a$ defines a strongly exact functor $f^{-1}: \mathrm{Shv}(Y, \C)\to \mathrm{Shv}(X, \C)$ which is the left adjoint of $f_*$.
	
	Replacing $\C$ by $LH(\C)$, we also have at our disposal the analogous functors between the categories of $LH(\C)$-sheaves, which we denote by $f^{-1}_{LH}$ and $f_{*, LH}$ to avoid confusion.
	\begin{lem}
		\leavevmode	
		\begin{enumerate}[(i)]
			\item There is a natural isomorphism $If_*(\mathcal{F})\cong f_{*, LH}I(\mathcal{F})$ for $\mathcal{F}\in \mathrm{Shv}(X, \C)$.
			\item There is a natural isomorphism $If^{-1}(\mathcal{F})\cong f^{-1}_{LH}I(\mathcal{F})$ for $\mathcal{F}\in \mathrm{Shv}(Y, \C)$.
		\end{enumerate}
	\end{lem}
	\begin{proof}
		(i) is immediate from the definitions. For (ii), note that $I: \C\to LH(\C)$ commutes with filtered colimits by \cite[Proposition 2.1.16]{Schneiders}, so $I$ commutes with $f^{-1}_\mathrm{pre}$. Since $I$ also commutes with sheafification by Lemma \ref{Iacommute}, the result follows.
	\end{proof}
	
	Since $f^{-1}$ is strongly exact, it admits a derived functor $f^{-1}: \mathrm{D}(\mathrm{Shv}(Y, \C))\to \mathrm{D}(\mathrm{Shv}(X, \C))$ on the unbounded derived categories, and in our earlier notation we have $\widetilde{f^{-1}}\cong f^{-1}_{LH}$.
	
	As $\mathrm{Shv}(X, \C)$ is a Grothendieck category, we can apply \cite[Tags 079P, 070Y]{stacksproj} to obtain a right derived functor $\mathrm{R}f_{*, LH}: \mathrm{D}(LH(\mathrm{Shv}(X, \C)))\to \mathrm{D}(LH(\mathrm{Shv}(Y, \C)))$ on the unbounded derived categories.
	
	By \cite[Theorem 14.4.5.(b)]{KS}, $\mathrm{R}f_{*, LH}$ is a right adjoint to $f^{-1}$. It therefore follows from \cite[Proposition 1.3.16]{Schneiders} that whenever $f_*$ is explicitly right derivable, then $\mathrm{R}f_*\cong \mathrm{R}f_{*, LH}$ under the natural identification of $\mathrm{D}(\mathrm{Shv}(X, \C))$ and $\mathrm{D}(LH(\mathrm{Shv}(X, \C)))\cong \mathrm{D}(\mathrm{Shv}(X, LH(\C)))$.
	
	We will not address the question of whether $f_*$ itself is derivable and usually write $\mathrm{R}f_*$ instead of $\mathrm{R}f_{*, LH}$ in an abuse of notation.
	
	We now turn to the corresponding functors between module categories.
	\begin{lem}
		\label{invimdistributes}
		There is a natural isomorphism $f^{-1}(\T(\mathcal{F}, \mathcal{G}))\cong \T(f^{-1}\mathcal{F}, f^{-1}\mathcal{G})$ for $ \mathcal{F}, \mathcal{G}\in \mathrm{Shv}(Y, \C)$.
	\end{lem}   
	\begin{proof}
		First note that $\T_\mathrm{pre}$ and $f^{-1}_\mathrm{pre}$ commute, since $T$ commutes with colimits. It follows from the adjunctions that $f^{-1}(\mathcal{F}^a)\cong (f^{-1}_\mathrm{pre}(\mathcal{F}))^a$ for any $\mathcal{F}\in \mathrm{Preshv}(X, \C)$, and
		\begin{equation*}
			\T(\mathcal{F}^a, \mathcal{G}^a)\cong (\T_\mathrm{pre}(\mathcal{F}, \mathcal{G}))^a
		\end{equation*}
		by Lemma \ref{Tacommute}, so there is a natural identification of both sides in the statement of the lemma with
		\begin{equation*}
			(f^{-1}_\mathrm{pre}\T_\mathrm{pre}(\mathcal{F}, \mathcal{G}))^a.\qedhere
		\end{equation*}
	\end{proof}
	
	Thus, $f^{-1}: \mathrm{Shv}(Y, \C)\to \mathrm{Shv}(X, \C)$ becomes a strong symmetric monoidal functor, and Proposition \ref{laxmodulechange} implies that if $\R$ is a monoid on $Y$, then $f^{-1}\R$ is a monoid on $X$ in a natural way, and $f^{-1}$ gives rise to a functor $\mathrm{Mod}(\R)\to \mathrm{Mod}(f^{-1}\R)$, which we also denote by $f^{-1}$. 
	
	Similarly, the adjunction between $f^{-1}$ and $f_*$ together with Lemma \ref{invimdistributes} yields natural morphisms
	\begin{equation*}
		\T(f_*\F, f_*\G)\to f_*(\T(\F, \G))
	\end{equation*}
	for $\F, \G\in \mathrm{Shv}(X, \C)$, making $f_*$ a lax symmetric monoidal functor.
	
	Hence, if $\R$ is a monoid on $X$, then $f_*\R$ is a monoid on $Y$, and $f_*$ gives rise to a functor 
	\begin{equation*}
		f_*: \mathrm{Mod}(\R)\to \mathrm{Mod}(f_*\R).
	\end{equation*}
	If $\R$ is a monoid on $Y$, then the adjunction yields a natural monoid morphism $\R\to f_*f^{-1}\R$, and thus $f_*$ yields a functor $\mathrm{Mod}(f^{-1}\R)\to \mathrm{Mod}(\R)$ in this case.
	\begin{lem}
		\label{Rmodimageadj}
		Let $\R$ be a monoid in $\mathrm{Shv}(Y, \C)$. Then there is a natural isomorphism
		\begin{equation*}
			\mathrm{Hom}_{\mathrm{Mod}(f^{-1}\R)}(f^{-1}\M, \N)\cong \mathrm{Hom}_{\mathrm{Mod}(\R)}(\M, f_*\N)
		\end{equation*}
		for any $\M\in \mathrm{Mod}(\R)$, $\N\in \mathrm{Mod}(f^{-1}\R)$.
	\end{lem}
	\begin{proof}
		By adjunction and Lemma \ref{invimdistributes}, 
		\begin{equation*}
			\mathrm{Hom}_{\mathrm{Shv}(X, \C)}(\T(f^{-1}\R, f^{-1}\M), \N)\cong \mathrm{Hom}_{\mathrm{Shv}(Y, \C)}(\T(\R, \M), f_*\N),
		\end{equation*}
		so the isomorphism follows from the description of $\mathrm{Hom}_\R$ as an equalizer.
	\end{proof}
	
	Note that by Lemma \ref{ModR}.(ii), $f^{-1}$ is also strongly exact as a functor between the module categories, and the functor $\widetilde{f^{-1}}$ is nothing but 
	\begin{equation*}
		f^{-1}_{LH}: \mathrm{Mod}(I(\R))\to \mathrm{Mod}(f^{-1}_{LH}(I(\R)))
	\end{equation*}
	under the equivalence from Proposition \ref{LHIR}.
	
	If $\R$ is a monoid on $X$, then
	\begin{equation*}
		f_{*,LH}: \mathrm{Mod}(I(\R))\to \mathrm{Mod}(f_{*, LH}I(\R))
	\end{equation*}
	can be derived to
	\begin{equation*}
		\mathrm{R}f_{*, LH}: \mathrm{D}(\mathrm{Mod}(I(\R)))\to \mathrm{D}(\mathrm{Mod}(f_{*, LH}I(\R)))
	\end{equation*}
	by \cite[Tags 079P, 070Y]{stacksproj}. Invoking Proposition \ref{LHIR} and \cite[Proposition 1.2.32]{Schneiders}, this corresponds to a triangulated functor
	\begin{equation*}
		\mathrm{R}f_*: \mathrm{D}(\mathrm{Mod}(\R))\to \mathrm{D}(\mathrm{Mod}(f_*(\R)))
	\end{equation*}
	such that $\mathrm{R}f_*$ is right adjoint to $f^{-1}$. In particular, if $f_*: \mathrm{Mod}(\R)\to \mathrm{Mod}(f_*(\R))$ is indeed derivable, it agrees with $\mathrm{R}f_*$ as defined here.
	
	To obtain the usual composition results for direct and inverse image functors, we follow \cite[section 5]{Spalt}.
	
	Recall that a complex $\F^\bullet$ in the homotopy category $K(\mathrm{Shv}(X, LH(\C)))$ is called \textbf{K-injective} (or homotopically injective) if for any acyclic complex $\A^\bullet$, the complex $\mathrm{Hom}^\bullet(\A^\bullet, \F^\bullet)$ given by
	\begin{equation*}
		\mathrm{Hom}^n(\A^\bullet, \F^\bullet)=\prod_i \mathrm{Hom}(\A^i, \F^{i+n}),
	\end{equation*}
	with differential $\mathrm{d}f_n:=\mathrm{d}_\F\circ f-(-1)^nf\circ \mathrm{d}_\A$, is acyclic. By \cite[Theorem 14.3.1]{KS}, any Grothendieck abelian category allows for resolutions by K-injective complexes, and these are used to compute right derived functors.
	
	As in Lemma \ref{genforshv}, let $G\in LH(\C)$ be a generator, and note that by Lemma \ref{genforLH}, we can assume that $G$ lies in the essential image of $I$. Given an open subset $V\subseteq X$, let $G_V$ denote the extension by zero already encountered in the proof of Lemma \ref{genforshv}. Let $\mathcal{P}$ denote the full subcategory of $K(\mathrm{Shv}(X, LH(\C)))$ consisting of bounded above complexes whose terms are direct sums of $G_V$ for various $V\subseteq X$.
	
	\begin{defn}
		A complex $\F^\bullet\in K(\mathrm{Shv}(X, LH(\C)))$ is called \textbf{K-limp} if for any acyclic $\A^\bullet\in \mathcal{P}$, the complex $\mathrm{Hom}^\bullet(\A^\bullet, \F^\bullet)$ is acyclic.
	\end{defn} 
	Note that any K-injective complex is a fortiori K-limp.
	\begin{lem}
		\label{limpres}
		\leavevmode
		\begin{enumerate}[(i)]
			\item K-limp complexes form a full triangulated subcategory containing all K-injective complexes.
			\item Let $\R$ be a monoid in $\mathrm{Shv}(X, \C)$. If $\M^\bullet\in K(\mathrm{Mod}(I(\R)))$ is K-injective, then $\mathrm{forget}(\M^\bullet)\in K(\mathrm{Shv}(X, LH(\C)))$ is K-limp.
			\item Let $V\subseteq X$ be open and let $(\F^\bullet, d)\in K(\mathrm{Shv}(X, LH(\C)))$ be an acyclic K-limp complex. Then $\Gamma(V, \F^\bullet)$ is an acyclic complex.
			\item For any morphism $f: X\to Y$, the derived functor $\mathrm{R}f_*$ on $\mathrm{D}(\mathrm{Shv}(X, LH(\C)))$ can be computed on K-limp complexes.
		\end{enumerate}
	\end{lem}
	\begin{proof}
		For (i), we only need to show a 2-out-of-3 property, which is immediate from the definition.
		
		For (ii), let $\A^\bullet\in \mathcal{P}$ be acyclic. In this case, $\widetilde{T}(I(\R), \A^\bullet)$ is acyclic, since $\A^i$ is in the essential image of $I$ for each $i$, on which $\widetilde{T}(I(\R), -)$ is exact by assumption and Lemma \ref{strictsheaftensor}. Note that by adjunction
		\begin{equation*}
			\mathrm{Hom}_{\mathrm{Shv}(X, LH(\C))}(\A^i, \mathrm{forget}(\M^{i+n}))\cong \mathrm{Hom}_{I(\R)}(\widetilde{\T}(I(\R), \A^i), \M^{i+n}),
		\end{equation*}
		so the result follows.  
		
		For (iii), we follow the argument in \cite[Proposition 5.16]{Spalt}. Consider the complex
		\begin{equation*}
			\begin{xy}
				\xymatrix{\hdots \ar[r] &\Gamma(V, \F^{n-1})\ar[r]^{\mathrm{d}_{n-1}(V)} &\Gamma(V, \F^n)\ar[r]^{\mathrm{d}_n(V)} &\Gamma(V, F^{n+1})\ar[r]& \hdots}
			\end{xy},
		\end{equation*}
		and let $S=\mathrm{ker}(\mathrm{d}_n(V))\in LH(\C)$. Let $\mathcal{S}'$ be the extension by zero of the constant sheaf $S$ on $V$, and let $j: \mathcal{S}\to \F^n$ denote the image of the morphism $\mathcal{S}'\to \F^{n}$ induced by adjunction. Note that $j$ factors through $\mathrm{ker}\ \mathrm{d}_n\cong \mathrm{Im}\ \mathrm{d}_{n-1}$ by construction.
		
		Set $\mathcal{S}^{n-1}=\F^{n-1}\times_{\F^n} \mathcal{S}$, $\mathcal{S}^i=\F^i$ for $i<n-1$, so that the complex
		\begin{equation*}
			\begin{xy}
				\xymatrix{\mathcal{S}^\bullet:=(\hdots \ar[r]& \F^{n-2}\ar[r]& \mathcal{S}^{n-1}\ar[r]& \mathcal{S}\ar[r]&0\ar[r]& 0\ar[r]&\hdots)}
			\end{xy}
		\end{equation*}
		is acyclic. By construction, $j$ yields a chain map $\mathcal{S}^\bullet\to \F^\bullet$. By definition of $\mathcal{P}$, there is now a complex $\A^\bullet$ in $\mathcal{P}$, concentrated in degrees $\leq n$, and a quasi-isomorphism $\A^\bullet\to \mathcal{S}^\bullet$, with $\A^n\to \mathcal{S}$ an epimorphism. In particular, $\A^\bullet$ is acyclic.
		
		As we assume $\F^\bullet$ to be K-limp, the complex $\mathrm{Hom}^\bullet(\A^\bullet, \F^\bullet)$ is acyclic, so that the chain map $\A^\bullet\to \mathcal{S}^\bullet\to \F^\bullet$ given above, which we denote by $f$, is in the image of the differential. Therefore, there exists a morphism $g_n: \A^n\to \F^{n-1}$ such that $f_n=\mathrm{d}_\F\circ g_n$, since $\A^{n+1}=0$. But then $\mathcal{S}$ is a subobject of the image of $\mathrm{d}$, so applying $\Gamma(V, -)$ yields that $S=\mathrm{ker}(\mathrm{d}_n(V))$ is equal to $\mathrm{Im}(d_{n-1}(V))$, as required.
		
		For (iv), we note that (iii) implies immediately that for any acyclic K-limp complex $\F^\bullet$, $f_*(\F^\bullet)$ is acyclic, so the result follows from \cite[Proposition 10.3.3]{KS}.   
	\end{proof}
	
	\begin{cor}
		\label{pushforwardforget}
		Let $\R$ be a monoid in $\mathrm{Shv}(X, \C)$. There is a natural isomorphism
		\begin{equation*}
			\mathrm{R}f_*(\mathrm{forget}\M^\bullet)\cong \mathrm{forget}\mathrm{R}f_*(\M^\bullet)
		\end{equation*}
		for $\M^\bullet\in \mathrm{D}(\mathrm{Mod}(I(\R)))\cong \mathrm{D}(\mathrm{Mod}(\R))$.
	\end{cor}
	\begin{proof}
		This is immediate from Lemma \ref{limpres}.(ii) and (iv) above.
	\end{proof}
	
	\begin{lem}
		\label{catcomposition}
		Let $f: X\to Y$, $g: Y\to Z$ be morphisms of G-topological spaces. There are natural isomorphisms
		\begin{equation*}
			f^{-1}g^{-1}(\M^\bullet)\cong (gf)^{-1}(\M^\bullet)
		\end{equation*}
		for $\M^\bullet\in \mathrm{D}(\mathrm{Shv}(Z, \C))$
		and
		\begin{equation*}
			\mathrm{R}g_*\mathrm{R}f_*(\M^\bullet)\cong \mathrm{R}(gf)_*(\M^\bullet)
		\end{equation*}
		for $\M^\bullet\in \mathrm{D}(\mathrm{Shv}(X, \C))$.
		
		The same holds for modules over monoids in $\C$-valued sheaves.
	\end{lem}
	\begin{proof}
		We use the equivalence $\mathrm{D}(\mathrm{Shv}(X, \C))\cong \mathrm{D}(\mathrm{Shv}(X, LH(\C)))$, which we have already seen to be compatible with the inverse and direct image functors.
		
		If $\M\in \mathrm{Shv}(X, LH(\C))$, it is clear by definition that
		\begin{equation*}
			g_*f_*(\M)\cong (gf)_*(\M),
		\end{equation*} 
		and thus by adjunction
		\begin{equation*}
			f^{-1}g^{-1}\N\cong (gf)^{-1}\N 
		\end{equation*}
		naturally for any $\N\in \mathrm{Shv}(Z, LH(\C))$. As inverse images are exact, this implies the same result for complexes of sheaves, and the corresponding result for derived direct images follows yet again by adjunction.
		
		In the case of $\R$-modules, it suffices to check that the natural morphism $\mathrm{R}(gf)_*\to\mathrm{R}g_*\mathrm{R}f_*$ becomes an isomorphism after applying the forgetful functor. This follows immediately from the above and Corollary \ref{pushforwardforget}.
	\end{proof}

	\section{Analytic background: Complete bornological vector spaces}
	\subsection{Bornological vector spaces}
	\label{BcK}
	Recall that we have fixed a complete nonarchimedean field $K$ of mixed characteristic $(0, p)$, with valuation ring $R$ and $\pi\in R$ satisfying $0<|\pi|<1$.\\
	In this subsection, we recall the definition of bornological $K$-vector spaces, its relation to locally convex topological vector spaces, and the closed symmetric monoidal structure on the category $\h{\B} c_K$ of complete bornological $K$-vector spaces. Our main references are \cite{Houzel}, \cite{ProsmansSchneiders}, \cite{Bamdagger}, \cite{BamStein}. Some of the references only address the case of complex vector spaces, but the relevant proofs carry over to our nonarchimedean setting without change.
	\begin{defn}
		\label{defborn}
		A \textbf{bornology} on a $K$-vector space $V$ is a collection $\mathcal{B}$ of subsets of $V$ such that
		\begin{enumerate}[(i)]
			\item if $B\in \mathcal{B}$ and $B'\subseteq B$, then $B'\in \mathcal{B}$.
			\item if $v\in V$, then $\{v\}\in \mathcal{B}$.
			\item $\mathcal{B}$ is closed under taking finite unions.
			\item if $B\in \mathcal{B}$ and $\lambda\in K$, then $\lambda \cdot B\in \mathcal{B}$.
			\item if $B\in \mathcal{B}$, then $R\cdot B$, the $R$-submodule spanned by $B$, is an element of $\mathcal{B}$.
		\end{enumerate} 
		We call $V$, or rather $(V, \mathcal{B})$, a \textbf{bornological $K$-vector space} (of convex type). Elements of $\mathcal{B}$ are called \textbf{bounded subsets}.
	\end{defn}
	A $K$-linear map $f: V\to W$ between bornological $K$-vector spaces is called \textbf{bounded} if $f$ sends bounded subsets to bounded subsets. We denote the category of bornological $K$-vector spaces together with bounded $K$-linear maps by $\mathcal{B}c_K$.
	
	$\B c_K$ is complete and cocomplete by \cite[Proposition 1.9]{ProsmansSchneiders}. The underlying vector space of a (co)limit is the corresponding (co)limit in the category $\mathrm{Vect}_K$ of abstract $K$-vector spaces. We refer to \cite[Proposition 1.2, Definition 1.4]{ProsmansSchneiders} for details on the natural bornologies on (co)products and (co)kernels.
	\begin{lem}
		The category $\B c_K$ is quasi-abelian.
	\end{lem}
	\begin{proof}
		See \cite[Proposition 1.8]{ProsmansSchneiders}.
	\end{proof}
	
	We briefly recall the relation between $\B c_K$ and the category of locally convex topological $K$-vector spaces $LCVS_K$ (see \cite[p.56, p.93]{Houzel}).
	
	If $V$ is a semi-normed $K$-vector space, its collection of bounded subsets with respect to the semi-norm defines a bornology on $V$. More generally, if $V$ is a locally convex topological vector space defined through a family of semi-norms $q_i$, then its collection of bounded subsets (i.e. those $B\subseteq V$ such $q_i(B)\subseteq \mathbb{R}$ is bounded for each $i$) defines a bornology on $V$. This is called the \textbf{von Neumann bornology} of $V$. We thus obtain a functor $V\mapsto V^b$ from $LCVS_K$ to $\B c_K$.
	
	Conversely, if $(V, \B)$ is a bornological $K$-vector space, we can endow it with a locally convex topology whose basis of open neighbourhoods of $0$ consists of all bornivorous subsets, i.e. those $L\subseteq V$ such that for any $B\in \B$ there exists $\lambda\in K$ such that $B\subseteq \lambda \cdot L$.
	
	This provides us with a functor $V\mapsto V^t$ from $\B c_K$ to $LCVS_K$.
	
	It is worth noting that $(-)^b$ and $(-)^t$ are not equivalences, but rather $(-)^b$ is the right adjoint of $(-)^t$ (see \cite[p. 93, Proposition 3]{Houzel}). Moreover, we have the following.
	\begin{lem}[{\cite[p. 102, p. 109]{Houzel}}]
		\label{metricequiv}
		Let $\mathscr{M}$ be the full subcategory of $LCVS_K$ consisting of metrizable spaces, and let $\mathscr{M}^b$ denote the essential image of $\mathscr{M}$ in $\B c_K$ under $(-)^b$.
		
		Then $(-)^b$ gives an equivalence of categories between $\mathscr{M}$ and $\mathscr{M}^b$, with quasi-inverse $(-)^t$.
	\end{lem}
	
	In this way, we will often view normed spaces or Fr\'echet spaces as objects in $\B c_K$. We point out however that $(-)^b$ is not exact on $\mathscr{M}$: if $f: V\to W$ is a morphism of Fr\'echet spaces, then $\mathrm{coker}(f^b)$ need not lie in $\mathscr{M}^b$ and is therefore not necessarily isomorphic to $(\mathrm{coker} f)^b$.
	\begin{lem}
		\label{borniexact}
		The functor $(-)^b: LCVS_K\to \B c_K$ preserves limits. It is exact on the full subcategory of semi-normed vector spaces. 
	\end{lem}
	\begin{proof}
		Right adjoints preserve limits. If $f: V\to W$ is a strict epimorphism of semi-normed spaces, then the unit ball $W^\circ$ is contained in some $f(\pi^rV^\circ)$, i.e. $f$ is also bornologically strict.
	\end{proof}
	As for locally convex topological vector spaces, there are notions of separatedness and completeness for bornological vector spaces (see \cite[sections 4 and 5]{ProsmansSchneiders}), which we recall now.
	
	Let $V$ be a bornological $K$-vector space and let $L$ be a bounded $R$-submodule. Denoting by $V_L$ the vector subspace spanned by $L$, we can endow $V_L$ with the bornology generated by $L$ (i.e. a subset is bounded if and only if it is contained in $\pi^nL$ for some $n\in \mathbb{Z}$ -- this is the von Neumann bornology of the gauge semi-norm associated to $L$, \cite[remark before Lemma 2.2]{SchneiderNFA}). By axiom (v) in Definition \ref{defborn}, we have
	\begin{equation*}
		V\cong \varinjlim_L V_L
	\end{equation*}
	in the category of bornological $K$-vector spaces, where the colimit is taken over all bounded $R$-submodules $L$.
	
	We call a bornology \textbf{separated} if it is generated by $R$-submodules which are $\pi$-adically separated. Note that in this case, the colimit above can be written as a colimit of normed spaces $V_L$ (equipped with the gauge norm).
	
	The inclusion of separated bornological vector spaces into $\B c_K$ admits a left adjoint, the separation functor $\mathrm{sep}$ (\cite[p. 48]{Houzel}).
	\begin{defn}
		A bornological $K$-vector space $V$ is \textbf{complete} if any bounded subset $B$ is contained in some bounded $R$-submodule $L$ such that $V_L$ is Banach.
	\end{defn}
	We denote the full subcategory of $\B c_K$ consisting of complete bornological $K$-vector spaces by $\h{\mathcal{B}}c_K$.
	
	By \cite[Proposition 5.11]{ProsmansSchneiders}, the natural inclusion of $\h{\B}c_K\to \B c_K$ admits a left adjoint, the \textbf{completion} functor $\h{(-)}: \mathcal{B}c_K\to \h{\mathcal{B}}c_K$, given by
	\begin{equation*}
		\h{V}=\mathrm{sep}(\varinjlim_L \h{V_L}),
	\end{equation*}
	where $\h{V_L}$ is the Banach completion of $V_L$ for $L$ a bounded $R$-submodule of $V$. For instance, if $V$ is a normed $K$-vector space, then $\h{V^b}\cong \h{V}^b$ is its Banach completion, viewed as a complete bornological space.
	
	By \cite[Corollary 1.18]{BambozziCGT}, an object $V$ of $\mathscr{M}$ is complete if and only if $V^b$ is a complete bornological vector space. In particular, we can view the category of Fr\'echet spaces as a full subcategory of $\h{\B} c_K$. Still, the relation between bornological completion and topological completion might be subtle in general.
	\begin{lem}
		\label{tcompl}
		Let $V\in LCVS_K$. Then $\h{V}$ is the Hausdorff completion of $(\h{V^b})^t$.
		
		If $V\in \mathscr{M}$, then $(\h{V^b})^t$ is Hausdorff, and the natural morphism $(\h{V^b})^t\to \h{V}$ (and hence $\h{V^b}\to \h{V}^b$) is injective.
	\end{lem}
	\begin{proof}
		First note that the adjunctions imply immediately that $\h{V}$ is the Hausdorff completion of $(\h{V^b})^t$.\\
		If $V$ is metrisable, then $V^b$ is proper in the sense of \cite[p. 112, Definition 5]{Houzel}. Let $\mathcal{B}$ be the set of bounded, $\pi$-adically separated $R$-submodules of $V^b$. Then by \cite[p. 113, Remark after Proposition 17]{Houzel}, $\varinjlim_{B\in \mathcal{B}} \h{V_B}$ is already separated, so 
		\begin{equation*}
			\h{V^b}\cong \varinjlim \h{V_B}
		\end{equation*}
		in $\mathcal{B}c_K$ by definition of the completion functor.
		
		So since $(-)^t$ commutes with colimits, we have $(\h{V^b})^t\cong \varinjlim \h{V_B}$ in $LCVS_K$.
		
		We now show that the map $(\h{V^b})^t\to \h{V}$ is injective, which implies that $(\h{V^b})^t$ is Hausdorff by \cite[Proposition 7.5]{SchneiderNFA}. Suppose $v\in (\h{V^b})^t$ is in the kernel. There exists some $B\in \mathcal{B}$ such that $v\in \h{V_B}$, and hence there is a sequence $(v_n)_n$ in $V_B$ tending to $v$. If the topology of $V$ is given by a family of semi-norms $|-|_n$, there exist $\alpha_n\in \mathbb{R}_{\geq 0}$ such that
		\begin{equation*}
			B\subseteq\{x\in V: |x|_n\leq \alpha_n\}.
		\end{equation*} 
		In particular, as $(v_n)$ is Cauchy in the normed space $V_B$, it is Cauchy in $V$, and its limit in $\h{V}$ is the image of $v$. But then by assumption on $v$, $v_n\to 0$ in $\h{V}$ and hence in $V$. So $(v_n)$ is a Cauchy sequence in $V$ tending to zero, and hence $v_n\to 0$ in $V^b$ by \cite[Proposition 1.17]{BambozziCGT}. But then $v_n\to v$ and $v_n\to 0$ in $\h{V^b}$, and since $\h{V^b}$ is a separated bornological space, we have $v=0$. Thus $(\h{V^b})^t\to \h{V}$ is injective.
		
		If we apply $(-)^t$ to the natural morphism $\h{V^b}\to \h{V}^b$, we obtain the morphism $(\h{V^b})^t\to \h{V}$, so as a linear map the two morphisms coincide. Hence injectivity of $\h{V^b}\to \h{V}^b$ follows from the above.
	\end{proof}

	\begin{prop}
		\label{hBciselementary}
		The category $\h{\B} c_K$ is a quasi-abelian category. It is complete and cocomplete. Direct sums are strongly exact and direct products are exact.
	\end{prop}
	\begin{proof}
		See \cite[Proposition 5.6]{ProsmansSchneiders}.
	\end{proof}
	As the completion functor makes $\h{\B}c_K$ a reflective subcategory of $\B c_K$, we know that the limit of a diagram in $\h{\B}c_K$ can also be computed in $\B c_K$, and the colimit of a diagram in $\h{\B}c_K$ is the completion of the colimit of the corresponding diagram in $\B c_K$. In particular, the definition of completeness yields immediately that any object of $\h{\B}c_K$ is a colimit of Banach spaces.
	
	\begin{lem}
		\label{cokerwithb}
		Let $f: V\to W$ be a morphism of Fr\'echet spaces, and consider the corresponding morphism $f^b: V^b\to W^b$ in $\h{\B}c_K$. Then the natural morphism $\mathrm{coker}(f^b)\to (\mathrm{coker}f)^b$ is a bounded bijection, where $\mathrm{coker} f$ denotes the cokernel in the category of Fr\'echet spaces. 
	\end{lem}
	\begin{proof}
		Write $U$ for the set-theoretic image of $f$. By \cite[Proposition 4.6.(b), Proposition 5.6.(a)]{ProsmansSchneiders}, the underlying vector space of $\mathrm{coker}(f^b)$ is the quotient of $W$ by the bornological closure of $U$ (see \cite[Definition 4.3]{ProsmansSchneiders}). Since $W$ is Fr\'echet, the bornological closure of $U$ agrees with the bornological limit points of $U$ (\cite[Proposition 3.15]{BamStein}), which are the topological limit points of $U$ by \cite[Proposition 1.15]{BambozziCGT}.
		
		Thus the underlying vector space of $\mathrm{coker}(f^b)$ is $W/\overline{U}$, and the natural bounded morphism $\mathrm{coker}(f^b)\to (\mathrm{coker}f)^b$ is simply the identity map. 
	\end{proof}
	Note that $\mathrm{coker}(f^b)\to (\mathrm{coker}f)^b$ is indeed an isomorphism if $V$ and $W$ are Banach. A large part of subsection 5.1 will be concerned with finding further conditions when this bijection is an isomorphism.
	
	Given $(V, \B)$, $(W, \B')\in \mathcal{B}c_K$, we can endow the tensor product $V\otimes_K W$ with a bornology as follows.
	
	If $B_1\in \B$ and $B_2\in \B'$ are bounded $R$-submodules, we can view $B_1\otimes_R B_2$ as a subset of $V\otimes_K W$ via the natural inclusion. We call $V\otimes_K W$, endowed with the bornology generated by all $B_1\otimes B_2$ as $(B_1, B_2)$ ranges over all pairs of bounded $R$-submodules, the \textbf{projective bornological tensor product}. 
	\begin{lem}
		Let $U, V, W$ be in ${\B}c_K$. If $f: V\times W\to U$ is a bounded bilinear map, then there exists a unique bounded linear map $\phi: V\otimes_K W\to U$ such that $f$ is the composition of the natural map $V\times W\to V\otimes_K W$ and $\phi$.
	\end{lem}
	\begin{proof}
		See \cite[p.174, Th\'eoreme 1.b]{Houzel}.
	\end{proof}
	We denote by $V\h{\otimes}_K W\in \h{\mathcal{B}}c_K$ the completion of $V\otimes_K W$ with respect to the projective tensor product bornology.
	
	Note that it follows immediately from the universal property above that there is a natural isomorphism $V\h{\otimes}_K W\cong \h{V}\h{\otimes}_K \h{W}$.
	
	For $V, W\in {\B}c_K$, we can endow $\mathrm{Hom}_{\B c_K}(V, W)$ with a bornology by declaring a set of morphisms $F$ to be bounded if it is equibounded, i.e. we require that
	\begin{equation*}
		\underset{\phi\in F}{\cup}\phi(B)\subseteq W
	\end{equation*}
	is a bounded subset of $W$ whenever $B$ is bounded in $V$. 
	
	It follows from \cite[p.97, Proposition 9]{Houzel} that this turns $\mathrm{Hom}_{\B c_K}(V, W)$ into a complete bornological vector space whenever $W$ is complete.
	\begin{thm}
		\label{ptisclosedmonoidal}
		\leavevmode
		\begin{enumerate}[(i)]
			\item $\otimes_K$ and $\mathrm{Hom}_{\B c_K}$ give $\B c_K$ the structure of a closed symmetric monoidal category.
			\item $\h{\otimes}_K$ and $\mathrm{Hom}_{\h{\B} c_K}$ give $\h{\B} c_K$ the structure of a closed symmetric monoidal category.
		\end{enumerate}
	\end{thm}
	\begin{proof}
		For (i), see \cite[Proposition 3.56]{Bamdagger}. (ii) then follows directly from (i), as completion is a left adjoint to the inclusion $\h{\B}c_K\to \B c_K$ (see \cite[Remark 3.58]{Bamdagger}).
	\end{proof}

	Note that if $V$, $W$ are metrisable locally convex topological vector spaces, then the tensor product $V\otimes_K W$ carries the structure of a locally convex topological vector space (again called the projective tensor product topology), which is metrisable by \cite[p. 188, Corollaire 1]{Houzel}, making its Hausdorff completion $V\h{\otimes}_K W$ a Fr\'echet space (\cite[remark after Proposition 17.6]{SchneiderNFA}).
	
	In the case of normed spaces, it is essentially tautological to compare the topological approach with the bornological one:
	
	\begin{lem}
		\label{normedtensor}
		\leavevmode
		\begin{enumerate}[(i)]
			\item If $V, W$ are normed spaces, then $V\otimes_KW$ is normed and the natural morphism $V^b\otimes_KW^b\to (V\otimes_KW)^b$ is an isomorphism. Moreover, the natural morphism 
			\begin{equation*}
				V^b\h{\otimes}_K W^b\to (V\h{\otimes}_K W)^b
			\end{equation*}
			is an isomorphism in $\h{\B}c_K$.
			\item If $V$ is Banach and
			\begin{equation*}
				0\to W_1\to W_2\to W_3\to 0
			\end{equation*}
			is a short strictly exact sequence of Banach spaces, then
			\begin{equation*}
				0\to V\h{\otimes}_K W_1\to V\h{\otimes}_K W_2\to V\h{\otimes}_K W_3\to 0
			\end{equation*}
			is strictly exact.
		\end{enumerate}
	\end{lem}
	\begin{proof}
		\begin{enumerate}[(i)]
			\item By definition, $V\otimes_KW$ is topologized by the lattice $V^\circ\otimes_R W^\circ$ and hence semi-normed. But we have already remarked that $V\otimes_KW$ is metrisable, so it is normed, and the identification of tensor products follows directly from the definition. Moreover, it is clear from the definition that for any normed space $V$, the completion of $V^b$ is precisely the Banach completion of $V$. This now yields the corresponding isomorphism for the completed tensor products.
			\item We can assume without loss of generality that $W_1^\circ=W_1\cap W_2^\circ$ and $W_3^\circ=W_2^\circ/W_1^\circ$.  Now the tensor product $V\otimes_K W_i$ is the normed space whose norm is the gauge norm for $V^\circ\otimes_R W_i^\circ$. Since the respective unit balls are $\pi$-torsionfree and thus flat $R$-modules, $V^\circ \otimes_RW_i^\circ$ embeds into $V\otimes_KW$. Thus the short exact sequence
			\begin{equation*}
				0\to V^\circ\otimes_RW_1^\circ\to V^\circ\otimes_R W_2^\circ\to V^\circ\otimes_RW_3^\circ\to 0
			\end{equation*}
			is exact, which implies that the sequence
			\begin{equation*}
				0\to V\otimes_KW_1\to V\otimes_KW_2\to V\otimes_KW_3\to 0
			\end{equation*} 
			is a strict exact sequence of normed $K$-vector spaces. But by \cite[Proposition 1.1.9/5]{BGR}, completion is exact on normed vector spaces, and the result follows.
		\end{enumerate}
	\end{proof}
	
	More generally, let $V$, $W$ be metrisable locally convex topological vector spaces. To compare $V^b{\otimes}_K W^b$ and $(V{\otimes}_K W)^b$, we first recall the following definition.
	
	\begin{defn}[{see \cite[Definition 4.2.1]{Schikhof}}]
		\leavevmode
		\begin{enumerate}[(i)]
			\item A normed $K$-vector space $E$ is said to be \textbf{of countable type} if it contains a countable subset spanning a dense subspace.
			\item A semi-norm $q$ on a $K$-vector space $E$ is called of countable type if the Hausdorff quotient $E_q=E/q^{-1}(0)$, viewed as a normed space in the obvious way, is of countable type.
			\item A locally convex $K$-vector space $E$ is said to be of countable type if each continuous semi-norm on $E$ is of countable type.
		\end{enumerate}
	\end{defn}
	For example, the Fr\'echet algebra $\O(\mathbb{A}_K^{d, \mathrm{an}})=\varprojlim_n K\langle \pi^nx_1, \hdots, \pi^n x_d\rangle$ is easily seen to be of countable type for any $d\geq 0$.
	
	\begin{lem}
		\label{Frtensor}
		Let $V, W$ be metrisable locally convex topological $K$-vector spaces, and suppose that at least one of the following is satisfied:
		\begin{enumerate}[(i)]
			\item $K$ is spherically complete.
			\item $V$ is of countable type.
		\end{enumerate}
		Then the natural morphism 
		\begin{equation*}
			V^b\otimes_KW^b\to (V\otimes_KW)^b 
		\end{equation*}
		is an isomorphism. 
	\end{lem}
	\begin{proof}
		This is \cite[Theorem 10.4.13]{Schikhof} together with \cite[Theorem 4.4.3]{Schikhof}.
	\end{proof} 
	\begin{lem}
		Let $V$, $W$ be metrisable locally convex topological $K$-vector spaces, and suppose that at least one of the following is satisfied:
		\begin{enumerate}[(i)]
			\item $K$ is spherically complete.
			\item $V$ is of countable type.
		\end{enumerate}
		Then the natural morphism $V^b\h{\otimes}_KW^b\to (V\h{\otimes}_KW)^b$ is injective.
	\end{lem}
	\begin{proof}
		As $V\otimes_K W$ is metrisable, this is a special case of Lemma \ref{tcompl}, noting that
		\begin{equation*}
			\h{(V\otimes W)^b}\cong \h{V^b\otimes_K W^b}=V^b\h{\otimes}_K W^b
		\end{equation*}
		by the Lemma above.
	\end{proof}

	We will discuss more general conditions for an isomorphism 
	\begin{equation*}
		V^b\h{\otimes}W^b\cong (V\h{\otimes}W)^b
	\end{equation*}
	in subsection 5.2.

	We now wish to show that $\h{\B}c_K$ has enough flat projectives stable under $\h{\otimes}_K$. Most of this can already be pieced together from e.g. \cite{ProsmansSchneiders} and \cite[Appendix A]{BBK}, but we give a slightly different presentation.
	
	For this purpose, we introduce the following class of spaces:
	
	Given a set $X$ and a Banach $K$-vector space $V$, we let $c_0(X, V)$ denote the space of all functions $\phi: X\to V$ such that for any $\epsilon>0$, the set
	\begin{equation*}
		\{x\in X: |\phi(x)|\geq \epsilon\}
	\end{equation*}
	is a finite set. In particular, each $\phi\in c_0(X, V)$ has bounded image, and $c_0(X, V)$ becomes a Banach space under the sup norm. 
	
	We abbreviate $c_0(X, K)$ to $c_0(X)$.
	
	\begin{lem}
		\label{czeroprops}
		\leavevmode
		\begin{enumerate}[(i)]
			\item There is a natural isomorphism
			\begin{equation*}
				c_0(X)\h{\otimes}_K V\cong c_0(X, V)
			\end{equation*}
			for any set $X$ and any Banach space $V$.
			\item There is a natural isomorphism
			\begin{equation*}
				c_0(X)\h{\otimes}_K c_0(Y)\cong c_0(X\times Y)
			\end{equation*}
			for any sets $X, Y$.
		\end{enumerate}
	\end{lem}
	\begin{proof}
		Let $X$ be a set and $V$ be a Banach space.
		
		There is a natural morphism of Banach spaces $c_0(X)\h{\otimes}_K V\to c_0(X, V)$, which maps the unit ball of the left-hand side (which is the $\pi$-adic completion of $c_0(X)^\circ\otimes_R V^\circ$) to the unit ball of the right-hand side, $c_0(X, V)^\circ$. We now note that for any $n\geq 1$, $c_0(X, V)^\circ/\pi^n$ consists of all functions $\phi: X\to V^\circ/\pi^n$ with finite support, i.e. $\phi(x)=0$ for all but finitely many $x\in X$. It is then straightforward to see that the map between the unit balls is an isomorphism mod $\pi^n$ for each $n$, and (i) follows.
		
		Applying (i) in the case when $V=c_0(Y)$ for some set $Y$, we immediately obtain
		\begin{equation*}
			c_0(X)\h{\otimes}_K c_0(Y)\cong c_0(X, c_0(Y))
		\end{equation*}
		via the natural morphism. We now verify via the usual argument that the right-hand side is naturally isomorphic to $c_0(X\times Y)$. 
		
		Indeed, given $\phi: X\to c_0(Y)$ in $c_0(X, c_0(Y))$, let $\phi_x: Y\to K$ denote the function $\phi(x)$, an element of $c_0(Y)$. Then $\Phi: X\times Y\to K$, sending $(x, y)$ to $\phi_x(y)$, is a well-defined map from $X\times Y$ to $K$. For $\epsilon>0$, let $X_\epsilon\subseteq X$ be a finite set such that $|\phi_x|<\epsilon$ for all $x\in X\setminus X_\epsilon$. For each $x\in X_\epsilon$, let $Y_x\subseteq Y$ be the (finite) set of all $y\in Y$ such that $|\phi_x(y)|\geq \epsilon$. Then $|\Phi(x, y)|<\epsilon$ for all $(x, y)$ outside of the finite set $\{(x, y): x\in X_\epsilon, y\in Y_x\}$, so $\Phi\in c_0(X\times Y)$.
		
		Conversely, if $\Phi: X\times Y\to K$ is an element of $c_0(X\times Y)$, let $\epsilon >0$ and let $X_\epsilon\subseteq X$ denote the finite set
		\begin{equation*}
			X_\epsilon=\{x\in X: |\Phi(x, y)|\geq \epsilon \ \text{for some } y\in Y\},
		\end{equation*}
		analogously $Y_\epsilon\subseteq Y$. For each $x\in X$, the function $\Phi(x, -): Y\to K$ is in $c_0(Y)$, as it takes values less then $\epsilon$ outside of $Y_\epsilon$. Thus $x\mapsto \Phi(x, -)$ defines a map from $X$ to $c_0(Y)$, which by construction takes values less than $\epsilon$ outside of $X_\epsilon$. So $\Phi\mapsto (x\mapsto \Phi(x, -))$ is a bounded linear map $c_0(X\times Y)\to c_0(X, c_0(Y))$, which is straightforwardly verified to be the inverse of the natural morphism $c_0(X, c_0(Y))\to c_0(X\times Y)$ given above. Thus, $c_0(X, c_0(Y))\cong c_0(X\times Y)$, and we have proven (ii).
	\end{proof}
	
	\begin{lem}
		\label{flatproj}
		The object $c_0(X)$ is a strongly flat projective object in $\h{\B}c_K$ for any set $X$.
	\end{lem}
	\begin{proof}
		Let $X$ be a set, which we view as a bounded subset of $c_0(X)$ by identifying $x\in X$ with the corresponding characteristic function.
		
		We first prove that $c_0(X)$ is strongly flat. Since completed tensor products are strongly right exact, we only need to show that $c_0(X)\h{\otimes}_K-$ preserves kernels.
		
		If $V\in \h{\B}c_K$, let $C(X, V)$ denote the vector space of all functions $f: X\to V$ such that there exists a bounded $R$-submodule $B\subseteq V$ such that for each $n$, $f(x)$ is contained in $\pi^nB$ for almost all $x\in X$. This carries the bornology given by declaring a collection $S$ of functions to be a bounded subset if there exists a bounded $R$-submodule $B\subseteq V$ such that $f(X)\subseteq B$ for all $f\in S$. In other words, $C(X, V)=\varinjlim c_0(X, V_B)\in \B c_K$, where the limit ranges over all bounded (without loss of generality, complete) $R$-submodules of $V$. It follows from the completeness of $V$ that $C(X, V)$ is already complete, and thus $C(X, V)$ is also the colimit of the $c_0(X, V_B)$ in $\h{\B}c_K$.
	
		But now we can write $V=\varinjlim V_B$ as the colimit of Banach spaces $V_B$ contained in $V$. Since $\h{\otimes}$ commutes with colimits, Lemma \ref{czeroprops}.(i) implies that 
		\begin{equation*}
			c_0(X)\h{\otimes}_KV\cong \varinjlim c_0(X, V_B)\cong C(X, V). 
		\end{equation*}
		This description now makes it immediate that $c_0(X)\h{\otimes}_K-$ preserves kernels in $\h{\B}c_K$: if $f: V\to W$ is a morphism, then $c_0(X)\h{\otimes}_K \mathrm{ker}f\cong C(X, \mathrm{ker}f)$, which is indeed the kernel of $C(X, V)\to C(X, W)$ since $\mathrm{ker}f\to V$ is a strict monomorphism.
		
		Next, we show that $c_0(X)$ is also projective. It is easy to see that bounded linear maps $c_0(X)\to V$ with $V$ Banach are in bijective correspondence with functions $X\to V$ such that the image of $X$ is bounded (cf. \cite[beginning of chapter 10]{SchneiderNFA}). If $V$ is an arbitrary complete bornological $K$-vector space and $f:X\to V$ is a function with bounded image, let $B$ be a bounded $R$-submodule of $V$ such that $f(X)\subseteq B$ and $V_B$ is Banach. Then $f$ gives rise to a morphism $c_0(X)\to V_B$ by the above, and hence to a morphism $c_0(X)\to V$. So even for general $V$, a function $X\to V$ with bounded image gives rise to a morphism $c_0(X)\to V$.
		
		Thus, if $f: V\to W$ is a strict epimorphism and $g: c_0(X)\to W$ is a morphism, then by strictness $V$ has a bounded subset $B$ such that each element of $g(X)$ has a preimage in $B$, yielding the desired morphism $c_0(X)\to V$. Hence, $c_0(X)$ is indeed projective.
	\end{proof}
	
	\begin{cor}
		\label{hBcflatproj}
		The category $\h{\B}c_K$ has enough flat projectives stable under $\h{\otimes}_K$.
	\end{cor}
	\begin{proof}
		Let $\mathcal{P} \subseteq \h{\B}c_K$ denote the full subcategory consisting of objects which are of the form $\oplus_{i\in I} c_0(X_i)$ for some sets $X_i$, $i\in I$. By the above, any $P\in \mathcal{P}$ is flat and projective. Since any Banach space $V$ admits a strict epimorphism $c_0(V^\circ)\to V$, it follows that any $V\in \h{\B}c_K$ admits a strict epimorphism from some object of $\mathcal{P}$, namely $\oplus_B c_0(V_B^\circ)\to V$.
		
		Moreover, we have seen in Lemma \ref{czeroprops}.(ii) that $c_0(X)\h{\otimes}_K c_0(Y)\cong c_0(X\times Y)$ for all sets $X, Y$, so that $\mathcal{P}$ is stable under $\h{\otimes}_K$. Thus $\h{\B}c_K$ has enough flat projectives stable under $\h{\otimes}_K$.
	\end{proof}
	
	We warn the reader that the proof in \cite[Theorm 3.50]{BamStein} is flawed (as we will see in the next subsection), so that we do not expect an arbitrary object in $V\in \h{\B}c_K$ to be flat.
	
	\begin{cor}
		\label{hBcquasiel}
		The category $\h{\B}c_K$ is a quasi-elementary category.
	\end{cor}
	\begin{proof}
		It remains to show that $c_0(X)$ is small for any set $X$. If $V_i$ is a collection of complete bornological vector spaces, then $\oplus V_i\in \h{\B}c_K$ has as underlying vector space the abstract direct sum $\oplus_iV_i$, and a set $B\subseteq \oplus_i V_i$ is bounded if there exist finitely many $i_1, \hdots, i_n$ and bounded subsets $B_{i_j}\subseteq V_{i_j}$ such that $B\subseteq \oplus_j B_{i_j}$.
		
		If $f: c_0(X)\to \oplus_i V_i$ is bounded, then $f(X)$ is a bounded set and hence contained in some finite direct sum $\oplus_j B_{i_j}$ as above. But then $f$ factors through $\oplus_j V_{i_j}$, and thus
		\begin{equation*}
			\mathrm{Hom}_{\h{\B} c_K}(c_0(X), \oplus_i V_i)\cong \oplus_i \mathrm{Hom}_{\h{\B}c_K}(c_0(X), V_i),
		\end{equation*}
		as required.
	\end{proof}
	
	We warn the reader that $\h{\B}c_K$ is however not an elementary category, contrary to the claim in \cite[Lemma 3.53]{Bamdagger}. Once again, this will become clear in the next subsection.
	
	
	
	\subsection{Comparison with $\mathrm{Ind}(\mathrm{Ban}_K)$}
	
	In this subsection, we will discuss how to embed $\h{\B}c_K$ into an elementary quasi-abelian category without changing the left heart.
	
	Firstly, we know from \cite[Proposition 2.1.12]{Schneiders} and Corollary \ref{hBcquasiel} that the left heart $LH(\h{\B}c_K)$ is an elementary abelian category. 
	
	Moreover, Corollary \ref{hBcflatproj} ensures that we can apply Lemma \ref{extendclosed} to endow $LH(\h{\B}c_K)$ with a closed symmetric monoidal structure, which we denote by $\widetilde{\otimes}_K$ and $\widetilde{H}$. Explicitly, $-\widetilde{\otimes}_K-=\mathrm{H}^0(- \h{\otimes}_K^{\mathbb{L}}-)$ and $\widetilde{H}(-, -)=\mathrm{H}^0(\mathrm{RHom}_{\h{\B}c_K}(-, -))$.
	
	We now collect various statements indicating to what extent the embedding $I: \h{\B}c_K\to LH(\h{\B}c_K)$ respects the various structures.
	
	\begin{lem}
		\label{IforhBc}
		\leavevmode
		\begin{enumerate}[(i)]
			\item The functor $I: \h{\B}c_K\to LH(\h{\B}c_K)$ is exact, fully faithful, and commutes with limits and direct sums.
			\item $I$ is lax symmetric monoidal.
			\item The left adjoint $C: LH(\h{\B}c_K)\to \h{\B}c_K$ is strong symmetric monoidal.
			\item If $V, W$ are Banach spaces, then the natural morphism
			\begin{equation*}
				I(V)\widetilde{\otimes}_K I(W)\to I(V\h{\otimes}_K W)
			\end{equation*}
			is an isomorphism. Moreover,
			\begin{equation*}
				V\h{\otimes}^{\mathbb{L}}_KW\cong I(V\h{\otimes}_K W).
			\end{equation*}
			\item $LH(\h{\B}c_K)$ has enough flat projectives stable under $\widetilde{\otimes}_K$.
		\end{enumerate}
	\end{lem}
	
	\begin{proof}
		(i) is \cite[Corollary 1.2.28]{Schneiders} together with \cite[Proposition 2.1.15]{Schneiders}. For (ii), we have already remarked after Proposition \ref{laxmodulechange} that $I$ is lax symmetric monoidal. (iii) follows directly from the definition of $\widetilde{\otimes}_K$ as the zeroth cohomology of $\h{\otimes}_K^{\mathbb{L}}$, as $C$ preserves cokernels. 
		
		We now prove (iv). As $W$ is a Banach space, there exists a strict epimorphism $c_0(W^\circ)\to W$. As the kernel is in turn a Banach space, we can find a set $X$ such that $W$ is the cokernel of a strict morphism $c_0(X)\to c_0(W^\circ)$. Thus by definition,
		\begin{equation*}
			I(V)\widetilde{\otimes}_K I(W)=[V\h{\otimes}_K c_0(X)\to V\h{\otimes}_K c_0(W^\circ)].
		\end{equation*}
		But by Lemma \ref{normedtensor}.(ii), $V\h{\otimes}_K -$ is exact on Banach spaces, so $I(V)\widetilde{\otimes}_K I(W)\cong I(V\h{\otimes}_K W)$.
		
		Moreover, if we let $W'$ denote the kernel of $c_0(W^\circ)\to W$, then the sequence
		\begin{equation*}
			0\to I(W')\to I(c_0(W^\circ))\to I(W)\to 0
		\end{equation*}
		is exact in $LH(\h{\B}c_K)$. Since $c_0(W^\circ)$ is strongly flat by Lemma \ref{flatproj}, we know that $I(c_0(W^\circ))$ is flat, so that $V\h{\otimes}_K^{\mathbb{L}} c_0(W^\circ)\cong I(V)\widetilde{\otimes}_K I(c_0(W^\circ))$. The long exact sequence of cohomology obtained by applying $V\h{\otimes}^{\mathbb{L}}_K-$ together with the observation above that
		\begin{equation*}
			0\to I(V)\widetilde{\otimes}_K I(W')\to I(V)\widetilde{\otimes}_K I(c_0(W^\circ))\to I(V)\widetilde{\otimes}_K I(W)\to 0
		\end{equation*}
		is exact (by Lemma \ref{normedtensor}.(ii) and the above) thus yields that $\mathrm{H}^{-1}(I(V)\h{\otimes}^{\mathbb{L}}_K I(W))=0$. Since $W'$ is also a Banach space, this shows inductively that $I(V)\h{\otimes}^{\mathbb{L}}_K I(W)$ is concentrated in degree $0$.
		
		For (v), note that any $V=(V^{-1}\to V^0)\in LH(\h{\B}c_K)$ admits an epimorphism $I(\oplus_i c_0(X_i))\to I(V^0)\to V$ for some collection of sets $X_i$. By Lemma \ref{flatproj} and \cite[Proposition 1.3.24]{Schneiders}, $I(\oplus_i c_0(X_i))$ is projective. Since $I$ commutes with direct sums and $c_0(X)$ is strongly flat for any set $I$, $I(\oplus c_0(X))$ is a direct sum of flat objects in $LH(\h{\B}c_K)$. As direct sums are exact by \cite[Proposition 2.1.15.(ii)]{Schneiders}, this shows that objects of the form $I(\oplus_i c_0(X_i))$ are flat projective. Moreover, we have
		\begin{align*}
			I(\oplus_i c_0(X_i))\widetilde{\otimes}_K I(\oplus_j c_0(Y_j))&\cong (\oplus_i I(c_0(X_i)))\widetilde{\otimes}_K (\oplus_j I(c_0(Y_j)))\\
			&\cong \oplus_{i, j} I(c_0(X_i)\h{\otimes}_K c_0(Y_j))\\
			&\cong \oplus_{i, j} I(c_0(X_i\times Y_j))\\&\cong I(\oplus_{i, j} c_0(X_i, Y_j))
		\end{align*}
		by (iv) above and Lemma \ref{czeroprops}.(ii). Hence $LH(\h{\B}c_K)$ has enough flat projectives stable under $\widetilde{\otimes}_K$.
	\end{proof}
	
	We now turn to the category $\mathrm{Ind}(\mathrm{Ban}_K)$ of Ind-Banach spaces.
	
	We refer to \cite[subsection 2.1]{Bamdagger} for basic facts on Ind-categories. Note that objects of $\mathrm{Ind}(\mathrm{Ban}_K)$ can be written as
	\begin{equation*}
		``\varinjlim" V_i
	\end{equation*} 
	for some diagram $V: I\to \mathrm{Ban}_K$ indexed by a small filtered category $I$, and morphisms are given by
	\begin{equation*}
		\mathrm{Hom}_{\mathrm{Ind}(\mathrm{Ban}_K)} (`` \varinjlim" V_i, ``\varinjlim" W_j)=\varprojlim_i \varinjlim_j \mathrm{Hom}_{\mathrm{Ban}_K}(V_i, W_j).
	\end{equation*}
	
	\begin{lem}
		\label{indbanelementary}
		The category $\mathrm{Ind}(\mathrm{Ban}_K)$ is an elementary quasi-abelian category.
	\end{lem}
	\begin{proof}
		As we have already seen that for any set $X$, $c_0(X)$ is a projective object in $\h{\B}c_K$ and thus a fortiori in $\mathrm{Ban}_K$, it follows from \cite[Proposition 2.1.17]{Schneiders} that $\mathrm{Ind}(\mathrm{Ban}_K)$ is elementary quasi-abelian. 
	\end{proof}
	
	The closed symmetric monoidal structure on $\mathrm{Ban}_K$ now extends straightforwardly to a closed symmetric monoidal structure on $\mathrm{Ind}(\mathrm{Ban}_K)$, by setting
	\begin{equation*}
		`` \varinjlim" V_i\overset{\rightarrow}{\otimes}_K ``\varinjlim" W_j:=\underset{i, j}{``\varinjlim"} V_i\h{\otimes}_K W_j.
	\end{equation*}
	and
	\begin{equation*}
		H(``\varinjlim" V_i, ``\varinjlim" W_j):=\varprojlim_i \varinjlim_j \mathrm{Hom}_{\mathrm{Ban}_K}(V_i, W_j)\in \mathrm{Ind}(\mathrm{Ban}_K),
	\end{equation*}
	which makes sense as $\mathrm{Hom}_{\mathrm{Ban}_K}(V_i, W_j)$ carries the structure of a Banach space and $\mathrm{Ind}(\mathrm{Ban}_K)$ is complete and cocomplete.
	
	\begin{prop}
		\label{indbanflatproj}
		\leavevmode
		\begin{enumerate}[(i)]
			\item Every object in $\mathrm{Ind}(\mathrm{Ban}_K)$ is flat.
			\item $\mathrm{Ind}(\mathrm{Ban}_K)$ has enough flat projectives stable under $\overset{\rightarrow}{\otimes}_K$.
		\end{enumerate}
	\end{prop}
	\begin{proof}
		\begin{enumerate}[(i)]
			\item Since $\mathrm{Ind}(\mathrm{Ban}_K)$ is elementary, filtered colimits are exact. Since colimits commute with the tensor product and every object can be written as the filtered colimit of Banach spaces, this follows now from Lemma \ref{normedtensor}.(ii).
			\item Our earlier discussion in Lemma \ref{flatproj} shows directly that $c_0(X)$ is a projective object in $\mathrm{Ind}(\mathrm{Ban}_K)$ (in fact, a tiny projective by the description of Hom sets above). Any object $``\varinjlim" V_i$ admits a strict surjection $\oplus_i c_0(V_i^\circ)\to \oplus V_i\to ``\varinjlim" V_i$, and clearly $c_0(X)\overset{\rightarrow}{\otimes}_K c_0(Y)=c_0(X)\h{\otimes}_K c_0(Y)=c_0(X\times Y)$ by definition of $\overset{\rightarrow}{\otimes}_K$ and Lemma \ref{czeroprops}.(ii). 
		\end{enumerate}
	\end{proof}
	
	Note that by the above, the closed symmetric monoidal structure on $\mathrm{Ind}(\mathrm{Ban}_K)$ lifts to a closed symmetric monoidal structure on $LH(\mathrm{Ind}(\mathrm{Ban}_K))$ such that $I$ is strong symmetric monoidal. 
	
	We denote the tensor product on $LH(\mathrm{Ind}(\mathrm{Ban}_K))$ by $\widetilde{\otimes}_K$ and the internal hom by $\widetilde{H}$. We will see below that this seeming conflict of notation is justified.
	
	We now consider the \textbf{dissection functor} $\mathrm{diss}: \h{\B}c_K\to \mathrm{Ind}(\mathrm{Ban}_K)$, which is given by sending $V\in \h{\B}c_K$ to $``\varinjlim" V_B$, where the colimit ranges over all bounded, $\pi$-adically complete $R$-submodules $B\subseteq V$.
	
	This functor admits as a left adjoint the functor $L: \mathrm{Ind}(\mathrm{Ban}_K)\to \h{\B}c_K$, given by sending $``\varinjlim" V_i$ to $\varinjlim V_i\in \h{\B}c_K$, see \cite[Proposition 5.15]{ProsmansSchneiders}.
	
	\begin{prop}
		\label{dissandL}
		\leavevmode
		\begin{enumerate}[(i)]
			\item The dissection functor $\mathrm{diss}: \h{\B}c_K\to \mathrm{Ind}(\mathrm{Ban}_K)$ is exact and fully faithful. Its essential image consists of the essentially monomorphic systems, i.e. those objects isomorphic to some $``\varinjlim" V_i$ with each $V_i\to V_j$ monomorphic.
			\item For any $V\in \h{\B}c_K$, $L(\mathrm{diss}(V))\cong V$ via the natural morphism, i.e. $\h{\B}c_K$ is a reflective subcategory of $\mathrm{Ind}(\mathrm{Ban}_K)$.
			\item $\mathrm{diss}$ commutes with all limits and with direct sums.
			\item $\mathrm{diss}$ is lax symmetric monoidal, and $L$ is strong symmetric monoidal.
			\item $\mathrm{diss}$ and $L$ induce an equivalence of triangulated categories
			\begin{equation*}
				\mathrm{D}(\h{\B}c_K)\cong \mathrm{D}(\mathrm{Ind}(\mathrm{Ban}_K))
			\end{equation*}
			and an equivalence of closed symmetric monoidal abelian categories
			\begin{equation*}
				LH(\h{\B}c_K)\cong LH(\mathrm{Ind}(\mathrm{Ban}_K)).
			\end{equation*}
		\end{enumerate}
	\end{prop}
	
	\begin{proof}
		\begin{enumerate}[(i)]
			\item This is \cite[Proposition 5.15]{ProsmansSchneiders}.
			\item This is also \cite[Proposition 5.15]{ProsmansSchneiders}.
			\item The commutation with limits is clear, as $\mathrm{diss}$ is a right adjoint. The commutation with direct sums follows from the construction of direct sums in $\h{\B}c_K$ already discussed in Corollary \ref{hBcquasiel}.
			\item The tensor products in $\h{\B}c_K$ and $\mathrm{Ind}(\mathrm{Ban}_K)$ agree for Banach spaces, yielding a natural transformation
			\begin{equation*}
				(``\varinjlim" V_B)\overset{\rightarrow}{\otimes}_K (``\varinjlim" W_{B'})\cong \varinjlim V_B\overset{\rightarrow}{\otimes}_K W_{B'}\to \mathrm{diss}(V\h{\otimes}_K W)
			\end{equation*}
			for $V\cong \varinjlim V_B$ and $W\cong \varinjlim W_{B'}$. We remark that this is not in general an isomorphism: while $V\h{\otimes}_K W\cong \varinjlim V_B\h{\otimes}_K W_{B'}$, this is not necessarily the filtered colimit used in the definition of $\mathrm{diss}$, see \cite[Remark 3.61]{Bamdagger}, \cite[subsection 1.5.5]{Meyer}. 
			
			Conversely, $L$ commutes with colimits, so in this case it follows that
			\begin{align*}
				L(``\varinjlim" V_i)\h{\otimes}_K L(``\varinjlim" W_j)&\cong \varinjlim_{i, j} L(V_i)\h{\otimes}_K L(W_j)\\&\cong\varinjlim L(V_i\overset{\rightarrow}{\otimes}_K W_j)\\&\cong L(``\varinjlim" V_i\overset{\rightarrow}{\otimes}_K``\varinjlim" W_j). 
			\end{align*}
			Thus $\mathrm{diss}$ is lax symmetric monoidal, and $L$ is strong symmetric monoidal.
			\item This is \cite[Proposition 5.16]{ProsmansSchneiders}. The fact that the tensor products in $LH(\h{\B}c_K)$ and on $LH(\mathrm{Ind}(\mathrm{Ban}_K))$ agree follows from the fact that they agree for Banach spaces by Lemma \ref{IforhBc} and Proposition \ref{indbanflatproj}, and every object can be written as a cokernel of some morphism $\oplus I(V_i)\to \oplus I(W_j)$ for $V_i, W_j$ Banach.
		\end{enumerate}
	\end{proof}
	
	\begin{cor}
		\label{disssubobjects}
		If $f:``\varinjlim" V_i\to ``\varinjlim" W_j$ is a monomorphism in $\mathrm{Ind}(\mathrm{Ban}_K)$ and $``\varinjlim" W_j$ is in the essential image of $\mathrm{diss}$, then $``\varinjlim" V_i$ is in the essential image of $\mathrm{diss}$.
	\end{cor}
	\begin{proof}
		By \cite[Proposition 2.10.(i)]{Bamdagger}, we can assume that $``\varinjlim" V_i$ and $``\varinjlim" W_j$ are indexed by the same diagram and $f$ is obtained from a system of compatible monomorphisms $f_i: V_i\to W_i$. 
		
		Since an ind-object is in the essential image of $\mathrm{diss}$ if and only if it is essentially monomorphic, we can further assume that $W_i\to W_{i'}$ is monomorphic for each $i\to i'$. But then $V_i\to V_{i'}$ is also monomorphic, and $``\varinjlim" V_i$ is in the essential image of $\mathrm{diss}$.
	\end{proof}
	
	We can now explain why $\h{\B}c_K$ itself is not elementary (contrary to the claim in \cite[Lemma 3.53]{Bamdagger}): if $\h{\B}c_K$ was elementary, then its colimits would commute with $I: \h{\B}c_K\to LH(\h{\B}c_K)$ by \cite[Proposition 2.1.16]{Schneiders}. Since $\mathrm{Ind}(\mathrm{Ban}_K)$ is also elementary and $\mathrm{diss}$ induces an equivalence on the level of left hearts, this would force $\mathrm{diss}: \h{\B}c_K\to \mathrm{Ind}(\mathrm{Ban}_K)$ to commute with filtered colimits, which would make it essentially surjective in contradiction to point (i) above.
	
	We thus arrive at the following picture:
	\begin{equation*}
		\begin{xy}
			\xymatrix{\h{\B}c_K\ar@<1ex>[r]^{\mathrm{diss}}\ar@<-1ex>[d]_I& \mathrm{Ind}(\mathrm{Ban}_K)\ar@<1ex>[l]^{L}\ar@<1ex>[d]^I\\
				LH(\h{\B}c_K)\ar@<1ex>[r]^{\widetilde{\mathrm{diss}}} \ar@<-1ex>[u]_{C}& LH(\mathrm{Ind}(\mathrm{Ban}_K))\ar@<1ex>[u]^{C} \ar@<1ex>[l]^{\widetilde{L}}}
		\end{xy}
	\end{equation*}
	where $\h{\B}c_K$ is in a sense the odd one out: the remaining three categories are all elementary, the functors between them are strong symmetric monoidal, and the bottom horizontal arrows are equivalences.
	
	Similarly (by \cite[Proposition 2.1.16]{Schneiders}), the other three categories have strongly exact filtered colimits, and the functors between them respect filtered colimits. On the other hand, the observation above shows that $I: \h{\B}c_K\to LH(\h{\B}c_K)$ does not preserve filtered colimits in general, and $\h{\B}c_K$ does not have exact filtered colimits, thanks to \cite[Proposition 2.1.16]{Schneiders}.
	
	At this point, one might ask why we are willing to notationally identify the tensor product on $LH(\h{\B}c_K)$ and on $LH(\mathrm{Ind}(\mathrm{Ban}_K))$, but keep a separate notation $\overset{\rightarrow}{\otimes}_K$ for the tensor product in $\mathrm{Ind}(\mathrm{Ban}_K)$. After all, the functors $I$ and $C$ on the right hand side of the diagram are also strong symmetric monoidal. The reason lies in the distinction of relative tensor products: If $A\in \mathrm{Ind}(\mathrm{Ban}_K)$ is a monoid, then the natural morphism
	\begin{equation*}
		I(M)\widetilde{\otimes}_{I(A)} I(N)\to I(M\overset{\rightarrow}{\otimes}_AN)
	\end{equation*}
	need not be an isomorphism for all $M\in \mathrm{Mod}_{\mathrm{Ind}(\mathrm{Ban}_K)}(A^{\mathrm{op}})$, $N\in \mathrm{Mod}_{\mathrm{Ind}(\mathrm{Ban}_K)}(A)$, as $I$ does not preserve arbitrary cokernels. 
	
	Note however that this subtlety disappears on the derived level: as we have seen in Lemma \ref{reltensoronLH}, there is a natural isomorphism $\widetilde{\otimes}_{I(A)}\cong \mathrm{H}^0(-\overset{\rightarrow}{\otimes}^{\mathbb{L}}_A-)$, and thus $\widetilde{\otimes}^{\mathbb{L}}\cong \overset{\rightarrow}{\otimes}^{\mathbb{L}}_{A}$.
	
	We emphasize that $\mathrm{Ind}(\mathrm{Ban}_K)$ satisfies all assumptions from the beginning of subsection 3.6:
	
	\begin{cor}
		\label{IndBansheafy}
		\leavevmode
		\begin{enumerate}[(i)]
			\item $\mathrm{Ind}(\mathrm{Ban}_K)$ is an elementary quasi-abelian category.
			\item $\mathrm{Ind}(\mathrm{Ban}_K)$ is a closed symmetric monoidal category where each object is flat.
			\item $\mathrm{Ind}(\mathrm{Ban}_K)$ has enough (flat) projectives stable under $\overset{\rightarrow}{\otimes}_K$.
		\end{enumerate}
	\end{cor}
	\begin{proof}
		(i) is Lemma \ref{indbanelementary}, (ii) and (iii) are Proposition \ref{indbanflatproj}. 
	\end{proof}
	
	As a consequence, we will generally apply results from the previous section to the case $\C=\mathrm{Ind}(\mathrm{Ban}_K)$, but in order to use analytic methods explicitly, we will carry out most of our work in $\h{\B}c_K$ and then apply $\mathrm{diss}$.
	
	Therefore, constructions which are not preserved by $\mathrm{diss}$ require special attention, e.g. arbitrary cokernels (only cokernels of strict morphisms are fine) and tensor products.
	
	The following Proposition ensures that while $\mathrm{diss}$ is only lax symmetric monoidal in general, it will respect all tensor products which appear in the remainder of this paper:
	
	\begin{prop}
		\label{dissandtensor}
		Let $V, W$ be metrisable, locally convex topological $K$-vector spaces. Assume that at least one of the following is satisfied:
		\begin{enumerate}[(i)]
			\item $K$ is spherically complete.
			\item $V$ is of countable type. 
		\end{enumerate}
		Then the natural morphism
		\begin{equation*}
			\mathrm{diss}(V^b)\overset{\rightarrow}{\otimes}_K \mathrm{diss}(W^b)\to \mathrm{diss}(V^b\h{\otimes}_KW^b)
		\end{equation*}
		is an isomorphism in $\mathrm{Ind}(\mathrm{Ban}_K)$.
	\end{prop} 
	\begin{proof}
		Consider the functor $\mathrm{diss}: \B c_K\to \mathrm{Ind}(\mathrm{sNorm}_K)$ sending a bornological $K$-vector space $E$ to the Ind-object of semi-normed spaces $``\varinjlim" E_B$ in analogy to before. Note that this functor is actually strong symmetric monoidal, where we equip the right hand side with the Kan extension of the semi-norm tensor product. We thus need to show that $\mathrm{diss}$ commutes with completion in the sense that $\mathrm{diss}(V^b\h{\otimes}_KW^b)\in \mathrm{Ind}(\mathrm{Ban}_K)$ is naturally isomorphic to $\h{\mathrm{diss}(V^b\otimes_KW^b)}$, where $\h{(-)}: \mathrm{Ind}(\mathrm{sNorm}_K)\to \mathrm{Ind}(\mathrm{Ban}_K)$ is the completion functor naturally induced by $\widehat{(-)}: \mathrm{sNorm}_K\to \mathrm{Ban}_K$.
		
		Since $V\otimes_K W$ is metrisable, it follows from Lemma \ref{Frtensor} that $V^b\otimes_KW^b\cong (V\otimes_KW)^b\in \B c_K$ is proper in the sense of \cite[p.112, Definition 5]{Houzel}. 
		
		But now \cite[Proposition 3.64]{Bamdagger} implies that for proper bornological spaces, dissection does indeed commute with completion -- see also \cite[Corollary 1.151]{Meyer} for another reference.
	\end{proof}

	\subsection{Bornological modules}
	We can now consider modules over monoids in $\B c_K$ and $\h{\B}c_K$.
	\begin{defn}
		A \textbf{bornological $K$-algebra} is a monoid in $\B c_K$, i.e. it is a bornological $K$-vector space $A\in \B c_K$ endowed with a $K$-algebra structure such that the multiplication map
		\begin{equation*}
			A\otimes_K A\to A
		\end{equation*} 
		is bounded.
	\end{defn}
	Note that a bornological $K$-algebra $A$ which happens to be complete is the same as a monoid in $\h{\B} c_K$, so there is no ambiguity in calling such an algebra a complete bornological $K$-algebra. For instance, if $A$ is a bornological $K$-algebra, then $\h{A}$ is naturally a complete bornological $K$-algebra.\\ 
	Following the formalism in subsection 3.3, a bornological module over a bornological $K$-algebra $A$ is a bornological $K$-vector space $M$, endowed with a (left) $A$-module structure with bounded action map
	\begin{equation*}
		A\otimes_K M\to M.
	\end{equation*}
	Let $A$ be a bornological $K$-algebra.
	
	\begin{lem}
		\label{propsofcompl}
		Let $A$ be a bornological $K$-algebra, and let $M$ be a bornological $A$-module.
		\begin{enumerate}[(i)]
			\item $\h{A}$ is a complete bornological $K$-algebra, and completion yields a functor $\mathrm{Mod}_{\B c_K}(A)\to \mathrm{Mod}_{\h{\B}c_K}(\h{A})$.
			\item For any positive integer $r$, any $A$-linear map $g: A^r\to M$ is bounded. 
		\end{enumerate}
	\end{lem}
	\begin{proof}
		(i) follows directly from Proposition \ref{laxmodulechange}, as completion is a strong symmetric monoidal functor by definition of the completed tensor product.
		
		For (ii), let $e_1, \hdots, e_r$ denote the free generators of $A^r$ and note that the image $g(B)$ of a bounded set $B=\oplus B_ie_i$ satisfies $g(B)=\sum B_ig(e_i)$, which is bounded, due to the boundedness of the action map. 
	\end{proof}
	Note that by definition, the morphisms in $\mathrm{Mod}_{\B c_K}(A)$ are precisely those bounded maps which are also abstract $A$-module morphisms.
	
	Example: Let $A$ be a Noetherian Banach $K$-algebra (e.g. an affinoid $K$-algebra). Then $A$ is naturally a complete bornological $K$-algebra, and endowing a finitely generated $A$-module with its canonical Banach structure provides a fully faithful embedding of the category of finitely generated $A$-modules into $\mathrm{Mod}_{\h{\B} c_K}(A)$.
	
	More generally, recall the definition of the category $\C_A$ of coadmissible modules over a Fr\'echet--Stein algebra $A$ in Definition \ref{FSandcoaddef}, and note that any coadmissible $A$-module carries a natural Fr\'echet structure (see Proposition \ref{coadprops}). 
	
	\begin{lem}
		\label{FSborno}
		If $A$ is a Fr\'echet--Stein $K$-algebra, then $A^b$ is naturally a complete bornological $K$-algebra.\\
		The functor $M\mapsto M^b$ yields a fully faithful functor
		\begin{equation*}
			\C_A\to \mathrm{Mod}_{\h{\B}c_K}(A^b).
		\end{equation*}
	\end{lem}
	\begin{proof}
		As $A$ is a Fr\'echet algebra, we have a continuous multiplication map $A\otimes_K A\to A$ satisfying the unit and associativity axioms. We thus have a bounded map
		\begin{equation*}
			A^b\otimes_K A^b\to (A\otimes_K A)^b\to A^b,
		\end{equation*}
		where the first map is the bounded bijection which is simply the identity on the underlying vector space, and the second map is obtained by applying $(-)^b$ to the multiplication map. Thus $A$ is a $K$-algebra with bounded multiplication map, and therefore a monoid in $\B c_K$. Applying the completion functor shows that $A^b$ is a complete bornological $K$-algebra.\\
		In the same way, if $M$ is a coadmissible $A$-module, equipped with its canonical Fr\'echet topology, then $M^b$ is a complete bornological $A^b$-module, and it suffices to show that this gives a fully faithful functor. Faithfulness is clear as $(-)^b$ does not change the underlying vector spaces and linear maps. By \cite[remark after Lemma 3.6]{ST}, any $A$-linear map between coadmissible $A$-modules is continuous, so in particular, any bounded $A$-linear map is continuous and $(-)^b$ is full. 
	\end{proof}
	
	We remark that we also have $A^b\cong \varprojlim A_n^b$ and $M^b\cong \varprojlim M_n^b$ in $\h{\B}c_K$, as $(-)^b$ commutes with limits.
	
	Let $A$ be a complete bornological $K$-algebra. As discussed in subsection 3.3, the tensor product $\otimes_K$ (resp., $\h{\otimes}_K$) gives rise to a tensor product 
	\begin{equation*}
		-\otimes_A-: \mathrm{Mod}_{\B c_K}(A^{\mathrm{op}})\times \mathrm{Mod}_{\B c_K}(A)\to \B c_K,
	\end{equation*}
	resp.
	\begin{equation*}
		-\h{\otimes}_A-: \mathrm{Mod}_{\h{\B} c_K}(A^{\mathrm{op}})\times \mathrm{Mod}_{\h{\B} c_K}(A)\to \h{\B} c_K.
	\end{equation*}
	As completion commutes with cokernels, $M\h{\otimes}_A N$ is naturally isomorphic to the completion of $M\otimes_A N$.
	
	If $A$ is a Fr\'echet $K$-algebra and $M$ is a right Fr\'echet $A$-module, $N$ a left Fr\'echet $A$-module, then we can also endow $M\otimes_A N$ with a locally convex topology by considering it as the cokernel of
	\begin{align*}
		M\otimes_K A\otimes_K N&\to M\otimes_K N\\
		m\otimes a\otimes n&\mapsto ma\otimes n-m\otimes an,
	\end{align*}
	where the tensor products over $K$ are equipped with the projective tensor product topology.
	\begin{lem}
		\label{moduletensorwithb}
		Let $A$ be a Fr\'echet $K$-algebra. The natural morphism
		\begin{equation*}
			M^b\otimes_{A^b}N^b\to (M\otimes_A N)^b
		\end{equation*}
		is a bounded bijection.\\
		If $K$ is spherically complete or if $M$, $A$, $N$ are of countable type, then the natural morphism
		\begin{equation*}
			(M^b{\otimes}_{A^b} N^b)^t\to M{\otimes}_AN
		\end{equation*}
		is an isomorphism. 
	\end{lem}
	\begin{proof}
		Note that the underlying vector space of both $M^b\otimes_{A^b} N^b$ and $(M\otimes_A N)^b$ is the algebraic tensor product $M\otimes_A N$ by the definition of cokernels in $\B c_K$ and $LCVS_K$, respectively. The natural morphism $M^b\otimes_{A^b}N^b\to (M\otimes_A N)^b$, induced by the bounded map $M^b\otimes_KN^b\to (M\otimes_K N)^b$, is then the identity map and hence bijetive.
		
		Now suppose that $K$ is spherically complete, or $M$, $A$, $N$ are of countable type. The bornological vector space $M^b\otimes_{A^b}N^b$ is then the cokernel of the natural morphism
		\begin{equation*}
			(M\otimes_K A\otimes_K N)^b\to (M\otimes_K N)^b,
		\end{equation*}
		thanks to Lemma \ref{Frtensor}. Note that all these tensor products are metrisable spaces, so applying first $(-)^b$ and then $(-)^t$ yields the identity functor. As $(-)^t$ preserves cokernels, we can deduce that
		\begin{equation*}
			(M^b\otimes_{A^b}N^b)^t\cong \mathrm{coker}(M\otimes_K A\otimes_K N\to M\otimes_K N),
		\end{equation*}
		as required.
	\end{proof}
	It is harder to make any general statements comparing the completed tensor products. We will return to this issue in subsection 5.4. We note the following special case:
	\begin{lem}
		\label{Banmoduletensor}
		Let $A$ be a Banach $K$-algebra. Let $M$ and $N$ be a right resp. left Banach module. Then there is a natural isomorphism
		\begin{equation*}
			M^b\h{\otimes}_{A^b}N^b\cong\left(M\h{\otimes}_AN\right)^b.
		\end{equation*}
	\end{lem}
	\begin{proof}
		Let $f: M\h{\otimes}_K A\h{\otimes}_K N\to M\h{\otimes}_K N$ be the usual morphism. Then 
		\begin{equation*}
			(M\h{\otimes}_A N)^b\cong (\mathrm{coker} f)^b,
		\end{equation*}
		and by Lemma \ref{normedtensor}.(i), $M^b\h{\otimes}_{A^b}N^b\cong \mathrm{coker} (f^b)$. The result thus follows from the remark after Lemma \ref{cokerwithb}.
	\end{proof}
	
	\begin{prop}
		\label{IBFrechetmod}
		Let $A$ be a complete bornological $K$-algebra. Then $\mathrm{diss}: \h{\B}c_K\to \mathrm{Ind}(\mathrm{Ban}_K)$ induces an exact and fully faithful functor
		\begin{equation*}
			\mathrm{diss}_A: \mathrm{Mod}_{\h{\B}c_K}(A)\to \mathrm{Mod}_{\mathrm{Ind}(\mathrm{Ban}_K)}(\mathrm{diss}(A)).
		\end{equation*}
		If $A$ is the bornologification of a Fr\'echet $K$-algebra, this yields an equivalence of categories
		\begin{align*}
			LH(\mathrm{Mod}_{\h{\B}c_K}(A))&\cong LH(\mathrm{Mod}_{\mathrm{Ind}(\mathrm{Ban}_K)}(\mathrm{diss}(A)))\\&\cong \mathrm{Mod}_{LH(\mathrm{Ind}(\mathrm{Ban}_K))}(I(\mathrm{diss}(A)))\\&\cong \mathrm{Mod}_{LH(\h{\B}c_K)}(I(A)),
		\end{align*}
		provided that at least one of the following is satisfied:
		\begin{enumerate}[(i)]
			\item $K$ is spherically complete.
			\item $A$ is of countable type.
		\end{enumerate}
	\end{prop}
	\begin{proof}
		Since $\mathrm{diss}: \h{\B}c_K\to \mathrm{Ind}(\mathrm{Ban}_K)$ is lax symmetric monoidal by Proposition \ref{dissandL}.(iv), we can apply Proposition \ref{laxmodulechange} to obtain a corresponding functor between the module categories. As $\mathrm{diss}$ is exact by Proposition \ref{dissandL}.(i), so is $\mathrm{diss}_A$.
		
		Since $\mathrm{diss}$ is fully faithful by Proposition \ref{dissandL}.(i), it follows immediately that $\mathrm{diss}_A$ is faithful. Moreover, if $M, N\in \mathrm{Mod}_{\h{\B}c_K}(A)$ and $f: \mathrm{diss}_A(M)\to \mathrm{diss}_A(N)$ is a $\mathrm{diss}(A)$-module morphism, i.e. a morphism in $\mathrm{Ind}(\mathrm{Ban}_K)$ such that the diagram
		\begin{equation*}
			\begin{xy}
				\xymatrix{\mathrm{diss}(A)\overset{\rightarrow}{\otimes}_K \mathrm{diss}(M)\ar[r]\ar[d] & \mathrm{diss}(A)\overset{\rightarrow}{\otimes}_K \mathrm{diss}(N)\ar[d]\\
					\mathrm{diss}(M)\ar[r]& \mathrm{diss}(N)}
			\end{xy}
		\end{equation*}
		commutes, then applying the strong symmetric monoidal functor $L$ shows that $L(f): M\to N$ is an $A$-module morphism such that $\mathrm{diss}_A(L(f))=f$. Thus $\mathrm{diss}_A$ is fully faithful.
		
		For the first equivalence, we will invoke once more \cite[Proposition 1.2.36]{Schneiders}, which we have already used in the proof of Proposition \ref{LHIR}. Let 
		\begin{equation*}
			J= I\circ \mathrm{diss}_A: \mathrm{Mod}_{\h{\B}c_K}(A)\to LH(\mathrm{Mod}_{\mathrm{Ind}(\mathrm{Ban}_K)}(\mathrm{diss}(A))).
		\end{equation*}
		We check the required properties:
		\begin{enumerate}[(a)]
			\item $J$ is fully faithful, since both $\mathrm{diss}_A$ and $I$ are.
			\item If $N\to J(M)$ is a monomorphism for some $N\in LH(\mathrm{Mod}_{\mathrm{Ind}(\mathrm{Ban}_K)}(\mathrm{diss}(A)))$, $M\in \mathrm{Mod}_{\h{\B}c_K}(A)$, then $N$ lies in the essential image of $J$: Since $J(M)$ is in the essential image of $I$, so is $N$, i.e. $N=I(N^0)$ for some $N^0\in \mathrm{Mod}_{\mathrm{Ind}(\mathrm{Ban}_K)}(\mathrm{diss}(A))$ with $N^0\to \mathrm{diss}_A(M)$ a monomorphism. In particular, $N^0$ is an essentially monomorphic system of Banach spaces by Corollary \ref{disssubobjects}, and $N^0\cong \mathrm{diss}(L(N^0))$ as an object of $\mathrm{Ind}(\mathrm{Ban}_K)$. Since $L$ is strong monoidal, $L(N^0)\in \mathrm{Mod}_{\h{\B}c_K}(A)$, and the natural transformation $\mathrm{id}\to \mathrm{diss}\circ L$ shows that the $\mathrm{diss}(A)$-module structure on $N^0$ is precisely the one on $\mathrm{diss}_A(L(N^0))$. Thus $N^0\cong \mathrm{diss}_A(L(N^0))$, and $N\cong I(N^0)$ is in the essential image of $J$.
			
			\item Any object in $\mathrm{LH}(\mathrm{Mod}_{\mathrm{Ind}(\mathrm{Ban}_K)}(\mathrm{diss}(A)))$ admits an epimorphism from an object in the essential image of $J$: It suffices to show that any $M\in \mathrm{Mod}_{\mathrm{Ind}(\mathrm{Ban}_K)}(\mathrm{diss}(A))$ admits a strict epimorphism from an object in the essential image of $\mathrm{diss}_A$. If $M=``\varinjlim" M_i\in \mathrm{Ind}(\mathrm{Ban}_K)$ carries a $\mathrm{diss}(A)$-module structure, then there are strict epimorphisms
			\begin{equation*}
				\mathrm{diss}(A)\overset{\rightarrow}{\otimes}_K (\oplus M_i)\to \mathrm{diss}(A)\overset{\rightarrow}{\otimes}_K M\to M
			\end{equation*} 
			of $\mathrm{diss}(A)$-modules. It thus suffices to show that $\mathrm{diss}(A)\overset{\rightarrow}{\otimes}_K V$ is in the essential image of $\mathrm{diss}_A$ for any Banach space $V$. But due to the assumptions, we have
			\begin{equation*}
				\mathrm{diss}(A)\overset{\rightarrow}{\otimes}_K V\cong \mathrm{diss}(A\h{\otimes}_K V)
			\end{equation*}
			by Proposition \ref{dissandtensor}.
		\end{enumerate}
		We can thus apply \cite[Proposition 1.2.36]{Schneiders} to deduce that $J$ induces an equivalence
		\begin{equation*}
			LH(\mathrm{Mod}_{\h{\B}c_K}(A))\cong LH(\mathrm{Mod}_{\mathrm{Ind}(\mathrm{Ban}_K)}(\mathrm{diss}(A))).
		\end{equation*}
		The second equivalence now follows from Proposition \ref{LHIR}, and the third from Proposition \ref{dissandL}.
	\end{proof}
	
	In other words: Even though  we have not shown that $I: \h{\B}c_K\to LH(\h{\B}c_K)$ is strong monoidal (and in fact, it is not), we obtain an analogue of Proposition \ref{LHIR}. In particular, we can apply Lemma \ref{reltensoronLH} to talk unambiguously about the tensor product $\widetilde{\otimes}_A$ for $LH(\mathrm{Mod}_{\h{\B}c_K}(A))$. 
	
	\begin{lem}
		\label{reltensorIBFrechet}
		Let $A$ be the bornologification of a Fr\'echet $K$-algebra. Assume that at least one of the following is satisfied:
		\begin{enumerate}[(i)]
			\item $K$ is spherically complete.
			\item $A$ is of countable type.
		\end{enumerate}
		Then there is a natural isomorphism
		\begin{equation*}
			\widetilde{\mathrm{diss}}(M\widetilde{\otimes}_{A} N)\cong \widetilde{\mathrm{diss}}(M)\widetilde{\otimes}_{\mathrm{diss}(A)} \widetilde{\mathrm{diss}}(N)\in LH(\mathrm{Ind}(\mathrm{Ban}_K))
		\end{equation*}
		for $M\in LH(\mathrm{Mod}_{\h{\B}c_K}(A^{\mathrm{op}}))$, $N\in LH(\mathrm{Mod}_{\h{\B}c_K}(A))$.
	\end{lem}
	
	\begin{proof}
		Thanks to Proposition \ref{IBFrechetmod}, we have equivalences
		\begin{equation*}
			LH(\mathrm{Mod}_{\h{\B}c_K}(A))\cong \mathrm{Mod}_{LH(\h{\B}c_K)}(I(A))
		\end{equation*}
		and
		\begin{equation*}
			LH(\mathrm{Mod}_{\mathrm{Ind}(\mathrm{Ban}_K)})(\mathrm{diss}(A))\cong \mathrm{Mod}_{LH(\mathrm{Ind}(\mathrm{Ban}_K))}(I(\mathrm{diss}(A))),
		\end{equation*}
		such that by Lemma \ref{reltensoronLH}, the tensor products on the left hearts get identified with the relative tensor products in the corresponding module categories.
		
		But now $\widetilde{\mathrm{diss}}: LH(\h{\B}c_K)\cong LH(\mathrm{Ind}(\mathrm{Ban}_K))$ is an equivalence of closed symmetric monoidal categories, so it also identifies the relative tensor products.
	\end{proof}
	
	\section{Complete bornological modules over Fr\'echet--Stein algebras}
	\subsection{Nuclearity}
	Coadmissible modules over any Fr\'echet--Stein algebra carry a natural Fr\'echet topology, and many of their elementary properties given e.g. in \cite[section 3]{ST} are expressed within this Fr\'echet (topological) context. To have the same results at our disposal when working in the complete bornological framework, we need the functor $(-)^b$ to preserve various constructions. We first collect a couple of problems which we have already encountered in the last section.
	\begin{enumerate}[(a)]
		\item We might ask whether $(-)^b$ is exact on a suitable class of Fr\'echet spaces. Let $f: V\to W$ be a continuous morphism of Fr\'echet spaces. When is $\mathrm{coker}(f^b)\cong (\mathrm{coker} f)^b$? We note that this question becomes particularly relevant with regard to completed tensor products of Fr\'echet modules, see Lemma \ref{moduletensorwithb}.
		
		\item Let $V\in \mathscr{M}$. Is there some topological condition which might ensure that the natural morphism $\h{V^b}\to \h{V}^b$ is an isomorphism? Note that as we do not know a priori whether $\h{V^b}\in \mathscr{M}^b$, this is far from clear. Lemma \ref{tcompl} only gives some partial information. 
		
		\item Let $V\cong \varprojlim V_n$ be a Fr\'echet space, where each $V_n$ is a Banach space and the transition maps have dense image. By the Mittag--Leffler property, 
		\begin{equation*}
			V\cong \mathrm{R}\varprojlim V_n, 
		\end{equation*}
		i.e. the sequence
		\begin{equation*}
			0\to V\to \prod V_n\to \prod V_n\to 0
		\end{equation*} 
		is strictly exact. We know that $V^b\cong \varprojlim (V_n^b)$, as right adjoints commute with limits, but is $V^b\cong \mathrm{R}\varprojlim V^b_n$? Applying $(-)^b$ to the strictly short exact sequence above yields a surjection $\prod (V_n^b)\to \prod (V_n^b)$, but it is not clear whether it is also strict. 
	\end{enumerate} 
	There are usually two cases in which one can prove the desired isomorphisms: firstly if the spaces in question are normed (see Lemma \ref{borniexact}, Lemma \ref{normedtensor}, Lemma \ref{Banmoduletensor}), and secondly if the spaces are \emph{nuclear}. The aim of the following subsections is the introduction of several generalizations of nuclearity which are broad enough to establish the desired stability properties for all spaces which we are considering.
	
	In this subsection, we consider the most natural notion, nuclearity relative to a Noetherian Banach algebra, which addresses (a). In subsection 5.2, we introduce the more general notion of pseudo-nuclearity, which addresses (b). 
	
	In subsection 5.3, we discuss (c) by introducing pre-nuclear inverse systems of Banach spaces. We then use these results to give basic properties of coadmissible modules when viewed as complete bornological modules.
	
	In each case, we also confirm that our results remain valid after applying the dissection functor $\mathrm{diss}$.
	
	We begin with (a) and the notion of a nuclearity relative to a Noetherian Banach algebra. 
	
	We fix a Noetherian Banach $K$-algebra $A$ and an open $R$-subalgebra $\A\subseteq A^\circ$. An $A$-module $M$ will be called a \textbf{contracting} Banach $A$-module (or, to be precise, a contracting Banach $(A, \A)$-module) if it is endowed with a Banach structure such that the unit ball $M^\circ$ is $\A$-stable. We remark that the notion of contracting Banach module, and much of what follows in this subsection, might a priori depend on the choice of $\A$, even though we usually suppress this in terminology and notation.
	\begin{defn}[{\cite[Definition 1.1]{Kiehl}}]
		\label{defnscc}
		Let $f: M\to N$ be a continuous $A$-module morphism between two contracting Banach $A$-modules.
		\begin{enumerate}[(i)]
			\item We say that $f$ is \textbf{completely continuous} if there exists a sequence of continuous $A$-module morphisms $f_i: M\to N$, converging uniformly to $f$, such that the image $f_i(M)$ is a finitely generated $A$-module for all $i$.
			\item We say that $f$ is \textbf{strictly completely continuous} (scc) if there exists a sequence of continuous $A$-module morphisms $f_i: M\to N$, converging uniformly to $f$, and an element $a\in R$, $a\neq 0$ such that for each $i$, $a\cdot f_i(M^\circ)$ is contained in some finitely generated $\A$-submodule of $N^\circ$.
		\end{enumerate}
	\end{defn}
	
	In the case $A=K$, $\A=R$, the notion of a completely continuous morphism has already been studied extensively. Recall that a continuous $K$-linear map $f: M\to N$ between Banach $K$-vector spaces is called \textbf{compactoid} if $f(M^\circ)$ is a compactoid subset of $N$ in the sense of \cite[Definition 3.8.1]{Schikhof}.  
	
	\begin{lem}
		\label{scccompactoid}
		Suppose that $A=K$, $\A=R$ and let $f: M\to N$ be a continuous morphism of Banach $K$-vector spaces. Then the following are equivalent:
		\begin{enumerate}[(i)]
			\item $f$ is completely continuous.
			\item $f$ is strictly completely continuous.
			\item $f$ is compactoid.
		\end{enumerate}
	\end{lem}
	\begin{proof}
		The equivalence between (i) and (iii) is \cite[Proposition 2]{Gruson}, and it is immediate that (ii) implies (i). It thus suffices to show that (i) implies (ii).\\
		Suppose that $f_i: M\to N$ are continuous linear maps exhibiting $f$ as a completely continuous map. By rescaling the norm on $N$ and discarding finitely many $f_i$, we can suppose that $f$ and $f_i$ send $M^\circ$ into $N^\circ$ for all $i$. Since $f_i(M)$ is a finite-dimensional $K$-vector space, we can invoke \cite[Proposition 4.13]{SchneiderNFA} to deduce that the subspace norm on $f_i(M)$ is equivalent to the standard norm, i.e. for each $i$, $f_i(M)\cap N^\circ$ is contained in some finitely generated $R$-module. Now pick any $a\in R$ with $|a|<1$, $a\neq 0$. It follows straightforwardly from \cite[Lemme 5.2/1]{Gruson} that there exists a finitely generated $R$-submodule $S_i\subseteq f_i(M)\cap N^\circ$ such that 
		\begin{equation*}
			af_i(M^\circ)\subseteq a\cdot(f_i(M)\cap N^\circ)\subseteq S_i.
		\end{equation*}
		Thus $f$ is strictly completely continuous.
	\end{proof}
	
	We also mention an example of an scc morphism for general $A$. The natural morphism 
	\begin{equation*}
		f: A\langle \pi x_1, \hdots, \pi x_d\rangle \to A\langle x_1, \hdots, x_d\rangle 
	\end{equation*}
	is a strictly completely continuous morphism of $A$-modules, as $f$ can be approximated by 
	\begin{align*}
		f_i: A\langle \pi x_1, \hdots, \pi x_d\rangle \to A\langle x_1, \hdots, x_d\rangle\\
		\sum a_\alpha x^\alpha\to \sum_{|\alpha|\leq i} f(a_\alpha x^\alpha).
	\end{align*}
	Both the notion of an scc morphism and the example above were already discussed in detail in \cite{Kiehl}.
	
	We also note that the finite direct sum of scc morphisms is scc, and if $f: M\to N$ is scc, then so is $gf$ for any continuous $A$-module morphism $g: N\to P$.
	
	\begin{defn}
		Let $M$ be a Fr\'echet $A$-module. We say that $M$ is \textbf{$A$-nuclear} (or nuclear over $A$, or nuclear relative to $(A, \A)$) if there exists an isomorphism of Fr\'echet modules $M\cong \varprojlim M_n$, where for each $n$ the following is satisfied:
		\begin{enumerate}[(i)]
			\item $M_n$ is a contracting Banach $A$-module, and the image of $M$ in $M_n$ is dense,
			\item there exists a contracting Banach $A$-module $F_n$ and a strict surjection 
			\begin{equation*}
				F_n\to M_n
			\end{equation*}
			such that the composition $F_n\to M_n\to M_{n-1}$ is strictly completely continuous over $A$. 
		\end{enumerate}
	\end{defn}
	
	As a first example, note that any finitely generated $A$-module equipped with its canonical Banach structure is nuclear over $A$ (taking the constant inverse system $M_n=M$).
	
	Classically, a Fr\'echet $K$-vector space $M$ is called nuclear if its topology is defined by a (countable) family of semi-norms such that the connecting maps between the Banach completions $M_{n+1}\to M_n$ are compactoid (see \cite[Theorem 8.5.1.($\beta$)]{Schikhof}). The lemma below proves that this is consistent with our definition.
	
	\begin{lem}
		A Fr\'echet $K$-vector space is nuclear if and only if it is $K$-nuclear as defined above.
	\end{lem}
	\begin{proof}
		We claim that if $p: F_n\to M_n$ is a continuous surjection of Banach spaces, and $g: M_n\to M_{n-1}$ is a continuous map of Banach spaces such that $gp$ is compactoid, then $g$ is compactoid. In fact, we can assume without loss of generality that $M_n^\circ\subseteq p(F_n^\circ)$, so that $g(M_n^\circ)\subseteq gp(F_n^\circ)$ is compactoid by \cite[Theorem 3.8.4.(ii)]{Schikhof}.
		
		The result now follows directly from Lemma \ref{scccompactoid}.
	\end{proof}
	
	For a further class of examples, we consider the completed enveloping algebras introduced in subsection 2.1.
	
	We fix an affinoid $K$-algebra $A$ with admissible affine formal model $\A$. Let $L$ be a smooth $(K, A)$-Lie--Rinehart algebra. Recall from Theorem \ref{UcapFS} that the completed enveloping algebra $\w{U_A(L)}$ is a Fr\'echet--Stein algebra.

	We say that $L$ \textbf{admits a smooth (resp. free) Lie lattice} if there exists some $(R, \A)$-Lie lattice $\L\subseteq L$ which is a finitely generated projective (resp. free) $\A$-module. For example, if $L$ is a free $A$-module, we can pick free generators for $L$ and rescale them appropriately such that they generate an $\A$-module which is a free, and a fortiori smooth, Lie lattice. This shows that `locally', smooth Lie lattices always exist.

	\begin{prop}
		\label{coadisnuc}
		Let $L$ be a smooth $(K, A)$-Lie--Rinehart algebra admitting a smooth Lie lattice. 
		\begin{enumerate}[(i)]
			\item Any coadmissible $\w{U_A(L)}$-module is nuclear over $A$.
			\item Any coadmissible $\w{U_A(L)}$-module is of countable type.
		\end{enumerate}
	\end{prop} 

	\begin{proof}
		We begin by showing that $\w{U_A(L)}$ is nuclear over $A$.
		
		Let $\L$ be an $(R, \A)$-Lie lattice of $L$ which is projective as an $\A$-module, with generators $x_1, \hdots, x_d$. We write $U_n=\h{U_{\A}(\pi^n\L)}\otimes_R K$ for each $n\geq 0$.
		
		By functoriality, the natural morphism
		\begin{equation*}
			\h{\mathrm{Sym}_{\A}(\A^d)}\cong \A\langle y_1, \hdots, y_d\rangle\to \h{\mathrm{Sym}_{\A}\pi^n\L}
		\end{equation*}
		sending $y_i$ to $\pi^nx_i$ admits a section and is thus a surjection of $\A$-modules.
		
		As $\L$ is projective, we can invoke Rinehart's version of the PBW theorem (\cite[Theorem 3.1]{Rinehart}) to deduce that 
		\begin{equation*}
			\h{\mathrm{Sym}_{\A} \pi^n\L}\cong \h{U_{\A}(\pi^n\L)}
		\end{equation*}
		as $\A$-modules.
		
		Tensoring with $K$, we have thus constructed continuous surjective $A$-module morphisms 
		\begin{equation*}
			\alpha_n: A\langle y_1, \hdots, y_d\rangle \to U_n
		\end{equation*}
		fitting into a commutative diagram
		\begin{equation*}
			\begin{xy}
				\xymatrix{
					A\langle y_1, \hdots, y_d\rangle \ar[r] \ar[d]_{\alpha_n}& A\langle y_1, \hdots, y_d\rangle\ar[d]^{\alpha_{n-1}}\\
					U_n\ar[r] & U_{n-1},
				}
			\end{xy}
		\end{equation*}
		where the top horizontal map is the algebra morphism determined by sending $y_i$ to $\pi y_i$. This is scc over $A$, as it is precisely the example given after Lemma \ref{scccompactoid}. Hence the composition $A\langle y_1, \hdots, y_d\rangle\to U_n\to U_{n-1}$ is also scc, and $\w{U_A(L)}$ is nuclear over $A$.
		
		If $M=\varprojlim M_n$ is a coadmissible $\w{U_A(L)}$-module, then for each $n$ there exists an integer $r_n$ such that we have a commutative diagram of Banach $A$-modules
		\begin{equation*}
			\begin{xy}
				\xymatrix{A\langle y_1, \hdots, y_d\rangle^{\oplus r_n}\ar[d]_{\alpha_n^{\oplus r_n}}\\
					U_n^{\oplus r_n}\ar[r]\ar[d]& U_{n-1}^{\oplus r_n}\ar[d]\\
					M_n\ar[r] & M_{n-1}
				}
			\end{xy}
		\end{equation*}
		with all vertical maps being continuous surjections (the right column is obtained by tensoring the left column with $U_{n-1}$).
		
		As $A\langle y_1, \hdots, y_d\rangle^{\oplus r_n}\to U_n^{\oplus r_n}\to U_{n-1}^{\oplus r_n}$ is scc, so is the composition 
		\begin{equation*}
			A\langle y_1, \hdots, y_d\rangle^{\oplus r_n}\to U_n^{\oplus r_n}\to M_n\to M_{n-1},
		\end{equation*}
		proving that $M$ is a nuclear $A$-module.
		
		Since $A$ is an affinoid $K$-algebra, it is a quotient of a Tate algebra and hence of countable type. The above description implies immediately that if $M=\varprojlim M_n$ is coadmissible, then $M_n$ is of countable type for each $n$, and thus $M$ is of countable type, e.g. by \cite[Theorem 4.2.13.(iii)]{Schikhof}.
	\end{proof}
	
	We will mainly consider coadmissible modules over algebras $\w{U_A(L)}$ as in the Proposition above, but we point out that the analogues of all results remaining in this section also hold for coadmissible modules over Fr\'echet--Stein algebras $U$ with the property that every coadmissible $U$-module is a nuclear $K$-vector space of countable type. This is for example the case for the distribution algebra $D(G, K)$ of a compact $p$-adic Lie group $G$ (\cite[Lemma 6.1]{ST}). 
	
	We now return to our more general setting, where $A$ is a (not necessarily commutative) Noetherian Banach $K$-algebra, and $\A\subseteq A^\circ$ is an open $R$-subalgebra.
	
	\begin{lem}
		\label{Cauchybounded}
		Let $M$ be a Fr\'echet $A$-module, and let $x_1, x_2, \hdots$ be a Cauchy sequence in $M$. Then
		\begin{equation*}
			\overline{\sum_{i=1}^\infty \A x_i}
		\end{equation*}
		is a bounded subset of $M$.
	\end{lem}
	\begin{proof}
		Cauchy sequences are bounded by \cite[Lemma 7.10]{SchneiderNFA}. Thus the boundedness of the $A$-action implies that $\sum \A x_i$ is bounded, and the closure of a bounded subset is bounded (see \cite[Lemma 4.10]{SchneiderNFA}).
	\end{proof}
	
	The crucial property of nuclear $A$-modules will be that under an additional hypothesis on $\A$, their bornology is generated by Cauchy sequences (Lemma \ref{Cauchygenerates}). To prove this, we first need several lemmas. Most of the strategy below is inspired by \cite{Schikhof}, where analogous results are discussed for the case $A=K$.
	
	\begin{lem}
		\label{nuclearimpliespre}
		Let 
		\begin{equation*}
			\begin{xy}
				\xymatrix{
					&F_{n+1}\ar[d]_{p} \ar[dr]^{f}\\
					M_{n+2}\ar[r]^{\rho_{n+1}}&M_{n+1}\ar[r]^{\rho_{n}}&M_n
				}
			\end{xy}
		\end{equation*}
		be a commutative diagram of contracting Banach $A$-modules such that $p$ is surjective, $f$ is scc, and the $\rho_i$ have dense image. Let $B\subseteq M_{n+1}$ be a bounded subset and let $V\subseteq M_n$ be an open $R$-submodule.
		\begin{enumerate}[(i)]
			\item There exist $m_1, \hdots, m_r\in M_{n+1}$ such that
			\begin{equation*}
				\rho_n(B)\subseteq V+\sum \A \rho_n(m_i).
			\end{equation*}
			\item There exists a bounded (even finitely generated) $\A$-submodule $B'\subseteq M_{n+2}$ such that
			\begin{equation*}
				\rho_n(B)\subseteq V+\rho_n(\rho_{n+1}(B')).
			\end{equation*}
		\end{enumerate}
	\end{lem}
	
	\begin{proof}
		Without loss of generality, $V$ consists of all elements of norm $< \epsilon$ for a given $\epsilon>0$, $B\subseteq p(F_{n+1}^\circ)$, $||\rho_n||\leq 1$, $||f||\leq 1$, and $M_{n+1}$ is equipped with the quotient norm induced by $p$. 
		
		By assumption, $f$ is the uniform limit of morphisms $f_i: F_{n+1}\to M_n$ such that there exists $a\in R$, $a\neq 0$ with $a\cdot f_i(F_{n+1}^\circ)$ being contained in some finitely generated $\A$-submodule of $M_n^\circ$.
		
		\begin{enumerate}[(i)]
			\item There exists $f_i$ such that $||f-f_i||<\epsilon$ and $x_1, \hdots, x_r\in M_n^\circ$ such that
			\begin{equation*}
				a\cdot f_i(F_{n+1}^\circ)\subseteq \sum_{j=1}^r \A x_j.
			\end{equation*}
			By the density assumption, there exist $m_1, \hdots, m_r\in M_{n+1}$ such that 
			\begin{equation*}
				|\rho_n(m_j)-a^{-1}\cdot x_j|<\epsilon
			\end{equation*}
			for all $j$.
			
			Now let $b\in B$ and $x\in F_{n+1}^\circ$ such that $p(x)=b$. There exist $a_1, \hdots, a_r\in \A$ with
			\begin{equation*}
				f_i(x)=\sum_{j=1}^r a^{-1}a_jx_j.
			\end{equation*}
			Since $\rho_n(b)=\rho_np(x)=f(x)$, we then have
			\begin{equation*}
				|\rho_n(b)-\sum a_j\rho_n(m_j)|=\left|f(x)-f_i(x)+\sum_{j=1}^r a_j(a^{-1}\cdot x_j -\rho_n(m_j))\right|<\epsilon,
			\end{equation*}
			as required.
			\item By the density assumption, there exist $b_1, \hdots, b_r\in M_{n+2}$ such that 
			\begin{equation*}
				|m_j-\rho_{n+1}(b_j)|<\epsilon
			\end{equation*} 
			for each $j$. Set $B'=\sum \A b_j$.\qedhere
		\end{enumerate}  
	\end{proof}
	
	\begin{lem}
		\label{compactoid}
		Let $M$ be a nuclear $A$-module, and let $B$ be a bounded subset of $M$. If $V$ is an open $\A$-submodule of $M$, then there exist finitely many elements $x_1, \hdots, x_s\in M$ such that
		\begin{equation*}
			B\subseteq V+\sum \A x_i.
		\end{equation*}
	\end{lem}
	\begin{proof}
		For each $n$, let $p_n: M\to M_n$ denote the natural map. By definition of the Fr\'echet topology, we can assume without loss of generality that there exists $N$ such that $V=p_N^{-1}(U)$ for some open $\A$-submodule $U\subseteq M_N$. Moreover, $p_{N+1}(B)$ is bounded in $M_{N+1}$.\\
		Without loss of generality, we can again assume that $U$ consists of all elements of $M_N$ of norm $<\epsilon$ and that the map $\rho_N: M_{N+1}\to M_N$ has sup norm $\leq 1$. Now Lemma \ref{nuclearimpliespre}.(i) yields $m_1, \hdots, m_s\in M_{N+1}$ such that 
		\begin{equation*}
			\rho_N(p_{N+1}(B))\subseteq U+\sum \A \rho_N(m_j).
		\end{equation*}
		As before, use the density of the map $M\to M_{N+1}$ to obtain $x_1, \hdots x_s\in M$ such that $|p_{N+1}(x_j)-m_j|<\epsilon$.
		
		Then by construction, for each $b\in B$ there exist $a_1, \hdots, a_s\in \A$ such that 
		\begin{equation*}
			\left|p_N(b)-p_N\left(\sum a_jx_j\right)\right|<\epsilon, 
		\end{equation*}
		i.e. $b-\sum a_jx_j\in p_N^{-1}(U)=V$.
	\end{proof}
	This result can be further improved if we make an additional assumption on $\A$: we assume from now on throughout that $\A$ is left Notherian if $K$ is discretely valued resp. almost left Noetherian if $K$ is not discretely valued. Note that this condition holds in particular whenever $\A$ is an admissible affine formal model of some affinoid $K$-algebra $A$.
	
	\begin{lem}[{compare \cite[Lemma 3.8.7]{Schikhof}}]
		\label{selfcompactoid}
		Let $M$ be a Fr\'echet $A$-module. Let $B$, $V$ be $\A$-submodules of $M$ and let $x\in M$ such that
		\begin{equation*}
			B\subseteq V+\A x.
		\end{equation*}
		Let $a\in R$, $a=1$ if $K$ is discretely valued, otherwise $0<|a|<1$. Then there exist $y_1, \hdots, y_n\in B$ such that
		\begin{equation*}
			a\cdot B\subseteq V+\sum_{i=1}^n \A y_i.
		\end{equation*}
	\end{lem}
	\begin{proof}
		Let $S=\{r\in \A: rx\in B+V\}$, a left ideal of $\A$. By our hypothesis on $\A$, there exists a finitely generated ideal $T\subseteq S$ such that $a\cdot S\subseteq T$.
		
		Let $T$ be generated by $r_1, \hdots, r_n$, say. Since $T\subseteq S$, we can find $y_i\in B$, $v_i\in V$ such that $y_i=v_i+r_ix$ for each $i$. 
		
		If now $b\in B$, then $b=v+rx$ for some $v\in V$, $r\in S$ by assumption. Then $ar=\sum a_ir_i$ for $a_i\in \A$, and hence
		\begin{align*}
			ab&=av+\sum a_ir_ix\\
			&=av+\sum a_i(y_i-v_i)\\
			&=(av-\sum a_iv_i) +\sum a_iy_i.
		\end{align*}
		Since $av-\sum a_iv_i\in V$, this proves the lemma.
	\end{proof}
	
	\begin{cor}[{compare \cite[Lemma 3.8.8]{Schikhof}}]
		\label{selfcpcdsequence}
		Let $M$ be a nuclear $A$-module, and let $B$ be a bounded $\A$-stable subset of $M$. Let $a\in R$, $a=1$ if $K$ is discretely valued, otherwise $0<|a|<1$. If $V$ is an open $\A$-submodule of $M$, then there exist $y_1, \hdots, y_s\in B$ such that
		\begin{equation*}
			a\cdot B\subseteq V+\sum_{i=1}^s \A y_i.
		\end{equation*}
	\end{cor}
	\begin{proof}
		Suppose first that $K$ is discretely valued. By Lemma \ref{compactoid}, there exist $x_1, \hdots, x_r\in M$ such that 
		\begin{equation*}
			B\subseteq V+\sum A^\circ x_i=(V+\sum_{i>1} \A x_i) + \A x_1.
		\end{equation*} 
		By Lemma \ref{selfcompactoid}, there are $y_1, \hdots, y_n\in B$ such that
		\begin{equation*}
			B\subseteq V+\sum_{i>1}\A x_i+\sum_{i=1}^n \A y_i=((V+\sum \A y_i) +\sum_{i>2}\A x_i)+\A x_2.
		\end{equation*}
		Repeating this process inductively yields the result in the discretely valued case.
		
		Now suppose that $K$ is not discretely valued, so that the value group is dense. Pick $a_1, \hdots, a_r\in R$ such that $0<|a_i|<1$ and $\prod |a_i|\geq |a|$.
		
		As before, we have
		\begin{equation*}
			B\subseteq V+\sum \A x_i=(V+\sum_{i>1} \A x_i) + \A x_1.
		\end{equation*} 
		by Lemma \ref{compactoid}. Hence by Lemma \ref{selfcompactoid}, there exist $y_1, \hdots, y_n\in B$ such that
		\begin{equation*}
			a_1\cdot B\subseteq (V+\sum \A y_i +\sum_{i>2} \A x_i)+ \A x_2.
		\end{equation*}
		Repeating inductively yields $y_1, \hdots y_s\in B$ such that
		\begin{equation*}
			\left(\prod_{i=1}^r a_i\right)\cdot B\subseteq V+\sum_{i=1}^s\A y_i,
		\end{equation*}
		and the result follows, as $a\cdot B\subseteq (\prod a_i)\cdot B$. 
	\end{proof}
	
	We remark that by applying the Corollary above in the case where $V$ gets replaced by the open submodule $a\cdot V$ and then multiplying throughout with $\lambda=a^{-1}$, we can find $y_1, \hdots, y_s\in \lambda\cdot B$ such that
	\begin{equation*}
		B\subseteq V+\sum_{i=1}^s \A y_s.
	\end{equation*} 
	\begin{lem}[{compare \cite[Theorem 3.8.25]{Schikhof}}]
		\label{Cauchygenerates}
		Let $M$ be a nuclear $A$-module, and let $B$ be a bounded subset of $M$. Then there exists a sequence $y_1, y_2, \hdots$ in $M$ tending to zero such that 
		\begin{equation*}
			B\subseteq \overline{\sum \A y_i}.
		\end{equation*}
		In fact, for each $b\in B$, there exist $a_1, a_2, \hdots\in \A$ such that $b=\sum_{i=1}^\infty a_iy_i$. 
	\end{lem}
	\begin{proof}
		Replacing $B$ by $\A\cdot B$, we can assume that $B$ is a bounded $\A$-submodule.
		
		Let $V_1\supseteq V_2\supseteq \hdots$ be a sequence of open $\A$-stable neighbourhoods of $0$ in $M$ such that $\cap V_i=\{0\}$. 
		
		We now pick $\lambda\in K$ with $\lambda=1$ if $K$ is discretely valued, $|\lambda|>1$ otherwise. Let $\lambda_1, \lambda_2, \hdots\in K$ be a sequence with $\lambda_i=1$ for all $i$ if $K$ is discretely valued, $|\lambda_i|>1$ with $\prod |\lambda_i|\leq|\lambda|$ otherwise.
		
		By Corollary \ref{selfcpcdsequence}, there exist $y_1, \hdots, y_s\in \lambda_1\cdot B$ such that
		\begin{equation*}
			B\subseteq V_1+\sum \A y_i.
		\end{equation*}
		Since $y_i\in \lambda_1\cdot B$ for each $i$ and $B\subseteq \lambda_1\cdot B$, we immediately deduce
		\begin{equation*}
			B\subseteq (V_1\cap \lambda_1\cdot B)+\sum \A y_i.
		\end{equation*}
		But now $V_1\cap \lambda_1 B$ is itself a bounded $\A$-submodule, so there exist $y_{s+1}, \hdots, y_r\in \lambda_2\cdot(V_1\cap \lambda_1 B)$ such that
		\begin{equation*}
			V_1\cap \lambda_1 B\subseteq V_2+\sum_{i=s+1}^r \A y_i
		\end{equation*}
		and hence $V_1\cap \lambda_1 B\subseteq (V_2\cap \lambda_1\lambda_2 B)+\sum_{i=s+1}^r \A y_i$.
		
		In particular,
		\begin{equation*}
			B\subseteq (V_2\cap \lambda_1\lambda_2 B)+\sum_{i=1}^r \A y_i.
		\end{equation*}
		Continuing in this way yields a sequence $y_1, y_2, \hdots$ in $M$ such that for each $i$, all but finitely many terms are contained in $(\prod_{j=1}^i \lambda_j)\cdot V_i\subseteq \lambda\cdot V_i$ and
		\begin{equation*}
			B\subseteq \left(V_i\cap \prod_{j=1}^i\lambda_j\cdot B\right)+\sum_{j=1}^\infty A^\circ y_j\subseteq V_i\cap \lambda\cdot B+\sum_{j=1}^\infty \A y_j.
		\end{equation*}
		Thus $(y_i)$ is a Cauchy sequence tending to zero such that $B\subseteq \overline{\sum \A y_i}$, and moreover, each $b\in B$ can be written as $\sum a_i y_i$ for some $a_1, a_2, \hdots \in \A$ by construction.
	\end{proof}

	\begin{prop}
		\label{strictepis}
		Let $f: M\to N$ be a continuous surjection of Fr\'echet $A$-modules, and assume $N$ is a nuclear $A$-module. Then $f^b: M^b\to N^b$ is a strict epimorphism in $\mathrm{Mod}_{\h{\B} c_K}(A)$.
	\end{prop}
	\begin{proof}
		By the Open Mapping Theorem, $f$ is a strict epimorphism in $LCVS_K$.\\
		Let $B$ be bounded subset in $N$. It suffices to show that $B$ is contained in the image of some bounded subset in $M$. By Lemma \ref{Cauchygenerates}, there exists a sequence $(x_i)$ in $N$ tending to zero such that $B$ is contained in $\overline{\sum \A x_i}$. By \cite[Lemma 3.5.8]{Schikhof}, there exists $y_i\in M$ such that $f(y_i)=x_i$ for each $i$, and the $y_i$ tend to zero in $M$.\\
		Then $B'=\overline{\sum \A y_i}$ is a bounded subset of $M$ (by Lemma \ref{Cauchybounded}) such that $f(B')$ contains $B$: if $b\in B$, then by Lemma \ref{Cauchygenerates} there exist $a_1, a_2, \hdots \in \A$ such that $b=\sum a_ix_i$. But now the series $\sum a_i y_i$ converges in $M$ since the $y_i$ form a null sequence, and $b=f(\sum a_iy_i)$. Since $\sum a_iy_i\in \overline{\sum \A y_i}$, this proves the claim. 
	\end{proof}
	
	\begin{cor}
		\label{embedding}
		Let $A$ be an affinoid $K$-algebra, and let $L$ be a smooth $(K, A)$-Lie--Rinehart algebra admitting a smooth Lie lattice. Then the functor
		\begin{equation*}
			(-)^b: \C_{\w{U_A(L)}}\to \mathrm{Mod}_{\h{\B} c_K}(\w{U_A(L)})
		\end{equation*}
		is exact and fully faithful.
	\end{cor}
	\begin{proof}
		We have already seen earlier that the functor $(-)^b$ is fully faithful on metrisable spaces.
		
		Given a short exact sequence 
		\begin{equation*}
			0\to M_1\to M_2\to M_3\to 0
		\end{equation*}
		of coadmissible $\w{U_A(L)}$-modules, we want to show that 
		\begin{equation*}
			0\to M_1^b\to M_2^b\to M_3^b\to 0
		\end{equation*} 
		is exact in $\h{\B} c_K$. 
		
		As $(-)^b$ is left exact, we only need to verify that $M_2^b\to M_3^b$ is a strict epimorphism. This follows from the previous proposition, since $M_3$ is $A$-nuclear.
	\end{proof}
	
	\begin{cor}
		\label{coadstrict}
		Let $A$ be an affinoid $K$-algebra, and let $L$ be a smooth $(K, A)$-Lie--Rinehart algebra admitting a smooth Lie lattice. If $M, N$ are coadmissible $\w{U_A(L)}$-modules, then any morphism $f:M^b\to N^b$ is strict.
	\end{cor}
	\begin{proof}
		As $(-)^b$ is fully faithful on Fr\'echet spaces, there exists a (continuous) $\w{U_A(L)}$-module morphism $f': M\to N$ such that $f'^b=f$. Since coadmissible modules form an abelian category (\cite[Corollary 3.4, Corollary 3.5]{ST}), we have exact sequences of coadmissible $\w{U_A(L)}$-modules
		\begin{equation*}
			0\to \ker f' \to M\to \coim f'\to 0
		\end{equation*} 
		and
		\begin{equation*}
			0\to \im f' \to N\to \coker f'\to 0,
		\end{equation*}
		with the natural morphism $\coim f'\to \im f'$ an isomorphism, since any morphism between coadmissible modules is topologically strict (\cite[Proposition 2.1.iii]{ST}).\\
		But by the above, $\coim f\cong(\coim f')^b$ and $\im f\cong(\im f')^b$, proving that $f$ is strict.
	\end{proof}
	
	We remark that our results remain valid upon passing to $\mathrm{Ind}(\mathrm{Ban}_K)$:
	
	\begin{cor}
		\label{coadIBembedding}
		Let $A$ be an affinoid $K$-algebra, and let $L$ be a smooth $(K, A)$-Lie--Rinehart algebra admitting a smooth Lie lattice. Then the functor
		\begin{equation*}
			\mathrm{diss}_{\w{U_A(L)}}\circ (-)^b: \C_{\w{U_A(L)}}\to \mathrm{Mod}_{\mathrm{Ind}(\mathrm{Ban}_K)}(\mathrm{diss}(\w{U_A(L)}))
		\end{equation*}
		is exact and fully faithful.
		
		Any morphism in the essential image is strict. 
	\end{cor}
	\begin{proof}
		By Proposition \ref{IBFrechetmod}, the functor $\mathrm{diss}_{\w{U_A(L)}}$ is exact and fully faithful, so the result follows from the above.  
	\end{proof}

	\subsection{Pseudo-nuclearity}
	While the notion of $A$-nuclearity is useful, it has its limitations. For example, if $M$, $N$ are nuclear over $A$, then $M\otimes_K N$ or $M\h{\otimes}_K N$ will often enjoy similar good properties, but simply do not fit into the framework established thus far. We therefore need to introduce a more general notion, pseudo-nuclearity, taking Lemma \ref{nuclearimpliespre} as our point of departure.
	\begin{defn}
		A metrisable locally convex topological $K$-vector $V$ is \textbf{pseudo-nuclear} if we can write $V\cong \varprojlim V_n$ for $V_n$ semi-normed, the transition maps $\rho_n: V_{n+1}\to V_n$ having dense images, satisfying the following property:
		
		For any bounded subset $B\subseteq V_{n+1}$ and any open $R$-submodule $U\subseteq V_n$, there exists a bounded subset $B'\subseteq V$ such that 
		\begin{equation*}
			\rho_n(B)\subseteq U+\rho_n(p_{n+1}(B')),
		\end{equation*}
		where $p_{n+1}: V\to V_{n+1}$ denotes the natural map.
	\end{defn}
	Obviously any normed space is pseudo-nuclear, considering the constant inverse system. If $V$ is metrisable but not normed and $V\cong \varprojlim V_n$ as a countable limit of semi-normed spaces, then the natural morphism $V\to V_n$ is continuous, but not necessarily strict, and we can think of pseudo-nuclear vector spaces as those metrisable vector spaces where the maps $V\to V_n$ are very close to being bornologically strict (they are `strict up to $\epsilon$ and applying $\rho$').\
	
	If $M$ is nuclear over some Noetherian Banach $K$-algebra $A$, then the same argument as in Lemma \ref{nuclearimpliespre} shows that $M$ is pseudo-nuclear. 
	\begin{lem}
		\label{prodofpn}
		\leavevmode
		\begin{enumerate}[(i)]
			\item A countable direct product of pseudo-nuclear spaces is pseudo-nuclear.
			\item If $V$, $W$ are pseudo-nuclear, then so is $V\otimes_K W$.
		\end{enumerate}
	\end{lem}
	\begin{proof}
		This is straightforward.
	\end{proof}

	\begin{prop}
		\label{bandcompl}
		Let $V$ be a pseudo-nuclear space. Then the natural morphism $\h{V^b}\to \h{V}^b$ is an isomorphism.
	\end{prop}
	\begin{proof}
		As always, we can assume without loss of generality that the transition maps have sup norm $\leq 1$.
		
		Invoking Lemma \ref{tcompl}, we see that the natural morphism $\h{V^b}\to \h{V}^b$ is always a bounded injection. It remains to show strict surjectivity. Let $B$ be a bounded set of $\h{V}$. Let $|-|_n$ be the semi-norm on $V$ induced from $V_n$, so that the topology of $\h{V}$ is defined by the extensions of $|-|_n$, which we will also denote by $|-|_n$. So without loss of generality, there exists $\alpha=(\alpha_1, \alpha_2, \hdots)\in \mathbb{R}_{>0}^\mathbb{N}$ such that
		\begin{equation*}
			B=\{v\in \h{V}: |v|_n\leq \alpha_n \ \forall n\}.
		\end{equation*}
		We wish to show that there exists a bounded set $B'\subseteq V$ such that $B$ is contained in (the image of) the $\pi$-adic completion of $B'$.
		
		Pick a decreasing sequence $\epsilon_n$ of real numbers satisfying 
		\begin{equation*}
			0<\epsilon_n<\mathrm{min}_{m\leq n}\{|\pi|^n, |\pi|^n\alpha_m\}
		\end{equation*}
		for each $n$. For any $\epsilon>0$, we write $V_n(\epsilon)$ to denote the set of all $v\in V_n$ with $|v|_n\leq \epsilon$. By definition of pseudo-nuclearity, we can inductively find bounded sets $C_n\subseteq V$ such that
		\begin{equation*}
			\rho_{n-1}(V_n(\alpha_n))+p_{n-1}(C_1+\hdots +C_{n-1})\subseteq V_{n-1}(\epsilon_{n-1})+p_{n-1}(C_{n}).
		\end{equation*}
		Now let $x\in B$. By definition of completion, there exists a Cauchy sequence $x_i$ in $V$ such that $|x-x_n|_n\leq \epsilon_n$. In particular, $|x_n|_n\leq \alpha_n$, i.e. it can be seen as an element of $V_n(\alpha_n)$.\\
		Choose inductively $y_n\in C_n$ such that
		\begin{equation*}
			\left|x_n-\sum_{i=1}^{n} y_i\right|_{n-1}\leq \epsilon_{n-1}.
		\end{equation*}
		Set $D_n:=C_1+\hdots +C_n$, a bounded set in $V$, and let $z_n=\sum_{i=1}^ny_i\in D_n$. For each $m<n$, we have
		\begin{equation*}
			|x-z_n|_m\leq \mathrm{max}\{|x-x_{n}|_m, |x_{n}-z_n|_m\}\leq \epsilon_{n-1}.
		\end{equation*} 
		Hence $z_n$ is a Cauchy sequence tending to $x$. Also, by construction, 
		\begin{equation*}
			|z_n|_{m}=|x-(x-z_n)|_m\leq \alpha_{m}
		\end{equation*} 
		for each $m\leq n-1$, since $\epsilon_{n-1}\leq \epsilon_m\leq \alpha_m$.
		
		Since $D_n$ is bounded in $V$ for each $n$, there exists $\delta_m$ such that $|v|_m\leq \delta_m$ for all $v\in D_{m+1}$. Let $B'\subseteq V$ denote the bounded set consisting of all $v\in V$ such that $|v|_m\leq \beta_m$, where $\beta_m:=|\pi|^{-m}\mathrm{max}\{\delta_m, \alpha_m\}$.
		
		Now for each $n$, 
		\begin{equation*}
			|z_n|_m\leq \begin{cases}
				\alpha_m \ \text{if $m\leq n-1$}\\
				\delta_m\ \text{if $m\geq n-1$}
			\end{cases}
		\end{equation*}
		since $z_n\in D_n\subseteq D_{m+1}$ if $m\geq n-1$ (the overlap of cases is deliberate here). This shows that $z_n\in B'$ for each $n$.
		
		Moreover,
		\begin{equation*}
			|z_n-z_{n-1}|_m\leq \begin{cases}
				\epsilon_{n-2}\leq |\pi|^{n-2}\alpha_m\leq |\pi|^{n-2-m}\alpha_m \ \text{if $m\leq n-2$}\\
				\delta_m\leq |\pi|^{n-2-m}\delta_m \ \text{if $m\geq n-1$}.
			\end{cases}
		\end{equation*}
		Here the first case uses $|(z_n-x)-(z_{n-1}-x)|_m\leq \epsilon_{n-2}$, and the second case is clear as $z_n, z_{n-1}\in D_{m+1}$ as soon as $n\leq m+1$. 
		Hence $z_n-z_{n-1}\in \pi^{n-2}B'$ for each $n$, and the sequence $(z_n)$ is also Cauchy in the semi-normed space $V_{B'}$. Thus any $x\in B$ is in the image of the $\pi$-adic completion of $B'$, a bounded subset of $\h{V^b}$. Therefore the bounded injection $\h{V^b}\to \h{V}^b$ is also surjective and strict, and hence an isomorphism. 
	\end{proof}
	
	\begin{cor}
		\label{capcompletion}
		Let $A$ be an affinoid $K$-algebra, and let $L$ be a smooth $(K, A)$-Lie--Rinehart algebra. Endowing $U_A(L)$ with the locally convex topology induced by $U_A(L)\to \w{U_A(L)}$, the complete bornological algebra $\w{U_A(L)}^b$ is the completion of $U_A(L)^b$.
	\end{cor}
	\begin{proof}
		Picking an $(R, \A)$-Lie lattice $\L$, endow $U_{\A}(\pi^n\L)\otimes_R K$ with the natural gauge semi-norm, so that
		\begin{equation*}
			U_A(L)\cong \varprojlim \left(U_{\A}(\pi^n\L)\otimes_R K\right).
		\end{equation*}
		Now $U_A(L)$ is pseudo-nuclear, as for any $j$
		\begin{equation*}
			U_{\A}(\pi^n\L)\subseteq \pi^jU_{\A}(\pi^{n-1}\L)+F_jU_{\A}(\pi^n\L),
		\end{equation*}
		where $F_jU(\pi^n\L)$ denotes the $j$th filtered piece of the PBW filtration (see \cite[section 3]{Rinehart}), a finitely generated $\A$-module and hence bounded in $U_A(L)$.
		
		Hence $\h{U_A(L)^b}\cong \h{U_A(L)}^b$ by Proposition \ref{bandcompl}, and $\h{U_A(L)}\cong \w{U_A(L)}$ by definition of the topology.
	\end{proof}

	\begin{cor}
		\label{pncompl}
		Let $V$, $W$ be pseudo-nuclear spaces, and suppose that at least one of the following is satisfied:
		\begin{enumerate}[(i)]
			\item $K$ is spherically complete.
			\item $V$ is of countable type.
		\end{enumerate}
		Then the natural morphism $V^b\h{\otimes} W^b\to (V\h{\otimes} W)^b$ is an isomorphism.
	\end{cor}
	\begin{proof}
		By Lemma \ref{Frtensor}, the morphism $V^b\otimes W^b\to (V\otimes W)^b$ is an isomorphism, so $V^b\h{\otimes} W^b$ is naturally isomorphic to the completion of $(V\otimes W)^b$. Since $V\otimes W$ is pseudo-nuclear, the result follow from Proposition \ref{bandcompl}.
	\end{proof}
	
	\begin{cor}
		\label{commutewprod}
		Let $V_i$ be a countable family of pseudo-nuclear spaces, and let $W$ be another pseudo-nuclear space. Suppose that at least one of the following is satisfied:
		\begin{enumerate}[(i)]
			\item $K$ is spherically complete.
			\item $W$ is of countable type.
			\item $V_i$ is of countable type for all $i$.
		\end{enumerate}
		Then
		\begin{equation*}
			(\prod_i V_i^b)\h{\otimes}_K W^b\cong \prod_i (V_i^b\h{\otimes}_K W^b)
		\end{equation*}
		via the natural morphism.
	\end{cor}
	\begin{proof}
		Note that we have a natural isomorphism $\prod V_i^b\cong (\prod V_i)^b$, and hence by Corollary \ref{pncompl}
		\begin{align*}
			(\prod V_i^b)\h{\otimes} W^b&\cong (\prod V_i)^b\h{\otimes} W^b\\
			&\cong ((\prod V_i)\h{\otimes}W)^b.
		\end{align*}
		In the case when $V_i$ is of countable type for all $i$, this uses the fact that $\prod V_i$ is then also of countable type by \cite[Theorem 4.2.13.(iii)]{Schikhof}.
		
		By \cite[p. 192, Proposition 9]{Houzel} and once more Corollary \ref{pncompl}, we now have
		\begin{align*}
			((\prod V_i)\h{\otimes}W)^b&\cong (\prod(V_i\h{\otimes}W))^b\\
			&\cong \prod (V_i\h{\otimes}W)^b\\
			&\cong \prod (V_i^b\h{\otimes}W^b),
		\end{align*}
		as required.
	\end{proof}
	
	\begin{cor}
		\label{commutewinv}
		Let $V_i$ be an inverse system (indexed by $\mathbb{N}$) of pseudo-nuclear Fr\'echet spaces such that each transition map has dense image. Assume that the Fr\'echet space $\varprojlim V_i$ is also pseudo-nuclear, and let $W$ be another pseudo-nuclear space. Suppose that at least one of the following is satisfied:
		\begin{enumerate}[(i)]
			\item $K$ is spherically complete.
			\item $W$ is of countable type.
			\item $V_i$ is of countable type for all $i$.
		\end{enumerate}
		Then
		\begin{equation*}
			(\varprojlim V_i^b)\h{\otimes}_K W^b\cong \varprojlim (V_i^b\h{\otimes}_K W^b)\cong \varprojlim (V_i\h{\otimes}_K W)^b
		\end{equation*}
		via the natural morphisms.
	\end{cor}
	\begin{proof}
		By \cite[p. 45, Lemme 1]{AST16}, we have a strictly exact sequence
		\begin{equation*}
			0\to \varprojlim V_i\to \prod V_i\to \prod V_i\to 0
		\end{equation*}
		in $LCVS_K$. As $-\h{\otimes}_K W$ is left exact in $LCVS_K$ (see e.g \cite[Fact 2.0.4.(3)]{CDNStein}) and $(-)^b$ preserves kernels by adjunction, the result follows from Corollary \ref{commutewprod} and Corollary \ref{pncompl}.
	\end{proof}
	
	We remark that our results remain valid upon passing to $\mathrm{Ind}(\mathrm{Ban}_K)$:
	
	\begin{cor}
		\label{commutewinvIB}
		Let $V_i$ be a countable family of pseudo-nuclear spaces and let $W$ be another pseudo-nuclear space. Suppose that at least one of the following is satisfied:
		\begin{enumerate}[(i)]
			\item $K$ is spherically complete.
			\item $W$ is of countable type.
			\item $V_i$ is of countable type for all $i$.
		\end{enumerate}
		Then we have the following:
		\begin{enumerate}[(i)]
			\item The natural morphism $\mathrm{diss}(V_i^b)\overset{\rightarrow}{\otimes}_K \mathrm{diss}(W^b)\to \mathrm{diss}((V_i\h{\otimes}_KW)^b)$ is an isomorphism in $\mathrm{Ind}(\mathrm{Ban}_K)$.
			\item The natural morphism
			\begin{equation*}
				\left(\prod_i \mathrm{diss}(V_i^b)\right)\overset{\rightarrow}{\otimes}_K \mathrm{diss}(W^b)\to \prod_i (\mathrm{diss}(V_i^b)\overset{\rightarrow}{\otimes}_K \mathrm{diss}(W^b))
			\end{equation*}
			is an isomorphism in $\mathrm{Ind}(\mathrm{Ban}_K)$.
			\item If the $V_i$ form an inverse system of Fr\'echet spaces with each transition map having dense image such that $\varprojlim V_i$ is also pseudo-nuclear, then the natural morphism
			\begin{equation*}
				\left(\varprojlim \mathrm{diss}(V_i^b)\right)\overset{\rightarrow}{\otimes}_K \mathrm{diss}(W^b)\to \varprojlim (\mathrm{diss}(V_i^b)\overset{\rightarrow}{\otimes}_K \mathrm{diss}(W^b))
			\end{equation*}
			is an isomorphism.
		\end{enumerate}
	\end{cor}
	\begin{proof}
		Everything follows from the above together with Proposition \ref{dissandtensor}, since $\mathrm{diss}$ commutes with arbitrary limits.
	\end{proof}

	\subsection{A Mittag-Leffler result}
	We will need some results concerning inverse limits in $\h{\B} c_K$. We begin with a definition.
	\begin{defn}
		An inverse system 
		\begin{equation*}
			\begin{xy}
				\xymatrix{\hdots \ar[r]& V_{n+1}\ar[r]^{\rho_n}& V_n\ar[r]& \hdots \ar[r]& V_0}
			\end{xy}
		\end{equation*} 
		of Banach $K$-vector spaces is called \textbf{pre-nuclear} if for each $n$, the following condition is satisfied:\\
		Given a bounded subset $B\subseteq V_{n+1}$ and an open $R$-submodule $\V\subseteq V_n$, there exists a bounded subset $B'\subseteq V_{n+2}$ such that
		\begin{equation*}
			\rho_n(B)\subseteq\V+\rho_n(\rho_{n+1}(B')).
		\end{equation*}
	\end{defn}
	As we are dealing with Banach spaces, the condition can be reformulated as: for each $n$ and any $\epsilon>0$, there exists $m\in \mathbb{Z}$ such that for any $v\in V_{n+1}^\circ$ we can find $w\in \pi^mV_{n+2}^\circ$ with
	\begin{equation*}
		|\rho_n(v)-\rho_n(\rho_{n+1}(w))|<\epsilon.
	\end{equation*}
	Also note that in a pre-nuclear system, the property in the definition can be iterated: considering $B'\subseteq V_{n+2}$, $\rho^{-1}(\V)\subseteq V_{n+1}$, we can find $B''\subseteq V_{n+3}$ such that
	\begin{equation*}
		\rho_n(B)\subseteq \V+\rho_n\rho_{n+1}\rho_{n+2}(B''),
	\end{equation*}
	and so on.
	
	We say that $(V_n)_n$ is \textbf{pre-nuclear with dense images} if furthermore each of the transition maps has dense image.
	
	If $A$ is an affinoid $K$-algebra, $L$ a smooth $(K, A)$-Lie--Rinehart algebra admitting a smooth Lie lattice, then for any coadmissible $\w{U_A(L)}$-module $M=\varprojlim M_n$, Lemma \ref{nuclearimpliespre} implies that the system formed by the $M_n$ is pre-nuclear with dense images. More generally, an inverse system of Banach spaces describing a Fr\'echet space as a pseudo-nuclear space is pre-nuclear.
	\begin{lem}
		\label{newprenuc}
		\leavevmode
		\begin{enumerate}[(i)]
			\item If $(V_i)$ and $(W_i)$ are two pre-nuclear systems, then $(V_i\h{\otimes}_KW_i)$ is pre-nuclear.
			\item Let $(V_i)$ be a pre-nuclear system and let $W_i\leq V_i$ be closed subspaces such that $(V_i/W_i)$ forms an inverse system of Banach spaces (i.e., $\rho_n(W_{n+1})\subseteq W_n$). Then $(V_i/W_i)$ is pre-nuclear.
		\end{enumerate}
	\end{lem}
	\begin{proof}
		(i) is immediate from the definition of $\h{\otimes}$. For (ii), any bounded set $B$ in $V_{n+1}/W_{n+1}$ is contained in the image of some bounded set $\widetilde{B}\subseteq V_{n+1}$. Then we can find a suitable $\widetilde{B}'$ bounded in $V_{n+2}$, whose image in $V_{n+2}/W_{n+2}$ has the desired property.
	\end{proof}
	
	Recall from \cite[Lemma 3.76]{BamStein} that for any inverse system $(V_n, \rho_n)$ in $\widehat{\mathcal{B}}c_K$ indexed by $\mathbb{N}$, the right derived inverse limit $\mathrm{R}\varprojlim V_n$ exists and can be represented by the \textbf{Roos complex}
	\begin{align*}
		\prod V_n&\to \prod V_n\\
		(v_0, v_1, v_2, \hdots)&\mapsto (v_0-\rho(v_1), v_1-\rho(v_2), v_2-\rho(v_3), \hdots).
	\end{align*} 
	
	\begin{thm}
		\label{MLBan}
		Let $\hdots \to V_{n+1}\to V_n\to \hdots\to V_0$ be a pre-nuclear system of Banach $K$-vector spaces such that each transition map has dense image. Then
		\begin{equation*}
			\mathrm{R}\varprojlim V^b_n\cong \varprojlim V^b_n
		\end{equation*}
		in $\mathrm{D}(\widehat{\mathcal{B}}c_K)$.
	\end{thm}
	\begin{proof}
		This is closely modelled on \cite[p. 45, Lemme 1]{AST16}, from which we already know that $\prod V_n\to \prod V_n$ is a topologically strict surjection of Fr\'echet spaces. It thus suffices to show that it is also bornologically strict.
		
		Suppose without loss of generality that each $\rho_n$ satisfies $||\rho_n||\leq 1$.
		
		As we will deal with sequences in product spaces, let us use the following notational convention: if $v\in \prod V_i$, we denote by $v_n\in V_n$ the image of $v$ under the natural projection map to $V_n$. Sequences of elements will be written in a functional way, e.g. $v(1), v(2), \hdots$.
		
		Let $B=\prod B_n\subseteq \prod V_n$ be a bounded subset, with each $B_n$ a bounded $R$-submodule of $V_n$. Let $v=(v_n)\in B$, and define $w(n)\in \prod_{i=0}^{n+1} V_i$ as follows:
		\begin{equation*}
			w(n)_i=v_i+\rho_i(v_{i+1})+\hdots+ \rho_i\circ \hdots \circ \rho_{n-1}(v_n)\in V_i
		\end{equation*}
		for $i\leq n$, and $w(n)_{n+1}=0$.
		
		Writing $B^n_i:=\rho_i\circ\hdots\circ  \rho_{n-1}(B_n)\subseteq  V_i$ for any $i\leq n$, set $C^n_i=\sum_{j=i}^n B^j_i\subseteq V_i$, a bounded set. Then $w(n)\in \prod_{i=0}^nC^n_i\times \{0\}$, and $w(n)$ maps to $(v_1, \hdots, v_n)$ under the truncated Roos map 
		\begin{align*}
			\prod^{n+1}_{i=0}V_i&\to \prod^n_{i=0} V_i\\
			(x_0, \hdots, x_{n+1})&\mapsto (x_0-\rho(x_1), \hdots, x_n-\rho(x_{n+1})).
		\end{align*}
		By pre-nuclearity, we can use induction to find bounded subsets $\widetilde{B}_n\subseteq V_n$ such that for any $x\in C_n^{n+1}+\rho_n(\widetilde{B}_{n+1})$, there exists $u\in \widetilde{B}_{n+2}$ such that
		\begin{equation*}
			|\rho_{n-1}\rho_n\rho_{n+1}(u)-\rho_{n-1}(x)|_{V_{n-1}}<1/2^n.
		\end{equation*}
		Write 
		\begin{equation*}
			\sigma_n: \prod^{n+1}_{i=0}V_i\to \prod^n_{i=0}V_i
		\end{equation*}
		for the projection map. As in \cite[Lemme 1]{AST16}, we will now construct inductively $w'(n)\in \prod^{n+1}_{i=0}V_i$ such that
		\begin{enumerate}[(i)]
			\item $w'(n)$ maps to $(v_1, \hdots, v_n)$ under the truncated Roos map,
			\item $|\sigma_{n-1}\sigma_n(w'(n))-\sigma_{n-1}\sigma_n\sigma_{n+1}(w'(n+1))|_{V_{n-1}}<1/2^n$ for each $n$.
			\item $\sigma_n(w'(n))\in \prod_{i=0}^n(C_i^{n+1}+\rho_i\hdots\rho_n(\widetilde{B}_{n+1}))$ for each $n$.
		\end{enumerate}
		Suppose we have found $w'(n)$. Then set
		\begin{equation*}
			x(n)=\sigma_{n+1}(w(n+1))-w'(n)\in \prod^{n+1}_{i=0} V_i.
		\end{equation*}
		The short exact sequence
		\begin{equation*}
			0\to V_{n+1}\to \prod^{n+1}_{i=0} V_i\to \prod^n_{i=0}V_i\to 0,
		\end{equation*}
		where the injection is the diagonal map and the surjection is the truncated Roos map, shows that $x(n)$ can be identified with an element of $V_{n+1}$. 
		
		As $w(n+1)_n$ is contained in $C_n^{n+1}$ and $w'(n)_n$ is in $C_n^{n+1}+\rho_n(\widetilde{B}_{n+1})$, we have $\rho_n(x(n))\in C_n^{n+1}+\rho_n(\widetilde{B}_{n+1})$.\\
		By construction, we can pick $u(n+1)\in \widetilde{B}_{n+2}$ such that 
		\begin{equation*}
			|\rho_{n-1}\rho_n\rho_{n+1}(u(n+1))-\rho_{n-1}\rho_n(x(n))|_{V_{n-1}}<1/2^n, 
		\end{equation*}
		and we set $w'(n+1)=w(n+1)-u(n+1)$. It is straightforward from the construction that this has the desired properties.
		
		Finally, fix $n$ and consider in $\prod^{n+1}_{j=0}V_j$ the sequence consisting of 
		\begin{equation*}
			\sigma_{n+1}\hdots \sigma_i(w'(i)),
		\end{equation*} 
		for $i>n$. By (ii), this is a Cauchy sequence, converging to some $y(n)\in \prod^{n+1}_{j=0}V_j$. By construction, $y(n)$ maps to $(v_1, \hdots, v_n)$ under the truncated Roos map, but moreover, $\sigma(y(n))=y(n-1)$. Hence the components of $y(n)$ for varying $n$ describe an element $y\in \prod_{i=0}^\infty V_i$ which is a preimage of $v$.
		
		By construction, 
		\begin{equation*}
			\sigma_{n-1}\sigma_n(w'(n))\in \prod_{i=0}^{n-1}(C_i^{n+1}+\rho_i\hdots\rho_n(\widetilde{B}_{n+1})).
		\end{equation*}
		Let $\alpha_n\in \mathbb{R}$ be such that 
		\begin{equation*}
			|C_{n-1}^{n+1}+\rho_{n-1}\rho_n(\widetilde{B}_{n+1})|_{V_{n-1}}<\alpha_{n-1},
		\end{equation*} 
		then by (ii), we have
		\begin{equation*}
			|\sigma_{n-1}\hdots\sigma_{m}(w'(m))|_{V_{n-1}}<\mathrm{max}\{\alpha_{n-1}, 1/2^n\}
		\end{equation*} 
		for all $m>n$. Thus $y_n$ is bounded by $\mathrm{max}\{\alpha_n, 1/2^{n+1}\}$, and the Roos map is strict.
	\end{proof}
	
	We remark that the Theorem does not necessarily hold without the pre-nuclearity assumption. For example, let $V_n=K\langle x\rangle$ for each $n$, with $\rho_n=\frac{\mathrm{d}}{\mathrm{d}x}$ for each $n$. The $\rho_n$ are continuous morphisms of Banach spaces with dense images, so the morphism $\prod_n K\langle x\rangle\to \prod_n K\langle x\rangle$ is a continuous surjection, but $\prod_n K\langle x\rangle^b\to \prod_n K\langle x\rangle^b$ is not bornologically strict:
	
	Suppose for a contradiction that the Roos map
	\begin{equation*}
		\prod_n K\langle x\rangle^b\to \prod_n K\langle x \rangle^b
	\end{equation*}
	is bornologically strict.
	
	For each $n$, choose $j_n\in \mathbb{Z}$ such that 
	\begin{equation*}
		|n! \pi^{j_n}|\to \infty.
	\end{equation*} 
	For instance, we can choose $j_n$ such that $|\pi^{j_n}|\geq p^n$.
	
	Let $B_n\subseteq K\langle x\rangle$ denote the subset of all elements of norm $\leq |\pi^{j_n}|$. Then $\prod B_n$ is a bounded subset of $\prod V_n$, so by assumption there exist some bounded sets $C_n\subseteq K\langle x\rangle$ such that $\prod B_n$ is contained in the image of $\prod C_n$ under the Roos map. Without loss of generality, we can assume that there exists $k_n\in \mathbb{Z}$ such that $C_n$ consists of elements of norm $\leq |\pi^{k_n}|$.
	
	Choose an increasing sequence of positive integers $r_0<r_1<\hdots <r_n<\hdots$ satisfying the following:
	\begin{equation*}
		\left|\frac{\pi^{j_n}}{p^{r_n}}\right|>|\pi^{k_{n+1}}| \ \text{for each } n.
	\end{equation*}
	Note that such a sequence also satisfies
	\begin{equation*}
		p^{r_n}-p^{r_s}>r_n-r_s\geq n-s
	\end{equation*}
	for all $s<n$.
	
	Moreover, as $p^{r_n}>n$ for all $n$, we have
	\begin{equation*}
		|(p^{r_n}-1)\cdot(p^{r_n}-2)\cdot \hdots \cdot (p^{r_n}-n)|=|n!|
	\end{equation*}
	for each $n$.
	
	Let
	\begin{equation*}
		(\pi^{j_n}x^{p^{r_n}-1})_n\in \prod B_n\subseteq \prod V_n
	\end{equation*}
	and let 
	\begin{equation*}
		(g_n)_n\in \prod C_n\subseteq \prod V_n
	\end{equation*}
	be a preimage.
	
	Thus by construction,
	\begin{equation*}
		\frac{\mathrm{d}}{\mathrm{d}x}(g_1)=\pi^{j_0}x^{p^{r_0}-1}+g_0.
	\end{equation*}
	As $g_1\in C_1$, the $x^{p^{r_0}}$-coefficient of $g_1$ is bounded by $|\pi^{k_1}|$, so the $x^{p^{r_0}-1}$-coefficient of $\frac{\mathrm{d}}{\mathrm{d}x}(g_1)$ is bounded by $|p^{r_0}\cdot \pi^{k_1}|$. But we have chosen the $r_n$ such that 
	\begin{equation*}
		|\pi^{j_0}|>|p^{r_0}\cdot \pi^{k_1}|,
	\end{equation*}
	so this is only possible if the $x^{p^{r_0}-1}$-coefficient of $g_0$ has norm exactly equal to $|\pi^{j_0}|$. 
	
	In the same way,
	\begin{align*}
		\frac{\mathrm{d}^2}{\mathrm{d}x^2}(g_2)&= \frac{\mathrm{d}}{\mathrm{d}x}(\pi^{j_1}x^{p^{r_1}-1}+g_1)\\
		&=\pi^{j_1}(p^{r_1}-1)x^{p^{r_1}-2}+\pi^{j_0}x^{p^{r_0}-1}+g_0.
	\end{align*}
	The $x^{p^{r_1}-2}$-coefficient of $\frac{\mathrm{d}^2}{\mathrm{d}x^2}(g_2)$ is bounded by $|p^{r_1}\cdot (p^{r_1}-1)\cdot \pi^{k_2}|<|\pi^{j_1}\cdot (p^{r_1}-1)|$. Thus the $x^{p^{r_1}-2}$-coefficient of $g_0$ must have norm equal to $|\pi^{j_1}\cdot (p^{r_1}-1)|$.
	
	The same argument shows in general that the $x^{p^{r_n}-n-1}$-coefficient of $g_0$ has norm equal to
	\begin{equation*}
		|\pi^{j_n}\cdot (p^{r_n}-1)\cdot (p^{r_n}-2)\cdot \hdots \cdot (p^{r_n}-n)|=|\pi^{j_n}\cdot n!|.
	\end{equation*}
	But these norms tend to infinity by assumption, so $g_0$ cannot be an element of $K\langle x\rangle$. We have thus arrived at the desired contradiction.
	
	\begin{cor}[{compare \cite[section 3, Theorem B]{ST}}]
		\label{MLcoad}
		Let $A$ be an affinoid $K$-algebra with admissible affine formal model $\A$, and let $L$ be a smooth $(K, A)$-Lie--Rinehart algebra admitting a smooth $(R, \A)$-Lie lattice $\L$. Write $U_n=\h{U_{\A}(\pi^n\L)}\otimes_RK$, so that $\w{U_A(L)}=\varprojlim U_n$. If $M\cong \varprojlim M_n$ is a coadmissible $\w{U_A(L)}$-module, then the system $(M_n)^b$ is $\varprojlim$-acyclic in $\h{\B} c_K$.
	\end{cor}

	\begin{cor}
		\label{MLcomplexes}
		Let
		\begin{equation*}
			V_i^\bullet=(\hdots \to V_i^j\to V_i^{j+1}\to\hdots )
		\end{equation*}
		be a strict complex of Banach spaces for each $i$ such that $(V_i^\bullet)$ is an inverse system of complexes of Banach spaces. Suppose that for each $j$, $(V_i^j)$ and $(\mathrm{H}^j(V_i^\bullet))$ form pre-nuclear systems with dense images.
		
		Then $V^\bullet:=\varprojlim (V_i^\bullet)^b$ is a strict complex of complete bornological vector spaces, and the natural morphism
		\begin{equation*}
			\mathrm{H}^j(V^\bullet)\to \varprojlim \mathrm{H}^j(V_i^\bullet)^b
		\end{equation*}
		is an isomorphism.
	\end{cor}
	\begin{proof}
		For each $i$ and $j$, we have the strictly exact sequences
		\begin{equation*}
			0\to \mathrm{ker}d\to V_i^j\to \mathrm{coim}d\to 0
		\end{equation*}
		and 
		\begin{equation*}
			0\to \mathrm{Im} d\to \mathrm{ker}d\to \mathrm{H}^j(V_i^\bullet)\to 0
		\end{equation*}
		of Banach spaces, where by strictness assumption $\mathrm{Coim}d\cong \mathrm{Im}d$. Recall that $(-)^b$ is exact on Banach spaces, so $\mathrm{Im}(d^b)\cong (\mathrm{Im}d)^b$, $\mathrm{ker}(d^b)\cong (\mathrm{ker}d)^b$. By Lemma \ref{newprenuc}, $(\mathrm{Im}d)_i$ is pre-nuclear with dense images, and hence both $(\mathrm{Im}d)^b_i$ and $(\mathrm{H}^j(V_i^\bullet))^b_i$ are $\varprojlim$-acyclic systems in $\h{\B}c_K$ by Theorem \ref{MLBan}. But then $(\mathrm{ker}d)^b_i$ is also $\varprojlim$-acyclic, so all terms appearing in the sequences above are acyclic for all $j$. Applying $\varprojlim$ in $\h{\B}c_K$ yields the result.
	\end{proof}
	
	\begin{lem}
		\label{MLforimages}
		Let $(f_n: B_n\to C_n)$ be a morphism of inverse systems of Banach spaces, and suppose that for each $n$ and every $y_n\in C_n$, there exists $x_{n-1}\in B_{n-1}$ such that $f_{n-1}(x_{n-1})=\rho_{n-1}(y_n)$.
		
		Suppose that $A_n:=\mathrm{ker}f_n$ is a pre-nuclear system with dense images. Then
		\begin{equation*}
			0\to \varprojlim A_n^b\to \varprojlim B_n^b\to \varprojlim C_n^b\to 0
		\end{equation*} 
		is strictly exact in $\h{\B} c_K$.
	\end{lem}
	\begin{proof}
		By the previous results, it suffices to show that $\varprojlim (\mathrm{Coim}f_n)\cong \varprojlim C_n$. But the assumption implies that the natural morphism is a continuous bijection, which is then an isomorphism by the Open Mapping Theorem for Fr\'echet spaces. 
	\end{proof}
	
	We remark that all these results remain valid upon passing to $\mathrm{Ind}(\mathrm{Ban}_K)$: if $(V_n)$ is an inverse system in $\h{\B}c_K$ such that
	\begin{equation*}
		0\to \varprojlim V_n\to \prod_n V_n\to \prod_n V_n\to 0
	\end{equation*}
	is strictly exact, then $(\mathrm{diss}(V_n))$ is also $\varprojlim$-acyclic in $\mathrm{Ind}(\mathrm{Ban}_K)$, since $\mathrm{diss}$ is exact and commutes with limits.
	\subsection{Coadmissible tensor products}
	
	Let $A$ be a Fr\'echet $K$-algebra. Then deriving the functor
	\begin{equation*}
		-\h{\otimes}_{A^b}- : \mathrm{Mod}_{\h{\B}c_K}((A^b)^{\mathrm{op}})\times \mathrm{Mod}_{\h{\B}c_K}(A^b)\to \h{\B}c_K
	\end{equation*}
	and taking zeroth cohomology yields a tensor product
	\begin{equation*}
		-\widetilde{\otimes}_{A^b}-: LH(\mathrm{Mod}_{\h{\B}c_K}((A^b)^{\mathrm{op}}))\times LH(\mathrm{Mod}_{\h{\B}c_K}(A^b))\to LH(\h{\B}c_K).
	\end{equation*}
	If $A$ is of countable type, we have already seen in Lemma \ref{reltensorIBFrechet} and Lemma \ref{reltensoronLH} that under the identifications $LH(\h{\B}c_K)\cong LH(\mathrm{Ind}(\mathrm{Ban}_K))$ and 
	\begin{align*}
		LH(\mathrm{Mod}_{\h{\B}c_K}(A^b))&\cong LH(\mathrm{Mod}_{\mathrm{Ind}(\mathrm{Ban}_K)}(\mathrm{diss}(A^b))\\
		& \cong \mathrm{Mod}_{LH(\mathrm{Ind}(\mathrm{Ban}_K))}(I(\mathrm{diss}(A^b))),
	\end{align*}
	this agrees with the corresponding derived tensor product over $\mathrm{diss}(A^b)$ resp. with the relative tensor product over the monoid $I(\mathrm{diss}(A^b))$ in the closed symmetric monoidal category $LH(\mathrm{Ind}(\mathrm{Ban}_K))$.
	
	In this subsection, we show that this tensor product agrees with more classical ones.
	
	\begin{lem}
		\label{BarinIB}
		Let $A$ be a monoid in $\mathrm{Ind}(\mathrm{Ban}_K)$. If $N\in \mathrm{Mod}_{\mathrm{Ind}(\mathrm{Ban}_K)}(A)$, then the Bar resolution
		\begin{equation*}
			F^{\bullet}=(\hdots\to A^{\overset{\rightarrow}{\otimes} n}\overset{\rightarrow}{\otimes}_K N\to A^{\overset{\rightarrow}{\otimes} n-1}\overset{\rightarrow}{\otimes}_K N\to \hdots \to A\overset{\rightarrow}{\otimes}_K N)
		\end{equation*}
		of $N$ from \cite[Lemma 2.9]{BamStein} is a flat resolution.
		If $M\in \mathrm{Mod}_{\mathrm{Ind}(\mathrm{Ban}_K)}(A^{\mathrm{op}})$, then
		\begin{equation*}
			M\overset{\rightarrow}{\otimes}^{\mathbb{L}}_A N\cong M\overset{\rightarrow}{\otimes}_A F^{\bullet}.
		\end{equation*}
	\end{lem}
	\begin{proof}
		Since every object in $\mathrm{Ind}(\mathrm{Ban}_K)$ is flat, this is \cite[Lemma 2.9]{BamStein}.
	\end{proof}
	
	\begin{lem}
		\label{BarhBc}
		Let $A$ be a Fr\'echet $K$-algebra, $N$ a left Fr\'echet $A$-module and $M$ a right Fr\'echet $A$-module. Assume that at least one of the following is satisfied:
		\begin{enumerate}[(i)]
			\item $K$ is spherically complete.
			\item $A$, $M$ and $N$ are of countable type.
		\end{enumerate}
		Let
		\begin{equation*}
			F^\bullet=(\hdots \to (A^b)^{\h{\otimes} n}\h{\otimes}_K N^b\to (A^b)^{\h{\otimes} n-1}\h{\otimes}_K N^b\to \hdots A^b\h{\otimes}_K N^b)
		\end{equation*}
		be the Bar resolution of $N^b$ in $\mathrm{Mod}_{\h{\B}c_K}(A^b)$.
		
		Then
		\begin{equation*}
			M\h{\otimes}^{\mathbb{L}}_{A^b} N^b\cong M\h{\otimes}_{A^b} F^\bullet,
		\end{equation*}
		i.e. the Bar resolution can be used to compute the derived tensor product.
	\end{lem}
	\begin{proof}
		Since $\widetilde{\mathrm{diss}}: LH(\h{\B}c_K)\to LH(\mathrm{Ind}(\mathrm{Ban}_K))$ is an equivalence, it suffices to show the isomorphism after applying $\widetilde{\mathrm{diss}}$. 
		
		By Lemma \ref{reltensorIBFrechet}, the left hand side is 
		\begin{equation*}
			I(\mathrm{diss}(M^b))\widetilde{\otimes}^{\mathbb{L}}_{I(\mathrm{diss}(A^b))}I(\mathrm{diss}(N^b))\cong \mathrm{diss}(M^b)\overset{\rightarrow}{\otimes}^{\mathbb{L}}_{\mathrm{diss}(A^b)} \mathrm{diss}(N^b),
		\end{equation*}
		and for the right hand side, Lemma \ref{tensorwithfree} and Proposition \ref{dissandtensor} yield the isomorphisms
		\begin{align*}
			\mathrm{diss}(M^b\h{\otimes}_{A^b} (A^b)^{\h{\otimes} n}\h{\otimes}_K N^b)&\cong \mathrm{diss}(M^b\h{\otimes}_K (A^b)^{\h{\otimes} n-1}\h{\otimes}_K N^b)\\
			&\cong \mathrm{diss}(M^b)\overset{\rightarrow}{\otimes}_K \mathrm{diss}(A^b)^{\overset{\rightarrow}{\otimes}n-1}\overset{\rightarrow}{\otimes}_K \mathrm{diss}(N^b)\\
			&\cong \mathrm{diss}(M^b)\overset{\rightarrow}{\otimes}_{\mathrm{diss}(A^b)} F_{\mathrm{Ind}(\mathrm{Ban}_K)}^{\bullet},
		\end{align*}
		where $F_{\mathrm{Ind}(\mathrm{Ban}_K)}^{\bullet}$ denotes the Bar resolution of $\mathrm{diss}(N^b)$ in $\mathrm{Mod}_{\mathrm{Ind}(\mathrm{Ban}_K)}(\mathrm{diss}(A^b))$. 
		
		We thus have
		\begin{equation*}
			\widetilde{\mathrm{diss}}(M^b\h{\otimes}_{A^b} F^{\bullet})\cong \mathrm{diss}(M^b)\overset{\rightarrow}{\otimes}_{\mathrm{diss}(A^b)}F_{\mathrm{Ind}(\mathrm{Ban}_K)}^{\bullet},
		\end{equation*}
		and we obtain the desired isomorphism thanks to Lemma \ref{BarinIB}.
	\end{proof}
	
	\begin{lem}
		\label{tensorfgBanach}
		Let $A, B$ be Noetherian Banach $K$-algebras. Let $N$ be a finitely generated left $A$-module with its canonical Banach structure. Let $M$ be a Banach $(B, A)$-bimodule which is finitely generated as a $B$-module, again equipped with its canonical Banach structure.
		
		Then
		\begin{equation*}
		\mathrm{H}^0(M^b\h{\otimes}^\mathbb{L}_{A^b}N^b)\cong I(M^b\h{\otimes}_{A^b}N^b)\cong I((M\otimes_A N)^b)\in LH(\h{\B}c_K),
		\end{equation*}
		where $M\otimes_A N$ is equipped with its canonical Banach structure as a finitely generated $B$-module.
		
		More generally,
		\begin{equation*}
			\mathrm{H}^{-j}(M^b\h{\otimes}^\mathbb{L}_{A^b}N^b)\cong I(\mathrm{Tor}_j^A(M, N)^b)
		\end{equation*}
		for all $j$, where the finitely generated abstract $B$-module $\mathrm{Tor}_j^A(M, N)$ is equipped with its canonical Banach structure.
	\end{lem}
	\begin{proof}
		By assumption, we have a strictly exact sequence of Banach spaces
		\begin{equation*}
			A^r\to A^s\to N\to 0.
		\end{equation*} 
		Since $M^b\widetilde{\otimes}_{A^b}A^b\cong I(M^b)$, we have $\mathrm{H}^0(M^b\h{\otimes}^\mathbb{L}_{A^b} (A^b)^r)\cong I((M^b)^r)$, as both $I$ and tensor products commute with direct sums. 
		
		Applying $M\otimes_A-$ to the sequence above, we obtain an exact sequence of finitely generated $B$-modules
		\begin{equation*}
			M^r\to M^s\to M\otimes_A N\to 0,
		\end{equation*}
		which is strictly exact by \cite[Corollary 3.7.3/5]{BGR}. Let $g: M^r\to M^s$ denote the first map.
		
		Applying instead $\mathrm{H}^0(M^b\h{\otimes}^\mathbb{L}_{A^b}-)$ to the sequence above yields an exact sequence
		\begin{equation*}
			\begin{xy}
				\xymatrix{
					I((M^b)^r)\ar[r]^{I(g)}& I((M^b)^s)\ar[r]& \mathrm{H}^0(M^b\h{\otimes}^\mathbb{L}_{A^b} N^b)\ar[r]&0.
				}
			\end{xy}
		\end{equation*}
		But $g$ was a strict morphism of Banach spaces, so it is also strict as a map of bornological spaces, and therefore $\mathrm{H}^0(M^b\h{\otimes}^\mathbb{L}_{A^b} N^b)$ is contained in the essential image of $I$ and is thus isomorphic to $I(M^b\h{\otimes}_{A^b }N^b)$, since $C(\mathrm{H}^0(M^b\h{\otimes}^\mathbb{L}_{A^b}N^b))\cong M^b\h{\otimes}_{A^b}N^b$. Comparing with the sequence above and noting that $(-)^b$ is exact on Banach spaces, we deduce that $(M\otimes_A N)^b$ is the cokernel of $g^b$, so $M^b\h{\otimes}_{A^b}N^b\cong (M\otimes_A N)^b$.
		
		Now let
		\begin{equation*}
			0\to P\to A^s\to N\to 0
		\end{equation*}
		be a short (strictly) exact sequence of finitely generated Banach $A$-modules. Since 
		\begin{equation*}
			\mathrm{Tor}_j^A(M, A^s)=0=\mathrm{H}^{-j}(M^b\h{\otimes}^\mathbb{L}_{A^b}(A^b)^s)
		\end{equation*}
		for all $j>0$, the result follows by working inductively along the long exact sequences of abstract Tors and bornological cohomology groups respectively.
	\end{proof}
	We now extend this result to the Fr\'echet--Stein setting. Let $U=\varprojlim U_n$ and $V=\varprojlim V_n$ be Fr\'echet--Stein algebras. By a \textbf{$V$-coadmissible $(V, U)$-bimodule} we will mean a coadmissible $V$-module $M$, equipped with its canonical Fr\'echet structure and a continuous right $U$-module structure (see \cite[Definition 7.3]{DcapOne}). Continuity implies that after relabelling, we can then always assume without loss of generality that $M_n:=V_n\otimes_V M$ is a right $U_n$-module.
	
	Let $M$ be a $V$-coadmissible $(V, U)$-bimodule. In \cite[section 7]{DcapOne}, Ardakov--Wadsley define the functor 
	\begin{align*}
		M\w{\otimes}_U-: &\C_U\to \C_V\\
		&N\mapsto \varprojlim (M_n\otimes_{U} N)=\varprojlim (M_n\otimes_{U_n}(U_n\otimes_U N)).
	\end{align*}
	We showed in \cite{Bitoun} that $M\w{\otimes}_U N$, equipped with its canonical Fr\'echet topology, is the Hausdorff completion of the topological projective tensor product $M\otimes_U N$ in $LCVS_K$. We will now study this functor from the bornological viewpoint.
	
	\begin{prop}
		\label{coadandtor}
		Let $U\cong \varprojlim U_n$, $V\cong \varprojlim V_n$ be (two-sided) Fr\'echet--Stein algebras over $K$ with the property that each coadmissible $U$- resp. $V$-module is pseudo-nuclear and of countable type (e.g. a completed enveloping algebra for a Lie--Rinehart algebra with smooth Lie lattice, or the distribution algebra of a compact $p$-adic group).
		
		Let $M$ be a $V$-coadmissible $(V, U)$-bimodule, and let $N\cong \varprojlim N_n$ be a coadmissible left $U$-module. Then
		\begin{equation*}
			\mathrm{H}^0(M^b\h{\otimes}^\mathbb{L}_{U^b}N^b)\cong I((M\w{\otimes}_UN)^b).
		\end{equation*} 
		More generally,
		\begin{equation*}
			\mathrm{H}^{-j}(M^b\h{\otimes}^\mathbb{L}_{U^b}N^b)\cong I(\varprojlim \mathrm{Tor}_j^{U_n}(M_n, N_n)^b) 
		\end{equation*}
		for each $j$. In particular, $\mathrm{H}^{-j}(M^b\h{\otimes}^\mathbb{L}_{U^b}N^b)$ is a coadmissile $V$-module.
	\end{prop}
	\begin{proof}
		By Lemma \ref{BarhBc}, we can compute $M^b_n\h{\otimes}^\mathbb{L}_{U^b_n} N^b_n$ by applying $M^b_n\h{\otimes}_{U^b_n}-$ to the Bar resolution $F_n^\bullet$ of $N^b_n$ given by
		\begin{equation*}
			F_n^\bullet:=(\hdots \to (U^b_n)^{\h{\otimes}s}\h{\otimes}_K N^b_n\to (U^b_n)^{\h{\otimes}s-1}\h{\otimes}N^b_n\to \hdots \to U^b_n\h{\otimes}N^b_n\to 0).
		\end{equation*}
		So by the previous Lemma, $M^b_n\h{\otimes}_{U^b_n}F_n^\bullet$ is a complex of Banach spaces such that
		\begin{equation*}
			\mathrm{H}^{-j}(M^b_n\h{\otimes}_{U^b_n} F_n^\bullet)=I(\mathrm{Tor}_j^{U_n}(M_n, N_n)^b).
		\end{equation*}
		In particular, it is a strict complex.
		
		Note that each term in $M^b_n\h{\otimes}_{U^b_n}F_n^\bullet$ is of the form $M^b_n\h{\otimes}_K (U^b_n)^{\h{\otimes} s}\h{\otimes}_K N^b_n$. Keeping $s$ fixed and letting $n$ vary, these form a pre-nuclear system with dense images by Lemma \ref{newprenuc}.
		
		Now the flatness of $U_{n+1}\to U_n$ and $V_{n+1}\to V_n$ yields isomorphisms
		\begin{align*}
			\mathrm{Tor}_j^{U_n}(M_n, N_n)&\cong \mathrm{Tor}_j^{U_n}(V_n\otimes_{V_{n+1}}M_{n+1}, U_n\otimes_{U_{n+1}}N_{n+1})\\
			&\cong \mathrm{Tor}_j^{U_{n+1}}(V_n\otimes_{V_{n+1}}M_{n+1}, N_{n+1})\\
			&\cong V_n\otimes_{V_{n+1}}\mathrm{Tor}_j^{U_{n+1}}(M_{n+1}, N_{n+1}).
		\end{align*}
		Thus $\varprojlim \mathrm{Tor}_j^{U_n}(M_n, N_n)$ is a coadmissible $V$-module. In particular, the terms $\mathrm{Tor}_j^{U_n}(M_n, N_n)$ form a pre-nuclear system with dense images for each $j$.
		
		We can thus apply Corollary \ref{MLcomplexes} to deduce that the complex $\varprojlim (M_n^b\h{\otimes}_{U^b_n}F_n^\bullet)$ is strict with cohomology groups isomorphic to $\varprojlim \mathrm{Tor}_j^{U_n}(M_n, N_n)^b$. By Corollary \ref{commutewinv}, we can identify $\varprojlim (M_n\h{\otimes}_K U_n^{\h{\otimes} s}\h{\otimes}_K N_n)^b$ with $M^b\h{\otimes}(U^b)^{\h{\otimes} s} \h{\otimes} N^b$, so the complex $\varprojlim(M^b_n\h{\otimes}_{U^b_n}F_n^\bullet)$ is the result of applying $M^b\h{\otimes}_{U^b}-$ to the Bar resolution of $N^b$.
		
		Therefore,
		\begin{equation*}
			\mathrm{H}^{-j}(M^b\h{\otimes}_{U^b}^\mathbb{L}N^b)\cong \mathrm{H}^{-j}(\varprojlim (M^b_n\h{\otimes}^\mathbb{L}_{U^b_n}F_n^\bullet))\cong I(\varprojlim \mathrm{Tor}_j^{U_n}(M_n, N_n)^b),
		\end{equation*}
		as required.
	\end{proof}
	In a similar fashion, we also have the following result.
	\begin{lem}
		\label{reltensorandinv}
		Let $A$ be a Banach $K$-algebra, and let $N\cong \varprojlim N_n$ be an inverse limit of left Banach $A$-modules, describing $N$ as a pseudo-nuclear module. Let $M$ be a right Banach $A$-module such that $M\h{\otimes}_A^\mathbb{L}N_n\cong M\h{\otimes}_AN_n$ for each $n$. Assume that at least one of the following is satisfied:
		\begin{enumerate}[(i)]
			\item $K$ is spherically complete.
			\item $A$, $M$ and $N_n$ are of countable type.
		\end{enumerate}
		Then 
		\begin{equation*}
			M\h{\otimes}_A^\mathbb{L}N\cong M\h{\otimes}_AN\cong \varprojlim (M\h{\otimes}_AN_n)
		\end{equation*} 
		via the natural morphism.
	\end{lem}
	\begin{proof}
		The complex $M\h{\otimes}_K A^{\h{\otimes}\bullet}\h{\otimes}_K N_n\to M\h{\otimes}_A N_n$ obtained from the Bar resolution of $N_n$ computes $M\h{\otimes}^\mathbb{L}_A N_n$, so it is strictly exact, and each term produces a pre-nuclear system of Banach spaces with dense images as $n$ varies. By Corollary \ref{MLcomplexes}, 
		\begin{equation*}
			\varprojlim (M\h{\otimes}A^{\h{\otimes} \bullet} \h{\otimes}N_n)
		\end{equation*}
		is a strictly exact complex, and the result follows from Corollary \ref{commutewinv}.
	\end{proof}
	
	We remark once again that we obtain the analogous results when working in $\mathrm{Mod}_{\mathrm{Ind}(\mathrm{Ban}_K)}(\mathrm{diss}(A^b))$. This is actually immediate by the identification of the corresponding left hearts.
	
	\subsection{The space $K\{x\}$}
	We write $K\{x\}=\varprojlim K\langle \pi^nx\rangle\in\h{\B} c_K$. Note that $K\{x\}$ is (the bornologification of) a nuclear Fr\'echet space, and the $K\langle \pi^nx\rangle$ are of countable type.
	
	Given $V\in \h{\B} c_K$, let $S(V)$ denote the vector space of sequences $(v_i)$ in $V$ satisfying the following property: there exists a bounded subset $B\subseteq V$ such that for each $n\geq 0$, $\{\pi^{-ni}v_i: i\in \mathbb{N}\}$ is a bounded subset of $V_B$.
	
	We can endow $S(V)$ with a bornology by declaring a subset $A\subseteq S(V)$ to be bounded if there exists a bounded subset $B\subseteq V$ such that for each $n\geq 0$,
	\begin{equation*}
		\cup_{(v_i)\in A} \{\pi^{-ni} v_i\}
	\end{equation*}
	is a bounded subset of $V_B$. As $V$ is complete, so is $S(V)$. This defines a functor $S: \h{\B}c_K\to \h{\B}c_K$, which by definition satisfies
	\begin{equation*}
		S(V)\cong \varinjlim S(V_B),
	\end{equation*}
	where $V\cong \varinjlim V_B$ is the usual description of $V$ as a colimit of Banach spaces.
	
	Note further that if $V$ is a Banach space, then
	\begin{equation*}
		S(V)=\varprojlim (V\h{\otimes}_K K\langle \pi^nx\rangle)\cong V\h{\otimes}_K K\{x\}
	\end{equation*}
	by Corollary \ref{commutewinv}.
	\begin{lem}
		\label{sequentialtensor}
		Let $V\in \h{\B} c_K$. Then $V\h{\otimes}_K K\{x\}\cong S(V)$ naturally in $V$. 
	\end{lem}
	\begin{proof}
		Write $V\cong \varinjlim V_B$ as the colimit of Banach spaces, so that 
		\begin{align*}
			V\h{\otimes}_K K\{x\}&\cong \varinjlim (V_B\h{\otimes}_K K\{x\})\\
			&\cong \varinjlim S(V_B)\cong S(V).\qedhere
		\end{align*}
	\end{proof}
	We thus obtain the following useful result:
	\begin{cor}
		\label{stronglyflat}
		The space $K\{x\}$ is strongly flat, i.e. the functor $-\h{\otimes}_K K\{x\}$ is strongly exact. In particular, the induced functor
		\begin{equation*}
			-\widetilde{\otimes}_K I(K\{x\}): LH(\h{\B} c_K)\to LH(\h{\B} c_K)
		\end{equation*}
		is exact.
	\end{cor}
	\begin{proof}
		We know from adjunction that $-\h{\otimes}_K K\{x\}$ preserves arbitrary cokernels, and the previous Lemma makes it clear that it also preserves arbitrary kernels. Thus $K\{x\}$ is strongly flat. 
		
		In particular, tensoring with $K\{x\}$ preserves monomorphisms, so 
		\begin{equation*}
			(E^{-1}\to E^0)\widetilde{\otimes} I(K\{x\})
		\end{equation*}
		can simply be represented by $(E^{-1}\h{\otimes} K\{x\}\to E^0\h{\otimes} K\{x\})$.
		
		By tensor-hom adjunction, the functor $-\widetilde{\otimes} I(K\{x\})$ is right exact. A mo\-no\-mor\-phism $E\to F$ in $LH(\h{\B}c_K)$ can be represented by a Cartesian diagram 
		\begin{equation*}
			\begin{xy}
				\xymatrix{
					E^{-1}\ar[r]\ar[d] & E^0\ar[d]\\
					F^{-1}\ar[r]& F^0
				}
			\end{xy}
		\end{equation*} 
		in $\h{\B}c_K$, i.e. we have a strictly exact sequence
		\begin{equation*}
			0\to E^{-1}\to E^0\oplus F^{-1}\to F^0.
		\end{equation*}  
		Since $-\h{\otimes} K\{x\}$ preserves all kernels, the diagram remains Cartesian after applying $-\h{\otimes} K\{x\}$. Hence $-\widetilde{\otimes} I(K\{x\})$ is both right exact and preserves monomorphisms, so it is exact.
	\end{proof}
	
	As $\widetilde{\mathrm{diss}}: LH(\h{\B}c_K)\to LH(\mathrm{Ind}(\mathrm{Ban}_K))$ is an equivalence of closed symmetric monoidal categories and $I\mathrm{diss}=\widetilde{\mathrm{diss}}I$, this implies that $\mathrm{diss}(K\{x\})$ is also a strongly flat object in $\mathrm{Ind}(\mathrm{Ban}_K)$.
	
	Writing $K\{x_1, \hdots, x_r\}=\varprojlim K\langle\pi^nx_1, \hdots, \pi^nx_r\rangle$, it follows from Corollary \ref{commutewinv} and Proposition \ref{dissandtensor} that
	\begin{equation*}
		K\{x_1, \hdots, x_r\}\cong K\{x_1\}\h{\otimes}_K \hdots \h{\otimes}_K K\{x_r\}
	\end{equation*}
	and
	\begin{equation*}
		\mathrm{diss}(K\{x_1, \hdots, x_r\})\cong \mathrm{diss}(K\{x_1\})\overset{\rightarrow}{\otimes}_K\hdots \overset{\rightarrow}{\otimes}_K \mathrm{diss}(K\{x_r\}).
	\end{equation*}
	In particular, $K\{x_1, \hdots, x_r\}\in \h{\B}c_K$ and $\mathrm{diss}(K\{x_1, \hdots, x_r\})\in \mathrm{Ind}(\mathrm{Ban}_K)$ are strongly flat objects for any $r$.
	\subsection{Further results on coadmissible modules}
	Throughout this subsection, let $U=\varprojlim U_n$ be a Fr\'echet--Stein $K$-algebra with the property that each coadmissible $U$-module is pseudo-nuclear. 
	
	If $K$ is not discretely valued, we make the additional assumption that $U$ is of countable type. In this case, $U_n$ is of countable type by definition, and hence any coadmissible $U$-module, being an inverse limit of finitely generated $U_n$-modules, is of countable type. 
	
	In this subsection, we discuss how some structural statements about coadmissible $U$-modules translate into the bornological framework.
	
	We first note the following immediate corollary of Proposition \ref{coadandtor}.
	
	\begin{prop}[{compare \cite[Remark 3.2]{ST}}]
		\label{flatbasechange}
		Let $U\cong \varprojlim U_n$, $V\cong\varprojlim V_n$ be Fr\'echet--Stein algebras (of countable type) with the property that coadmissible modules are pseudo-nuclear. 
		
		Let $M\cong \varprojlim M_n$ be a $V$-coadmissible $(V, U)$-bimodule equipped with its natural Fr\'echet topology, and suppose that for each $n$, $M_n$ is flat as an abstract right $U_n$-module.
		
		Then
		\begin{equation*}
			M^b\widehat{\otimes}^\mathbb{L}_{U^b} N^b\cong I(M^b\widehat{\otimes}_{U^b} N^b)
		\end{equation*}
		for any left coadmissible $U$-module $N$.
	\end{prop}
	
	\begin{proof}
		This is clear from Proposition \ref{coadandtor}.
	\end{proof}

	\begin{cor}[{compare \cite[Corollary 3.1]{ST}}]
		\label{Anacyclic}
		For any coadmissible $U$-module $M\cong \varprojlim M_n$, there is an isomorphism
		\begin{equation*}
			U_n^b\widehat{\otimes}^\mathbb{L}_{U^b}M^b\cong I(U_n^b\widehat{\otimes}_{U^b} M^b)\cong I(M_n^b)
		\end{equation*}
		for every $n$.
	\end{cor}
	\begin{proof}
		Take $V=U_n$ in Proposition \ref{flatbasechange}. This is a Noetherian Banach algebra, which can be viewed as a Fr\'echet--Stein algebra with $V_m=U_n$ for each $m$. Coadmissible $V$-modules are precisely finitely generated $U_n$-modules, which are Banach and thus pseudo-nuclear. Since $U_n$ is of countable type, so is every finitely generated $U_n$-module, and moreover, $V_m=U_n$ is flat over $U_m$ for $m\geq n$ by definition of Fr\'echet--Steinness. Thus the conditions in Proposition \ref{flatbasechange} are satisfied.
		
		Moreover, Proposition \ref{coadandtor} implies that $U_n^b\h{\otimes}_{U^b} M^b$ is isomorphic to $(U_n\otimes_U M)^b$, equipped with its canonical Banach structure as a finitely generated $U_n$-module. But this is precisely $M_n^b$ by \cite[Corollary 3.1]{ST}.
	\end{proof}

	\begin{cor}[{compare \cite[Corollary 3.3]{ST}}]
		\label{coadhBc}
		Let $M$ be a complete bornological $U$-module. Then $M\cong N^b$ for some coadmissible $U$-module $N$ if and only if $U^b_n\h{\otimes}_{U^b} M$ is a finitely generated $U_n$-module with its canonical Banach structure (viewed as a bornological module) and the natural morphism $M\to \varprojlim (U^b_n\h{\otimes}_{U^b} M)$ is an isomorphism.
	\end{cor}
	\begin{proof}
		This is immediate from the above and the fact that $(-)^b$ commutes with limits.
	\end{proof}

	We remark that our results remain valid upon passing to $\mathrm{Ind}(\mathrm{Ban}_K)$:
	
	\begin{cor}
		Let $M\in \mathrm{Mod}_{\mathrm{Ind}(\mathrm{Ban}_K)}(\mathrm{diss}(U))$. Then $M\cong \mathrm{diss}_U(N^b)$ for some coadmissible $U$-module $N$ if and only if $U_n\overset{\rightarrow}{\otimes}_{\mathrm{diss}(U)} M$ is a finitely generated $U_n$-module endowed with its canonical Banach structure (viewed as an object in $\mathrm{Mod}_{\mathrm{Ind}(\mathrm{Ban}_K)}(U_n)$) and the natural morphism $M\to \varprojlim (U_n\overset{\rightarrow}{\otimes}_{\mathrm{diss}(U)}M)$ is an isomorphism.
	\end{cor}
	\begin{proof}
		If $N=\varprojlim N_n$ is a coadmissible $U$-module, then 
		\begin{align*}
			I(U_n)\widetilde{\otimes}_{I(\mathrm{diss}(U))} I(\mathrm{diss}(N^b))&\cong
			\widetilde{\mathrm{diss}}(\mathrm{H}^0(U_n\h{\otimes}^{\mathbb{L}}_{U} N^b))\\
			&\cong \widetilde{\mathrm{diss}}(I(N_n^b))\\
			&=I(\mathrm{diss}(N_n^b))
		\end{align*}
		due to Corollary \ref{reltensorIBFrechet} and Corollary \ref{Anacyclic}. 
		
		Thus $\mathrm{H}^0(U_n\overset{\rightarrow}{\otimes}^{\mathbb{L}}_{\mathrm{diss}(U^b)}\mathrm{diss}(N^b))\cong I(U_n)\widetilde{\otimes}_{I(\mathrm{diss}(U))}I(\mathrm{diss}(N^b))\in LH(\mathrm{Ind}(\mathrm{Ban}_K))$ is contained in the essential image of $I$ and is thus isomorphic to the tensor product $I(U_n\overset{\rightarrow}{\otimes}_{\mathrm{diss}(U^b)}\mathrm{diss}(N^b))$. As $I$ is fully faithful, the above then yields $U_n\overset{\rightarrow}{\otimes}_{\mathrm{diss}(U^b)}\mathrm{diss}(N^b)\cong \mathrm{diss}(N_n^b)$.
		
		Moreover, $\mathrm{diss}(N^b)\cong \varprojlim (N_n^b)$, since $\mathrm{diss}$ commutes with limits.
		
		Conversely, if $M\in \mathrm{Mod}_{\mathrm{Ind}(\mathrm{Ban}_K)}(\mathrm{diss}(U))$ satisfies the given condition, write $M_n$ for the Banach $U_n$-module $U_n\overset{\rightarrow}{\otimes}_{\mathrm{diss}(U)} M$. Then 
		\begin{align*}
			M&\cong \varprojlim \mathrm{diss}(M_n^b)\\
			&\cong \mathrm{diss}(\varprojlim M_n^b)\\
			&\cong \mathrm{diss}((\varprojlim M_n)^b),
		\end{align*}
		and the Fr\'echet $U$-module $N:=\varprojlim M_n$ is indeed coadmissible: Since $N$ is a Fr\'echet space of countable type, Lemma \ref{BarhBc} implies that  
		\begin{align*}
			\mathrm{diss}(U_n\h{\otimes}_{U^b} N^b)&\cong
			U_n\overset{\rightarrow}{\otimes}_{\mathrm{diss}(U^b)}\mathrm{diss}(N^b)\\
			&\cong U_n\overset{\rightarrow}{\otimes}_{\mathrm{diss}(U^b)} M\\
			&\cong \mathrm{diss}(M_n^b)
		\end{align*}
		so $U_n\h{\otimes}_U N^b\cong M_n^b$. Hence $N$ is coadmissible by Corollary \ref{coadhBc}.   
	\end{proof}
	
	An analogous characterization for objects in the left hearts is now immediate: if $M\in \mathrm{Mod}_{LH(\h{\B}c_K)}(U^b)$ is of the form $\varprojlim I(M_n^b)$ for $M_n$ Banach, then $M$ is in the essential image of $I$, as $I$ commutes with limits, and we are reduced to the earlier discussion.

	\section{Bornological $\wideparen{\mathcal{D}}$-modules}
	\subsection{The setting}
	Let $X$ be a rigid analytic $K$-variety, viewed as a G-topological space with its strong Grothendieck topology. Before we finally turn to a discussion of $\w{\D}$-modules, we summarize the categorical setting which we have developed so far.
	
	Let $\C=\mathrm{Ind}(\mathrm{Ban}_K)$. This is an elementary quasi-abelian category with a closed symmetric monoidal structure such that each object is flat and there are enough flat projectives stable under $\overset{\rightarrow}{\otimes}_K$ (see Corollary \ref{IndBansheafy}).
	
	In particular, we can invoke the sheaf theory from Theorem \ref{Shvnice}: the category $\mathrm{Shv}(X, \C)$ is quasi-abelian and closed symmetric monoidal with enough flat objects, and $LH(\mathrm{Shv}(X, \C))\cong \mathrm{Shv}(X, LH(\C))$ is a Grothendieck abelian category with a closed symmetric monoidal structure. 
	
	If $\R\in \mathrm{Shv}(X, \C)$ is a monoid, analogous results hold for the category of $\R$-modules. We let $\mathrm{D}(\R):=\mathrm{D}(\mathrm{Mod}_{\mathrm{Shv}(X, \C)}(\R))\cong \mathrm{D}(\mathrm{Mod}_{\mathrm{Shv}(X, LH(\C))}(I(\R)))$ denote the corresponding unbounded derived category.
	
	Thanks to the results in subsections 3.6 and 3.7, we have at our disposal the derived functors
	\begin{equation*}
		-\widetilde{\otimes}^{\mathbb{L}}_{\R}-: \mathrm{D}(\R^{\mathrm{op}})\times \mathrm{D}(\R)\to \mathrm{D}(\mathrm{Shv}(X, \C))
	\end{equation*}
	\begin{equation*}
		\mathrm{R}\mathcal{H}om_{\R}(-, -): \mathrm{D}(\R)^{\mathrm{op}}\times \mathrm{D}(\R)\to \mathrm{D}(\mathrm{Shv}(X, \C))
	\end{equation*}
	as well as
	\begin{equation*}
		f^{-1}: \mathrm{D}(\R)\to \mathrm{D}(f^{-1}\R)
	\end{equation*}
	for $f: Y\to X$ a morphism of rigid spaces and
	\begin{equation*}
		\mathrm{R}f_*: \mathrm{D}(\R)\to \mathrm{D}(f_*\R)
	\end{equation*}
	for any morphism $f: X\to Y$. These satisfy the usual adjunction and composition statements (see Proposition \ref{deriverelthom}, Lemma \ref{catcomposition} and the discussion after Lemma \ref{Rmodimageadj}). In the case of bimodules, the functors can be upgraded to functors between suitable module categories.
	
	We use throughout the sheaf tensor product $\widetilde{\otimes}_K$ on $\mathrm{Shv}(X, LH(\C))$, and note that if $\M, \N\in \mathrm{Shv}(X, \C)$, then
	\begin{equation*}
		I(\M)\widetilde{\otimes}_K I(\N)\cong I(\M\overset{\rightarrow}{\otimes}_K \N)
	\end{equation*}
	by Lemma \ref{strictsheaftensor}.
	
	If $\R\in \mathrm{Shv}(X, \C)$ is a monoid, we have
	\begin{equation*}
		-\widetilde{\otimes}^{\mathbb{L}}_{\R}-\cong -\overset{\rightarrow}{\otimes}^{\mathbb{L}}_{\R}-,
	\end{equation*}
	by Proposition \ref{deriverelqabtensor}, but
	\begin{equation*}
		\M\widetilde{\otimes}_{\R}\N=\mathrm{H}^0(\M\widetilde{\otimes}^{\mathbb{L}}_{\R}\N)
	\end{equation*}
	is only containd in $\mathrm{Shv}(X, \C)$ if additional conditions are met, see e.g. Proposition \ref{coadandtor}
	
	In addition to $\mathrm{Ind}(\mathrm{Ban}_K)$, we have encountered the quasi-abelian category $\h{\B}c_K$ of complete bornological vector spaces, which we can regard as a full subcategory of $\mathrm{Ind}(\mathrm{Ban}_K)$ via the exact and fully faithful functor $\mathrm{diss}: \h{\B}c_K\to \mathrm{Ind}(\mathrm{Ban}_K)$, with left adjoint $L$.
	
	Note that $\mathrm{diss}$ induces an exact and fully faithful functor 
	\begin{equation*}
		\mathrm{diss}^{\mathrm{pre}}: \mathrm{Preshv}(X, \h{\B}c_K)\to \mathrm{Preshv}(X, \mathrm{Ind}(\mathrm{Ban}_K)).
	\end{equation*}

	Since $\mathrm{diss}$ is a left adjoint, it is strongly left exact and commutes with products. Hence, $\mathrm{diss}^{\mathrm{pre}}$ sends $\h{\B}c_K$-sheaves to $\mathrm{Ind}(\mathrm{Ban}_K)$-sheaves, and conversely, if $\F\in \mathrm{Preshv}(X, \h{\B}c_K)$ such that $\mathrm{diss}^{\mathrm{pre}}(\F)$ is an $\mathrm{Ind}(\mathrm{Ban}_K)$-sheaf, then $\F$ is a $\h{\B}c_K$-sheaf. In particular, there is no ambiguity in the terminology `complete bornological sheaf', as we can identify $\mathrm{Shv}(X, \h{\B}c_K)$ with the full subcategory of $\mathrm{Shv}(X, \mathrm{Ind}(\mathrm{Ban}_K))$ consisting of sheaves $\F$ such that for each admissible open $U$, $\F(U)$ is contained in $\h{\B}c_K$ (i.e. is an essentially monomorphic system).
	
	We have also noted that the left heart $LH(\h{\B}c_K)$ is an elementary abelian, closed symmetric monoidal category which is equivalent to $LH(\mathrm{Ind}(\mathrm{Ban}_K))$. In particular, if $\R$ is a monoid in $\mathrm{Shv}(X, \mathrm{Ind}(\mathrm{Ban}_K))$, then
	\begin{equation*}
		\mathrm{D}(\R)\cong \mathrm{D}(\mathrm{Mod}_{\mathrm{Shv}(X, LH(\h{\B}c_K))}(\widetilde{L}(\R))).
	\end{equation*} 
	
	A monoid $\R\in \mathrm{Shv}(X, \mathrm{Ind}(\mathrm{Ban}_K))$ which is a complete bornological sheaf will be called a \textbf{sheaf of complete bornological $K$-algebras}. This terminology might appear circumspect, since we have not defined a symmetric monoidal structure on $\mathrm{Shv}(X, \h{\B}c_K)$, but rather use the one on $\mathrm{Shv}(X, \mathrm{Ind}(\mathrm{Ban}_K))$ instead -- it is once again justified by the fact that $\mathrm{Shv}(X, LH(\h{\B}c_K))$ and $\mathrm{Shv}(X, LH(\mathrm{Ind}(\mathrm{Ban}_K)))$ are equivalent even as closed symmetric monoidal categories, i.e. we could equivalently define sheaves of complete bornological $K$-algebras as those monoids in $\mathrm{Shv}(X, LH(\h{\B}c_K))$ which only take values in $\h{\B}c_K$.
	
	The terminology is further justified by the following Proposition.
	
	\begin{prop}
		\label{hBcextension}
		Let $X$ be a rigid analytic $K$-variety and let $\B$ be a basis which is locally a site (e.g. the collection of all affinoid subdomains of $X$). 
		
		If $\R\in \mathrm{Shv}(\B, \h{\B}c_K)$ is a monoid in $\mathrm{Preshv}(\B, \h{\B}c_K)$ (i.e. with respect to the presheaf tensor product $\h{\otimes}_K$), then $\R$ extends uniquely to a sheaf of complete bornological $K$-algebras $\R^{\mathrm{ext}}\in \mathrm{Shv}(X, \mathrm{Ind}(\mathrm{Ban}_K))$.
		
		There is an exact functor
		\begin{equation*}
			(-)^{a, \mathrm{ext}}: \mathrm{Mod}_{\mathrm{Preshv}(\B, \h{\B}c_K)}(\R)\to \mathrm{Mod}_{\mathrm{Shv}(X, \mathrm{Ind}(\mathrm{Ban}_K))}(\R^{\mathrm{ext}})
		\end{equation*}
		given by sending $\M$ to $((\mathrm{diss}^{\mathrm{pre}}(\M))^a)^{\mathrm{ext}}$.
		
		In particular, if $\F\in \mathrm{Shv}(X, \mathrm{Ind}(\mathrm{Ban}_K))$ is a complete bornological sheaf, then in order to give $\F$ an $\R^{\mathrm{ext}}$-module structure, it suffices to specify a presheaf $\R$-module structure on $\F|_\B\in \mathrm{Shv}(\B, \h{\B}c_K)$. 
	\end{prop}
	\begin{proof}
		The functor $\mathrm{diss}^{\mathrm{pre}}$ is lax symmetric monoidal (Proposition \ref{dissandL}.(iv)), the sheafification functor $(-)^a: \mathrm{Preshv}(\B, \mathrm{Ind}(\mathrm{Ban}_K))\to \mathrm{Shv}(\B, \mathrm{Ind}(\mathrm{Ban}_K))$ is strong symmetric monoidal by definition, and the functor $(-)^{\mathrm{ext}}$ from Lemma \ref{extfrombasis} is strong symmetric monoidal by Lemma \ref{extstrongmonoidal}. Each of the three functors is exact by Proposition \ref{dissandL}.(i), Lemma \ref{Presheaflimits}.(iii), and Lemma \ref{extfrombasis}.
		
		Hence everything follows directly from Proposition \ref{laxmodulechange} and Lemma \ref{basisandsheafification}.
	\end{proof}
	
	In practice, we will usually suppress notational decorations like $(-)^{\mathrm{ext}}$ and simply note that monoid and module structures in $\mathrm{Preshv}(\B, \h{\B}c_K)$ can be used to define monoid and module structures for $\mathrm{Ind}(\mathrm{Ban}_K)$-sheaves on $X$ if they take values in $\h{\B}c_K$.
	
	We also remark that if $\F\in \mathrm{Shv}(X, \mathrm{Ind}(\mathrm{Ban}_K))$ such that $\F|_{\B}\in \mathrm{Shv}(\B, \h{\B}c_K)$, then $\F\in \mathrm{Shv}(X, \h{\B}c_K)$ by Proposition \ref{dissandL}.(iii). 
	
	Moreover, we can regard $\F$ in this case also as a sheaf of vector spaces on $X$, since the forgetful functor $\h{\B}c_K\to \mathrm{Vect}_K$ commutes with kernels and products. Then by the uniqueness of extension, $\mathrm{forget}(\F)$ is the unique extension of $\mathrm{forget}(\F|_{\B})$ to a $\mathrm{Vect}_K$-sheaf on $X$. 
	
	\subsection{The sheaf $\w{\D}$ and coadmissible modules}
	Let $X$ be a rigid analytic $K$-variety. In this subsection, we describe the structure sheaf on $X$ as a sheaf of complete bornological algebras, before we do the same for the sheaves $\w{\mathscr{U}(\mathscr{L})}$ introduced in \cite{DcapOne} for any Lie algebroid on $X$. If $X$ is smooth and $\mathscr{L}$ is the tangent sheaf of $X$, this produces the sheaf $\w{\D}_X$.
	
	Let $X_w$ denote the collection of all affinoid subdomains of $X$, which is a basis for $X$ and locally a site.

	Note that for any $U$ in $X_w$, the sections $\mathcal{O}_X(U)$ of the structure sheaf carry a natural Banach algebra structure such that the restriction maps are continuous. Moreover, for any finite affinoid covering $\mathfrak{U}=(U_i)$ of $U$, the corresponding Cech complex $\check{\mathrm{C}}(\mathfrak{U}, \mathcal{O}_X)$ is a strictly exact complex of Banach spaces (strictness follows from exactness in combination with the Open Mapping Theorem). Applying Lemma \ref{borniexact}, we see that the assignment
	\begin{equation*}
		U\mapsto \mathcal{O}_X(U)^b
	\end{equation*}
	for each $U$ of $X_w$ defines an object of $\mathrm{Shv}(X_w, \h{\B} c_K)$ with vanishing higher Cech cohomology on each affinoid, endowed with a presheaf-monoid structure. Applying Proposition \ref{hBcextension} yields a sheaf of complete bornological $K$-algebras $\O_X\in \mathrm{Shv}(X, \mathrm{Ind}(\mathrm{Ban}_K))$.
	
	Let $A$ be an affinoid $K$-algebra. Recall that we have seen in subsection 4.3 that $\mathrm{diss}\circ (-)^b$ yields an exact, fully faithful functor
	\begin{equation*}
		\{\text{finitely generated $A$-modules}\}\to \mathrm{Mod}_{\h{\B}c_K}(A^b)\to \mathrm{Mod}_{\mathrm{Ind}(\mathrm{Ban}_K)}(\mathrm{diss}(A^b))
	\end{equation*}
	by endowing a finitely generated $A$-module with its canonical Banach structure.
	
	Repeating the same procedure as above for coherent $\mathcal{O}_X$-modules then yields an exact, fully faithful monoidal functor $\mathrm{diss}\circ (-)^b$ from the category of coherent $\mathcal{O}_X$-modules to $\mathrm{Mod}_{\mathrm{Shv}(X, \mathrm{Ind}(\mathrm{Ban}_K))}(\O_X)$. In the following, we will often cease to make a distinction between an abstract coherent module $\mathcal{M}$ and $\mathrm{diss}(\mathcal{M}^b)$, in the understanding that we work with complete bornological sheaves from now on.
	
	In order to introduce $\w{\D}_X$ in a similar vein, we recall from \cite{DcapOne} (for $K$ discretely valued) and \cite{Ardakov} (in the general case) the sheaf analogues of the completed enveloping algebras $\w{U_A(L)}$ mentioned in section 2. Let $\T_X$ denote the tangent sheaf of $X$.
	\begin{defn}
		A \textbf{Lie algebroid} on $X$ is a pair $(\mathscr{L}, \rho)$, where $\mathscr{L}$ is a locally free coherent $\O_X$-module which is also a sheaf of $K$-Lie algebras, and $\rho: \mathscr{L}\to \T_X$ is an $\O_X$-linear morphism of sheaves of $K$-Lie algebras satisfying
		\begin{equation*}
			[x, ay]=a[x, y]+\rho(x)(a)y
		\end{equation*}
		for all $x, y\in \mathscr{L}(U)$, $a\in\O_X(U)$ for $U\subseteq X$ any admissible open subset of $X$. 
	\end{defn}
	Note that whenever $U=\mathrm{Sp} A$ is an affinoid subdomain and $\mathscr{L}$ is a Lie algebroid on $X$, then $\mathscr{L}(U)$ is a smooth $(K, A)$-Lie--Rinehart algebra. It is also immediate that $\T_X$ is a Lie algebroid as soon as $X$ is smooth, with the anchor map $\rho$ being the identity map in that case.
	
	We now assume for simplicity that $X=\Sp A$ is affinoid. Let $\mathscr{L}$ be a Lie algebroid on $X$ and write $L=\mathscr{L}(X)$. 
	
	Let $\mathcal{A}$ be an admissible affine formal model for $A$ and let $\L\subseteq L$ be an $(R, \A)$-Lie lattice. Arguing locally if necessary, we can suppose that $L$ admits a free Lie lattice, so that we will assume that $\L$ is a finitely generated free $\A$-module. 
	
	While we will not recall the full definition here, we note that \cite[section 4.2]{Ardakov} yields a Grothendieck topology $X_n=X(\pi^n\mathcal{L})$ of $\pi^n\mathcal{L}$-\textbf{accessible subdomains} for $n\geq 0$ with the following properties:
	\begin{enumerate}[(i)]
		\item If $Y$ is in $X_n$, then it is in $X_m$ for $m\geq n$. Any covering in $X_n$ also defines a covering in $X_m$ for $m\geq n$.
		\item Let $Y$ be an affinoid subdomain of $X$. Then $Y$ is in $X_n$ for sufficiently large $n$, and any finite affinoid covering of $Y$ is a covering in $X_n$ for sufficiently large $n$.
		\item If $Y=\Sp B$ is an affinoid subdomain in $X_n$, then $B$ contains an affine formal model $\mathcal{B}$, containing the image $\mathcal{A}$ under restriction, such that $\B\otimes_\A \pi^n\L$ is an $(R, \B)$-Lie lattice in $B\otimes_A L=\T(Y)$ and $\mathcal{U}_n(\mathscr{L})(Y):=\h{U_{\mathcal{B}}(\mathcal{B}\otimes_\mathcal{A} \pi^n\mathcal{L})}_K$ defines a Noetherian Banach $K$-algebra, which is independent of the choice of $\mathcal{B}$.
		\item The assignment
		\begin{equation*}
			Y\mapsto \mathcal{U}_n(\mathscr{L})(Y)
		\end{equation*}
		defines a sheaf of $K$-algebras on $X_n$ which has vanishing higher Cech cohomology with respect to any $X_n$-covering (see \cite[Theorem 4.2.3]{Ardakov}).
	\end{enumerate}
	By the same arguments as above, this defines a monoid $\mathcal{U}_n(\mathscr{L})$ in $\mathrm{Shv}(X_n, \h{\B} c_K)$ which has vanishing higher Cech cohomology, and an exact, fully faithful functor 
	\begin{equation*}
		\mathrm{diss}\circ(-)^b: \{\text{coherent $\mathcal{U}_n(\mathscr{L})$-modules}\}\to \mathrm{Mod}_{\mathrm{Shv}(X_n, \mathrm{Ind}(\mathrm{Ban}_K))}(\mathcal{U}_n(\mathscr{L})).
	\end{equation*} 
	If $Y$ is any affinoid subdomain of $X$, we can set
	\begin{equation*}
		\wideparen{\mathscr{U}(\mathscr{L})}(Y)=\varprojlim \mathcal{U}_n(\mathscr{L})(Y)\cong \w{U_{\O_X(Y)}(\mathscr{L}(Y))},
	\end{equation*}
	which defines a complete bornological $K$-algebra for each $Y$. By \cite[Lemma 6.2, Proposition 6.2]{DcapOne}, this does not depend on the choices made. 
	
	If $\mathfrak{U}$ is a finite affinoid covering of $X$, then $\check{\mathrm{C}}^\bullet(\mathfrak{U}, 
	\wideparen{\mathscr{U}(\mathscr{L})})=\varprojlim \check{\mathrm{C}}^\bullet(\mathfrak{U}, \mathcal{U}_n(\mathscr{L}))$, so the Cech complex of $\wideparen{\mathscr{U}(\mathscr{L})}$ is strictly exact by Theorem \ref{MLBan}, and $\wideparen{\mathscr{U}(\mathscr{L})}$ is a monoid in $\mathrm{Shv}(X, \mathrm{Ind}(\mathrm{Ban}_K))$ by Lemma \ref{FSborno} and Proposition \ref{hBcextension}. 
	
	Now let $X$ be an arbitrary rigid analytic $K$-variety and let $\mathscr{L}$ be a Lie algebroid. Let $X(\mathscr{L})$ denote the collection of affinoid subdomains $U$ such that $\mathscr{L}(U)$ is a free $\O_X(U)$-module. In particular, if $\A\subseteq \O_X(U)$ is any admissible affine formal model, we can find a free $(R, \A)$-Lie lattice. Since $\mathscr{L}$ is locally free, $X(\mathscr{L})$ is a basis of $X$ which is locally a site -- in fact, if $U\in X(\mathscr{L})$, then so is any affinoid subdomain of $U$.
	
	We can thus apply Proposition \ref{hBcextension} to obtain a sheaf $\w{\mathscr{U}(\mathscr{L})}$ of complete bornological $K$-algebras on $X$.
	
	We also note that, as previously remarked, the constructions above recover the sheaf $\wideparen{\mathscr{U}(\mathscr{L})}$ from \cite{DcapOne} and \cite{Ardakov} in the sense that for any admissible (not necessarily affinoid) $Y$, the underlying vector space of $\wideparen{\mathscr{U}(\mathscr{L})}(Y)$ is the corresponding space of sections in \cite{DcapOne}.
	
	We now discuss modules over $\w{\mathscr{U}(\mathscr{L})}$.
	\begin{defn}[{see \cite[section 8]{DcapOne}}]
		Let $X$ be a rigid analytic $K$-variety. Let $\mathscr{L}$ be a Lie algebroid on $X$, and write $\mathscr{U}=\w{\mathscr{U}(\mathscr{L})}$. An (abstract) $\mathscr{U}$-module $\M$ is called \textbf{coadmissible} if there exists an $X(\mathscr{L})$-covering $(U_i)$ of $X$ by affinoid subdomains such that
		\begin{enumerate}[(i)]
			\item $\M(U_i)$ is a coadmissible module over the Fr\'echet--Stein algebra $\mathscr{U}(U_i)$.
			\item for any affinoid subdomain $Y\subseteq U_i$, the natural morphism
			\begin{equation*}
				\mathscr{U}(Y)\w{\otimes}_{\mathscr{U}(U_i)}\M(U_i)\to \M(Y)
			\end{equation*}
			is an isomorphism.
		\end{enumerate}
	\end{defn}
	Analogously to Kiehl's Theorem for coherent $\O$-modules, \cite[Theorem 8.4]{DcapOne} shows that if a module $\M$ is coadmissible with respect to one  $X(\mathscr{L})$-covering, then it is coadmissible with respect to any such covering. While \cite{DcapOne} assumes $K$ to be discretely valued, \cite[Theorem 4.2.11]{Ardakov} ensures that the same argument works for general $K$.
	
	Suppose now once more that $X$ is affinoid, and $\mathscr{L}$ is a Lie algebroid such that $\mathscr{L}(X)$ is a free $\O(X)$-module. If $\M$ is a coadmissible $\wideparen{\mathcal{U}(\mathscr{L})}$-module, then 
	\begin{equation*}
		\M_n:=\mathcal{U}_n(\mathscr{L})\widetilde{\otimes}_{\wideparen{\mathscr{U}(\mathscr{L})}|_{X_n}}\M|_{X_n}
	\end{equation*}
	is a coherent $\mathcal{U}_n(\mathscr{L})$-module with vanishing higher Cech cohomology by Corollary \ref{Anacyclic} (we will often omit the restriction symbols in a slight abuse of notation). Applying Theorem \ref{MLBan} and Proposition \ref{hBcextension} again, this shows that equipping $\M(Y)$ with its canonical Fr\'echet structure for each affinoid $Y$ gives $\M$ the structure of a complete bornological $\wideparen{\mathscr{U}(\mathscr{L})}$-module with vanishing higher Cech cohomology.
	
	Finally invoking Corollary \ref{coadIBembedding} and glueing, we have thus proved the following:
	\begin{thm}
		\label{Dcapembedding}
		Let $X$ be a rigid analytic $K$-variety and let $\mathscr{L}$ be a Lie algebroid on $X$. Then $\mathcal{O}_X$ and $\wideparen{\mathscr{U}(\mathscr{L})}$ are sheaves of complete bornological $K$-algebras, and $(-)^b$ induces exact and fully faithful functors
		\begin{align*}
			\{\text{coherent $\mathcal{O}_X$-modules}\}\to \mathrm{Mod}_{\mathrm{Shv}(X, LH(\h{\B}c_K))}(\O_X)\\
			\{\text{coadmissible $\wideparen{\mathscr{U}(\mathscr{L})}$-modules}\}\to \mathrm{Mod}_{\mathrm{Shv}(X, LH(\h{\B}c_K))}(\w{\mathscr{U}(\mathscr{L})}),
		\end{align*} 
		taking values in complete bornological sheaves.
	\end{thm}
	
	\begin{rmks}
		\leavevmode
		\begin{enumerate}[(i)]
			\item Let $X$ be affinoid such that $\mathscr{L}(X)$ is free. As $(-)^b$ and $\mathrm{diss}$ commute with limits, we have $\wideparen{\mathscr{U}(\mathscr{L})}(X)\cong \varprojlim \mathcal{U}_n(\mathscr{L})(X)$ in $\mathrm{Ind}(\mathrm{Ban}_K)$, as well as $\M(X)\cong \varprojlim \M_n(X)$ for any coadmissible $\wideparen{\mathscr{U}(\mathscr{L})}$-module with 
			\begin{equation*}
				\M_n:=\mathcal{U}_n(\mathscr{L})\widetilde{\otimes}_{\wideparen{\mathscr{U}(\mathscr{L})}}\M.
			\end{equation*} 
			\item Let $\M$ be a coadmissible $\wideparen{\mathscr{U}(\mathscr{L})}$-module on some affinoid $X$, and let $Y$ be an affinoid subdomain. Assume that $\mathscr{L}(X)$ is free. By Proposition \ref{coadandtor} and its analogue in $\mathrm{Ind}(\mathrm{Ban}_K)$, we have an isomorphism
			\begin{equation*}
				\M(Y)\cong \wideparen{\mathscr{U}(\mathscr{L})}(Y)\widetilde{\otimes}_{\wideparen{\mathscr{U}(\mathscr{L})}(X)}\M(X).
			\end{equation*}
			In other words, the description of coadmissible modules as localisations of coadmissible $\wideparen{\mathscr{U}(\mathscr{L})}(X)$-modules in \cite{DcapOne} could also be carried out in the complete bornological resp. $\mathrm{Ind}(\mathrm{Ban}_K)$-setting. We content ourselves with taking the results in \cite{DcapOne} as given and applying $\mathrm{diss}\circ(-)^b$, but the inclined reader may check that \cite{DcapOne} could actually be rephrased entirely in bornological language. 
		\end{enumerate}
	\end{rmks}
	
	\begin{lem}
		\label{flatoverO}
		The sheaf $\w{\mathscr{U}(\mathscr{L})}$ is a strongly flat object in $\mathrm{Mod}_{\mathrm{Shv}(X, \mathrm{Ind}(\mathrm{Ban}_K))}(\O_X)$.
	\end{lem}
	\begin{proof}
		As a left $\O_X$-module, $\w{\mathscr{U}(\mathscr{L})}$ is locally of the form 
		\begin{align*}
			\varprojlim(\O_X\widetilde{\otimes}_K K\langle \pi^nx_1, \hdots, \pi^nx_r\rangle)&\cong \O_X\widetilde{\otimes}_K (\varprojlim_n K\langle \pi^nx_1, \hdots \pi^nx_r\rangle)\\
			&\cong \O_X\widetilde{\otimes}_K K\{x_1\}\widetilde{\otimes}_K \hdots \widetilde{\otimes}_K K\{x_r\},
		\end{align*}
		using Corollary \ref{commutewinvIB}. We can thus apply Corollary \ref{stronglyflat}.
	\end{proof}
	
	One difficulty of the theory lies in the fact that $\mathcal{U}_n(\mathscr{L})$ is only defined as a monoid on $X_n$. We remark however that if $X=\Sp A$ is an affinoid and $\mathscr{L}$ is free as an $\O_X$-module, then the right $\w{\mathscr{U}(\mathscr{L})}$-module $\F_n:=\mathcal{U}_n(\mathscr{L})(X)\widetilde{\otimes}_A \O_X$ satisfies $\F_n|_{X_n}\cong \mathcal{U}_n(\mathscr{L})$ (by \cite[Proposition 2.3]{DcapOne}). In particular, if $\M$ is a coadmissible $\w{\mathscr{U}(\mathscr{L})}$-module, then $\F_n\widetilde{\otimes}_{\w{\mathscr{U}(\mathscr{L})}}\M$ is an object of $\mathrm{Shv}(X, LH(\h{\B}c_K))$ such that
	\begin{equation*}
		(\F_n\widetilde{\otimes}_{\w{\mathscr{U}(\mathscr{L})}}\M)|_{X_n}\cong \mathcal{U}_n(\mathscr{L})\widetilde{\otimes}_{\w{\mathscr{U}(\mathscr{L})}|_{X_n}}\M|_{X_n}.
	\end{equation*} 
	Furthermore, if we are given for each $n$ a $\mathcal{U}_n(\mathscr{L})$-module $\mathcal{S}_n$ together with compatible $\mathcal{U}_{n+1}(\mathscr{L})|_{X_n}$-module morphisms $\mathcal{S}_{n+1}|_{X_n}\to \mathcal{S}_n$, we can still construct a $\w{\mathscr{U}(\mathscr{L})}$-module $\varprojlim \mathcal{S}_n$ by associating to an affinoid subdomain $Y$ the module $\varprojlim (\mathcal{S}_n(Y))$, which makes sense by the properties (i) and (ii) of $X_n$.
	
	For example, if $\M$ is a coadmissible $\w{\mathscr{U}(\mathscr{L})}$-module and we set 
	\begin{equation*}
		\M_n:=\mathcal{U}_n(\mathscr{L})\widetilde{\otimes}_{\w{\mathscr{U}(\mathscr{L})}}\M,
	\end{equation*}
	then $\varprojlim \M_n$ is a $\w{\mathscr{U}(\mathscr{L})}$-module which is naturally isomorphic to $\varprojlim (\F_n\widetilde{\otimes}_{\w{\mathscr{U}(\mathscr{L})}}\M)$ as an object of $\mathrm{Shv}(X, LH(\h{\B}c_K))$.
	\begin{lem}
		Let $X$ be an affinoid $K$-variety, let $\mathscr{L}$ be a free Lie algebroid on $X$, and let $\M$ be a complete bornological $\w{\mathscr{U}(\mathscr{L})}$-module. Then the following are equivalent:
		\begin{enumerate}[(i)]
			\item $\M$ is coadmissible.
			\item $\M_n:=\mathcal{U}_n(\mathscr{L})\widetilde{\otimes}_{\w{\mathscr{U}(\mathscr{L})}}\M$ is a coherent $\mathcal{U}_n(\mathscr{L})$-module for each $n$, and 
			\begin{equation*}
				\M\cong \varprojlim \M_n
			\end{equation*}
			via the natural morphism.
		\end{enumerate}
	\end{lem} 
	\begin{proof}
		Suppose that $\M$ is coadmissible. Then by Corollary \ref{Anacyclic}, 
		\begin{equation*}
			M_n:=\mathcal{U}_n(\mathscr{L})(X)\widetilde{\otimes}_{\w{\mathscr{U}(\mathscr{L})}(X)}\M(X)
		\end{equation*}
		is a finitely generated $\mathcal{U}_n(\mathscr{L})(X)$-module equipped with its canonical Banach structure.
		
		As the same holds for any affinoid subdomain $Y$ of $X$, we have
		\begin{align*}
			\mathcal{U}_n(\mathscr{L})(Y)\widetilde{\otimes}_{\w{\mathscr{U}(\mathscr{L})}(Y)}\M(Y)&\cong \mathcal{U}_n(\mathscr{L})(Y)\widetilde{\otimes}_{\w{\mathscr{U}(\mathscr{L})}(Y)}\left(\w{\mathscr{U}(\mathscr{L})}(Y)\widetilde{\otimes}_{\w{\mathscr{U}(\mathscr{L})}(X)}\M(X)\right)\\
			&\cong \mathcal{U}_n(\mathscr{L})(Y)\widetilde{\otimes}_{\w{\mathscr{U}(\mathscr{L})}(X)}\M(X)\\
			&\cong \mathcal{U}_n(\mathscr{L})(Y)\widetilde{\otimes}_{\mathcal{U}_n(\mathscr{L})(X)}(\mathcal{U}_n(\mathscr{L})(X)\widetilde{\otimes}_{\w{\mathscr{U}(\mathscr{L})}(X)} \M(X))\\
			&\cong\mathcal{U}_n(\mathscr{L})(Y)\otimes_{\mathcal{U}_n(\mathscr{L})(X)}M_n,
		\end{align*}
		whenever $U$ is $\pi^n\L$-accessible by Lemma \ref{tensorfgBanach}, and thus $\M_n$ is coherent, being the localization of $M_n$. In particular, $\M_n(X)=M_n$. 
		
		Since $\M(X)$ is a coadmissible $\w{\mathscr{U}(\mathscr{L})}(X)$-module, we have 
		\begin{equation*}
			\M(X)\cong \varprojlim M_n=\varprojlim \left((\mathcal{U}_n(\mathscr{L})\widetilde{\otimes}_{\w{\mathscr{U}(\mathscr{L})}}\M)(X)\right)
		\end{equation*} 
		via the natural morphism. As this is also true for an arbitrary affinoid subdomain $Y$, we have $\M\cong \varprojlim (\mathcal{U}_n(\mathscr{L})\widetilde{\otimes}_{\w{\mathscr{U}(\mathscr{L})}}\M)$.
		
		Conversely, suppose that (ii) holds, and let $M_n:=(\mathcal{U}_n(\mathscr{L})\widetilde{\otimes}_{\w{\mathscr{U}(\mathscr{L})}}\M)(X)$. Note that the $M_n$ satisfy $\M(X)\cong \varprojlim M_n$. By definition, $M_n$ is a finitely generated $\mathcal{U}_n(\mathscr{L})(X)$-module, equipped with its canonical Banach structure, and 
		\begin{equation*}
			M_n\cong\mathcal{U}_n(\mathscr{L})(X)\otimes_{\mathcal{U}_{n+1}(\mathscr{L})(X)}M_{n+1}.
		\end{equation*}
		Hence, $\M(X)$ is a coadmissible $\w{\mathscr{U}(\mathscr{L})}(X)$-module.
		
		If $Y$ is an affinoid subdomain of $X$, the isomorphism 
		\begin{align*}
			\M(Y)&\cong \varprojlim (\mathcal{U}_n(\mathscr{L})(Y)\otimes_{\mathcal{U}_n(\mathscr{L})(X)}M_n)\\
			& \cong\w{\mathscr{U}(\mathscr{L})}(Y)\w{\otimes}_{\w{\mathscr{U}(\mathscr{L})}(X)} \M(X)
		\end{align*} 
		allows us to conclude that $\M$ is a coadmissible $\w{\mathscr{U}(\mathscr{L})}$-module.
	\end{proof}
	The above lemma will be our point of departure for the discussion of $\C$-complexes in section 8.
	\begin{lem}
		\label{Mndefinescoad}
		Let $X$ be affinoid and let $\mathscr{L}$ be a free Lie algebroid on $X$ (with notation as before). For each $n$, let $\M_n$ be a coherent $\mathcal{U}_n(\mathscr{L})$-module with $\mathcal{U}_{n+1}(\mathscr{L})|_{X_n}$-module morphisms $\M_{n+1}|_{X_n}\to \M_n$ inducing isomorphisms
		\begin{equation*}
			\mathcal{U}_n(\mathscr{L})\widetilde{\otimes}_{\mathcal{U}_{n+1}(\mathscr{L})|_{X_n}}\M_{n+1}|_{X_n}\cong \M_n
		\end{equation*} 
		for each $n$. Then $\M:=\varprojlim \M_n$ is a coadmissible $\w{\mathscr{U}(\mathscr{L})}$-module, and
		\begin{equation*}
			\mathcal{U}_n(\mathscr{L})\widetilde{\otimes}_{\w{\mathscr{U}(\mathscr{L})}}\M\cong \M_n
		\end{equation*}
		for each $n$.
	\end{lem}
	\begin{proof}
		By definition, $\M(X)$ is coadmissible, and for any admissible open affinoid subdomain $Y$,
		\begin{equation*}
			\M(Y)\cong \varprojlim (\mathcal{U}_n(\mathscr{L})(Y)\widetilde{\otimes}_{\mathcal{U}_n(\mathscr{L})(X)}\M_n(X))\cong \w{\mathscr{U}(\mathscr{L})}(Y)\w{\otimes}_{\w{\mathscr{U}(\mathscr{L})}(X)}\M(X)
		\end{equation*}
		by Proposition \ref{coadandtor}. Hence $\M$ is coadmissible, and the last isomorphism follows from Corollary \ref{Anacyclic}.
	\end{proof}
	We will from now on concentrate on the case when $X$ is smooth, $\mathscr{L}=\T_X$, and write $\w{\D}_X:=\w{\mathscr{U}(\T_X)}$. We write $\B=X(\T_X)$ for the basis of $X$ consisting of affinoid subdomains with free tangent sheaf.
	
	When $X$ is itself affinoid with a free tangent sheaf and suitable lattices have been chosen, we write $\D_{X_n}=\mathcal{U}_n(\T_X)$, or even $\D_n$ if the space is understood. The category of coadmissible $\w{\D}_X$-modules is denoted by $\C_X$.

	\subsection{Side-changing and the Spencer resolution}
	Let $X$ be a smooth rigid analytic $K$-variety. Consider the sheaf $\D_X$ of finite-order differential operators, which is given by
	\begin{equation*}
		\D_X(Y)=U_{\O(Y)}(\mathcal{T}(Y))
	\end{equation*}
	for any affinoid subdomain $Y$ of $X$.
	
	Note that we have seen in Corollary \ref{capcompletion} that for any $Y\in \B$, $\w{\D}_X(Y)$ is the bornological completion of $\D_X(Y)$, where we endow $\D_X(Y)$ with the bornology induced by the injection $\D_X(Y)\to \w{\D}_X(Y)$.
	
	Since $\D_X$ is a sheaf of $K$-vector spaces strictly embedding into $\w{\D}_X$ and $\w{\D}_X$ is a $\h{\B}c_K$-sheaf (a fortiori, a $\B c_K$-sheaf), we know that $\D_X$ satisfies the sheaf condition for ${\B}c_K$-sheaves on $\B$. 
	
	We remark that there is a natural completion functor
	\begin{equation*}
		\widehat{(-)}: \mathrm{Shv}(\B, \B c_K)\to \mathrm{Shv}(\B, \mathrm{Ind}(\mathrm{Ban}_K))
	\end{equation*}
	given by applying $\mathrm{diss}\circ \widehat{(-)}$ sectionwise and then sheafifying. Note that there is no straightforward way to define a completion functor which always takes values in $\h{\B}c_K$-sheaves, since we have not defined sheafification in this setting.
	
	Quite often, we will however be in a situation where the sheafification is not necessary, as the sectionwise operation already yields a sheaf. For instance, the above and Corollary \ref{capcompletion} then make it clear that $\w{\D}_X$ is the completion of $\D_X$. We will often exploit this when we describe $\w{\D}_X$-module structures.
	
	Locally, we can give an explicit description of $\D_X$ and $\w{\D}_X$. Since $X$ is smooth, there exists an admissible covering $(X_i)$ of $X$ by affinoid subdomains such that $X_i=\mathrm{Sp} A_i$ is \'etale over some polydisk, of dimension $m$, say (see \cite[Proposition 4.15]{ABB}). In particular, we have $x_1, \hdots, x_m\in A_i$ and an $A_i$-basis $\partial_1, \hdots, \partial_m\in \T(X_i)$ such that
	\begin{equation*}
		[\partial_i, \partial_j]=0, \ \partial_i(x_j)=\delta_{ij}
	\end{equation*} 
	for all $i, j$. We call this a \textbf{local coordinate system} on $X_i$.
	
	In this case,
	\begin{equation*}
		\D_X(X_i)=\left\{\sum_{\alpha\in \mathbb{N}^m} f_\alpha \partial^\alpha: \ f_\alpha\in A_i, f_\alpha=0 \ \text{for all but finitely many $\alpha$}\right\} 
	\end{equation*}
	and
	\begin{equation*}
		\w{\D}_X(X_i)=\left\{\sum f_\alpha \partial^\alpha: |\pi^{-n|\alpha|}f_\alpha|\to 0 \ \text{as} \ |\alpha|\to \infty \ \forall n\right\}.
	\end{equation*}
	In what follows, we will often define a bounded $\D_X$-module structure explicitly and then use completion to obtain a $\w{\D}_X$-module structure, by the following lemma.
	\begin{lem}
		\label{completeDmods}
		There is an exact functor
		\begin{equation*}
			\widehat{(-)}: \mathrm{Mod}_{\mathrm{Preshv}(\B, \B c_K)}(\D_X)\to \mathrm{Mod}_{\mathrm{Shv}(X, \mathrm{Ind}(\mathrm{Ban}_K))}(\w{\D}_X)
		\end{equation*}
		given by completing a presheaf $\M$ sectionwise and then applying the functor $(-)^{a, \mathrm{ext}}$ from Proposition \ref{hBcextension}.
		
		In particular, if $\F\in \mathrm{Shv}(X, \mathrm{Ind}(\mathrm{Ban}_K))$ is a complete bornological sheaf, then in order to give $\F$ a $\w{\D}_X$-module structure, it suffices to specify a presheaf $\D_X$-module structure on $\F|_\B\in \mathrm{Shv}(\B, \B c_K)$.
	\end{lem}
	\begin{proof}
		The completion functor $\widehat{(-)}: \B c_K\to \h{\B}c_K$ is strong symmetric monoidal and exact and sends $\D_X(Y)$ to $\w{\D}_X(Y)$ for any $Y\in \B$. The result thus follows from Proposition \ref{hBcextension}.
	\end{proof}

	If $\M\in \mathrm{Shv}(X, \h{\B}c_K)$, let $\mathcal{E}nd_{\mathrm{Shv}(X, \h{\B}c_K)}(\M)=\mathcal{H}om_{\mathrm{Shv}(X, \h{\B}c_K)}(\M, \M)$ be the inner Hom sheaf defined before Lemma \ref{shvclosed}. Explicitly, if $U$ is an admissible open subset of $X$, then $\mathcal{E}nd_{\mathrm{Shv}(X, \h{\B}c_K)}(\M)(U)$ is the vector space $\mathrm{End}_{\mathrm{Shv}(X, \h{\B}c_K)}(\M|_U)$, endowed with the subspace bornology induced from 
	\begin{equation*}
		\prod_{V\subseteq U} \mathrm{Hom}_{\h{\B}c_K}(\M(V), \M(V)),
	\end{equation*}
	where the product runs over all admissible open subsets $V\subseteq U$.
	
	We remark that $\mathrm{diss}(\mathcal{E}nd_{\mathrm{Shv}(X, \h{\B}c_K)}(\M))\cong \mathcal{E}nd_{\mathrm{Shv}(X, \mathrm{Ind}(\mathrm{Ban}_K))}(\mathrm{diss}(\M))$, since $\mathrm{diss}: \h{\B}c_K\to \mathrm{Ind}(\mathrm{Ban}_K)$ preserves limits and intertwines the internal Homs by \cite[Proposition 1.139]{Meyer}.
	
	\begin{prop}
		Let $\M\in \mathrm{Mod}_{\mathrm{Shv}(X, \mathrm{Ind}(\mathrm{Ban}_K))}(\O_X)$ be a complete bornological $\O_X$-module. To give $\M$ the structure of a complete bornological $\w{\D}_X$-module extending the $\O_X$-module structure is equivalent to providing a  morphism $\nabla: \T_X\to \mathcal{E}nd_{\mathrm{Shv}(X, \h{\B} c_K)}(\M)$ in $\mathrm{Shv}(X, \mathrm{Vect}_K)$ satisfying 
		\begin{enumerate}[(i)]
			\item $f\nabla(\partial)(m)=\nabla(f\partial)(m)$ for all $\partial\in \T_X$, $f\in \O_X$, $m\in \M$
			\item $\nabla(\partial)( fm)=\partial(f)m+f\nabla(\partial)( m)$ for all $\partial\in \T_X$, $f\in \O_X$, $m\in \M$
			\item $\nabla([\partial_1, \partial_2])(m)=(\nabla(\partial_1)\nabla(\partial_2)-\nabla(\partial_2)\nabla(\partial_1))(m)$ for all $\partial_1, \partial_2\in \T_X$, $m\in \M$
		\end{enumerate} 
		such that the resulting morphism
		\begin{equation*}
			\D_X\otimes^{\mathrm{pre}}_K \M \to \M
		\end{equation*} 
		is bounded.
	\end{prop}
	
	\begin{proof}
		The map $\T_X\otimes \M \to \M$ gives rise to an abstract $\D_X$-action, as the conditions (i), (ii), (iii) correspond to defining relations for $\D_X$. If the resulting action map is bounded, $\M\in \mathrm{Mod}_{\mathrm{Preshv}(\B, \B c_K)}(\D_X)$. But as $\M$ is already complete, it follows that $\h{\M}=\M\in \mathrm{Mod}_{\mathrm{Shv}(X, \mathrm{Ind}(\mathrm{Ban}_K))}(\w{\D}_X)$ by Lemma \ref{completeDmods}.
		
		Conversely, if $\M$ is endowed with a $\w{\D}_X$-module structure, then for every $U\in \B$, we have an action map
		\begin{equation*}
			\w{\D}_X(U)\widetilde{\otimes}_K \M(U)\to \M(U).
		\end{equation*}
		Since $L: \mathrm{Ind}(\mathrm{Ban}_K)\to \h{\B}c_K$ is strong symmetric monoidal and $\w{\D}_X(U)$, $\M(U)$ are in the essential image of $\mathrm{diss}$, applying $L$ yields a bounded action map
		\begin{equation*}
			\w{\D}_X(U)\h{\otimes}_K \M(U)\to \M(U)
		\end{equation*}
		in $\h{\B}c_K$. Composing with the natural morphism $\D_X(U)\otimes_K\M(U)\to \w{\D}_X(U)\h{\otimes}_K \M(U)$ in $\B c_K$, this makes $\M$ an object in $\mathrm{Mod}_{\mathrm{Preshv}(\B, \B c_K)}(\D_X)$, yielding the desired morphism $\nabla$.
	\end{proof}
	We note that if $\nabla$ is a morphism as above, then $\nabla$ is $\O_X$-linear by property (i). Hence $\nabla$ is bounded by Lemma \ref{propsofcompl}, since $\T_X$ is coherent. Applying $\mathrm{diss}$, we can regard $\nabla$ as a morphism $\nabla: \T_X\to \mathcal{E}nd_{\mathrm{Shv}(X,\mathrm{Ind}(\mathrm{Ban}_K))}(\M)$ in $\mathrm{Mod}_{\mathrm{Shv}(X, \mathrm{Ind}(\mathrm{Ban}_K))}(\O_X)$.
	
	As usual, tensor-hom adjunction then allows us to write $\nabla$ as an $\O_X$-linear morphism $\T_X\widetilde{\otimes}_{K}\M\to \M$, and then as a morphism $\M\to \Omega_X^1\widetilde{\otimes}_{\O_X} \M$, where $\Omega_X^1$ is the cotangent bundle on $X$.
	
	We now extend the results in \cite[section 3]{DcapTwo} from coadmissible modules to all $\w{\D}_X$-modules.
	
	Let $\Omega_X=(\Omega_X^1)^{\wedge \mathrm{dim}X}$ be the sheaf of top differentials. By \cite[Corollary 3.5]{DcapTwo}, it is a right coadmissible $\w{\D}_X$-module. 
	\begin{prop}
		The functor $\Omega_X\widetilde{\otimes}_{\O_X}-$ gives rise to a functor
		\begin{equation*}
			\mathrm{Mod}_{\mathrm{Shv}(X, \mathrm{Ind}(\mathrm{Ban}_K))}(\w{\D}_X)\to \mathrm{Mod}_{\mathrm{Shv}(X, \mathrm{Ind}(\mathrm{Ban}_K))}(\w{\D}_X^{\mathrm{op}}).
		\end{equation*}
	\end{prop}
	\begin{proof}
		Let $Y=\Sp A\in \B$ and let $M\in \mathrm{Mod}_{\h{\B}c_K}(\w{\D}_X(Y))$. Then the formula
		\begin{equation*}
			(d\otimes m)\cdot \partial=(d\cdot \partial)\otimes m-d\otimes (\partial m), \ \partial\in \T_X(Y), \ d\in \Omega_X(Y), \ m\in M
		\end{equation*}
		defines an abstract right $\D_X(Y)$-module structure on $\Omega_X(Y)\otimes_A M$. We now show that this action is bounded.
		
		Let $\A\subseteq A$ be an admissible affine formal model and let $\L\subseteq \T_X(Y)$ be an $(R, \A)$-Lie lattice which is free as an $\A$-module, generated by elements $\partial_1, \hdots, \partial_r$. We now consider the $\D_X(Y)$-action on $\Omega_X(Y)\otimes_A M$ defined as above. Let $B\subseteq \D_X(Y)$ be a bounded subset. Without loss of generality, we can suppose that there exist integers $r_n\in \mathbb{Z}$ such that
		\begin{equation*}
			B=\underset{n}{\cap}\  \pi^{r_n}U_\A(\pi^n\L),
		\end{equation*}
		where we view each $U_\A(\pi\L)$ as a subset of $\D_X(Y)$.
		
		By definition, the bornology on the tensor product is generated by subsets of the form $B_1\otimes B_2$, where $B_1\subseteq \Omega_X(Y)$, $B_2\subseteq M$ are bounded $R$-submodules. 
		
		Let $P\in B$, which by the PBW theorem can be written as
		\begin{equation*}
			P=\sum_{\alpha\in \mathbb{N}^r} \partial^\alpha f_\alpha, \ f_\alpha\in A,
		\end{equation*}
		with all but finitely many $f_\alpha$ being equal to zero. The description of $B$ above then translates to
		\begin{equation*}
			f_\alpha\in \pi^{r_n+n|\alpha|}\A 
		\end{equation*}
		for each $n$ and each $\alpha$, where $|\alpha|=\alpha_1+\hdots+\alpha_r$.
		
		We write $\beta\leq \alpha$ if for all $i$, $\beta_i\leq \alpha_i$.
		
		It then follows that 
		\begin{equation*}
			(B_1\otimes B_2)\cdot P\subseteq \sum_\alpha f_\alpha\left( \sum_{\beta\leq \alpha} B_1\cdot \partial^\beta\otimes \partial^{\alpha-\beta}\cdot B_2\right).
		\end{equation*}
		For each $\alpha$ with $f_{\alpha}\neq 0$, pick $s_\alpha\in \mathbb{Z}$ such that
		\begin{equation*}
			f_\alpha\in \pi^{2s_\alpha}\A, \ f_\alpha\notin \pi^{2s_\alpha+2}\A,
		\end{equation*}
		which we can do as $\A$ is $\pi$-adically separated by \cite[Corollary 7.3/9]{Boschlectures}.
		
		Now choose $r'_n\in \mathbb{Z}$ with
		\begin{equation*}
			r'_n\leq\frac{r_{2n}}{2}-1
		\end{equation*}
		for each $n$.
		Since $f_\alpha\in \pi^{r_{2n}+2n|\alpha|}\A$, but $f_\alpha\notin\pi^{2s_\alpha+2}\A$, we must have 
		\begin{equation*}
			2s_\alpha+2>r_{2n}+2n|\alpha|
		\end{equation*}
		and hence
		\begin{equation*}
			r'_n+n|\alpha|<s_\alpha.
		\end{equation*}
		
		Thus we have firstly $\pi^{s_\alpha}\in \pi^{r'_n+n|\alpha|}\A$ for each $n$, and secondly 
		\begin{equation*}
			\pi^{-s_\alpha}f_\alpha\in \pi^{s_\alpha}\A\subseteq \pi^{r'_n+n|\alpha|}\A.
		\end{equation*}
		A fortiori, this yields
		\begin{equation*}
			\pi^{s_\alpha}\in \pi^{r'_n+n|\beta|}\A
		\end{equation*}
		and
		\begin{equation*}
			\pi^{-s_\alpha}f_\alpha\in \pi^{r'_n+n|\beta|}\A
		\end{equation*}
		for each $n$ and any $\beta$ with $\beta\leq \alpha$.
		
		Writing $f_\alpha=\pi^{s_\alpha}\pi^{-s_\alpha}f_\alpha$, this shows that
		\begin{equation*}
			(B_1\otimes B_2)\cdot B\subseteq (B_1\cdot B')\otimes (B'\cdot B_2),
		\end{equation*}
		where $B'\subseteq \D_X(Y)$ is the bounded subset
		\begin{equation*}
			B'=\underset{n}\cap\  \pi^{r'_n}U_\A(\pi^n\L).
		\end{equation*}
		Since the actions on $\Omega_X(Y)$ and $M$ are bounded, this shows that $(B_1\otimes B_2)\cdot B$ is a bounded set, as required. Thus $\Omega(Y)\otimes_AM\in \mathrm{Mod}_{\B c_K}(\D_X(Y)^{\mathrm{op}})$.
		
		Since $\widehat{(-)}: \B c_K\to \h{\B}c_K$ is strong symmetric monoidal, this shows that $\Omega_X(Y)\h{\otimes}_A M$ is in $\mathrm{Mod}_{\h{\B}c_K}(\w{\D}_X(Y)^{\mathrm{op}})$, and we have produced an exact functor
		\begin{equation*}
			\Omega_X(Y)\h{\otimes}_A-: \mathrm{Mod}_{\h{\B}c_K}(\w{\D}_X(Y))\to \mathrm{Mod}_{\h{\B}c_K}(\w{\D}_X(Y)^{\mathrm{op}}).
		\end{equation*}
		By exactness, this produces the functor
		\begin{equation*}
			\Omega_X(Y)\widetilde{\otimes}_A-: LH(\mathrm{Mod}_{\h{\B}c_K}(\w{\D}_X(Y)))\to LH(\mathrm{Mod}_{\h{\B}c_K}(\w{\D}_X(Y)^{\mathrm{op}})).
		\end{equation*}
		
		By Proposition \ref{IBFrechetmod}, we now have the identification $LH(\mathrm{Mod}_{\h{\B}c_K}(\w{\D}_X(Y)))\cong LH(\mathrm{Mod}_{\mathrm{Ind}(\mathrm{Ban}_K)}(\w{\D}_X(Y)))$, i.e. we have a functor
		\begin{equation*}
			\Omega_X(Y)\widetilde{\otimes}_A-: LH(\mathrm{Mod}_{\mathrm{Ind}(\mathrm{Ban}_K)}(\w{\D}_X(Y)))\to LH(\mathrm{Mod}_{\mathrm{Ind}(\mathrm{Ban}_K)}(\w{\D}_X(Y)^{\mathrm{op}})).
		\end{equation*}
		Since $\Omega_X(Y)$ is a finitely generated projective $A$-module, this functor preserves objects in the essential image of $I$, yielding the desired functor $\mathrm{Mod}_{\mathrm{Ind}(\mathrm{Ban}_K)}(\w{\D}_X(Y))\to \mathrm{Mod}_{\mathrm{Ind}(\mathrm{Ban}_K)}(\w{\D}_X(Y)^{\mathrm{op}})$.
		
		By functoriality of the construction, this produces a corresponding functor on the level of presheaves and then on the level of sheaves, as required.
	\end{proof}
	
	Similarly, suppose that $M\in \mathrm{Mod}_{\h{\B}c_K}(\w{\D}_X(Y))$ for some $Y\in \B$. As usual, the adjunction gives a morphism $\Omega_X^{\otimes-1}(Y)\h{\otimes}_A M\to \mathrm{Hom}_A(\Omega_X(Y), M)$ in $\mathrm{Mod}_{\h{\B}c_K}(A)$, which is an isomorphism as $\Omega_X$ is a line bundle. Now $\mathrm{Hom}_A(\Omega_X(Y), M)$ is a left $\D_X(Y)$-module via
	\begin{equation*}
		(\partial\cdot \phi)(d)=\phi(d\cdot \partial)-\phi(d)\cdot\partial, \ \partial\in\T_X(Y), \ d\in\Omega_X(Y), \ m\in M
	\end{equation*}
	where it is again easy to check that the action is bounded, following a similar strategy to the above. We thus have an exact functor $\Omega_X(Y)^{\otimes-1}\h{\otimes}_{\O_X}-$ from $\mathrm{Mod}_{\h{\B}c_K}(\w{\D}_X(Y)^{\mathrm{op}})$ to $\mathrm{Mod}_{\h{\B}c_K}(\w{\D}_X(Y))$, and the same argument as above allows us to derive and glue to obtain a functor
	\begin{equation*}
		\Omega_X^{\otimes-1} \widetilde{\otimes}_{\O_X}-: \mathrm{Mod}_{\mathrm{Shv}(X, \mathrm{Ind}(\mathrm{Ban}_K))}(\w{\D}_X^{\mathrm{op}})\to \mathrm{Mod}_{\mathrm{Shv}(X, \mathrm{Ind}(\mathrm{Ban}_K))}(\w{\D}_X).
	\end{equation*}
	
	\begin{thm}
		\label{sidechanging}
		The functor $\Omega_X\h{\otimes}_{\O_X}-$ is an equivalence of categories between $\mathrm{Mod}_{\mathrm{Shv}(X, \mathrm{Ind}(\mathrm{Ban}_K))}(\w{\D}_X)$ and $\mathrm{Mod}_{\mathrm{Shv}(X, \mathrm{Ind}(\mathrm{Ban}_K))}(\w{\D}_X^{\mathrm{op}})$, with quasi-inverse $\Omega_X^{\otimes -1}\h{\otimes}_{\O_X}-$.
	\end{thm}
	\begin{proof}
		Let $Y=\Sp A\in \B$, $M\in \mathrm{Mod}_{\mathrm{Ind}(\mathrm{Ban}_K)}(\w{\D}_X(Y))$. By adjunction in $\mathrm{Mod}_{\mathrm{Ind}(\mathrm{Ban}_K)}(A)$, we have the unit morphism
		\begin{equation*}
			M\to \Omega_X^{\otimes-1}(Y)\widetilde{\otimes}_A \Omega_X(Y)\widetilde{\otimes}_A M,
		\end{equation*}
		which is an isomorphism of $A$-modules. We claim that it also $\w{\D}_X(Y)$-linear. Since $M$ is the cokernel of a strict morphism between some complete bornological $\w{\D}_X(Y)$-modules (by considering e.g. $\w{\D}_X(Y)\widetilde{\otimes}_K (\oplus_i c_0(X_i))\cong \mathrm{diss}(\w{\D}_X(Y)\h{\otimes}_K (\oplus_i c_0(X_i)))$ for some sets $X_i$), it suffices to consider the case where $M\in \mathrm{Mod}_{\h{\B}c_K}(\w{\D}_X(Y))$. 
		
		By writing out the action, the morphisms above are then morphisms of $\D_X(Y)$-modules (see \cite[Proposition 1.2.12]{Hotta} for the classical result), and hence of $\w{\D}_X$-modules by Lemma \ref{propsofcompl}.(ii). 
		
		Thus if $\M\in \mathrm{Mod}_{\mathrm{Shv}(X, \mathrm{Ind}(\mathrm{Ban}_K))}(\w{\D}_X)$, then the natural morphism
		\begin{equation*}
			\M\to \Omega_X^{\otimes -1}\widetilde{\otimes}_{\O_X}\Omega_X \widetilde{\otimes}_{\O_X} \M
		\end{equation*}
		is an isomorphism in $\mathrm{Mod}_{\mathrm{Shv}(X, \mathrm{Ind}(\mathrm{Ban}_K))}(\w{\D}_X)$. The same argument applies mutatis mutandis for the counit morphism.
		
		Thus $\Omega_X\widetilde{\otimes}-$ and $\Omega_X^{\otimes -1}\widetilde{\otimes}-$ are quasi-inverse equivalences.
	\end{proof}
	
	Let $X$ be a smooth rigid analytic $K$-space of dimension $n$. We now give an analogue of the Spencer resolution as a locally free resolution of $\O_X$ in $\mathrm{Mod}_{\mathrm{Shv}(\mathrm{Ind}(\mathrm{Ban}_K))}(\w{\D}_X)$.
	
	Let $\T_X$ be the tangent sheaf of $X$. As a coherent $\O_X$-module, the module $\wedge^i \T_X$ is naturally a complete bornological $\O_X$-module for any $i=0, \hdots, n$ (see Theorem \ref{Dcapembedding}). Then $S^{-i}:=\w{\D}_X\widetilde{\otimes}_{\O_X}\wedge^i \T_X$ is a complete bornological $\w{\D}_X$-module (with action only on the first factor) which is locally isomorphic to $\w{\D}_X^{\oplus r_i}$, where $r_i=\binom{n}{i}$. In particular, it is coadmissible.\\
	Note that $S^{-i}$ is the completion of the bornological $\D_X$-module $\D_X\otimes_{\O_X}\wedge^i\T_X$.
	
	For $i=1, \hdots, n$, we have a $\D_X$-module morphism 
	\begin{align*}
		\D_X\otimes \wedge^i\T_X&\to \D_X\otimes \wedge^{i-1}\T_X\\
		P\otimes (\wedge \partial_j)&\mapsto \sum_j (-1)^{j+1}P\partial_j\otimes \partial_1\wedge\hdots \wedge \h{\partial_j}\hdots \wedge \partial_i\\
		&+\sum_{j<j'} (-1)^{j+j'}P[\partial_j, \partial_{j'}]\otimes \partial_1\wedge \hdots \wedge \h{\partial_j}\hdots \wedge \h{\partial_{j'}}\hdots \wedge \partial_i
	\end{align*}
	as in \cite[Lemma 1.5.27]{Hotta}. Tensoring with $\w{\D}_X$, we have a morphism between coadmissible $\w{\D}_X$-modules $S^{-i}\to S^{-i+1}$, which is then automatically bounded by Theorem \ref{Dcapembedding}.
	
	We also have the natural bounded map $S^0\cong \w{\D}_X\to \O_X$ sending any differential operator $P$ to $P(1)$.
	\begin{thm}
		\label{spencer}
		With the morphisms as above, the map $S^\bullet\to \O_X$ is an isomorphism in $\mathrm{D}(\w{\D}_X)$.
	\end{thm}
	\begin{proof}
		By the same argument as in \cite[Lemma 1.5.27]{Hotta}, the uncompleted complex $\D_X\otimes \wedge^\bullet \T_X$ is a resolution of $\O_X$ in the category of abstract $\D_X$-modules. From this, we can obtain $S^\bullet$ by applying $\w{\D}_X\otimes_{\D_X}-$ (note that $\w{\D}_X{\otimes}_{\D_X}\O_X\cong \O_X$ by \cite[Proposition 7.2]{DcapTwo}). As $\w{\D}_X$ is a flat $\D_X$-module in the category of abstract $\D_X$-modules by Theorem \ref{flatcompletion}, we obtain an exact complex of coadmissible $\w{\D}_X$-modules. Viewed as a complex of complete bornological $\w{\D}_X$-modules, it is then strictly exact by Theorem \ref{Dcapembedding}.
	\end{proof}
	Applying side-changing, we immediately obtain the following.
	\begin{cor}
		The complex
		\begin{equation*}
			0\to \wedge^0\Omega_X^1\widetilde{\otimes}_{\O_X} \w{\D}_X\to \wedge^1 \Omega^1_X\widetilde{\otimes}_{\O_X}\w{\D}_X\to \hdots \to \Omega_X\widetilde{\otimes}_{\O_X}\w{\D}_X\to \Omega_X\to 0
		\end{equation*}
		is strictly exact in $\mathrm{Mod}_{\mathrm{Shv}(X, \mathrm{Ind}(\mathrm{Ban}_K))}(\w{\D}_X^{\mathrm{op}})$, giving a locally free resolution of $\Omega_X$.
	\end{cor}
	Under this transformation, the maps in the resolution are now locally given as the completion of
	\begin{align*}
		\wedge^i \Omega^1\otimes \D_X&\to \wedge^{i+1}\Omega^1\otimes \D_X\\
		\omega\otimes P&\mapsto \mathrm{d}\omega\otimes P+\sum_j \mathrm{d}x_j\wedge \omega\otimes \partial_jP,
	\end{align*}
	where $\{x_j, \partial_j\}_{1\leq j\leq n}$ is a local coordinate system on $X$.
	\begin{cor}
		Let $X$ be a smooth rigid analytic $K$-variety of dimension $n$. In $\mathrm{D}(\mathrm{Shv}(X, LH(\h{\B} c_K)))$, there is a canonical isomorphism $\Omega_X\widetilde{\otimes}^\mathbb{L}_{\w{\D}_X}\O_X\cong \Omega_{X/K}^\bullet[n]$.
	\end{cor}
	\section{The six operations}
	\subsection{Tensor product}
	
	Let $X$ be a smooth rigid analytic $K$-variety. If $\M, \M'\in \mathrm{Mod}_{\mathrm{Shv}(X, \mathrm{Ind}(\mathrm{Ban}_K))}(\w{\D}_X)$, we would like to endow $\M\widetilde{\otimes}_{\O_X} \M'$ with a $\w{\D}_X$-module structure as well. For this, we follow the same strategy as for side-changing: we define section-wise a $\D_X$-module structure in the bornological case, which we complete to obtain a $\w{\D}_X$-module, and then we derive in order to pass to $LH(\h{\B}c_K)\cong LH(\mathrm{Ind}(\mathrm{Ban}_K))$.
	
	\begin{lem}
		Let $X=\Sp A$ be a smooth affinoid with free tangent sheaf. If $M, M'\in \mathrm{Mod}_{LH(\h{\B}c_K)}(\w{\D}_X(X))$, then $M\widetilde{\otimes}_A M'$ carries a natural $\w{\D}_X(X)$-module structure extending the $A$-module structure.
	\end{lem}
	\begin{proof}
		First suppose that $M, M'\in \mathrm{Mod}_{\h{\B}c_K}(\w{\D}_X(X))$. Then the action
		\begin{equation*}
			\partial\cdot(m\otimes m')=(\partial m)\otimes m'+m\otimes(\partial m')
		\end{equation*}
		makes $M \otimes_A M'$ an abstract $\D_X(X)$-module. The same argument as in the previous subsection ensures that the action is bounded, so that $M\otimes_A M'\in \mathrm{Mod}_{\B c_K}(\D_X(X))$. 
		
		Since the completion functor $\widehat{(-)}: \B c_K\to \h{\B}c_K$ is strong symmetric monoidal and strongly right exact, this makes $M\h{\otimes}_A M'\cong \h{M\otimes_AM'}$ an object of $\mathrm{Mod}_{\h{\B}c_K}(\w{\D}_X(X))$.
		
		As the construction is functorial, we now have a bifunctor
		\begin{equation*}
			-\h{\otimes}_A -: \mathrm{Mod}_{\h{\B}c_K}(\w{\D}_X(X))\times \mathrm{Mod}_{\h{\B}c_K}(\w{\D}_X(X))\to \mathrm{Mod}_{\h{\B}c_K}(\w{\D}_X(X)).
		\end{equation*}
		
		Since $\mathrm{Mod}_{\h{\B}c_K}(\w{\D}_X(X))$ contains enough objects which are flat over $A$ (indeed, $\h{\B}c_K$ has enough flat projectives, and $\w{\D}_X(X)\cong A\h{\otimes}_K K\{d_1\}\h{\otimes}_K\hdots \h{\otimes}_K K\{d_n\}$ is a flat $A$-module by Corollary \ref{stronglyflat}), we can derive this functor to
		\begin{equation*}
			-\widetilde{\otimes}_A-: LH(\mathrm{Mod}_{\h{\B}c_K}(\w{\D}_X(X)))\times LH(\mathrm{Mod}_{\h{\B}c_K}(\w{\D}_X(X)))\to LH(\mathrm{Mod}_{\h{\B}c_K}(\w{\D}_X(X))).
		\end{equation*}
		Using the equivalence $LH(\mathrm{Mod}_{\h{\B}c_K)}(\w{\D}_X(X))\cong \mathrm{Mod}_{LH(\h{\B}c_K)}(\w{\D}_X(X))$, this yields the result.
	\end{proof}
	
	\begin{prop}
		Let $X$ be a smooth rigid analytic $K$-variety.
		
		There is a bifunctor
		\begin{equation*}
			-\widetilde{\otimes}^{\mathbb{L}}_{\O_X} -: \mathrm{D}(\w{\D}_X)\times \mathrm{D}(\w{\D}_X)\to \mathrm{D}(\w{\D}_X)
		\end{equation*}
		which under the forgetful functor $\mathrm{D}(\w{\D}_X)\to \mathrm{D}(\O_X)$ agrees with the derived tensor product $\widetilde{\otimes}_{\O_X}$ on $\mathrm{D}(\O_X)$.
	\end{prop}
	\begin{proof}
		Using the identification $LH(\h{\B}c_K)\cong LH(\mathrm{Ind}(\mathrm{Ban}_K))$, the functor above yields
		\begin{align*}
			-\widetilde{\otimes}^{\mathrm{pre}}_{\O_X}-: \mathrm{Mod}_{\mathrm{Shv}(X, LH(\mathrm{Ind}(\mathrm{Ban}_K)))}(\w{\D}_X)\times \mathrm{Mod}_{\mathrm{Shv}(X, LH(\mathrm{Ind}(\mathrm{Ban}_K)))}(\w{\D}_X)\\
			\to \mathrm{Mod}_{\mathrm{Preshv}(X, LH(\mathrm{Ind}(\mathrm{Ban}_K)))}(\w{\D}_X).
		\end{align*}
		Composing this with sheafification yields the desired bifunctor for sheaves.
		
		Since $\mathrm{Mod}_{\mathrm{Shv}(X, LH(\mathrm{Ind}(\mathrm{Ban}_K)))}(\w{\D}_X)\cong \mathrm{Mod}_{LH(\mathrm{Shv}(X, \mathrm{Ind}(\mathrm{Ban}_K)))}(\w{\D}_X)$ has enough $\O_X$-flat objects (again, consider objects of the form $\w{\D}_X\widetilde{\otimes}_K I(\F)$), we can invoke \cite[Corollary 14.4.9]{KS} to obtain the desired derived functor between the unbounded derived categories
		\begin{equation*}
			-\widetilde{\otimes}^{\mathbb{L}}_{\O_X}-: \mathrm{D}(\w{\D}_X)\times \mathrm{D}(\w{\D}_X)\to \mathrm{D}(\w{\D}_X).\qedhere
		\end{equation*}
	\end{proof}

	\subsection{Duality}
	By Proposition \ref{deriverelthom}, $\mathrm{R}\mathcal{H}om_{\w{\D}_X}(-, \w{\D}_X)$ is a well-defined functor from $\mathrm{D}(\w{\D}_X)$ to $\mathrm{D}(\w{\D}_X^{\mathrm{op}})^{\mathrm{op}}$. We define the \textbf{duality functor}
	\begin{equation*}
		\mathbb{D}=\mathrm{R}\mathcal{H}om_{\w{\D}_X}(-, \w{\D}_X)\widetilde{\otimes}_{\O_X} \Omega_X^{\otimes-1}[\dim X]: \mathrm{D}(\w{\D}_X)\to \mathrm{D}(\w{\D}_X)^{\mathrm{op}},
	\end{equation*}
	using the side-changing results from subsection 6.3.
	\begin{prop}
		Let $\M$ be a vector bundle with integrable connection on $X$ (i.e. a coadmissible $\w{\D}_X$-module which is also a locally free coherent $\O_X$-module). Then 
		\begin{equation*}
			\mathbb{D}\M\cong \mathcal{H}om_{\O_X}(\M, \O_X).
		\end{equation*} 
	\end{prop}
	\begin{proof}
		Having already established the Spencer resolution (Theorem \ref{spencer}), the argument in \cite[Example 2.6.10]{Hotta} goes through unchanged.
	\end{proof}
	\subsection{Inverse image}
	Let $f: X\to Y$ be a morphism of smooth rigid analytic $K$-varieties. Let $\w{\D}_{X\to Y}:= \O_X\widetilde{\otimes}_{f^{-1}\O_Y} f^{-1} \w{\D}_Y$ denote the \textbf{transfer bimodule}. 
	\begin{lem}
		$\w{\D}_{X\to Y}$ is a $(\w{\D}_X, f^{-1}\w{\D}_Y)$-bimodule in $\mathrm{Shv}(X, \mathrm{Ind}(\mathrm{Ban}_K))$.
	\end{lem}
	\begin{proof}
		As in the classical situation, the morphism $\theta: \T_X\to \O_X\otimes_{f^{-1}\O_Y}f^{-1}\T_Y$ induces a left $\D_X$-module structure on $\O_X\otimes_{f^{-1}\O_Y}f^{-1}\D_Y$ such that the action commutes with the right $f^{-1}\D_Y$-action. We now claim that this action is bounded.
		
		Let $\B'$ be the collection of affinoid subdomains $U$ of $X$ which have free tangent sheaf and such that $f(U)$ is contained in an affinoid subdomain of $Y$ with free tangent sheaf. This is a basis of $X$.
		
		Let $U=\Sp B\in \B'$ and let $V=\Sp A\subseteq Y$ be an affinoid with free tangent sheaf such that $f(U)\subseteq V$. Let $\A\subseteq A$ and $\B\subseteq B$ be admissible affine formal models such that $\B$ contains the image of $\A$. Let $\L_Y\subseteq \T_Y(V)$ be a smooth $(R, \A)$-Lie lattice. Let $S$ be a bounded subset of $\D_Y(V)$, i.e. for each $i\geq 0$ there exist $n_i\in \mathbb{Z}$ such that $S\subseteq \pi^{n_i}U_{\A}(\pi^i\L)$. 
		
		Let $\L'\subseteq \T_X(X)$ be an $(R, \B)$-Lie lattice with the property that $\theta(\L')\subseteq \B\otimes \L$. Then
		\begin{equation*}
			\pi^i\L'\cdot (\B\otimes S)\subseteq \B\otimes \pi^{n_i}U_{\A}(\pi^i\L)
		\end{equation*} 
		by definition of the action, since the $\L'$-action on $B$ preserves $\B$. But then it is clear that
		\begin{equation*}
			U_{\B}(\pi^i\L')\cdot (\B\otimes S)\subseteq \B\otimes \pi^{n_i}U_{\A}(\pi^i\L),
		\end{equation*}
		showing that the action is bounded.
		
		Hence $B\otimes_A \D_Y(V)\in \mathrm{Mod}_{\B c_K}(\D_X(U)\otimes_K \D_Y(V)^{\mathrm{op}})$, and therefore
		\begin{equation*}
			B\h{\otimes}_A \w{\D}_Y(V)\in \mathrm{Mod}_{\h{\B}c_K}(\w{\D}_X(U)\h{\otimes}_K \w{\D}_Y(V)^{\mathrm{op}})
		\end{equation*}
		by applying the completion functor.
		
		Now
		\begin{align*}
			\mathrm{diss}(B\h{\otimes}_A \w{\D}_Y(V))&\cong \mathrm{diss}(B\h{\otimes}_K K\{d_1, \hdots, d_n\})\\&\cong \mathrm{diss}(B)\widetilde{\otimes}_K \mathrm{diss}(K\{d_1, \hdots, d_n\})\\
			&\cong B\widetilde{\otimes}_A \w{\D}_Y(V)
		\end{align*}
		by Proposition \ref{dissandtensor}. Since $\mathrm{diss}$ is lax symmetric monoidal, we thus have 
		\begin{equation*}
			B\widetilde{\otimes}_A \w{\D}_Y(V)\in \mathrm{Mod}_{\mathrm{Ind}(\mathrm{Ban}_K)}(\w{\D}_X(U)\widetilde{\otimes}_K \w{\D}_Y(V)^{\mathrm{op}}).
		\end{equation*}

		By functoriality of the construction, this makes the presheaf $\O_X\widetilde{\otimes}^{\mathrm{pre}}_{f^{-1}\O_X}f^{-1}\w{\D}_Y$ a $(\w{\D}_X, f^{-1}\w{\D}_Y)$-bimodule, and sheafification yields the result.
	\end{proof}
	
	We can now define the \textbf{extraordinary inverse image functor}
	\begin{align*}
		f^!: \ \mathrm{D}(\w{\D}_Y)&\to \mathrm{D}(\w{\D}_X)\\
		\M^\bullet\ \ \ &\mapsto\w{\D}_{X\to Y} \widetilde{\otimes}^{\mathbb{L}}_{f^{-1}\w{\D}_Y} f^{-1}\M^\bullet [\mathrm{dim}X-\mathrm{dim}Y].
	\end{align*}
	By associativity of the tensor product, $f^!\M^\bullet\cong \O_X\widetilde{\otimes}^{\mathbb{L}}_{f^{-1}\O_Y} f^{-1}\M^\bullet [\mathrm{dim}X-\mathrm{dim}Y]$ in $\mathrm{D}(\O_X)$.
	
	\begin{prop}
		\label{invimcomposition}
		Let $f: X\to Y$ and $g: Y\to Z$ be morphisms of smooth rigid analytic $K$-spaces. Then there is a natural isomorphism
		\begin{equation*}
			(gf)^!\cong f^!g^!: \mathrm{D}(\w{\D}_Z)\to \mathrm{D}(\w{\D}_X). 
		\end{equation*}
	\end{prop}
	\begin{proof}
		For any $\M^\bullet\in \mathrm{D}(\w{\D}_Z)$, we have
		\begin{align*}
			f^!g^!\M^\bullet[\mathrm{dim}Z-\mathrm{dim}X]&\cong\O_X\widetilde{\otimes}^{\mathbb{L}}_{f^{-1}\O_Y}f^{-1}(\O_Y\widetilde{\otimes}^{\mathbb{L}}_{g^{-1}\O_Z}g^{-1}\w{\D}_Z\widetilde{\otimes}^{\mathbb{L}}_{g^{-1}\w{\D}_Z}g^{-1}\M^\bullet)\\
			&\cong \O_X\widetilde{\otimes}^{\mathbb{L}}_{f^{-1}g^{-1}\O_Z}f^{-1}g^{-1}\w{\D}_Z\widetilde{\otimes}^{\mathbb{L}}_{f^{-1}g^{-1}\w{\D}_Z}f^{-1}g^{-1}\M^\bullet,
		\end{align*}
		and the result follows from Lemma \ref{catcomposition}.
	\end{proof}
	We also define $f^+=\mathbb{D}_Xf^{!}\mathbb{D}_Y$, the \textbf{inverse image functor}.
	\subsection{Direct image}
	As before, let $f: X\to Y$ be a morphism of smooth rigid analytic $K$-spaces. We now define the direct image functor for right $\w{\D}$-modules,
	\begin{equation*}
		f^r_+: \mathrm{D}(\w{\D}_X^{\mathrm{op}})\to \mathrm{D}(\w{\D}_Y^{\mathrm{op}}),
	\end{equation*} 
	by setting
	\begin{equation*}
		f^r_+(\M^\bullet)=\mathrm{R}f_*(\M^\bullet\widetilde{\otimes}^{\mathbb{L}}_{\w{\D}_X} \w{\D}_{X\to Y}).
	\end{equation*}
	Using Theorem \ref{sidechanging}, we can define the \textbf{direct image functor} for left $\w{\D}$-modules as
	\begin{equation*}
		f_+(\M^\bullet)=f^r_+(\Omega_X \widetilde{\otimes}_{\O_X} \M^\bullet)\widetilde{\otimes}_{\O_Y}\Omega_Y^{\otimes -1}.
	\end{equation*}
	\begin{lem}
		Let $\w{\D}_{Y\leftarrow X}=\Omega_X\widetilde{\otimes}_{\O_X} \w{\D}_{X\to Y}\widetilde{\otimes}_{f^{-1}\O_Y}f^{-1}\Omega^{\otimes -1}_Y$. Then there is a natural isomorphism
		\begin{equation*}
			f_+(\M^\bullet)\cong \mathrm{R}f_*(\w{\D}_{Y\leftarrow X}\widetilde{\otimes}^{\mathbb{L}}_{\w{\D}_X} \M^\bullet)
		\end{equation*}
		for $\M^\bullet\in \mathrm{D}(\w{\D}_X)$.
	\end{lem}
	\begin{proof}
		As $\Omega_Y^{\otimes -1}$ is a line bundle, we clearly have 
		\begin{equation*}
			\mathrm{R}f_*(\F)\widetilde{\otimes}_{\O_Y}\Omega_Y^{\otimes -1}\cong \mathrm{R}f_*(\F\widetilde{\otimes}_{f^{-1}\O_Y}f^{-1}\Omega_Y^{\otimes -1})
		\end{equation*}
		for any $\F\in \mathrm{Mod}_{\mathrm{Shv}(Y, LH(\h{\B}c_K))}(f^{-1}\O_Y)$. Comparing the actions yields the result.
	\end{proof}
	We also note that for closed embeddings $i: X\to Y$, we recover the functor $i_+$ described in \cite{DcapTwo} -- see section 9.1 for details.
	
	Lastly, we set $f_!=\mathbb{D}_Yf_+\mathbb{D}_X$, the \textbf{extraordinary (shriek) direct image functor}.

	\subsection{Composition results for direct image functors}	
	The composition of direct image functors is far more subtle than for inverse images.
	\begin{lem}
		\label{ifprojthendirectimcomp}
		Let $f: X\to Y$ and $g: Y\to Z$ be morphisms of smooth rigid analytic $K$-spaces. Let $\M^\bullet\in \mathrm{D}(\w{\D}_X)$. If 
		\begin{equation*}
			\w{\D}_{Z\leftarrow Y}\widetilde{\otimes}^\mathbb{L}_{\w{\D}_Y}f_+\M^\bullet\cong \mathrm{R}f_*\left(\w{\D}_{Z\leftarrow X}\widetilde{\otimes}^\mathbb{L}_{\w{\D}_X} \M^\bullet\right) 
		\end{equation*}
		via tha natural morphism, then
		\begin{equation*}
			(gf)_+(\M^\bullet)\cong g_+f_+(\M^\bullet)
		\end{equation*}
		via the natural morphism.
	\end{lem}
	\begin{proof}
		This follows directly from the definitions and the composition result for sheaf theoretic direct images (Lemma \ref{catcomposition}). In the classical case of $\D_X$-modules on quasi-projective complex varieties, where the corresponding projection formula is always satisfied, the argument is given in \cite[Proposition 1.5.21]{Hotta}.
	\end{proof}
	The right module version of the above holds mutatis mutandis.
	\begin{lem}
		\label{compforsmooth}
		Let $f: X\to Y$ and $g: Y\to Z$ be morphisms of smooth rigid analytic $K$-spaces. Suppose that $g$ is smooth. Then
		\begin{equation*}
			(gf)_+(\M^\bullet)\cong g_+f_+(\M^\bullet)
		\end{equation*}
		via the natural morphism for all $\M^\bullet\in \mathrm{D}(\w{\D}_X)$.
	\end{lem}
	\begin{proof}
		As the direct image functors intertwine side-changing, it suffices to show the claim for right $\w{\D}$-modules.
		
		As $g$ is smooth, the transfer bimodule $\w{\D}_{Y\to Z}$ is locally given by 
		\begin{equation*}
			\w{\D}_{Y\to Z}\cong\w{\D}_Y/(\sum_{i=d+1}^r\w{\D}_Y\partial_i) 
		\end{equation*}
		(as a left $\w{\D}_Y$-module) for suitable coordinates on $Y$. In particular, the same arguments as for the Spencer resolution in subsection 6.3 show that locally on $Y$, $\w{\D}_{Y\to Z}$ can be represented by a finite complex of finite free $\w{\D}_Y$-modules. We therefore have the projection formula
		\begin{equation*}
			\mathrm{R}f_*(\N^\bullet)\widetilde{\otimes}^\mathbb{L}_{\w{\D}_Y}\w{\D}_{Y\to Z} \cong \mathrm{R}f_*(\N^\bullet\widetilde{\otimes}^\mathbb{L}_{f^{-1}\w{\D}_Y}f^{-1}\w{\D}_{Y\to Z})
		\end{equation*} 
		for any $\N^\bullet\in \mathrm{D}(f^{-1}\w{\D}_Y^\mathrm{op})$.
		
		Setting $\N^\bullet=\M^\bullet\widetilde{\otimes}^\mathbb{L}_{\w{\D}_X}\w{\D}_{X\to Y}$ for $\M^\bullet\in \mathrm{D}(\w{\D}_X^\mathrm{op})$, this becomes the isomorphism required in Lemma \ref{ifprojthendirectimcomp}.
	\end{proof}
	It remains to verify an analogous statement in the case when $g$ is a closed embedding, ideally resting on a similar projection formula as above. We are currently not aware of any argument which can treat this case for arbitrary $\M^\bullet$, but we will provide a proof in the case when $\M^\bullet$ is a $\C$-complex in subsection 9.1.
	
	Alternatively, we can restrict the class of spaces under consideration, as in the lemma below.
	\begin{lem}
		\label{clsmimcomp}
		Let $f:X\to Y$ and $g:Y\to Z$ be morphisms of smooth rigid analytic $K$-spaces. Assume that $f_*$ has finite cohomological dimension and that for any admissible open affinoid $V\subseteq Y$, the subset $f^{-1}V$ is quasi-compact. Then
		\begin{equation*}
			(gf)_+\M^\bullet\cong g_+f_+(\M^\bullet)
		\end{equation*}
		via the natural morphism for all $\M^\bullet\in \mathrm{D}(\w{\D}_X)$.
	\end{lem}
	\begin{proof}
		Factoring $g$ as a closed immersion followed by a smooth morphism, Lemma \ref{compforsmooth} allows us to reduce to the case where $g$ is a closed immersion. We now verify that the (right module version of the) condition in Lemma \ref{ifprojthendirectimcomp} is satisfied. As this is a local question on $Y$, we assume that $Z$ is affinoid admitting a local coordinate system $x_1, \hdots, x_d$, $\partial_1, \hdots, \partial_d$ such that $Y$ is cut out by 
		\begin{equation*}
			x_{r+1}=\hdots=x_d=0.
		\end{equation*}
		In this case, 
		\begin{equation*}
			\w{\D}_{Y\to Z}\cong \w{\D}_Y\widetilde{\otimes}_K K\{\partial_{r+1}\}\widetilde{\otimes}_K \hdots \widetilde{\otimes}_K K\{\partial_d\}=: \w{\D}_Y\widetilde{\otimes}_K K\{\underline{\partial}\}
		\end{equation*}
		as left $\w{\D}_Y$-modules.
		
		It thus suffices to verify that
		\begin{equation*}
			\mathrm{R}f_*(\N^\bullet)\widetilde{\otimes}_K K\{\underline{\partial}\}\to \mathrm{R}f_*(\N^\bullet\widetilde{\otimes}_K K\{\underline{\partial}\})
		\end{equation*} 
		is an isomorphism for $\N^\bullet\in \mathrm{D}(\mathrm{Shv}(X, LH(\h{\B} c_K)))$.
		
		By Corollary \ref{stronglyflat}, $K\{\underline{\partial}\}$ is strongly flat. Let $\N\in \mathrm{Shv}(X, LH(\h{\B} c_K))$ be an injective sheaf. Then by Lemma \ref{tensoronqcompact}, 
		\begin{equation*}
			(\N\widetilde{\otimes}_K K\{\underline{\partial}\})(U)\cong \N(U)\widetilde{\otimes}_K K\{\underline{\partial}\} 
		\end{equation*}
		for any quasi-compact admissible open subset $U\subseteq X$. 
		
		In particular, $\N\widetilde{\otimes}_K K\{\underline{\partial}\}$ is still $f_*$-acyclic and
		\begin{equation*}
			f_*(\N)\widetilde{\otimes}_K K\{\underline{\partial}\}\cong f_*(\N\widetilde{\otimes}_K K\{\underline{\partial}\})
		\end{equation*}
		via the natural morphism.
		
		Thus we have shown the claim in the case when $\N$ is an injective sheaf in the left heart.
		
		If $\N^\bullet$ is bounded below, we can represent it by a complex consisting of injective sheaves, and the desired isomorphism follows directly from the above (noting as in the above that if $\N$ is injective, $\N\widetilde{\otimes} K\{\underline{\partial}\}$ is still $f_*$-acyclic). More generally, if $\N^\bullet$ is arbitrary, we use the assumption that $f_*$ has finite cohomological dimension to reduce to the case of a bounded below complex. This proves the projection formula and hence the Lemma.
	\end{proof}
	This covers a large class of cases (in particular, when $f$ and $g$ are closed embeddings), but not all morphisms in which we might be interested (e.g. quasi-Stein spaces).
	
	We remark that the additional quasi-compactness assumption was only used to appeal to Lemma \ref{tensoronqcompact}, to ensure that the tensor product commuted with the products involved in a Cech complex. It is thus not surprising that we can drop this assumption whenever we can invoke Corollary \ref{commutewprod} instead.
	
	\begin{lem}
		Let $X$ be a separated rigid analytic $K$-variety. Let $\M\in \mathrm{Shv}(X, \h{\B}c_K)$ and $\B$ be a basis of the topology on $X$, consisting of admissible open affinoid subdomains such that for any $U\in \B$, the following properties hold:
		\begin{enumerate}[(i)]
			\item Any affinoid subdomain of $U$ is also in $\B$. In particular, $\B$ is closed under finite intersections.
			\item $\M|_U$ has vanishing higher Cech cohomology with respect to any finite affinoid covering.
			\item $\M(U)$ is a pseudo-nuclear space of countable type.
		\end{enumerate}
		Let $V\subseteq X$ be an admissible open subspace admitting a countable admissible covering by elements of $\B$. Then 
		\begin{equation*}
			(\M\h{\otimes}_K K\{x_1, \hdots, x_n\})(V)\cong \M(V)\h{\otimes}_K K\{x_1, \hdots, x_n\} 
		\end{equation*}
		for any $n$. 
		
		More generally, 
		\begin{equation*}
			\mathrm{H}^j(V, \M\h{\otimes}_K K\{\underline{x}\})\cong \mathrm{H}^j(V, \M)\widetilde{\otimes}_K K\{\underline{x}\}
		\end{equation*}
		for any $j$.
	\end{lem}
	\begin{proof}
		By Lemma \ref{tensoronqcompact}, 
		\begin{equation*}
			(\M\widetilde{\otimes}_K K\{\underline{x}\})(U)\cong \M(U)\widetilde{\otimes}_K K\{\underline{x}\}
		\end{equation*}
		for any affinoid $U$. In particular, $\M\widetilde{\otimes}_K K\{\underline{x}\}$ has vanishing higher Cech cohomology on affinoids in $\B$ relative to finite affinoid coverings.
		
		Let $V$ be an admissible open subspace of $X$ and let $\mathfrak{U}$ be a countable $\B$-covering of $V$. By the above, we can deduce as in \cite[tags 01EW, 01ET]{stacksproj} that
		\begin{equation*}
			\mathrm{H}^j(V, \M)\cong\check{\mathrm{H}}^j(\mathfrak{U}, \M)
		\end{equation*}
		and
		\begin{equation*}
			\mathrm{H}^j(V, \M\widetilde{\otimes}_K K\{\underline{x}\})\cong \check{\mathrm{H}}^j(\mathfrak{U}, \M\widetilde{\otimes}_K K\{\underline{x}\})
		\end{equation*}
		for any $j$.
		
		Note that by assumption, $\M(U)$ is pseudo-nuclear of countable type for any $U\in \mathfrak{U}$, so that Corollary \ref{commutewinvIB} implies that 
		\begin{align*}
			\check{C}^j(\mathfrak{U}, \M\widetilde{\otimes}_K K\{\underline{x}\})&\cong \prod \left((\M\widetilde{\otimes}_K K\{\underline{x}\})(U_{i_0\hdots  i_j})\right)\\
			&\cong \prod \left( \M(U_{i_0\hdots i_j})\widetilde{\otimes}_K K\{\underline{x}\}\right)\\
			&\cong (\prod \M(U_{i_0\hdots i_j}))\widetilde{\otimes}_K K\{\underline{x}\}\\
			&\cong \check{C}^j(\mathfrak{U}, \M)\h{\otimes}_K K\{\underline{x}\}.
		\end{align*}
		Thus the strong flatness of $K\{\underline{x}\}$ implies that
		\begin{equation*}
			\check{\mathrm{H}}^j(\mathfrak{U}, \M\widetilde{\otimes}_K K\{\underline{x}\})\cong \check{\mathrm{H}}^j(\mathfrak{U}, \M)\widetilde{\otimes}_K K\{\underline{x}\},
		\end{equation*}
		finishing the proof.
	\end{proof}
	
	\begin{lem}
		Let $X$ be separated rigid analytic $K$-variety with countably many connected components, and let $f: X\to Y$ be a morphism of rigid analytic $K$-varieties. Let $\M$ and $\B$ be as in the previous lemma.
		Then the morphism
		\begin{equation*}
			\mathrm{R}f_*(\M)\widetilde{\otimes}_K K\{x_1, \hdots,x_n\}\to \mathrm{R}f_*(\M\widetilde{\otimes}_K K\{x_1, \hdots, x_n\}) 
		\end{equation*}
		is an isomorphism for any $n$.
	\end{lem}
	\begin{proof}
		Let $W$ be an affinoid subdomain of $Y$. In the light of the previous lemma, we only need to show that $f^{-1}W$ admits a countable $\B$-covering, since $\mathrm{H}^j(\mathrm{R}f_*(\M))$ is the sheafification of $W\mapsto \mathrm{H}^j(f^{-1}W, \M)$ for any $j$ due to \cite[tag 0D5X]{stacksproj}.
		
		Since $X$ is quasi-paracompact with countably many connected components, it admits a countable admissible covering by affinoid subspaces $U_i$ (see \cite[p. 212]{Boschlectures}). As $Y$ is quasi-separated, the graph morphism 
		\begin{equation*}
			X\cong X\times_ Y Y\to X\times Y
		\end{equation*} 
		is quasi-compact, since it is obtained via base change from the diagonal map for $Y$. In particular, $U_i\cap f^{-1}W$ is a quasi-compact admissible open subspace of $X$ for each $i$, i.e. $f^{-1}W$ admits a countable admissible covering by quasi-compact spaces. Since $\B$ is a basis, we can now cover each $U_i\cap f^{-1}W$ by finitely many affinoids in $\B$, yielding the desired countable admissible covering of $f^{-1}W$. 
	\end{proof}
	
	\begin{cor}
		\label{projforCcomplexes}
		Let $X$ be a separated rigid analytic $K$-variety with countably many connected components, and let $f: X\to Y$ be a morphism of rigid analytic spaces. Let $\M^\bullet$ be a bounded below chain complex of objects in $\mathrm{Shv}(X, \h{\B}c_K)$, and $\B$ be a basis of $X$ consisting of admissible open affinoid subspaces such that for any $U\in \B$, the following properties hold:
		\begin{enumerate}[(i)]
			\item Any affinoid subdomain of $U$ is also in $\B$.
			\item For each $j$, $\M^j|_U$ has vanishing higher Cech cohomology with respect to any finite affinoid covering.
			\item For each $j$, $\M^j(U)$ is a pseudo-nuclear space of countable type.
		\end{enumerate}
		Then the morphism
		\begin{equation*}
			\mathrm{R}f_*(\M^\bullet)\widetilde{\otimes}_K K\{x_1, \hdots, x_n\}\to \mathrm{R}f_*(\M^\bullet\widetilde{\otimes}_K K\{x_1, \hdots, x_n\})
		\end{equation*}
		is an isomorphism in $\mathrm{D}(\mathrm{Shv}(Y, LH(\h{\B}c_K)))$ for any $n$.
	\end{cor}
	\begin{proof}
		Let $W\subseteq Y$ be an affinoid subdomain, and let $\mathfrak{U}$ be a countable admissible $\B$-covering of $f^{-1}W$, as obtained in the previous lemma.
		
		By \cite[tag 0FLH]{stacksproj}, $\mathrm{R}\Gamma(f^{-1}W, \M^\bullet)$ can be computed as the total complex 
		\begin{equation*}
			\mathrm{Tot}(\check{C}^\bullet(\mathfrak{U}, \M^\bullet)).
		\end{equation*}
		It now follows as above that
		\begin{equation*}
			\mathrm{Tot}(\check{C}^\bullet(\mathfrak{U}, \M^\bullet))\widetilde{\otimes}_K K\{\underline{x}\}\cong \mathrm{Tot}(\check{C}^\bullet(\mathfrak{U}, \M^\bullet\widetilde{\otimes}_K K\{\underline{x}\})).
		\end{equation*}
		By the lemma above, \cite[tag 0FLH]{stacksproj} can also be applied to $\M^\bullet \widetilde{\otimes}_K K\{\underline{x}\}$, so the right-hand side is isomorphic to $\mathrm{R}\Gamma(f^{-1}W, \M^\bullet\widetilde{\otimes}_K K\{\underline{x}\})$. 
		
		Since \cite[tag 0D5X]{stacksproj} implies that $\mathrm{H}^i(\mathrm{R}f_*(\M^\bullet))$ is the sheafification of
		\begin{equation*}
			W\mapsto \mathrm{H}^i(f^{-1}W, \M^\bullet),
		\end{equation*}
		analogously for $\M^\bullet\widetilde{\otimes}_K K\{\underline{x}\}$, this finishes the proof.
	\end{proof}
	\section{$\C$-complexes}
	\subsection{Definition and basic properties}
	In this section we define a full triangulated subcategory $\mathrm{D}_{\C}\subseteq \mathrm{D}(\w{\D}_X)$ which plays an analogous role to $\mathrm{D}^b_{\mathrm{coh}}(\D_X)$ in the algebraic theory. The construction relies on a bounded coherence condition on the level of each $\D_n$, but this time in a derived sense.
	
	We first define $\mathrm{D}_\C$ on a smooth affinoid $X=\Sp A$ where $\T(X)$ admits a smooth Lie lattice. Fixing an affine formal model $\A$ and a smooth $(R, \A)$-Lie lattice $\L$ inside $\T(X)$, we have at our disposal the sheaves $\D_n$ on the site $X_n=X(\pi^n\L)$ for $n\geq 0$. 
	
	Denote by $\mathrm{D}^b_\mathrm{coh}(\D_n)$ the full subcategory of $\mathrm{D}(\D_n)$ consisting of bounded complexes with coherent cohomology. The dual version of \cite[Theorem 13.2.8]{KS} implies that one can identify $\mathrm{D}^b_\mathrm{coh}(\D_n)$ with the bounded derived category of $\mathrm{Coh}(\D_n)$, the abelian category of coherent $\D_n$-modules.
	
	Recall from subsection 6.2  that if $\M_n$ is a complete bornological $\D_n$-module, together with connecting morphisms $\M_{n+1}|_{X_n}\to \M_n$, we can form the module $\varprojlim \M_n\in \mathrm{Mod}_{\mathrm{Shv}(X, LH(\h{\B}c_K))}(\w{\D}_X)$, a complete bornological $\w{\D}_X$-module.
	\begin{defn}
		\label{defnConaffinoid}
		An object $\M^\bullet\in \mathrm{D}(\w{\D}_X)$ is called a \textbf{$\C$-complex} if the following two properties are satisfied:
		\begin{enumerate}[(i)]
			\item for every $n$, $\D_n\widetilde{\otimes}^{\mathbb{L}}_{\w{\D}_X}\M^\bullet\in \mathrm{D}^b_{\mathrm{coh}}(\D_n)$.
			\item for every $i$, the natural morphism $\mathrm{H}^i(\M^\bullet)\to \varprojlim \mathrm{H}^i(\D_n\widetilde{\otimes}^{\mathbb{L}}_{\w{\D}_X}\M^\bullet)$ is an isomorphism.
		\end{enumerate}
		We denote the full subcategory of $\mathrm{D}(\w{\D}_X)$ consisting of $\C$-complexes by $\mathrm{D}_\C$, or by $\mathrm{D}_{\C}(\wideparen{\mathcal{D}}_X)$ if we need to stress the underlying space.
	\end{defn}
	Note that a priori this definition depends on the choices of $\A$ and $\L$. We will verify below that a different choice yields the same subcategory.

	\begin{prop}
		\label{coadcohomology}
		If $\M^\bullet\in \mathrm{D}_\C$, then $\mathrm{H}^i(\M^\bullet)\in \C_X$ for each $i$. In particular, $\M^\bullet$ is a strict complex.
		
		Moreover, 
		\begin{equation*}
			\mathrm{H}^i(\D_n\widetilde{\otimes}^\mathbb{L}\M^\bullet)\cong \D_n\widetilde{\otimes}\mathrm{H}^i(\M^\bullet)
		\end{equation*}
		in $\mathrm{Mod}_{\mathrm{Shv}(X_n, LH(\h{\B}c_K))}(\D_n)$ for each $i$ and each $n$. 
	\end{prop}
	
	\begin{proof}
		As $\D_n\widetilde{\otimes}_{\D_{n+1}}-$ is exact on coherent $\D_{n+1}$-modules, we have 
		\begin{equation*}
			\mathrm{H}^i(\D_n\widetilde{\otimes}^{\mathbb{L}}_{\w{\D}_X}\M^\bullet)\cong \D_n\widetilde{\otimes}_{\D_{n+1}}\mathrm{H}^i(\D_{n+1}\widetilde{\otimes}^{\mathbb{L}}_{\w{\D}_X}\M^\bullet),
		\end{equation*}
		so $\mathrm{H}^i(\M^\bullet)\cong \varprojlim \mathrm{H}^i(\D_n\widetilde{\otimes}^{\mathbb{L}}\M^\bullet)$ is a coadmissible $\w{\D}_X$-module by Lemma \ref{Mndefinescoad}, with
		\begin{equation*}
			\mathrm{H}^i(\D_n\widetilde{\otimes}^{\mathbb{L}}\M^\bullet)\cong \D_n\widetilde{\otimes}\mathrm{H}^i(\M^\bullet),
		\end{equation*}
		also following from Lemma \ref{Mndefinescoad}.
	\end{proof}
	\begin{prop}
		If $\M$ is an object in $\mathrm{Mod}_{\mathrm{Shv}(X, LH(\h{\B}c_K))}(\w{\D}_X)$, viewed as a complex concentrated in degree zero, then $\M\in \mathrm{D}_\C$ if and only if $\M\in \C_X$.
	\end{prop}
	\begin{proof}
		If $\M$ is coadmissible, then $\D_n\widetilde{\otimes}^{\mathbb{L}}\M\cong \D_n\widetilde{\otimes}\M$ by Corollary \ref{Anacyclic}, which is a coherent $\D_n$-module by Corollary \ref{Anacyclic}, and $\M\to \varprojlim \D_n\widetilde{\otimes}\M$ is an isomorphism by assumption. Hence $\M$ is a $\C$-complex.
		
		The other directon is clear from Proposition \ref{coadcohomology}.
	\end{proof}
	
	\begin{prop}
		\label{Ccomplextriangulated}
		The category $\mathrm{D}_\C$ is a triangulated subcategory of $\mathrm{D}(\w{\D}_X)$.
	\end{prop}
	
	\begin{proof}
		It suffices to show that for any distinguished triangle 
		\begin{equation*}
			\M_1^\bullet\to \M_2^\bullet\to \M_3^\bullet\to \M_1^\bullet[1]
		\end{equation*} in $\mathrm{D}(\w{\D}_X)$ with $\M_1^\bullet, \M_2^\bullet\in \mathrm{D}_\C$, we have $\M_3^\bullet\in \mathrm{D}_\C$.
		
		As $\D_n\widetilde{\otimes}^{\mathbb{L}}-$ preserves distinguished triangles and $\mathrm{D}^b_{\mathrm{coh}}(\D_n)$ is triangulated, we know that $\D_n\widetilde{\otimes}^{\mathbb{L}}\M_3^\bullet\in \mathrm{D}^b_{\mathrm{coh}}(\D_n)$. Moreover, taking sections on the long exact sequence of cohomology, we have strictly exact sequences
		\begin{equation*}
			C_n: \hdots \to  \mathrm{H}^{i-1}(\D_n\widetilde{\otimes}^\mathbb{L}\M_3^\bullet)(U)\to  \mathrm{H}^i(\D_n\widetilde{\otimes}^{\mathbb{L}}\M_1^\bullet)(U)\to \mathrm{H}^i(\D_n\widetilde{\otimes}^{\mathbb{L}}\M_2^\bullet)(U)\to \hdots
		\end{equation*}
		of Banach $\D_n(U)$-modules for any admissible affinoid $U$ in $X_n$ (note that the functor $\Gamma(U, -)$ is exact on coherent $\D_n$-modules). As
		\begin{equation*}
			\D_n(U)\widetilde{\otimes}_{\D_{n+1}(U)}\mathrm{H}^i(\D_{n+1}\widetilde{\otimes}^{\mathbb{L}}\M_j^\bullet)(U)\cong \mathrm{H}^i(\D_n\widetilde{\otimes}^{\mathbb{L}}\M_j^\bullet)(U)
		\end{equation*}
		for any $j=1, 2,3 $, any $i\in \mathbb{Z}$, $n\geq 0$, we see that $\varprojlim \mathrm{H}^i(\D_n\widetilde{\otimes}^\mathbb{L}\M_j^\bullet)(U)$ is a coadmissible $\w{\D}(U)$-module, so the system $(\mathrm{H}^i(\D_n\widetilde{\otimes}^\mathbb{L}\M_j^\bullet)(U))$ is pre-nuclear. Hence the Mittag-Leffler property of Theorem \ref{MLBan} is satisfied, so that $\varprojlim C_n$ is still a strictly exact complex. 
		
		Assuming that $\mathrm{H}^i(\M_j)(U)\cong \varprojlim \mathrm{H}^i(\D_n\widetilde{\otimes}^\mathbb{L}\M_j)(U)$ for $j=1, 2$ and for all $i$, the 5-Lemma then forces $\mathrm{H}^i(\M_3)(U)\cong \varprojlim \mathrm{H}^i(\D_n\widetilde{\otimes}^{\mathbb{L}}\M_3)(U)$, proving that $\M_3\in \mathrm{D}_\C$, as required.
	\end{proof}

	\begin{prop}
		\label{Ccomplexequiv}
		Let $\M^\bullet\in \mathrm{D}(\wideparen{\mathcal{D}}_X)$. The following are equivalent:
		\begin{enumerate}[(i)]
			\item $\M^\bullet$ is a $\C$-complex.
			\item $\mathrm{H}^i(\M^\bullet)\in \C_X$ for each $i$, and for each $n$
			\begin{equation*}
				\D_n\widetilde{\otimes}\mathrm{H}^i(\M^\bullet)=0
			\end{equation*}
			for all but finitely many $i$.
		\end{enumerate}
	\end{prop}
	\begin{proof}
		(i) implies (ii) is immediate from Proposition \ref{coadcohomology}. 
		
		Suppose that $\M^\bullet$ satisfies (ii). Then the dual version of Lemma \ref{exactcohom} implies that
		\begin{equation*}
			\mathrm{H}^i(\D_n\widetilde{\otimes}^\mathbb{L}\M^\bullet)\cong \D_n\widetilde{\otimes}\mathrm{H}^i(\M^\bullet)
		\end{equation*}
		for each $i$ and each $n$. Thus $\D_n\widetilde{\otimes}^\mathbb{L}\M^\bullet\in \mathrm{D}^b_{\mathrm{coh}}(\D_n)$ for each $n$. By coadmissibililty,
		\begin{equation*}
			\mathrm{H}^i(\M^\bullet)\cong \varprojlim \D_n\widetilde{\otimes}\mathrm{H}^i(\M^\bullet)\cong \varprojlim \mathrm{H}^i(\D_n\widetilde{\otimes}^\mathbb{L}\M^\bullet).
		\end{equation*}
		Hence (ii) implies (i).
	\end{proof}
	Characterizing $\C$-complexes as those complexes satisfying (ii) makes it immediate that our definition does not depend on the choices of $\A$ and $\L$. We also record the following corollaries.
	\begin{cor}
		\label{truncatedCcompl}
		If $\M^\bullet\in \mathrm{D}_\C$, then so are $\tau^{\geq n}\M^\bullet$ and $\tau^{\leq n}\M^\bullet$ for each $n$.
	\end{cor}
	
	\begin{cor}
		\label{boundedCcomplex}
		If $\M^\bullet\in \mathrm{D}(\wideparen{\mathcal{D}}_X)$ is bounded, then $\M^\bullet$ is a $\C$-complex if and only if all its cohomologies are coadmissible.
	\end{cor}
	As there are usually many coadmissible $\w{\D}_X$-modules $\M$ such that $\D_n\widetilde{\otimes}\M=0$ for some $n$, we point out that there are examples of unbounded $\C$-complexes. 
	
	Finally, we can now give the general definition of $\C$-complexes on arbitrary smooth rigid analytic $K$-varieties.
	
	\begin{defn}
		\label{defnCgeneral}
		Let $X$ be an arbitrary smooth rigid analytic $K$-variety. An object $\M^\bullet\in \mathrm{D}(\w{\D}_X)$ is called a \textbf{$\C$-complex} if there exists an admissible $\B$-covering $\mathfrak{U}=(U_i)$ of $X$ such that $\M^\bullet|_{U_i}\in \mathrm{D}_\C(\wideparen{\mathcal{D}}_{U_i})$ for each $i$.
	\end{defn}
	As before we use $\mathrm{D}_\C$ to denote the full subcategory of $\C$-complexes on $X$.
	
	By exactness of the restriction functor and the definition of coherence being local, this definition is consistent with the earlier definition in the affinoid case. In fact, if one admissible affinoid covering $\mathfrak{U}$ satisfies the condition in Definition \ref{defnCgeneral}, so does any suitable affinoid covering.
	
	Arguing locally, our previous work implies immediately the following.
	
	\begin{thm}
		Let $X$ be a smooth rigid analytic $K$-variety. 
		\begin{enumerate}[(i)]
			\item A $\C$-complex has coadmissible cohomology groups.
			\item A bounded complex $\M^\bullet\in \mathrm{D}(\w{\D}_X)$ is a $\C$-complex if and only if all its cohomologies are coadmissible.
			\item The category $\mathrm{D}_\C$ is a triangulated subcategory of $\mathrm{D}(\w{\D}_X)$, stable under $\tau^\geq$ and $\tau^\leq$.
		\end{enumerate}
	\end{thm}
	
	We remark that we can define in the same way $\C$-complexes for $\w{\mathscr{U}(\mathscr{L})}$-modules, where $\mathscr{L}$ is any Lie algebroid on a rigid analytic space $X$. 
	
	\subsection{$\C$-complexes via localization and homotopy limits}
	In this subsection we provide two further descriptions of $\C$-complexes: one in terms of derived localization, the other in terms of homotopy limits.
	
	Let $A=\varprojlim A_n$ be a Fr\'echet--Stein $K$-algebra with the property that any coadmissible $A$-module is pseudo-nuclear. If $K$ is not spherically complete, we assume additionally that $A$ is of countable type. Analogously to the previous subsection, we can make the following definition.
	
	\begin{defn}
		An object $M^\bullet\in \mathrm{D}(\mathrm{Mod}_{LH(\h{\B}c_K)}(A))$ is called a \textbf{$\C$-complex} over $A$ if 
		\begin{enumerate}[(i)]
			\item $M_n^\bullet:=A_n\widetilde{\otimes}^\mathbb{L}_A M^\bullet\in \mathrm{D}^b_{\mathrm{f.g.}}(A_n)$ for each $n$, and
			\item for each $i$, the natural morphism $\mathrm{H}^i(M^\bullet)\to \varprojlim \mathrm{H}^i(M_n^\bullet)$ is an isomorphism.
		\end{enumerate} 
	\end{defn}
	We denote by $\mathrm{D}_\C(A)$ the full subcategory consisting of $\C$-complexes over $A$.
	
	It is straightforward to obtain analogues of the results in the previous subsection for $\C$-complex over $A$, where we require pseudo-nuclearity to have the Mittag-Leffler results (Theorem \ref{MLBan}) available, and countable type for Corollary \ref{Anacyclic}.
	
	Moreover, we have the following alternative description.
	\begin{lem}
		\label{AChomotopy}
		An object $M^\bullet\in \mathrm{D}(\mathrm{Mod}_{LH(\h{\B}c_K)}(A))$ is a $\C$-complex over $A$ if and only if the following holds:
		
		For each $n$, $M_n^\bullet\in \mathrm{D}^b_{\mathrm{f.g.}}(A_n)$, and the natural map $M^\bullet\to \prod M_n^\bullet$ exhibits $M^\bullet$ as the homotopy limit of the $M_n^\bullet$ in $ \mathrm{D}(\mathrm{Mod}_{LH(\h{\B}c_K)}(A))$.
	\end{lem}
	
	\begin{proof}
		Note that $\mathrm{Mod}_{LH(\h{\B}c_K}(A))$ is a Grothendieck abelian category with exact countable products, so that we can work within the framework discussed in subsection 3.2.
		
		Let $M^\bullet\in \mathrm{D}(\mathrm{Mod}_{LH(\h{\B}c_K)}(A))$, and suppose that $M_n^\bullet\in \mathrm{D}^b_{\mathrm{f.g.}}(A_n)$ for each $n$. By Lemma \ref{homotopyproductexact}, $\prod M_n^\bullet\in \mathrm{D}(\mathrm{Mod}_{LH(\h{\B}c_K)}(A))$ exists and satisfies
		\begin{equation*}
			\mathrm{H}^i(\prod M_n^\bullet)\cong \prod \mathrm{H}^i(M_n^\bullet)
		\end{equation*}
		for each $i$. 
		By construction, we have
		\begin{equation*}
			A_{n-1}\widetilde{\otimes}_{A_n}\mathrm{H}^i(M_n^\bullet)\cong \mathrm{H}^i(M_{n-1}^\bullet),
		\end{equation*}
		so that $\varprojlim \mathrm{H}^i(M_n^\bullet)$ is a coadmissible $A$-module. In particular, since we assume that coadmissible $A$-modules are pseudo-nuclear, Lemma \ref{homotopyproductexact} and Theorem \ref{MLBan} imply that
		\begin{equation*}
			\mathrm{H}^i(\mathrm{holim}M_n^\bullet)\cong \varprojlim \mathrm{H}^i(M_n^\bullet).
		\end{equation*}
		Thus the natural morphism $M^\bullet\to \prod M_n^\bullet$ yields an isomorphism $M^\bullet\cong \mathrm{holim}M_n^\bullet$ in $\mathrm{D}(\mathrm{Mod}_{LH(\h{\B}c_K)}(A))$ if and only if it induces an isomorphism $\mathrm{H}^i(M^\bullet)\cong \varprojlim \mathrm{H}^i(M_n^\bullet)$ for each $i$.
	\end{proof}
	
	Now let $A$ be an affinoid $K$-algebra, and suppose that $X=\Sp A$ is smooth. Let $\A$ be an admissible affine formal model, $\L$ an $(R, \A)$-Lie lattice in $\mathrm{Der}_K(A)$ which is assumed to be a free $\A$-module, and let $\D_n$ be the sheaf on $X_n:=X(\pi^n\L)$ given by 
	\begin{equation*}
		\D_n(U)=\h{U_\B(\B\otimes_\A \pi^n\L)}_K
	\end{equation*}
	for $\B$ a suitable affine formal model of $U=\Sp B\in X_n$.
	
	Recall that the functor $\Gamma(X, -)$ induces an equivalence of categories between coadmissible $\w{\D}_X$-modules and coadmissible $\w{\D}_X(X)$-modules, see \cite[Corollary 1.4]{DcapOne} (in the discretely valued case) and \cite[Theorem 3.6.11]{Ardakov} (with $G=1$, for the general case). The quasi-inverse is the localisation functor $\mathrm{Loc}$, which we can now define more generally as follows:
	
	\begin{align*}
		\mathrm{Loc}: \mathrm{Mod}_{LH(\h{\B}c_K)}(\w{\D}_X(X))&\to \mathrm{Mod}_{\mathrm{Shv}(X, LH(\h{\B}c_K))}(\w{\D}_X)\\
		M&\mapsto \w{\D}_X\widetilde{\otimes}_{\w{\D}_X(X)}M, 
	\end{align*}
	i.e. $\mathrm{Loc}M$ is the sheafification of 
	\begin{equation*}
		U\mapsto \w{\D}_X(U)\widetilde{\otimes}_{\w{\D}_X(X)}M.
	\end{equation*}
	It is immediate that $\mathrm{Loc}$ is the left adjoint of $\Gamma(X, -)$, and we have an adjoint pair of functors $(\mathbb{L}\mathrm{Loc}, \mathrm{R}\Gamma(X, -))$ between the unbounded derived categories.
	
	In the same way, we have localisation functors from $\D_n(X)$-modules to $\D_n$-modules, which we likewise denote by $\mathrm{Loc}$.
	
	We will now show that the localisation equivalence for coadmissible modules extends naturally to an equivalence for $\C$-complexes.

	\begin{thm}
		\label{affinoidloc}
		The functors $\mathrm{R}\Gamma(X, -)$ and $\mathbb{L}\mathrm{Loc}$ yield an equivalence between $\mathrm{D}_\C(\w{\D}_X)$ and $\mathrm{D}_\C(\w{\D}_X(X))$.
	\end{thm}
	
	First note that for bounded $\C$-complexes, this is actually immediate, as coadmissible modules are acyclic for both functors, so that we can reduce to the corresponding abelian statement. Our task is thus to carefully check the behaviour of unbounded complexes.

	\begin{lem}
		\label{tensorloc}
		Let $M^\bullet$ be a $\C$-complex of $\w{\D}_X(X)$-modules. Then $\mathbb{L}\mathrm{Loc}(M^\bullet)\in \mathrm{D}_\C(\w{\D}_X)$, and
		\begin{equation*}
			\mathrm{H}^i(\mathbb{L}\mathrm{Loc}(M^\bullet))\cong \mathrm{Loc}(\mathrm{H}^i(M^\bullet))
		\end{equation*}
		for each $i$.
	\end{lem}
	\begin{proof}
		The description of cohomology is immediate for bounded above complexes, as coadmissible modules are $\mathrm{Loc}$-acyclic. As direct sums are exact in the category $\mathrm{Mod}_{LH(\h{\B}c_K)}(\w{\D}_X(X))$, we have the same isomorphism in the general case by writing $M^\bullet$ as a homotopy colimit of its truncations, i.e. the dual version of Lemma \ref{exactcohom}.
		
		We deduce that $\mathrm{H}^i(\mathbb{L}\mathrm{Loc}M^\bullet)$ is a coadmissible $\w{\D}_X$-module, and for each $n$, 
		\begin{equation*}
			\D_n\widetilde{\otimes}_{\w{\D}_X} \mathrm{H}^i(\mathbb{L}\mathrm{Loc}M^\bullet)\cong \mathrm{Loc}(\D_n(X)\widetilde{\otimes}_{\w{\D}_X(X)} \mathrm{H}^i(M^\bullet)),
		\end{equation*}
		which is zero for almost all $i$ by assumption. Thus $\mathbb{L}\mathrm{Loc}(M^\bullet)$ is a $\C$-complex by Proposition \ref{Ccomplexequiv}.
	\end{proof}
	
	\begin{lem}
		The functors $\mathrm{R}\Gamma(X, -)$ and $\mathbb{L}\mathrm{Loc}$ yield an equivalence between bounded below $\C$-complexes of $\w{\D}_X$-modules and bounded below $\C$-complexes of $\w{\D}_X(X)$-modules.
	\end{lem}
	
	\begin{proof}
		Let $\M^\bullet\in \mathrm{D}_\C(\w{\D}_X)$ be bounded below. Since coadmissible $\w{\D}_X$-modules are $\Gamma(X, -)$-acyclic, we see immediately that
		\begin{equation*}
			\mathrm{H}^i(\mathrm{R}\Gamma(X, \M^\bullet))=\Gamma(X, \mathrm{H}^i(\M^\bullet))
		\end{equation*}
		for all $i$.
		
		Moreover, for any coadmissible $\w{\D}_X$-module $\M$, we have 
		\begin{equation*}
			\D_n(X)\widetilde{\otimes}_{\w{\D}_X(X)} \M(X)\cong (\D_n\widetilde{\otimes}_{\w{\D}_X}\M)(X).
		\end{equation*}
		This makes $\mathrm{R}\Gamma(X, \M^\bullet)$ a bounded below $\C$-complex of $\w{\D}_X(X)$-modules.
		
		The lemma above then shows that the counit morphism $\mathbb{L}\mathrm{Loc}\ \mathrm{R}\Gamma(X, \M^\bullet)\to \M^\bullet$ is a quasi-isomorphism.
		
		Conversely, if $M^\bullet$ is a bounded below $\C$-complex of $\w{\D}_X(X)$-modules, the lemma above shows that $\mathbb{L}\mathrm{Loc}M^\bullet$ is a bounded below $\C$-complex, and the unit morphism $M^\bullet\to \mathrm{R}\Gamma(X, \mathbb{L}\mathrm{Loc}M^\bullet)$ is again a quasi-isomorphism due to the description of the cohomology groups given above. 
	\end{proof}
	To extend from bounded below $\C$-complex to arbitrary $\C$-complexes, we wish to invoke Lemma \ref{holimtruncation}.
	
	Let $\M^\bullet$ be a $\C$-complex of $\w{\D}_X$-modules. Let $\tau^{\geq -n}\M^\bullet \to \mathcal{I}_n^\bullet$ be a quasi-isomorphism with $\mathcal{I}_n^j$ an injective object in $\mathrm{Mod}_{\mathrm{Shv}(X, LH(\h{\B}c_K))}(\w{\D}_X)$ for all $j$, satisfying the same conditions as in \cite[Lemma 070F]{stacksproj}, also given before Lemma \ref{holimtruncation}. Then by \cite[Lemma 070M]{stacksproj} and its proof, $\mathcal{I}^\bullet:=\varprojlim \mathcal{I}_n^\bullet$ exists in $\mathrm{D}(\w{\D}_X)$ and is a homotopy limit of $\tau^{\geq -n}\M^\bullet$.
	\begin{lem}
		\label{truncatedCcomplex}
		The natural morphism $\M^\bullet \to \mathcal{I}^\bullet$ is a quasi-isomorphism. In particular, $\M^\bullet \cong \mathrm{holim} \tau^{\geq -n}\M^\bullet$.
	\end{lem}
	
	\begin{proof}
		Let $U$ be an affinoid subdomain of $X$. As $\mathrm{R}\Gamma(U, -)$ preserves homotopy limits (by \cite[Lemma 0D60]{stacksproj}), we have
		\begin{equation*}
			\mathrm{H}^i(\mathrm{R}\Gamma(U, \mathcal{I}^\bullet))\cong \mathrm{H}^i(\mathrm{holim} \ \mathrm{R}\Gamma(U, \tau^{\geq -n}\M^\bullet)).
		\end{equation*}
		Recall from Corollary \ref{truncatedCcompl} that $\tau^{\geq -n}\M^\bullet$ is a bounded below $\C$-complex for all $n$. Thus for any fixed $i$, 
		\begin{equation*}
			\mathrm{H}^i(\mathrm{R}\Gamma(U, \tau^{\geq -n}\M^\bullet))\cong \Gamma(U, \mathrm{H}^i(\tau^{\geq -n}\M^\bullet))\cong\Gamma(U, \mathrm{H}^i(\M^\bullet))
		\end{equation*}
		for all sufficiently large $n$, so that
		\begin{equation*}
			\mathrm{R}^1\varprojlim \mathrm{H}^i(\mathrm{R}\Gamma(U, \tau^{\geq -n}\M^\bullet))=0.
		\end{equation*}
		Hence Lemma \ref{homotopyproductexact} implies that 
		\begin{equation*}
			\mathrm{H}^i(\mathrm{R}\Gamma(U, \mathcal{I}^\bullet))\cong \varprojlim \mathrm{H}^i(\mathrm{R}\Gamma(U, \tau^{\geq -n}\M^\bullet))\cong\Gamma(U, \mathrm{H}^i(\M^\bullet)).
		\end{equation*}
		
		By \cite[Lemma 0BKJ]{stacksproj}, $\mathrm{H}^i(\mathcal{I}^\bullet)$ is thus the sheafification of
		\begin{equation*}
			U\mapsto \Gamma(U, \mathrm{H}^i(\M^\bullet)),
		\end{equation*} 
		which of course requires no sheafification, and the natural morphism $\M^\bullet\to \mathcal{I}^\bullet$ is a quasi-isomorphism.
		
		Therefore $\M^\bullet\cong \mathrm{holim}\tau^{\geq -n}\M^\bullet$ by Lemma \ref{holimtruncation}.
	\end{proof}
	
	\begin{proof}[Proof of Theorem \ref{affinoidloc}]
		Let $\M^\bullet\in \mathrm{D}_\C(\w{\D}_X)$. As $\mathrm{R}\Gamma(X, -)$ respects homotopy limits, it follows from the bounded below case and Lemma \ref{truncatedCcomplex} that $\mathrm{H}^i(\mathrm{R}\Gamma(X, \M^\bullet))=\Gamma(X, \mathrm{H}^i(\M^\bullet))$. In particular, $\mathrm{R}\Gamma(X, \M^\bullet)$ is a $\C$-complex. Moreover, Lemma \ref{tensorloc} yields $\mathrm{H}^i(\mathbb{L}\mathrm{Loc}(M^\bullet))=\mathrm{Loc}(\mathrm{H}^i(M^\bullet))$ for any $M^\bullet\in \mathrm{D}_\C(\w{\D}_X(X))$. Thus the unit and co-unit morphisms yield quasi-isomorphisms, since $\mathrm{Loc}$ and $\Gamma(X, -)$ are quasi-inverses on coadmissible modules.
	\end{proof}
	
	We conclude this subsection by proving the analogue of Lemma \ref{AChomotopy} for $\C$-complexes of $\w{\D}_X$-modules. We first record the following lemma.
	
	\begin{lem}
		Let $\M^\bullet\in \mathrm{D}(\w{\D}_X)$, and assume that $\M_n^\bullet=\D_n\widetilde{\otimes}^\mathbb{L}_{\w{\D}_X}\M^\bullet\in \mathrm{D}^b_\mathrm{coh}(\D_n)$ for each $n$. Then for each $m$, 
		\begin{equation*}
			\mathrm{H}^i(\prod_{n\geq m} \M_n^\bullet)=\prod_{n\geq m} \mathrm{H}^i(\M_n^\bullet)
		\end{equation*}
		as $\w{\D}_X$-modules on $X_m$.
	\end{lem}
	
	\begin{proof}
		For any affinoid subdomain $U\in X_m$, we have 
		\begin{align*}
			\mathrm{H}^i(\mathrm{R}\Gamma(U, \prod_{n\geq m}\M_n^\bullet))&\cong \mathrm{H}^i(\prod_{n\geq m} \mathrm{R}\Gamma(U, \M_n^\bullet))\\
			&\cong \prod_{n\geq m} \mathrm{H}^i(\mathrm{R}\Gamma(U, \M_n^\bullet))\\
			&\cong \prod \Gamma(U, \mathrm{H}^i(\M_n^\bullet)),
		\end{align*}
		since products in $LH(\h{\B}c_K)$ are exact and coherent $\D_n$-modules are acyclic on affinoids. \\
		Finally by \cite[Lemma 0BKJ]{stacksproj}, $\mathrm{H}^i(\prod \M_n^\bullet)$ is the sheafification of 
		\begin{equation*}
			U\mapsto \mathrm{H}^i(\mathrm{R}\Gamma(U, \prod \M_n^\bullet))\cong \prod \mathrm{H}^i(\M_n^\bullet)(U).
		\end{equation*}
		But this already satisfies the sheaf condition, and
		\begin{equation*}
			\mathrm{H}^i(\prod \M_n^\bullet)\cong \prod \mathrm{H}^i(\M_n^\bullet). \qedhere
		\end{equation*}
	\end{proof}
	
	\begin{cor}
		\label{Ccomplexholim}
		An object $\M^\bullet\in \mathrm{D}(\w{\D}_X)$ is a $\C$-complex if and only if the following holds:
		
		For each $n$, $\M_n^\bullet\in \mathrm{D}^b_{\mathrm{coh}}(\D_n)$, and the natural map $\M^\bullet|_{X_m}\to \prod_{n\geq m} \M_n^\bullet|_{X_m}$ exhibits $\M^\bullet|_{X_m}$ as the homotopy limit of the $\M_n^\bullet|_{X_m}$ in $\mathrm{D}(\w{\D}_X|_{X_m})$ for all $m$.	
	\end{cor}
	\begin{proof}	
		If $\M$ is a coadmissible $\w{\D}_X$-module, then we have a (strict) short exact sequence 
		\begin{equation*}
			0\to \M(U)\to \prod_{n\geq m} \M_n(U)\to \prod_{n\geq m} \M_n(U)\to 0
		\end{equation*}
		for any $U\in X_m$, thanks to Corollary \ref{MLcoad}. Thus $\M|_{X_m}=\varprojlim \M_n|_{X_m}\cong \mathrm{R}\varprojlim \M_n|_{X_m}$. In particular, $\mathrm{R}^1 \varprojlim \M_n|_{X_m}=0$.
		
		Now let $\M^\bullet\in \mathrm{D}(\w{\D}_X)$, and suppose that $\M_n^\bullet\in \mathrm{D}^b_{\mathrm{coh}}(\D_n)$ for each $n$. We fix an $m$ and work throughout on $X_m$, omitting the restriction symbols. The defining triangle 
		\begin{equation*}
			\mathrm{holim} \M_n^\bullet \to \prod_{n\geq m} \M_n^\bullet\to \prod_{n\geq m} \M_n^\bullet
		\end{equation*}
		induces a long exact sequence of cohomology, which by the lemma above reads
		\begin{equation*}
			\hdots \to \mathrm{H}^i(\mathrm{holim}\M_n^\bullet)\to \prod \mathrm{H}^i(\M_n^\bullet)\to \prod \mathrm{H}^i(\M_n^\bullet)\to\dots,
		\end{equation*}
		yielding short exact sequences
		\begin{equation*}
			0\to \mathrm{R}^1\varprojlim \mathrm{H}^{i-1}(\M_n^\bullet)\to \mathrm{H}^i(\mathrm{holim}\M_n^\bullet)\to \varprojlim \mathrm{H}^i(\M_n^\bullet)\to 0.
		\end{equation*}
		
		Since $\varprojlim \mathrm{H}^{i-1}(\M_n^\bullet)$ is coadmissible, we have $\mathrm{R}^1\varprojlim \mathrm{H}^{i-1}(\M_n^\bullet)=0$ by the above, and
		\begin{equation*}
			\mathrm{H}^i(\mathrm{holim} \M_n^\bullet)\cong \varprojlim \mathrm{H}^i(\M_n^\bullet).
		\end{equation*}
		Thus we obtain $\M^\bullet\cong\mathrm{holim} \M_n^\bullet$ if and only if the natural morphism $\mathrm{H}^i(\M^\bullet)\to \varprojlim \mathrm{H}^i(\M_n^\bullet)$ is an isomorphism for each $i$. 
	\end{proof}
	
	We remark that the restriction to $X_m$ for various $m$ is only necessary as $\D_n$ is only defined on the subsite $X_n$. Alternatively, recall from subsection 6.2 the right $\w{\D}_X$-modules $\F_n=\D_n(X)\widetilde{\otimes}_A \O_X$, satisfying $\F_n|_{X_n}\cong\D_n$. It follows immediately that $\C$-complexes can also be characterized as those $\M^\bullet\in \mathrm{D}(\w{\D}_X)$ such that $\M_n^\bullet\in \mathrm{D}^b_\mathrm{coh}(\D_n)$ for each $n$, and $\M^\bullet\cong \mathrm{holim} \F_n\widetilde{\otimes}^\mathbb{L}_{\w{\D}_X} \M^\bullet$ in $\mathrm{D}(\mathrm{Shv}(X, LH(\h{\B}c_K)))$ via the natural morphism, as this description is valid on $X_m$ for each $m$ by the above.

	\section{Properties of $\C$-complexes}
	\subsection{Kashiwara's equivalence}
	Let $Y$ be a smooth rigid analytic $K$-variety and let $i: X\to Y$ be a closed subvariety. We say that a sheaf $\F\in \mathrm{Shv}(Y, LH(\h{\B}c_K))$ is \textbf{supported on $X$} if $\F(U)=0$ for any $U$ satisfying $U\cap X=\emptyset$. 
	\begin{prop}
		\label{supportequivalence}
		The functors $i_*$ and $i^{-1}$ establish an equivalence of categories between $\mathrm{Shv}(X, LH(\h{\B}c_K))$ and the full subcategory of $\mathrm{Shv}(Y, LH(\h{\B}c_K))$ consisting of sheaves which are supported on $X$.
	\end{prop}
	\begin{proof}
		This is the same argument as in \cite[Appendix A]{DcapTwo}.
	\end{proof}
	For any smooth rigid analytic $K$-variety $X$, denote the category of coadmissible $\w{\D}_X$-modules by $\C_X$, and let $\C_X^r$ be the category of coadmissible right $\w{\D}_X$-modules.
	
	In \cite{DcapTwo}, Ardakov--Wadsley proved the following analogue of Kashiwara's equivalence (with the condition of discrete valuation being removed in \cite{Ardakovinduction}):
	\begin{thm}[{\cite[Theorem A]{DcapTwo}, \cite[Theorem B]{Ardakovinduction}}]
		\label{KashiwaraDcapTwo}
		Let $i:X\to Y$ be a closed embedding of smooth rigid analytic $K$-varieties, defined by the vanishing of the ideal sheaf $\I$. The functor 
		\begin{align*}
			i_+^0:&\C_X^r\to \C_Y^r\\
			& \M\mapsto i_*(\M\w{\otimes}_{\w{\D}_X}i^*\w{\D}_Y)
		\end{align*}
		gives an equivalence of categories between coadmissible right $\w{\D}_X$-modules and coadmissible right $\w{\D}_Y$-modules supported on $X$.
		
		The quasi-inverse is given by $i^{\natural}: \M\mapsto i^{-1}(\M[\I])$. 
	\end{thm}
	Here, $\M[\I]$ denotes the subsheaf of sections annihilated by $\I$. Of course, side-changing implies the corresponding result for left modules.
	
	In this subsection, we generalize Kashiwara's equivalence to our derived setting. We say that $\M^\bullet\in \mathrm{D}_\C(\w{\D}_Y)$ is a \textbf{$\C^X$-complex} if each cohomology is supported on $X$ and write $\mathrm{D}_{\C^X}(\w{\D}_Y)$ for the full subcategory of $\C^X$-complexes. 
	
	We will first concentrate on the following local picture: Suppose that $X$ and $Y$ are affinoid, $Y$ admits local coordinates $(y_1, \hdots, y_m, \partial_1, \hdots, \partial_m)$, and $X=V(y_m)$ is defined by the vanishing of $y_m$. Note that in this case
	\begin{equation*}
		i^*\w{\D}_Y\cong \w{\D}_X\widetilde{\otimes}_K K\{\partial_m\}
	\end{equation*}
	as left $\w{\D}_X$-modules, and
	\begin{equation*}
		i^\natural(\M)=\mathrm{ker}(y_m: i^{-1}\M\to i^{-1}\M)\cong\mathrm{H}^0(i_r^!\M)
	\end{equation*}
	for any coadmissible right $\w{\D}_Y$-module $\M$ supported on $X$, where $i^!_r$ is the extraordinary inverse image for right $\w{\D}$-modules obtained via side-changing.
	
	We choose compatible affine formal models for $X$ and $Y$ and choose Lie lattices such that $\D_{X_n}$ is generated by $\pi^n\partial_i$, $i=1, \hdots, m-1$, and $\D_{Y_n}$ is generated by $\pi^n\partial_i$, $i=1, \hdots, m$.  
	\begin{lem}
		\label{Kashiwaratermwise}
		\leavevmode
		\begin{enumerate}[(i)]
			\item Let $\M^\bullet\in \mathrm{D}_\C(\w{\D}_X^\mathrm{op})$. Then 
			\begin{equation*}
				\mathrm{H}^j(i^r_+\M^\bullet)\cong i_+^0(\mathrm{H}^j(\M^\bullet)).
			\end{equation*}
			In particular, $i_+^r\M^\bullet$ is a $\C^X$-complex of right $\w{\D}_Y$-modules.
			\item Let $\N$ be a coadmissible right $\w{\D}_Y$-module supported on $X$. Then $\mathrm{H}^j(i_r^!\N)=0$ whenever $j\neq 0$. In particular, if $\N^\bullet$ is a $\C^X$-complex on $Y$, then
			\begin{equation*}
				\mathrm{H}^j(i_r^!\N^\bullet)\cong i^\natural(\mathrm{H}^j(\N^\bullet)),
			\end{equation*}
			and $i_r^!\N^\bullet$ is a $\C$-complex of right $\w{\D}_X$-modules.
		\end{enumerate}
	\end{lem}
	\begin{proof}
		\begin{enumerate}[(i)]
			\item By the description above, $i^*\w{\D}_Y$ is a strongly flat $\w{\D}_X$-module, and $i_*$ is exact. Hence
			\begin{align*}
				\mathrm{H}^j(\mathrm{R}i_*(\M^\bullet\widetilde{\otimes}^\mathbb{L}_{\w{\D}_X}i^*\w{\D}_Y))&\cong i_*(\mathrm{H}^j(\M^\bullet\widetilde{\otimes}_{\w{\D}_X}i^*\w{\D}_Y))\\
				&\cong i_*(\mathrm{H}^j(\M^\bullet)\widetilde{\otimes}_{\w{\D}_X}i^*\w{\D}_Y)\\
				&\cong i_+^0(\mathrm{H}^j(\M^\bullet))
			\end{align*} 
			for any $\C$-complex $\M^\bullet$, where the last isomorphism is a consequence of Proposition \ref{coadandtor}. Note that in particular, each cohomology group of $i^r_+\M^\bullet$ is coadmissible and supported on $X$ by Ardakov--Wadsley's version of Kashiwara's equivalence.
			
			Since
			\begin{equation*}
				i_+^0(\M)\widetilde{\otimes}_{\w{\D}_Y} \D_{Y_n}
			\end{equation*}
			is the coherent $\D_{Y_n}$-module associated to the global sections
			\begin{equation*}
				\M(X)\widetilde{\otimes}_{\w{\D}_X(X)}\frac{\w{\D}_Y(Y)}{y_m\w{\D}_Y(Y)}\widetilde{\otimes}_{\w{\D}_Y(Y)}\D_{Y_n}(Y)\cong \M(X)\widetilde{\otimes}_{\w{\D}_X(X)}\frac{{\D}_{Y_n}(Y)}{y_m\D_{Y_n}(Y)},
			\end{equation*}
			it follows that 
			\begin{equation*}
				i_+^0(\M)\widetilde{\otimes}_{\w{\D}_Y}\D_{Y_n}\cong i_*(\M\widetilde{\otimes}_{\w{\D}_X}i^*\D_{Y_n})\cong i_*(\M\widetilde{\otimes}_{\w{\D}_X}\D_{X_n}\widetilde{\otimes}_{\D_{X_n}}i^*\D_{Y_n})
			\end{equation*}
			for any coadmissible right $\w{\D}_X$-module $\M$. Thus if $\M^\bullet$ is a $\C$-complex, then $i^r_+\M^\bullet$ is a $\C$-complex by Proposition \ref{Ccomplexequiv}, since $\M^\bullet\widetilde{\otimes}\D_{X_n}$ is bounded for each $n$.
			\item Let $\N$ be a coadmissible right $\w{\D}_Y$-module supported on $X$. We wish to show that multiplication by $y_m$ gives a strict epimorphism $i^{-1}\N\to i^{-1}\N$.
			
			Write $N_n:=(\N\widetilde{\otimes}_{\w{\D}_Y}\D_{Y_n})(Y)$, a finitely generated right $\D_n(Y)$-module. We have a commutative diagram of right $\O(Y)$-modules
			\begin{equation*}
				\begin{xy}
					\xymatrix{
						\D_n(Y)^r\ar[d]\ar[r]^{\cdot y_m}& \D_n(Y)^r\ar[d]\\
						N_n\ar[r]^{\cdot y_m}& N_n
					}
				\end{xy}
			\end{equation*}
			where the vertical maps correspond to the choice of $r$ generators of $N_n$ which are all annihilated by $y_m$ (which is possible by \cite[Theorem 4.9]{DcapTwo}, which also works for non-discretely valued $K$). 
			
			We will now verify the conditions in Lemma \ref{MLforimages}. Let $x\in N_n$, and let $(P_i)\in \D_n(Y)^r$ be a preimage. Let $D'_n$ be the Banach subalgebra of $\D_n(Y)$ generated by $\O(Y)$ and $\partial_1, \hdots, \partial_{m-1}$, and note that $y_m$ commutes with every element in $D'_n$. Writing
			\begin{equation*}
				P_i=\sum_{j=0}^\infty Q_{ij}\pi^{nj}\partial_m^j, \ Q_{ij}\in D'_n, \ |Q_{ij}|\to 0\ \text{ as } j\to \infty,
			\end{equation*}   
			we use the relation
			\begin{equation*}
				\partial_m^j=\frac{1}{j+1}(\partial_m^{j+1}y_m-y_m\partial_m^{j+1})
			\end{equation*} 
			to write
			\begin{equation*}
				P_i=\left(\sum_{j=0}^\infty \frac{1}{j+1}Q_{ij}\pi^{nj}\partial_m^{j+1}\right)y_m-y_m\left(\sum_{j=0}^\infty \frac{1}{j+1}Q_{ij}\pi^{nj}\partial_m^{+1}j\right),
			\end{equation*}
			whenever the expressions in brackets actually converge. Since $|\frac{p^j}{j+1}|\leq 1$ for each $j$, there exists an $l$ such that for any $n\geq l$, the expressions converge in $\D_{n-l}(Y)$ -- simply choose $l$ such that $|\pi^l|\leq |p|$. Thus the image of $P_i$ in $\D_{n-l}(Y)$ is contained in $y_m\D_{n-l}(Y)+\D_{n-l}(Y)y_m$. As the distinguished generators of $N_n$ (and compatibly, of $N_{n-l}$) are annihilated by $y_m$, it follows that the image of $x$ in $N_{n-l}$ is contained in $N_{n-l}y_m$.
			
			Since the kernels of the `multiplication by $y_m$' map form a pre-nuclear system (the system describes the coadmissible $\w{\D}_X(X)$-module $i^\natural(\N)(X)$, by \cite[Corollary 5.7]{DcapTwo} resp. \cite[Theorem 3.3.13]{Ardakovinduction} with $G=1$), we can apply Lemma \ref{MLforimages} to obtain that $y_m: \N(Y)\to \N(Y)$ is a strict epimorphism (both in $\h{\B}c_K$ and in $\mathrm{Ind}(\mathrm{Ban}_K)$), similarly for all other affinoid subdomains. We thus obtain a strict exact sequence
			\begin{equation*}
				0\to i^{\natural}(\N)\to i^{-1}\N\to i^{-1}\N\to 0,
			\end{equation*}  
			in $\mathrm{Shv}(X, \mathrm{Ind}(\mathrm{Ban}_K))$as required.
			
			As $i^!_r$ has bounded cohomogical dimension on $\C^X$-complexes, we can invoke Proposition \ref{exactcohom}.(ii) to see that 
			\begin{equation*}
				\mathrm{H}^j(i_r^!\N^\bullet)\cong i^{\natural}\mathrm{H}^j(\N^\bullet)
			\end{equation*}
			for any $j$ and any $\C^X$-complex $\N^\bullet$ of right $\w{\D}_Y$-modules.
			
			It remains to verify that $i_r^!$ sends $\C^X$-complexes to $\C$-complexes, i.e. that for every given $n$, $\mathrm{H}^j_n(i_r^!\N^\bullet):=\mathrm{H}^j(i_r^!\N^\bullet)\widetilde{\otimes}_{\w{\D}_X}\D_{X_n}$ is zero for almost all $j$. For this, it suffices to note that
			\begin{equation*}
				i_*(\mathrm{H}^j_n(i_r^!\N^\bullet)\widetilde{\otimes}_{\D_{X_n}}i^*\D_{Y_n})\cong i_+^0(\mathrm{H}^j(i_r^!\N^\bullet))\widetilde{\otimes}_{\w{\D}_Y}\D_{Y_n},
			\end{equation*}
			and
			\begin{equation*}
				i_+^0(\mathrm{H}^j(i_r^!\N^\bullet))\cong i_+^0i^\natural(\mathrm{H}^j(\N^\bullet))\cong \mathrm{H}^j(\N^\bullet)
			\end{equation*}
			by the above and Theorem \ref{KashiwaraDcapTwo}. Therefore by \cite[Theorem 4.10]{DcapTwo} resp. \cite[Theorem 3.2.1]{Ardakovinduction}, $\mathrm{H}^j_n(\N^\bullet)$ vanishes whenever $\mathrm{H}^j(\N^\bullet)\widetilde{\otimes}_{\w{\D}_Y}\D_{Y_n}$ vanishes, and hence $i_r^!\N^\bullet$ is a $\C$-complex.
		\end{enumerate}
	\end{proof}
	
	In particular, if $i:X\to Y$ is an arbitrary closed embedding of smooth rigid analytic $K$-varieties, $i^!$ sends $\C^X$-complexes to $\C$-complexes, and $i_+$ sends $\C$-complexes to $\C^X$-complexes, since we can argue locally and use Proposition \ref{invimcomposition} to reduce to the case considered above.\\
	As before, we can obtain the corresponding result for left modules by using side-changing.
	\begin{lem}
		\label{invashom}
		Let $i: X\to Y$ be a closed embedding of smooth rigid analytic $K$-varieties. For any $\N^\bullet\in \mathrm{D}_{\C^X}(\w{\D}_Y)$, there is a natural isomorphism
		\begin{equation*}
			i^!\N^\bullet\cong \mathrm{R}\mathcal{H}om_{i^{-1}\w{\D}_Y}(\w{\D}_{Y\leftarrow X}, i^{-1}\N^\bullet).
		\end{equation*}
	\end{lem}
	\begin{proof}
		This follows \cite[Proposition 1.5.14]{Hotta}.\\
		First, we claim that the natural morphism
		\begin{equation*}
			i^{-1}\M^\bullet\widetilde{\otimes}^\mathbb{L}_{i^{-1}\w{\D}_Y}\mathrm{R}\mathcal{H}om_{i^{-1}\w{\D}_Y}(\w{\D}_{X\to Y}, i^{-1}\w{\D}_Y)\to \mathrm{R}\mathcal{H}om_{i^{-1}\w{\D}_Y}(\w{\D}_{X\to Y}, i^{-1}\M^\bullet)
		\end{equation*}
		is an isomorphism for any right $\C^X$-complex $\M^\bullet$. For this, note that tensor-hom adjunction turns the left hand side into
		\begin{equation*}
			i^{-1}\M^\bullet\widetilde{\otimes}^\mathbb{L}_{i^{-1}\w{\D}_Y}\mathrm{R}\mathcal{H}om_{i^{-1}\O_Y}(\O_X, i^{-1}\w{\D}_Y),
		\end{equation*} 
		and the right hand side into
		\begin{equation*}
			\mathrm{R}\mathcal{H}om_{i^{-1}\O_Y}(\O_X, i^{-1}\M^\bullet).
		\end{equation*}
		Locally, the Koszul resolution yields a resolution $\mathcal{K}^\bullet$ of $\O_X$ by locally free $i^{-1}\O_Y$-modules of finite rank. So the claimed isomorphism follows from the natural isomorphism
		\begin{align*}
			i^{-1}\M^\bullet\widetilde{\otimes}^\mathbb{L}_{i^{-1}\w{\D}_Y}\mathrm{R}\mathcal{H}om_{i^{-1}\O_Y}(i^{-1}\O_Y, i^{-1}\w{\D}_Y)&\cong i^{-1}\M^\bullet\\&\cong \mathrm{R}\mathcal{H}om_{i^{-1}\O_Y}(i^{-1}\O_Y, i^{-1}\M^\bullet).
		\end{align*}
		We thus wish to show that there is an isomorphism
		\begin{equation*}
			\mathrm{R}\mathcal{H}om_{i^{-1}\w{\D}_Y}(\w{\D}_{X\to Y}, i^{-1}\w{\D}_Y)\cong \w{\D}_{Y\leftarrow X}[-d],
		\end{equation*}
		where $d$ denotes the codimension of $X$ in $Y$, as the result will then follow from side-changing.
		
		By tensor-hom adjunction and the same arguments as above, the left hand side is isomorphic to
		\begin{equation*}
			\mathrm{R}\mathcal{H}om_{i^{-1}\O_Y}(\O_X, i^{-1}\O_Y)\widetilde{\otimes}^\mathbb{L}_{i^{-1}\O_Y} i^{-1}\w{\D}_Y,
		\end{equation*}
		and using the Koszul resolution again as in \cite[Proposition 1.5.14]{Hotta} yields the desired isomorphism
		\begin{equation*}
			\mathrm{R}\mathcal{H}om_{i^{-1}\O_Y}(\O_X, i^{-1}\O_Y)\cong i^{-1}\Omega_Y^{\otimes-1}\widetilde{\otimes}_{i^{-1}\O_Y} \Omega_X[-d].\qedhere
		\end{equation*}
	\end{proof}

	We now lift the adjunction from \cite[Theorem 6.10.(b)]{DcapTwo} to the derived level.
	\begin{prop}
		Let $i: X\to Y$ be a closed embedding of smooth rigid analytic $K$-varieties. Then there is a natural isomorphism
		\begin{equation*}
			\mathrm{RHom}_{\w{\D}_Y}(i_+\M^\bullet, \N^\bullet)\cong \mathrm{RHom}_{\w{\D}_X}(\M^\bullet, i^!\N^\bullet)
		\end{equation*}
		for $\M^\bullet\in \mathrm{D}_{\C}(\w{\D}_X)$, $\N^\bullet\in \mathrm{D}_{\C^X}(\w{\D}_Y)$.\\
		In particular, $i^!$ and $i_+$ are adjoint functors between $\mathrm{D}_{\C^X}(\w{\D}_Y)$ and $\mathrm{D}_\C(\w{\D}_X)$.
	\end{prop}
	\begin{proof}
		There are natural isomorphisms
		\begin{align*}
			\mathrm{RHom}_{\w{\D}_{Y}}(i_+\M^\bullet, \N^\bullet)&\cong \mathrm{RHom}_{\w{\D}_Y}(i_*(\w{\D}_{Y\leftarrow X}\widetilde{\otimes}^\mathbb{L}_{\w{\D}_X}\M^\bullet), i_*i^{-1}\N^\bullet)\\
			&\cong \mathrm{RHom}_{i^{-1}\w{\D}_Y}(i^{-1}i_*(\w{\D}_{Y\leftarrow X}\widetilde{\otimes}^\mathbb{L}_{\w{\D}_X}\M^\bullet), i^{-1}\N^\bullet)\\
			&\cong \mathrm{RHom}_{i^{-1}\w{\D}_Y}(\w{\D}_{Y\leftarrow X}\widetilde{\otimes}^\mathbb{L}_{\w{\D}_X}\M^\bullet, i^{-1}\N^\bullet)\\
			&\cong \mathrm{RHom}_{\w{\D}_X}(\M^\bullet, \mathrm{R}\mathcal{H}om_{i^{-1}\w{\D}_Y}(\w{\D}_{Y\leftarrow X}, i^{-1}\N^\bullet)),
		\end{align*}
		where the first and third isomorphism follow from Proposition \ref{supportequivalence}.
		
		Thus the result follows from Lemma \ref{invashom}.
	\end{proof}
	
	\begin{thm}
		\label{KashiwaraCcomplex}
		The functor $i_+$ provides an equivalence between $\mathrm{D}_\C(\w{\D}_X)$ and $\mathrm{D}_{\C^X}(\w{\D}_Y)$, with quasi-inverse $i^!$.
	\end{thm}
	\begin{proof}
		The unit and counit of the adjunction yield the desired natural transformations, inducing natural morphisms between the corresponding cohomology groups. By Lemma \ref{Kashiwaratermwise}, these can in turn be interpreted as the unit and counit of the adjunction from \cite[Theorem 6.10.(b)]{DcapTwo}, which are isomorphisms by \cite[Theorem 6.10.(c)]{DcapTwo} resp. \cite[Theorem 3.2.1]{Ardakovinduction}. 
	\end{proof}
	
	\begin{prop}
		\label{imcompKashiwara}
		Let $f: X\to Y$ and $g:Y\to Z$ be morphisms of smooth rigid analytic $K$-varieties. Assume that the following is satisfied:
		\begin{enumerate}[(i)]
			\item $f$ is smooth and separated.
			\item $g$ is a closed embedding.
			\item For any affinoid $U\subseteq Y$, the preimage $f^{-1}U$ has countably many connected components.
		\end{enumerate} 
		Let $\M^\bullet\in \mathrm{D}^+_\C(\w{\D}_X)$. Then
		\begin{equation*}
			(gf)_+(\M^\bullet)\cong g_+f_+\M^\bullet
		\end{equation*}
		via the natural morphism.
		
		The same holds for arbitrary $\M^\bullet\in \mathrm{D}_\C(\w{\D}_X)$ provided that $f_*$ has finite cohomological dimension.
	\end{prop}
	\begin{proof}
		We prove the statement for right modules and invoke side-changing.
		
		Let $\M$ be a coadmissible right $\w{\D}_X$-module. In the light of Lemma \ref{ifprojthendirectimcomp}, we wish to show that
		\begin{equation*}
			\mathrm{R}f_*(\M\widetilde{\otimes}^\mathbb{L}_{\w{\D}_X} \w{\D}_{X\to Y})\widetilde{\otimes}^\mathbb{L}_{\w{\D}_Y}\w{\D}_{Y\to Z}\to \mathrm{R}f_*(\M\widetilde{\otimes}^\mathbb{L}_{\w{\D}_X}\w{\D}_{X\to Y}\widetilde{\otimes}^{\mathbb{L}}_{f^{-1}\w{\D}_Y}f^{-1}\w{\D}_{Y\to Z})
		\end{equation*}
		is an isomorphism. Arguing locally on $Y$, we can assume that $Y$ is a smooth affinoid with a local coordinate system, and the above morphism reduces to
		\begin{equation*}
			\mathrm{R}f_*(\N^\bullet)\widetilde{\otimes}_K K\{\underline{\partial}\}\to \mathrm{R}f_*(\N^\bullet\widetilde{\otimes}_K K\{\underline{\partial}\})
		\end{equation*}
		as in Lemma \ref{clsmimcomp}, where $\N^\bullet\cong \M\widetilde{\otimes}^\mathbb{L}_{\w{\D}_X}\w{\D}_{X\to Y}$. Since $f$ is smooth, $\wideparen{\mathcal{D}}_{X\to Y}$ admits a finite resolution by locally free $\wideparen{\mathcal{D}}_X$-modules of finite rank, analogously to the Spencer resolution in subsection 6.3. In particular, $\mathcal{N}^\bullet$ can be represented by a bounded below complex of complete bornological sheaves $\mathcal{F}^\bullet$ such that there exists an affinoid covering $(U_i)$ of $X$ with $\mathcal{F}^j|_{U_i}$ isomorphic to some finite direct sum of copies of $\M|_{U_i}$, for each $i$ and each $j$.
		
		As $Y$ is assumed to be affinoid, $X=f^{-1}Y$ is now separated.
		Let $\B$ be the collection of admissible open affinoid subspaces of $X$ which are contained in some $U_i$. This is a basis of the topology. Moreover, since $\M$ has vanishing higher Cech cohomology on affinoids and $\M(U)$ is a pseudo-nuclear space of countable type for any affinoid $U\in \B$ by Proposition \ref{coadisnuc} and the remark before Lemma \ref{prodofpn}, it follows that $\mathcal{F}^\bullet$ and $\B$ satisfy the conditions from Corollary \ref{projforCcomplexes}.
		
		Thus, 
		\begin{align*}
			\mathrm{R}f_*(\mathcal{N}^\bullet)\widetilde{\otimes}_K K\{\underline{\partial}\}&\cong \mathrm{R}f_*(\mathcal{F}^\bullet)\widetilde{\otimes}_K K\{\underline{\partial}\}\\
			&\cong\mathrm{R}f_*(\mathcal{F}^\bullet\widetilde{\otimes}_K K\{\underline{\partial}\})\\
			&\cong \mathrm{R}f_*(\mathcal{N}^\bullet\widetilde{\otimes}_K K\{\underline{\partial}\}),
		\end{align*}
		as required.
		
		Invoking Lemma \ref{ifprojthendirectimcomp}, we have
		\begin{equation*}
			(gf)^r_+\M\cong g^r_+f^r_+\M.
		\end{equation*}
		An induction argument on cohomological length now proves the same for arbitrary $\M^\bullet\in \mathrm{D}^b_\C(\w{\D}_X^\mathrm{op})$. Since $-\widetilde{\otimes}_{\w{\D}_X} \w{\D}_{X\to Y}$ has finite cohomological dimension and $f_*$ is left exact, we can invoke Lemma \ref{ifprojthendirectimcomp} even for bounded below $\C$-complexes, and for arbitrary $\C$-complexes as soon as $f_*$ has finite cohomological dimension (see e.g. \cite[Tag 07K7]{stacksproj}).
	\end{proof}
	\begin{cor}
		Let $f:X\to Y$, $g:Y\to Z$ be morphisms of smooth rigid analytic $K$-spaces. Assume that $f$ is separated and that for any affinoid $U\subseteq Y$, the preimage $f^{-1}U$ has countably many connected components. Let $\M^\bullet\in \mathrm{D}^+_\C(\w{\D}_X)$. Then
		\begin{equation*}
			(gf)_+\M^\bullet\cong g_+f_+\M^\bullet
		\end{equation*}
		via the natural morphism.
		
		The same holds for $\M^\bullet\in \mathrm{D}_\C(\w{\D}_X)$ provided that $\mathrm{pr}_*$ has finite cohomological dimension for the projection morphism $\mathrm{pr}: X\times Y\to Y$.
	\end{cor}
	\begin{proof}
		We factor $f$ as the compositon $f=\mathrm{pr}\circ i_1$, where $i_1: X\to X\times Y$ is the graph embedding of $f$ (which is a closed embedding by assumption), and $\mathrm{pr}: X\times Y\to Y$ the (smooth) projection morphism.
		
		We can also write $g=g^{\mathrm{sm}}i_2$, where $i_2: Y\to Z'$ is a closed embedding and $g^\mathrm{sm}: Z'\to Z$ is smooth (e.g. by considering the graph embedding, which is locally closed -- in particular, $Z'$ can be chosen to be smooth).
		
		By Lemma \ref{compforsmooth}, we have
		\begin{equation*}
			g_+f_+\M^\bullet\cong g^\mathrm{sm}_+{i_2}_+f_+\M^\bullet
		\end{equation*}	
		and
		\begin{equation*}
			(gf)_+\M^\bullet\cong g^\mathrm{sm}_+(i_2f)_+\M^\bullet.
		\end{equation*}
		It thus suffices to prove that
		\begin{equation*}
			(i_2f)_+\M^\bullet\to {i_2}_+f_+\M^\bullet
		\end{equation*}
		is an isomorphism. We can thus argue locally and assume that $Z'$ and $Y$ are affinoid. 
		
		In this case, $X$ is separated, so that the projection $\mathrm{pr}$ is both smooth and separated. Moreover, $\mathrm{pr}^{-1}U=X\times U$ has countably many components for any affinoid $U\subseteq Y$ by assumption.
		
		By Theorem \ref{KashiwaraCcomplex}, ${i_1}_+$ sends $\mathrm{D}^+_{\C}(\w{\D}_X)$ to $\mathrm{D}^+_{\C}(\w{\D}_{X\times Y})$.
		
		We can thus invoke Lemma \ref{compforsmooth} and Proposition \ref{imcompKashiwara} to obtain
		\begin{align*}
			{i_2}_+f_+\M^\bullet&\cong {i_2}_+\mathrm{pr}_+{i_1}_+\M^\bullet\\
			&\cong (i_2\mathrm{pr})_+{i_1}_+\M^\bullet
		\end{align*}
		for any $\M^\bullet\in \mathrm{D}^+_{\C}(\w{\D}_X)$. If $\mathrm{pr}_*$ has finite cohomological dimension, the same holds for arbitrary $\M^\bullet\in \mathrm{D}_{\C}(\w{\D}_X)$.
		
		Now write $i_2\mathrm{pr}$ again as a composition of a closed embedding $i_3$ and a smooth morphism $h$. By Lemma \ref{clsmimcomp} and Lemma \ref{compforsmooth}, the above is isomorphic to
		\begin{equation*}
			h_+{i_3}_+{i_1}_+(\M^\bullet)\cong h_+(i_3i_1)_+(\M^\bullet)\cong (hi_3i_1)_+(\M^\bullet)=(i_2f)_+(\M^\bullet).
		\end{equation*}
		This finishes the proof.
	\end{proof}
	\subsection{Extraordinary inverse image along smooth morphisms}
	Let $f:X\to Y$ be a smooth morphism of smooth rigid analytic $K$-varieties. In this subsection, we are going to show that $f^!$ preserves $\C$-complexes.
	
	Suppose that $X=\Sp B$, $Y=\Sp A$ are affinoid, with $X$ resp. $Y$ admitting local coordinates $(x_1, \hdots, x_m, \partial_1, \hdots, \partial_m)$ resp. $(y_1, \hdots, y_r, \theta_1, \hdots, \theta_r)$ such that the map $\T_X\to f^*\T_Y$ sends $\partial_i$ to $1\otimes \theta_i$ for $i\leq r$ and to zero for $i>r$. Assume further that $A$ and $B$ admit admissible affine formal models $\A$ and $\B$ respectively, such that $f$ is obtained from a \emph{flat} morphism $\A\to \B$. We choose Lie lattices such that $\D_{X_n}$ is generated by $\pi^n\partial_i$, $i=1, \hdots, m$, and $\D_{Y_n}$ is generated by $\pi^n\theta_i$, $i=1, \hdots, r$.
	
	\begin{lem}
		\label{smoothtransfer}
		The $\w{\D}_X$-module $\w{\D}_{X\to Y}$ is coadmissible.
	\end{lem}
	\begin{proof}
		Note that by assumption, $\w{\D}_{X\to Y}\cong \w{\D}_X/\sum_{i>r} \w{\D}_X\partial_i$ is the cokernel of a morphism $\w{\D}_X^{m-r}\to \w{\D}_X$ of coadmissible $\w{\D}_X$-modules. 
	\end{proof}
	It follows from \cite[Lemma 7.3]{DcapOne}  and Proposition \ref{coadandtor} that 
	\begin{equation*}
		\mathrm{H}^{r-m}(f^!\M)=\w{\D}_{X\to Y}\widetilde{\otimes}_{f^{-1}\w{\D}_Y}f^{-1}\M
	\end{equation*}
	is a coadmissible $\w{\D}_X$-module for any coadmissible $\w{\D}_Y$-module $\M$, and furthermore 
	\begin{equation*}
		\D_{X_n}\widetilde{\otimes}_{\w{\D}_X}\mathrm{H}^{r-m}(f^!\M)\cong f^*(\D_{Y_n}\widetilde{\otimes}_{\w{\D}_Y}\M).
	\end{equation*}
	
	\begin{prop}
		\label{flatbasechangeaffinoid}
		For any coadmissible $\w{\D}_Y$-module $\M$, we have
		\begin{equation*}
			\w{\D}_{X\to Y}\widetilde{\otimes}^\mathbb{L}_{f^{-1}\w{\D}_Y}f^{-1}\M\cong \w{\D}_{X\to Y}\widetilde{\otimes}_{f^{-1}\w{\D}_Y} f^{-1}\M.
		\end{equation*}
	\end{prop}
	\begin{proof}
		Write $U_n=\D_{Y_n}(Y)$, so that 
		\begin{equation*}
			\w{\D}_{X\to Y}(X)\cong B\h{\otimes}_A \varprojlim U_n\cong \varprojlim (B\h{\otimes}_A U_n)
		\end{equation*} 
		by Lemma \ref{reltensorandinv}.
		
		In the light of Proposition \ref{flatbasechange}, it thus suffices to show that $B\h{\otimes}_A U_n$ is (algebraically) flat over $U_n$. For this, it is enough to consider short exact sequences of finitely generated $U_n$-modules, which carry canonical Banach structures. Note that such a sequence is then automatically a strictly exact sequence of Banach $A$-modules, and we can use the flatness of $\B$ over $\A$ to invoke \cite[Remark after Lemma 2.13]{Bode1} and finish the proof.
	\end{proof}

	\begin{thm}
		\label{smoothinvimage}
		Let $f: X\to Y$ be an arbitrary smooth morphism of smooth rigid analytic $K$-varieties. If $\M$ is a coadmissible $\w{\D}_Y$-module, then
		\begin{equation*}
			f^!\M\cong \w{\D}_{X\to Y}\widetilde{\otimes}_{f^{-1}\w{\D}_Y}f^{-1}\M[\dim X-\dim Y].
		\end{equation*}
		Moreover, $f^!$ preserves $\C$-complexes.
	\end{thm}
	\begin{proof}
		Due to the existence of flat formal models (\cite[Theorem 9.4/1]{Boschlectures}), we can reduce to the case discussed above. Now by Proposition \ref{flatbasechangeaffinoid}, 
		\begin{equation*}
			\w{\D}_{X\to Y}\widetilde{\otimes}^\mathbb{L}_{f^{-1}\w{\D}_Y}f^{-1}\M\cong \w{\D}_{X\to Y}\widetilde{\otimes}_{f^{-1}\w{\D}_Y}f^{-1}\M.
		\end{equation*}
		Thus Proposition \ref{exactcohom}.(i) implies that for any $\C$-complex $\M^\bullet$ on $Y$, we have
		\begin{equation*}
			\mathrm{H}^j(f^!\M^\bullet)\cong f^*\mathrm{H}^{j+\dim X-\dim Y}(\M^\bullet).
		\end{equation*}
		Hence $f^!\M^\bullet$ is a $\C$-complex on $X$ thanks to the remarks after Lemma \ref{smoothtransfer}.
	\end{proof}
	
	\subsection{Relative de Rham cohomology}
	We slightly generalize the results concerning the Spencer resolution (subsection 6.3) in order to obtain the analogue of \cite[p. 45]{Hotta}.
	Let $Y$, $Z$ be smooth rigid analytic $K$-spaces, and let $X=Y\times Z$. Consider the natural projection maps $p_1: X=Y\times Z\to Y$, $p_2: X\to Z$. Note that 
	\begin{equation*}
		\w{\D}_{X\to Y}\cong \frac{\w{\D}_X}{\w{\D}_Xp_2^*\T_Z},
	\end{equation*}
	where we view $p_2^*\T_X$ as a subsheaf of $\T_X$ via the natural isomorphism $\T_X\cong p_1^*\T_Y\oplus p_2^*\T_Z$.
	
	Using the same argument as in subsection 6.3, we can find a finite resolution of $\w{\D}_{X\to Y}$ by locally free $\w{\D}_X$-modules of finite rank. Applying side-changing then yields the locally free resolution
	\begin{equation*}
		\hdots \to\Omega_{X/Y}^{\dim Z+k}\widetilde{\otimes}_{\O_X}\w{\D}_X\to \hdots \to \Omega_{X/Y}^{\dim Z}\widetilde{\otimes}_{\O_X}\w{\D}_X\to \w{\D}_{Y\leftarrow X}\to 0.
	\end{equation*}
	In particular, if $\M\in \mathrm{Mod}_{\mathrm{Shv}(X, LH(\h{\B}c_K))}(\w{\D}_X)$, then
	\begin{equation*}
		\w{\D}_{Y\leftarrow X}\widetilde{\otimes}^{\mathbb{L}}_{\w{\D}_X} \M\cong DR_{X/Y}(\M),
	\end{equation*}
	where 
	\begin{equation*}
		DR_{X/Y}(\M)^k=\Omega_{X/Y}^{\dim Z+k}\widetilde{\otimes}_{\O_X} \M
	\end{equation*}
	for $-\dim Z\leq k\leq 0$, and we immediately obtain the following:
	\begin{thm}
		There is a natural isomorphism $p_{1+}\M\cong \mathrm{R}p_{1*}(DR_{X/Y}(\M))$ for any $\M\in \mathrm{Mod}_{\mathrm{Shv}(X, LH(\h{\B}c_K))}(\w{\D}_X)$. 
	\end{thm}
	In particular, if $Y=\Sp K$ is a point, we recover the de Rham cohomology of $X$ via $p_{1+}\O_X$. The relationship between $p_{1+}\O_X$ and the individual de Rham cohomology groups is however a bit subtle: $p_{1+}\O_X$ is generally not a strict complex, so $\mathrm{H}^i(p_{1+}\O_X)$ is only an object of $LH(\h{\B}c_K)$. It is tempting then to apply the cokernel functor $C: LH(\h{\B}c_K)\to \h{\B}c_K$, but this takes the completed cokernel in $\h{\B}c_K$. For example, if $X=\Sp K\langle x\rangle$, then we have
	\begin{equation*}
		p_{1+}\O_X=\mathrm{R}p_{1*}\Omega^\bullet_{X/K}[1]\cong (0\to \O(X)\to \Omega_X(X)\to 0)
	\end{equation*}
	with $\Omega_X(X)$ in degree zero as expected, but
	\begin{equation*}
		C\mathrm{H}^0(p_{1+}\O_X)=0,
	\end{equation*}
	as the morphism has dense image.
	
	A more useful formulation is thus the following.
	\begin{prop}
		Let $X$ be a smooth rigid analytic $K$-variety of dimension $d$ with structure map $p: X\to \Sp K$. Let $J: LH(\h{\B}c_K)\to LH(\B c_K)$ be the functor induced by the exact embedding $\h{\B}c_K\to \B c_K$ (see \cite[Proposition 5.14]{ProsmansSchneiders}). Then
		\begin{equation*}
			\mathrm{H}^i_{\mathrm{dR}}(X)\cong CJ\mathrm{H}^{i-d}(p_+\O_X)
		\end{equation*}
		as $K$-vector spaces.
	\end{prop}
	It looks like this is endowing $\mathrm{H}^i_\mathrm{dR}(X)$ with a bornological structure, but this is rarely interesting: if the space is finite-dimensional, it is the usual Banach structure, and in pathological cases like $\mathrm{H}^1_\mathrm{dR}(\Sp K\langle x\rangle)$ the bornology is the trivial one, i.e. every subset is bounded.
	
	Crucially the discussion above also shows that if $X$ is a smooth rigid analytic $K$-space with infinite-dimensional de Rham cohomology (e.g. the closed unit disk), and $p: X\to \Sp K$, then $p_+\O_X$ need not be an object of $\mathrm{D}_{\C}(\w{\D}_{\Sp K})=\mathrm{D}^b_{\mathrm{findim}}(\h{\B} c_K)$. This is in contrast to the complex theory, where stability of holonomicity ensures that the direct image of the trivial connection has $\D$-coherent cohomology for \emph{any} morphism.
	\subsection{Direct image along projective morphisms} 
	Suppose that $K$ is discretely valued, so that we have the results from \cite{Bodeproper} at our disposal.
	
	Let $Y$ be a smooth affinoid and consider the projection morphism $f: X=\mathbb{P}^m\times Y\to Y$.
	
	Note that in this case, $\T_X$ is globally generated, and we can define $\D_{X_n}$ globally on a suitable site $X_n$ (see \cite[Proposition 3.21]{Bodeproper}). Choosing lattices as usual, we have at our disposal the sheaves $\w{\D}_{X\to Y}=f^*\w{\D}_Y$ and $f^*\D_{Y_n}$ (the latter defined on a suitable site). Note that in this special setting, $f^*\T_Y$ is a Lie algebroid on $X$ (a direct summand of $\T_X$), and $\w{\D}_{X\to Y}\cong \w{\mathscr{U}(f^*\T_Y)}$, $f^*\D_{Y_n}\cong \mathcal{U}_n(f^*\T_Y)$.
	
	We write $\mathscr{U}=\w{\mathscr{U}(f^*\T_Y)}$ and $\mathcal{U}_n:=\mathscr{U}_n(f^*\T_Y)$.
	
	\begin{prop}
		Let $\M^\bullet\in \mathrm{D}_\C(\mathscr{U}^\mathrm{op})$. Then $\mathrm{R}f_*(\M^\bullet)\in \mathrm{D}_\C(\w{\D}_Y^\mathrm{op})$.
	\end{prop}
	\begin{proof}
		Write $\M_n:=\M^\bullet\widetilde{\otimes}^\mathbb{L}_\mathscr{U}\mathcal{U}_n$. It follows from \cite[Lemma 4.1, Proposition 5.7]{Bodeproper} and an easy induction argument on the cohomological length that $\mathrm{R}f_*(\M_n)\in \mathrm{D}^b_{\mathrm{coh}}(\D_{Y_n})$, and
		\begin{equation*}
			\mathrm{H}^i(\mathrm{R}f_*(\M_n))\h{\otimes}_{\D_{Y_n}}\D_{Y_{n-1}}\cong \mathrm{H}^i(\mathrm{R}f_*(\M_{n-1}))
		\end{equation*} 
		for all $i$ and all $n$. In particular, $\varprojlim \mathrm{H}^i(\mathrm{R}f_*(\M_n))$ is coadmissible, and the inverse system $(\mathrm{H}^i(\mathrm{R}f_*(\M_n)))_n$ is $\varprojlim$-acyclic.
		
		As $f_*$ commutes with products, we thus get 
		\begin{equation*}
			\mathrm{H}^i(\mathrm{R}f_*(\M^\bullet))\cong \mathrm{H}^i(\mathrm{holim} \mathrm{R}f_*(\M_n))\cong \varprojlim \mathrm{H}^i(\mathrm{R}f_*(\M_n))
		\end{equation*} 
		by the $\mathscr{U}$-analogue of Corollary \ref{Ccomplexholim}.
		Hence $\mathrm{R}f_*(\M^\bullet)$ is a $\C$-complex by Lemma \ref{Mndefinescoad}.
	\end{proof}
	
	\begin{thm}
		\label{elementaryKiehl}
		Let $\M^\bullet\in \mathrm{D}_\C(\w{\D}_X^\mathrm{op})$. Then $f_+^r(\M^\bullet)\in \mathrm{D}_\C(\w{\D}_Y^\mathrm{op})$.
	\end{thm}
	
	\begin{proof}
		By the above, it remains to show that $\N^\bullet:=\M^\bullet\widetilde{\otimes}^\mathbb{L}_{\w{\D}_X}\mathscr{U}$ is a right $\C$-complex over $\mathscr{U}$.
		
		First suppose that $\M^\bullet=\M\in \C$. Write $\M_n\cong \M\widetilde{\otimes}_{\w{\D}_X} \D_{X_n}$, so that $\M\cong \varprojlim \M_n$. Now $\mathcal{U}_n$ is a coherent $\D_{X_n}$-module, which by the discussion in subsection 6.3 admits a finite resolution by locally free $\D_{X_n}$-modules of finite rank. Thus
		\begin{equation*}
			\M\widetilde{\otimes}^\mathbb{L}_{\w{\D}_X} \mathcal{U}_n\cong \M_n\widetilde{\otimes}^\mathbb{L}_{\D_{X_n}} \mathcal{U}_n\in \mathrm{D}^b(\mathcal{U}_n),
		\end{equation*}
		and since $\M_n$ can locally be represented by a (bounded above) complex of free $\D_{X_n}$-modules of finite rank, we have
		\begin{equation*}
			\M_n\widetilde{\otimes}^\mathbb{L}_{\D_{X_n}}\mathcal{U}_n\in \mathrm{D}^b_{\mathrm{coh}}(\mathcal{U}_n).
		\end{equation*}
		Moreover,
		\begin{equation*}
			\mathrm{H}^j(\M\widetilde{\otimes}^\mathbb{L}_{\w{\D}_X} \mathcal{U}_n)\cong \mathrm{H}^j(\M\widetilde{\otimes}^\mathbb{L}_{\w{\D}_X} \mathcal{U}_{n+1})\widetilde{\otimes}_{\mathcal{U}_{n+1}} \mathcal{U}_n
		\end{equation*}
		for all $j$ and all $n$, due to the flatness properties of $\mathcal{U}_{n+1}\to \mathcal{U}_n$.
		
		In particular, $\varprojlim \mathrm{H}^j(\M\widetilde{\otimes}^\mathbb{L}_{\w{\D}_X} \mathcal{U}_n)$ is a coadmissible $\mathscr{U}$-module.
		
		It thus suffices now to pass to the limit with a Mittag-Leffler argument, as follows.
		
		By the discussion in subsection 9.3, 
		\begin{equation*}
			\M\widetilde{\otimes}^\mathbb{L}_{\w{\D}_X}\mathcal{U}_n\cong \M_n\widetilde{\otimes}^\mathbb{L}_{\D_{X_n}} \mathcal{U}_n
		\end{equation*}
		can be represented by a finite complex with terms
		\begin{equation*}
			\F_n^{-j}j=\M_n\widetilde{\otimes}_{\O_X} \wedge^{j} \T_{X/Y}
		\end{equation*}
		for $0\leq j\leq m$.
		
		These are sheaves whose sections on each affinoid subdomain $U$ in $X_n$ form a Banach space such that $(\F_m^j(U))_{m\geq n}$ is a pre-nuclear system (since $\M_m(U)$ is a pre-nuclear system).
		
		Since $\M\widetilde{\otimes}^\mathbb{L}\mathcal{U}_n\in \mathrm{D}^b_\mathrm{coh}(\mathcal{U}_n)$, the complex is strict, and the corresponding sections of the cohomology groups also form pre-nuclear systems due to coadmissibility.
		
		Applying Theorem \ref{MLBan}, the complex formed by $\varprojlim_n \F_n ^{-j}$ is strict with cohomology groups
		\begin{equation*}
			\varprojlim \mathrm{H}^{-j}(\M\widetilde{\otimes}^\mathbb{L}\mathcal{U}_n).
		\end{equation*}
		But now by Lemma \ref{reltensorandinv}, we have 
		\begin{equation*}
			\varprojlim \F_n^{-j}\cong \M\widetilde{\otimes}_{\O_X}\wedge^{j}\T_{X/Y},
		\end{equation*}
		which represents the complex $\M\widetilde{\otimes}^\mathbb{L}\mathscr{U}$.
		
		Hence $\M\widetilde{\otimes}^\mathbb{L} \mathcal{U}$ is a $\C$-complex over $\mathscr{U}$.
		
		If $\M^\bullet$ is a bounded $\C$-complex, an induction argument implies now that $\M^\bullet\widetilde{\otimes}^\mathbb{L}\mathscr{U}$ is a $\C$-complex over $\mathscr{U}$, and the result follows in full generality from $-\widetilde{\otimes}^\mathbb{L}\mathscr{U}$ having finite cohomological dimension.
	\end{proof}

	We call a morphism $f: X\to Y$ \textbf{projective} if there exists a closed embedding $i: X\to \mathbb{P}^d\times Y$ such that $f$ factors as $i$ followed by the natural projection.
	\begin{cor}
		\label{KiehlCcomplex}
		Let $f:X\to Y$ be a projective morphism between smooth rigid analytic $K$-varieties. Then $f_+$ preserves $\C$-complexes.
	\end{cor}
	\begin{proof}
		By Lemma \ref{compforsmooth} and arguing locally, it suffices to consider separately the cases of a closed embedding (Lemma \ref{Kashiwaratermwise}) and the projection to an affinoid (Theorem \ref{elementaryKiehl}). Side-changing implies that the result holds both for left and right $\C$-complexes.
	\end{proof}

	\subsection{Duality}
	As in the previous subsection, suppose that $K$ is discretely valued.
	Recall that in \cite[Proposition 5.3]{DcapThree}, the authors associate to any coadmissible $\w{\D}_X$-module $\M$ right coadmissible modules $\mathcal{E}xt^i_{\w{\D}_X}(\M, \w{\D}_X)$, whose sections over affinoids are given by the corresponding algebraic Ext groups. To avoid confusion with our functor $\mathrm{R}\mathcal{H}om$ and its cohomology groups, we denote the sheaves from \cite{DcapThree} by $E^i(\M, \w{\D}_X)$. We use the notation $E^i_n$ for the analogous construction for coherent $\D_n$-modules.
	
	In this subsection, we prove that the duality functor $\mathbb{D}$ preserves $\C$-complexes and recovers the sheaves $E^i(\M, \w{\D}_X)$. 
	\begin{thm}
		\label{dualCcomplex}
		Let $X$ be a smooth rigid analytic $K$-space. Then $\mathbb{D}_X$ sends $\mathrm{D}_{\C}(\w{\D}_X)$ to $\mathrm{D}_{\C}(\w{\D}_X)^{\mathrm{op}}$. Moreover, there is a natural isomorphism $\mathbb{D}^2\cong \mathrm{id}$.
	\end{thm}
	\begin{proof}
		Suppose that $X$ is affinoid with free tangent sheaf. By \cite[Theorem A.(i)]{DcapThree}, we can assume that $\D_n(X)$ has finite self-injective dimension. Note that this result is currently only available if $K$ is discretely valued. 
		
		If $\M$ is a coherent $\D_n$-module, then any $\D_n$-linear map $\M\to \D_n$ is bounded, so $\mathcal{H}om_{\D_n}(\M, \D_n)$ is the coherent right $\D_n$-module $E_n^0(\M, \D_n)$. Considering a short exact sequence
		\begin{equation*}
			0\to \N\to \D_n^r\to \M\to 0
		\end{equation*}
		and the corresponding long exact sequence of cohomology, we find inductively that $\mathrm{H}^j(\mathrm{R}\mathcal{H}om_{\D_n}(\M, \D_n))$ is the coherent right $\D_n$-module $E_n^j(\M, \D_n)$. In particular, $\mathrm{R}\mathcal{H}om_{\D_n}(-, \D_n)\widetilde{\otimes}_{\O_X}\Omega_X^{\otimes -1}$ sends $\mathrm{D}^b_{\mathrm{coh}}(\D_n)$ to itself, where the boundedness follows from the finite self-injective dimension of $\D_n(X)$.
		
		Note that this implies that if $\M\cong \varprojlim \M_n$ is a coadmisible $\w{\D}_X$-module, then 
		\begin{equation*}
			E^i(\M, \w{\D}_X)\cong \varprojlim \mathrm{H}^i(\mathrm{R}\mathcal{H}om_{\D_n}(\M_n, \D_n))\cong \varprojlim \mathrm{H}^i(\mathrm{R}\mathcal{H}om_{\w{\D}_X}(\M, \D_n)).
		\end{equation*}
		Now let $\M^\bullet$ be a $\C$-complex. Note that by the above, the module
		\begin{equation*}
			\varprojlim \mathrm{H}^j(\mathrm{R}\mathcal{H}om_{\w{\D}_X}(\M^\bullet, \D_n))\cong \varprojlim \mathrm{H}^j(\mathrm{R}\mathcal{H}om_{\D_n}(\D_n\widetilde{\otimes}^\mathbb{L}\M^\bullet, \D_n))
		\end{equation*}
		is a coadmissible right module.
		
		Let $\F_n=\O_X\widetilde{\otimes}_A\D_n(X)$, a complete bornological $\w{\D}$-module such that 
		\begin{equation*}
			\F_n|_{X_n}\cong\D_n.
		\end{equation*}
		As before, 
		\begin{equation*}
			\varprojlim \mathrm{H}^j(\mathrm{R}\mathcal{H}om_{\w{\D}_X}(\M^\bullet, \F_n))\cong \varprojlim \mathrm{H}^j(\mathrm{R}\mathcal{H}om_{\w{\D}_X}(\M^\bullet, \D_n))
		\end{equation*}
		as complete bornological sheaves, as the sections naturally agree on each affinoid subdomain. We have $\w{\D}_X(U)\cong \mathrm{R}\varprojlim \D_n(U)$ for each sufficiently small affinoid subdomain $U$ by Theorem \ref{MLBan}, and hence $\w{\D}_X\cong \mathrm{R}\varprojlim \F_n$ by exactness of the corresponding Roos complex. 
		
		Thus $\mathrm{R}\mathcal{H}om_{\w{\D}_X}(\M^\bullet, \w{\D}_X)$ is the homotopy limit of $\mathrm{R}\mathcal{H}om_{\w{\D}_X}(\M^\bullet, \F_n))$, and by Corollary \ref{MLcoad}, we thus obtain
		\begin{equation*}
			\mathrm{H}^j(\mathrm{R}\mathcal{H}om_{\w{\D}_X}(\M^\bullet, \w{\D}_X))\cong \varprojlim \mathrm{H}^j(\mathrm{R}\mathcal{H}om_{\D_n}(\D_n\widetilde{\otimes}^\mathbb{L}\M^\bullet, \D_n)).
		\end{equation*}  
		Applying side-changing, this shows that $\mathbb{D}$ preserves $\C$-complexes.
		
		Lastly, the identity $\mathbb{D}^2\cong \mathrm{id}$ follows immediately from the corresponding statement for $\mathrm{D}^b_{\mathrm{coh}}(\D_n)$.
	\end{proof}
	
	\begin{cor}
		Let $X$ be a smooth rigid analytic $K$-space of dimension $d$ and let $\M$ be a coadmissible $\w{\D}_X$-module. Then $\Omega_X\widetilde{\otimes}_{\mathcal{O}_X}\mathrm{H}^j(\mathbb{D}\M)$ is isomorphic to $E^{j+d}(\M, \w{\D}_X)$.
	\end{cor}

\end{document}